\documentclass{amsart}
\usepackage{amssymb,amsmath}
\usepackage{color}
\usepackage{mathtools}
\usepackage{bbm}
\usepackage{ulem}

\date{August 22, 2019}

\numberwithin{equation}{section}

\newcommand\cZt{\mathcal{X}}
\newcommand{\Cut}{\operatorname{Cut}}   
\newcommand{\ribbongraph}{G}
\newcommand{\poles}{p}
\newcommand{\Graph}{\Gamma}
\newcommand{\Separating}{\Delta}
\newcommand{\Petal}{\Gamma}
\newcommand{\cyl}{\mathit{cyl}}
\newcommand{\separatingQ}{\chi_{\Graph}}
\newcommand{\carea}{c_{\mathit{area}}}
\newcommand{\zeroes}{\ell} 
\newcommand{\dprinc}{d} 
\newcommand{\dVolMV}{d\!\Vol}

\newcommand{\elow}{\varepsilon_{\mathit{below}}(g)}
\newcommand{\eup}{\varepsilon_{\mathit{above}}(g)}

\newcommand{\VSumH}{S^H}
\newcommand{\VSumZ}{S^\zeta}

\newlength{\halfbls}

\newcommand{\cD}{\mathcal{D}}
\newcommand{\cI}{\mathcal{I}}
\newcommand{\cG}{\mathcal{G}}

\newcommand{\cL}{\mathcal{L}}
\newcommand{\cM}{\mathcal{M}}
\newcommand{\cML}{\mathcal{ML}}
\newcommand{\cN}{\mathcal{N}}
\newcommand{\cP}{\mathcal{P}}
\newcommand{\cQ}{\mathcal{Q}}
\newcommand{\cR}{\mathcal{R}}
\newcommand{\cS}{\mathcal{S}}
\newcommand{\cT}{\mathcal{T}}
\newcommand{\cY}{\mathcal{Y}}
\newcommand{\cZ}{\mathcal{Z}}
\newcommand{\cST}{\mathcal{ST}}
\newcommand{\cSTgn}{\mathcal{ST}{\hspace*{-3pt}}_{g,n}}

\newcommand{\E}{{\mathbb E}}
\newcommand{\N}{{\mathbb N}}
\newcommand{\Z}{{\mathbb Z}}

\newcommand{\Q}{{\mathbb Q}}
\newcommand{\R}{{\mathbb R}}
\newcommand{\C}{{\mathbb C}}

\newcommand{\Area}{\operatorname{Area}}
\newcommand{\Vol}{\operatorname{Vol}}
\newcommand{\Aut}{\operatorname{Aut}}
\newcommand{\CP}{{\mathbb C}\!\operatorname{P}^1}
\newcommand{\card}{\operatorname{card}}
\newcommand{\Mod}{\operatorname{Mod}}
\newcommand{\Stab}{\operatorname{Stab}}
\newcommand{\Sym}{\operatorname{Sym}}

\newcommand{\SLZ}{\operatorname{SL}(2,{\mathbb Z})}

\newcommand{\GL}{\operatorname{GL}(2,\mathbb{R})}

\renewcommand{\epsilon}{\varepsilon}

\newtheorem{Theorem}{Theorem}[section]
\newtheorem{CondTheorem}[Theorem]{Conditional Theorem}
\newtheorem{Proposition}[Theorem]{Proposition}
\newtheorem{Lemma}[Theorem]{Lemma}
\newtheorem{Corollary}[Theorem]{Corollary}
\newtheorem{Conjecture}[Theorem]{Conjecture}
\newtheorem{Guess}[Theorem]{Guess}

\newtheorem*{NNTheorem}{Theorem}

\newtheorem*{NNLemma}{Lemma}

\theoremstyle{definition}
\newtheorem{Definition}[Theorem]{Definition}
\newtheorem{Convention}[Theorem]{Convention}

\theoremstyle{remark}
\newtheorem{Example}[Theorem]{Example}
\newtheorem{Remark}[Theorem]{Remark}

\newtheorem*{NNRemark}{Remark}

\title[Masur--Veech volumes, simple closed geodesics, intersection numbers]
{Masur-Veech volumes, frequencies of simple closed geodesics and intersection numbers of moduli spaces of curves
}

\author[V.~Delecroix]{Vincent Delecroix}
\address{
LaBRI,
Domaine universitaire,
351 cours de la Lib\'eration, 33405 Talence, FRANCE
}
\email{20100.delecroix@gmail.com}

\author[\'E.~Goujard]{\'Elise Goujard}
\thanks{Research of the second author was partially supported by
PEPS}
\address{
Institut de Math\'ematiques de Bordeaux,
Universit\'e de Bordeaux,
351 cours de la Lib\'eration, 33405 Talence, FRANCE
}
\email{elise.goujard@gmail.com}

\author[P.~G.~Zograf]{Peter~Zograf}
\thanks{The research of Section~\ref{s:2:correlators} was supported by the Russian Science Foundation grant 19-71-30002.
}
\address{
St.~Petersburg Department, Steklov Math. Institute, Fontanka 27,
St. Petersburg 191023, and Chebyshev Laboratory,
St. Petersburg State University, 14th
Line V.O. 29B, St.Petersburg 199178 Russia}
\email{zograf@pdmi.ras.ru}

\author[A.~Zorich]{Anton Zorich}
\thanks{This material is based upon work supported by the NSF Grant DMS-1440140 while
part of the authors were in residence at the MSRI during
the Fall 2019 semester}
\address{
Center for Advanced Studies, Skoltech;
Institut de Math\'ematiques de Jussieu --
Paris Rive Gauche,
Case 7012,
8 Place Aur\'elie Nemours,
75205 PARIS Cedex 13, France}
\email{anton.zorich@gmail.com}

\begin{document}

\begin{abstract}
We express the Masur--Veech volume $\Vol\cQ_{g,n}$ and
the area Siegel--Veech constant $\carea(\cQ_{g,n})$ of
the moduli space $\cQ_{g,n}$ of meromorphic quadratic
differential with $n$ simple poles and no other poles as
polynomials in the intersection numbers of $\psi$-classes
supported on the boundary cycles of the
Deligne--Mumford compactification $\overline{\cM}_{g,n}$.
The formulae obtained in this article are derived from
lattice point count involving the Kontsevich volume
polynomials $N_{g,n}(b_1, \ldots, b_n)$ that also appear
in Mirzakhani's recursion for the Weil--Petersson volumes
of the moduli space $\cM_{g,n}(b_1,\dots,b_n)$ of
bordered hyperbolic Riemann surfaces.

A similar formula for the Masur--Veech volume
$\Vol\cQ_{g}$ (though without explicit evaluation) was
obtained earlier by M.~Mirzakhani through completely
different approach. We prove further result: up to a
normalization factor depending only on $g$ and $n$ (which
we compute explicitly), the density of the orbit
$\Mod_{g,n}\cdot\gamma$ of any simple closed multicurve
$\gamma$ inside the ambient set $\cML_{g,n}(\Z)$ of
integral measured laminations computed by Mirzakhani,
coincides with the density of square-tiled surfaces
having horizontal cylinder decomposition associated to
$\gamma$ among all square-tiled surfaces in $\cQ_{g,n}$.

We study the resulting densities (or, equivalently,
volume contributions) in more detail in the special case
when $n=0$. In particular, we compute explicitly the
asymptotic frequencies of separating and non-separating
simple closed geodesics on a closed hyperbolic surface of
genus $g$ for all small genera $g$ and we show that in
large genera the separating closed geodesics are
$\sqrt{\frac{2}{3\pi g}}\cdot\frac{1}{4^g}$
times less frequent.

We conclude with detailed conjectural description of
combinatorial geometry of a random simple closed
multicurve on a surface of large genus and of a random
square-tiled surface of large genus. This description is
conditional to the conjectural asymptotic formula for the
Masur--Veech volume $\Vol\cQ_g$ in large genera and to the
conjectural uniform asymptotic formula for certain sums
of intersection numbers of $\psi$-classes in large genera.
\end{abstract}

\maketitle

\tableofcontents

\newpage

\section{Introduction and statements of main theorems}
\label{s:intro}

\subsection{Masur--Veech volume of the moduli space of
quadratic differentials}
\label{ss:MV:volume}
Consider the moduli space $\cM_{g,n}$ of complex curves of
genus $g$ with $n$ distinct labeled marked points. The
total space $\cQ_{g,n}$ of the cotangent bundle over
$\cM_{g,n}$ can be identified with the moduli space of
pairs $(C,q)$, where $C\in\cM_{g,n}$ is a smooth complex
curve and $q$ is a meromorphic quadratic differential on
$C$ with simple poles at the marked points and no other
poles. (In the case $n=0$ the quadratic differential $q$ is
holomorphic.) Thus, the \textit{moduli space of quadratic
differentials} $\cQ_{g,n}$ is endowed with the canonical
symplectic structure. The induced volume element $\dVolMV$
on $\cQ_{g,n}$ is called the \textit{Masur--Veech volume
element}. (In Section~\ref{ss:background:strata}
we provide alternative more common definition
of the Masur--Veech volume element and explain why
two definitions are equivalent.)

A meromorphic quadratic differential $q$ defines a flat
metric $|q|$ on the complex curve $C$. The resulting metric
has conical singularities at zeroes and simple poles of
$q$. When all poles of $q$ (if any) are simple, and when
$C$ is a smooth compact curve, the resulting total flat
area
$$
\Area(C,q)=\int_C |q|
$$
is finite; it is strictly positive as soon as $q$ is not
identically zero. Consider the following subsets in $\cQ_{g,n}$:
\begin{align}
\label{eq:ball:of:radius:1:2:in:Q:g:n}
\cQ^{\Area\le\frac{1}{2}}_{g,n}
&=
\left\{(C,q)\in\cQ_{g,n}\,|\, \Area(C,q) \le \tfrac{1}{2}\right\}\,,
\\
\label{eq:sphere:of:radius:1:2:in:Q:g:n}
\cQ^{\Area=\frac{1}{2}}_{g,n}
&=
\left\{(C,q)\in\cQ_{g,n}\,|\, \Area(C,q) = \tfrac{1}{2}\right\}\,.
\end{align}
These subsets might be seen as analogs of a ``ball
(respectively a sphere) of radius $\tfrac{1}{2}$'' in
$\cQ_{g,n}$. However, this analogy is not quite adequate
since none of these two subsets is compact. By the
independent results of H.~Masur~\cite{Masur:82} and
W.~Veech~\cite{Veech:Gauss:measures}, the volume $\Vol
\cQ^{\Area\le\frac{1}{2}}_{g,n}$ is finite.

The Masur--Veech volume element $\dVolMV$ in $\cQ_{g,n}$
induces the canonical volume element
$\dVolMV_1$ on any noncritical
level hypersurface of any real-valued function $f$. In
particular, it induces the canonical volume element
$\dVolMV_1$ on the level hypersurface
$\cQ^{\Area=\frac{1}{2}}_{g,n}$
satisfying
\begin{equation}
\label{eq:def:Vol:Q:g:n}
\Vol_1 \cQ^{\Area=\frac{1}{2}}_{g,n}
=
(12g-12+4n)\cdot\Vol \cQ^{\Area\le\frac{1}{2}}_{g,n}\,.
\end{equation}
Here $12g-12+4n=\dim_{\R}\cQ_{g,n}=2\dim_{\C}\cQ_{g,n}$.
The above formula can be considered as the definition of
the expression on the left hand side.

\begin{Convention}
\label{conv:MV:volume:notation}
The Masur--Veech volume of the total space $\cQ_{g,n}$ is,
clearly, infinite. Following the common convention,
speaking of the \textit{Masur--Veech volume of the moduli
space} $\cQ_{g,n}$ we always mean the
quantity~\eqref{eq:def:Vol:Q:g:n} and we use the
abbreviated notation $\Vol \cQ_{g,n}$ for this quantity.
\end{Convention}

\begin{Remark}
The choice of ``radius $\tfrac{1}{2}$'' instead of ``radius
$1$'' in~\eqref{eq:def:Vol:Q:g:n} might seem unexpected. We
explain the reasons for this choice in
Section~\ref{ss:background:strata} where we discuss the
Masur--Veech volume element in more details. Further
details on natural normalizations of the Masur--Veech
volume can be found in appendix~A
in~\cite{DGZZ:meanders:and:equidistribution}.
\end{Remark}

\subsection{Square-tiled surfaces, simple
closed multicurves and stable graphs}
\label{ss:Square:tiled:surfaces:and:associated:multicurves}

We have already mentioned that a meromorphic quadratic
differential $q$ defines a flat metric with conical
singularities on the complex curve $C$. One can construct
this kind of flat surfaces by assembling together identical
\textit{polarized oriented} flat squares. Namely, we assume that we
know which pair of opposite sides of each square is
horizontal; the remaining pair of sides is vertical. Gluing
the squares we identify sides to sides respecting the
polarization and the orientation.

\begin{figure}[htb]
   %
   %
\includegraphics{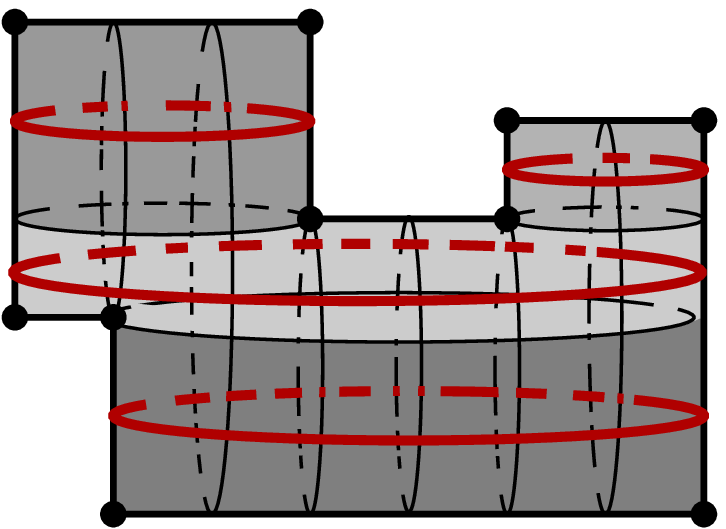}
\includegraphics{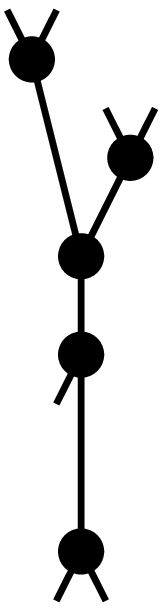}

\begin{picture}(0,0)(175,10)
\put(2,-13){$2\gamma_1$}
\put(113,-23){$\gamma_2$}
\put(1.5,-46){$\phantom{2}\gamma_3$}
\put(23,-77){$2\gamma_4$}
\end{picture}

\includegraphics{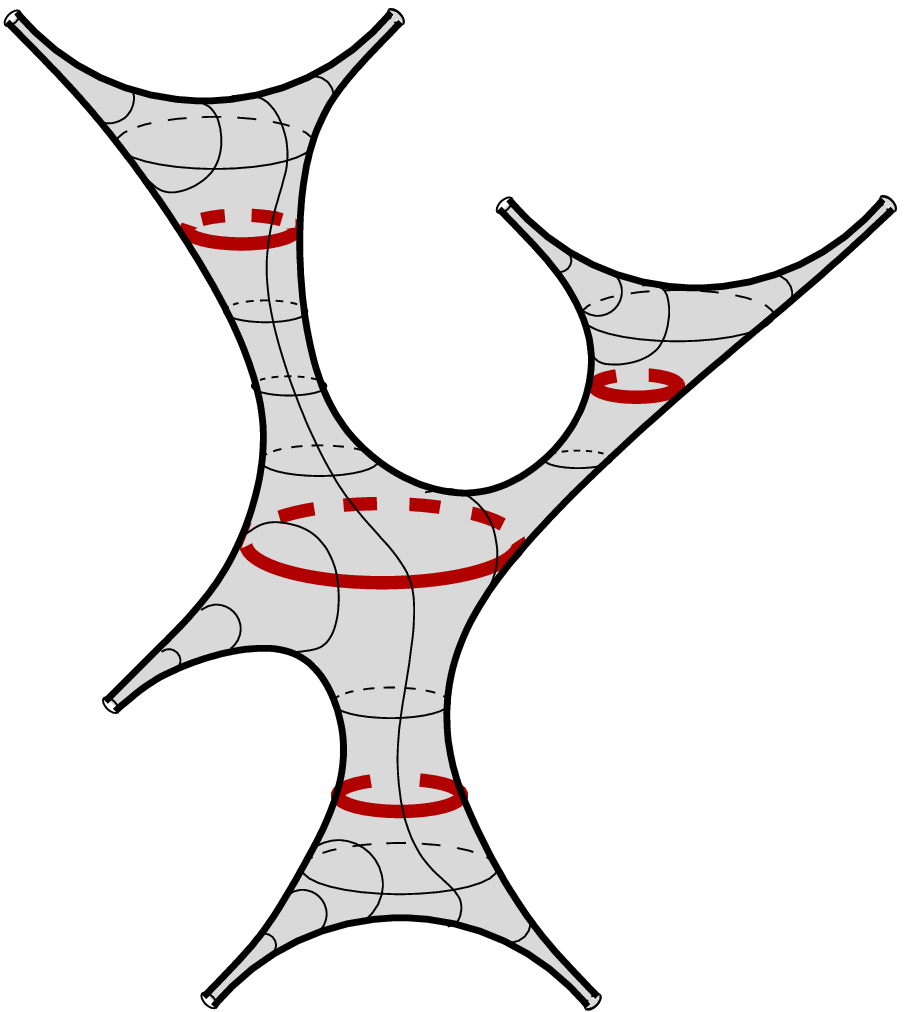}

\begin{picture}(0,0)(-26,-6)
\put(4,-20){$2\gamma_1$}
\put(53.5,-36.5){$\gamma_2$}
\put(15,-54){$\gamma_3$}
\put(20,-82){$2\gamma_4$}
\end{picture}
\vspace{95pt}
\caption{
\label{fig:square:tiled:surface:and:associated:multicurve}
Square-tiled surface in $\cQ_{0,7}$,
and associated multicurve and
stable graph
}
\end{figure}

Suppose that the resulting closed \textit{square-tiled
surface} has genus $g$ and $n$ conical singularities with
angle $\pi$, i.e. $n$ vertices adjacent to only two
squares. For example, the square-tiled surfaces in
Figure~\ref{fig:square:tiled:surface:and:associated:multicurve}
has genus $g=0$ and $n=7$ conical singularities with angle
$\pi$. Consider the complex coordinate $z$ in each square
and a quadratic differential $(dz)^2$. It is easy to check
that the resulting square-tiled surface inherits the
complex structure and globally defined meromorphic
quadratic differential $q$ having simple poles at $n$
conical singularities with angle $\pi$ and no other poles.
Thus, any square-tiled surface of genus $g$ having $n$
conical singularities with angle $\pi$ canonically defines
a point $(C,q)\in\cQ_{g,n}$. Fixing the size of the square
once and forever and considering all resulting square-tiled
surfaces in $\cQ_{g,n}$ we get a discrete subset $\cSTgn$
in $\cQ_{g,n}$.

Define $\cSTgn(N)\subset\cSTgn$ to be the subset of
square-tiled surfaces in $\cQ_{g,n}$
tiled with at most $N$ identical squares.
We shall see in Section~\ref{ss:background:strata} that
square-tiled surfaces form a lattice in period coordinates
of $\cQ_{g,n}$, which justifies the following alternative
definition of the Masur--Veech volume\footnote{See
Section~\ref{ss:background:strata}
for justification of this normalization.}:
\begin{equation}
\label{eq:Vol:sq:tiled}
\Vol\cQ_{g,n}
= 2(6g-6+2n)\cdot
\lim_{N\to+\infty}
\frac{\card(\cSTgn(2N))}{N^{d}}\,,
\end{equation}
where $d=6g-6+2n=\dim_{\C}\cQ_{g,n}$.
In this formula we assume that all conical singularities
of square-tiled surfaces are labeled.
\smallskip

\noindent\textbf{Multicurve associated to a cylinder decomposition.}
Any square-tiled surface admits a decomposition into
maximal horizontal cylinders filled with isometric closed
regular flat geodesics. Every such maximal horizontal
cylinder has at least one conical singularity on each of
the two boundary components. The square-tiled surface in
Figure~\ref{fig:square:tiled:surface:and:associated:multicurve}
has four maximal horizontal cylinders which are represented
in the picture by different shades. For every maximal
horizontal cylinder choose the corresponding waist curve
$\gamma_i$.

By construction each resulting simple closed curve
$\gamma_i$ is nonperiferal (i.e. it does not bound a
topological disc without punctures or with a single
puncture) and different $\gamma_i, \gamma_j$ are not freely
homotopic on the underlying $n$-punctured topological
surface. In other words, pinching simultaneously all waist
curves $\gamma_i$ we get a legal stable curve in the
Deligne--Mumford compactification $\overline{\cM}_{g,n}$.

We encode the number of circular horizontal bands of
squares contained in the corresponding maximal horizontal
cylinder by the integer weight $H_i$ associated to the
curve $\gamma_i$. The above observation implies that the
resulting formal linear combination $\gamma=\sum
H_i\gamma_i$ is a simple closed integral multicurve in the
space $\cML_{g,n}(\Z)$ of measured laminations. For
example, the simple closed multicurve associated to the
square-tiled surface as in
Figure~\ref{fig:square:tiled:surface:and:associated:multicurve}
has the form $2\gamma_1+\gamma_2+\gamma_3+2\gamma_4$.

Given a simple closed integral multicurve $\gamma$ in
$\cML_{g,n}(\Z)$ consider the subset
$\cSTgn(\gamma)\subset\cSTgn$ of those square-tiled
surfaces, for which the associated horizontal multicurve is
in the same $\Mod_{g,n}$-orbit as $\gamma$ (i.e. it is
homeomorphic to $\gamma$ by a homeomorphism sending $n$
marked points to $n$ marked points and preserving their
labeling). Denote by $\Vol(\gamma)$ the contribution to
$\Vol\cQ_{g,n}$ of square-tiled surfaces from the subset
$\cSTgn(\gamma)\subset\cSTgn$:
$$
\Vol(\gamma)
= 2(6g-6+2n)\cdot
\lim_{N\to+\infty}
\frac{\card(\cSTgn(2N)\cap\cSTgn(\gamma))}{N^{d}}\,.
$$
The results in~\cite{DGZZ:meanders:and:equidistribution}
imply that for any $\gamma$ in $\cML_{g,n}(\Z)$ the above
limit exists, is strictly positive, and that
\begin{equation}
\label{eq:Vol:Q:as:sum:of:Vol:gamma}
\Vol\cQ_{g,n}
=\sum_{[\gamma]\in\mathcal{O}} \Vol(\gamma)\,,
\end{equation}
where similarly to~\eqref{eq:b:g:n:as:sum:of:c:gamma} the
sum is taken over representatives $[\gamma]$ of all orbits
$\mathcal{O}$ of the mapping class group $\Mod_{g,n}$ in
$\cML_{g,n}(\Z)$.
Formula~\eqref{eq:Vol:Q:as:sum:of:Vol:gamma} allows to
interpret the ratio $\Vol(\gamma)/\Vol\cQ_{g,n}$ as the
``asymptotic probability'' to get a square-tiled surface in
$\cSTgn(\gamma)$ taking a random square-tiled surface in
$\cSTgn(N)$ as $N\to+\infty$.
\smallskip

\noindent\textbf{Stable graph associated to a multicurve.}
Given a simple closed integral multicurve $\gamma=\sum
H_i\gamma_i$ on a topological surface $S_{g,n}$ of genus
$g$ with $n$ punctures as above define the associated
\textit{reduced} multicurve $\gamma_{\mathit{reduced}}$ as
$\gamma_{\mathit{reduced}}=\sum \gamma_i$. Here we assume
that $\gamma_i$ and $\gamma_j$ are pairwise non-isotopic
for $i\neq j$.

Following M.~Kontsevich~\cite{Kontsevich} we assign to any
multicurve $\gamma$ a \textit{stable graph}
$\Gamma(\gamma)=\Gamma(\gamma_{\mathit{reduced}})$. We
provide a detailed formal definition of a stable graph in
Appendix~\ref{s:stable:graphs} limiting ourselves to a more
intuitive description in the current section.

The stable graph $\Gamma(\gamma)$ is a decorated graph dual
to $\gamma_{reduced}$. It consists of vertices, edges, and
``half-edges'' also called ``legs''. Vertices of
$\Gamma(\gamma)$ represent the connected components of the
complement $S_{g,n}\setminus\gamma_{\mathit{reduced}}$.
Each vertex is decorated with the integer number recording
the genus of the corresponding connected component of
$S_{g,n}\setminus\gamma_{reduced}$. Edges of
$\Gamma(\gamma)$ are in the natural bijective
correspondence with curves $\gamma_i$; an edge joins a
vertex to itself when on both sides of the corresponding
simple closed curve $\gamma_i$ we have the same connected
component of $S_{g,n}\setminus\gamma_{reduced}$. Finally,
the $n$ punctures are encoded by $n$ \textit{legs}. The
right picture in
Figure~\ref{fig:square:tiled:surface:and:associated:multicurve}
provides an example of the stable graph associated to the
multicurve $\gamma$.

Pinching a complex curve in $\cM_{g,n}$ by all components
of a reduced multicurve $\gamma_{reduced}$ we get a stable
curve in the Deligne--Mumford compactification
$\overline{\mathcal{M}}_{g,n}$. In this way stable graphs
encode the boundary cycles of
$\overline{\mathcal{M}}_{g,n}$. In particular, the set
$\cG_{g,n}$ of all stable graphs is finite. It is in the
natural bijective correspondence with boundary cycles of
$\overline{\mathcal{M}}_{g,n}$ or, equivalently, with
$\Mod_{g,n}$-orbits of reduced multicurves in
$\cML_{g,n}(\Z)$. Table~\ref{tab:2:0} in
Section~\ref{ss:intro:Masur:Veech:volumes}
and Table~\ref{tab:1:2} in
Appendix~\ref{a:1:2} list all stable graphs in $\cG_{2,0}$ and
$\cG_{1,2}$ respectively.

\subsection{Ribbon graphs, intersection numbers and volume polynomials}
\label{ss:intro:volume:polynomials}

In this section we introduce multivariate polynomials
$N_{g,n}(b_1, \ldots, b_n)$ that appear in different
contexts. They are an essential ingredient to our formula
for the Masur--Veech volume.

Let $g$ be a nonnegative integer and $n$ a positive integer. Let the
pair $(g,n)$ be different from $(0,1)$ and $(0,2)$. Let
$d_1,\dots,d_n$ be an ordered partition of $3g - 3 + n$ into a sum of
nonnegative integers, $|d|=d_1+\dots+d_n=3g-3+n$, let $\boldsymbol{d}$
be a multiindex $(d_1,\dots,d_n)$ and let $b^{2\boldsymbol{d}}$
denote $b_1^{2d_1}\cdot\cdots\cdot b_n^{2d_n}$.

Define the following homogeneous polynomial
$N_{g,n}(b_1,\dots,b_n)$ of degree $3g-3 + n$ in variables
$b_1,\dots,b_n$ in the following way.
\begin{equation}
\label{eq:N:g:n}
N_{g,n}(b_1,\dots,b_n)=
\sum_{|d|=3g-3+n}c_{\boldsymbol{d}} b^{2\boldsymbol{d}}\,,
\end{equation}
where
\begin{equation}
\label{eq:c:subscript:d}
c_{\boldsymbol{d}}=\frac{1}{2^{5g-6+2n}\, \boldsymbol{d}!}\,
\langle \psi_1^{d_1} \dots \psi_n^{d_n}\rangle
\end{equation}
\begin{equation}
\label{eq:correlator}
\langle \psi_1^{d_1} \dots \psi_n^{d_n}\rangle
=\int_{\overline{\cM}_{g,n}} \psi_1^{d_1}\dots\psi_n^{d_n}\,,
\end{equation}
and $\boldsymbol{d}!=d_1!\cdots d_n!$. Note that $N_{g,n}(b_1,\dots,b_n)$
contains only even powers of $b_i$, where $i=1,\dots,n$.
For small $g$ and $n$ we get:
$$
\begin{array}{ll}
N_{0,3}(b_1,b_2,b_3)&=1\\
[-\halfbls] &\\
N_{0,4}(b_1,b_2,b_3,b_4)&=\cfrac{1}{4}(b_1^2+b_2^2+b_3^2+b_4^2)\\
[-\halfbls] &\\
N_{1,1}(b_1)&=\cfrac{1}{48}(b_1^2)\\
[-\halfbls] &\\
N_{1,2}(b_1,b_2)&=\cfrac{1}{384}(b_1^2+b_2^2)(b_1^2+b_2^2)
\end{array}
$$

\begin{NNTheorem}[Kontsevich]
Consider a collection of positive integers $b_1,\dots, b_n$
such that $\sum_{i=1}^n b_i$ is even.
The weighted count of genus $g$ connected trivalent metric ribbon graphs $G$
with integer edges and with $n$ labeled boundary components of lengths $b_1,\dots,b_n$
is equal to $N_{g,n}(b_1,\dots,b_n)$ up to the lower order terms:
$$
\sum_{G \in \cR_{g,n}} \frac{1}{|\Aut(G)|}\, N_G(b_1,\dots,b_n)
=N_{g,n}(b_1,\dots,b_n)+\text{lower order terms}\,,
\,
$$
where $\cR_{g,n}$ denotes the set of (nonisomorphic)
trivalent ribbon graphs $G$ of genus $g$ and with $n$
boundary components.
\end{NNTheorem}

This Theorem is a part of Kontsevich's
proof~\cite{Kontsevich} of Witten's
conjecture~\cite{Witten}.

\begin{Remark}
P.~Norbury~\cite{Norbury} and
K.~Chapman--M.~Mulase--B.~Safnuk~\cite{Chapman:Mulase:Safnuk}
refined the count of Kontsevich proving that the function
counting lattice points in the moduli space $\cM_{g,n}$
corresponding to ramified covers of the sphere over three
points (the so-called \textit{dessins d'enfants}) is a
quasi-polynomial in variables $b_i$. In other terms, when
considering all ribbon graphs (and not only trivalent ones)
the lower order terms in Kontsevich's theorem form a
quasi-polynomial. This quasi-polynomiality of the
expression on the left hand side endows the notion of
``lower order terms'' with natural formal sense.
\end{Remark}

\begin{Remark}
\label{rm:N:g:n:V:g:n}
Up to a normalization constant depending only on $g$ and
$n$ (given by the explicit
formula~\eqref{eq:Vgn:through:Ngn}), the polynomial
$N_{g,n}(b_1,\dots,b_n)$ coincides with the top homogeneous
part of Mirzakhani's volume polynomial
$V_{g,n}(b_1,\dots,b_n)$ providing the Weil--Petersson
volume of the moduli space of bordered Riemann
surfaces~\cite{Mirzakhani:simple:geodesics:and:volumes}.
This relation between correlators $\langle \psi_1^{d_1}
\ldots \psi_n^{d_n}\rangle$ and Weil--Petersson volumes is
one of the key elements of Mirzakhani's alternative proof
of Witten's
conjecture~\cite{Mirzakhani:volumes:and:intersection:theory}.
\end{Remark}

We also use the following common notation for the
intersection numbers~\eqref{eq:correlator}. Given an
ordered partition $d_1+\dots+d_n=3g-3+n$ of $3g - 3 + n$
into a sum of nonnegative integers we define

\begin{equation}
\label{eq:tau:correlators}
\langle \tau_{d_1} \dots \tau_{d_n}\rangle
=\int_{\overline{\cM}_{g,n}} \psi_1^{d_1}\dots\psi_n^{d_n}\,.
\end{equation}

\subsection{Formula for the Masur--Veech volumes}
\label{ss:intro:Masur:Veech:volumes}

Following~\cite{AEZ:genus:0} we consider the following
linear operators $\cY(\boldsymbol{H})$ and $\cZ$ on the spaces of
polynomials in variables $b_1,b_2,\dots$.

The operator $\cY(\boldsymbol{H})$ is defined for any collection of
strictly positive integer parameters $H_1, H_2, \dots$. It
is defined on monomials as
\begin{equation}
\label{eq:cV}
\cY(\boldsymbol{H})\ :\quad
\prod_{i=1}^{k} b_i^{m_i} \longmapsto
\prod_{i=1}^{k} \frac{m_i!}{H_i^{m_i+1}}\,,
\end{equation}
and extended to arbitrary polynomials by linearity.

Operator $\cZ$ is defined on monomials as
\begin{equation}
\label{eq:cZ}
\cZ\ :\quad
\prod_{i=1}^{k} b_i^{m_i} \longmapsto
\prod_{i=1}^{k} \big(m_i!\cdot \zeta(m_i+1)\big)\,,
\end{equation}
and extended to arbitrary polynomials by linearity.
In the above
formula $\zeta$ is the Riemann zeta function
$$
\zeta(s) = \sum_{n \geq 1} \frac{1}{n^s}\,,
$$
so for any collection of strictly positive integers
$(m_1,\dots,m_k)$ one has
$$
\cZ\left(\prod_{i=1}^{k} b_i^{m_i}\right)
=\sum_{\boldsymbol{H}\in\N^k}
\cY(\boldsymbol{H})\left(\prod_{i=1}^{k} b_i^{m_i}\right)\,.
$$

\begin{Remark}
For even integers $2m$ we have
$$
\zeta(2m) = (-1)^{m+1} \frac{B_{2m} (2\pi)^{2m}}{2\, (2m)!}
$$
where $B_{2m}$ are the Bernoulli numbers. Consider a homogeneous
polynomial of degree $2m$ with rational coefficients, such
that all powers of all variables in each monomial are odd.
Observation above implies that the value of $\cZ$ on any
such polynomial is a rational number multiplied by
$\pi^{2m}$.
\end{Remark}


Given a stable graph $\Graph$ denote by $V(\Gamma)$ the set
of its vertices and by $E(\Gamma)$ the set of its edges. To
each stable graph $\Gamma\in\cG_{g,n}$ we associate the
following homogeneous polynomial $P_\Gamma$
of degree $6g-6+2n$. To
every edge $e\in E(\Gamma)$ we assign a formal variable
$b_e$. Given a vertex $v\in V(\Gamma)$ denote by $g_v$ the
integer number decorating $v$ and denote by $n_v$ the
valency of $v$, where the legs adjacent to $v$ are counted
towards the valency of $v$. Take a small neighborhood of
$v$ in $\Gamma$. We associate to each half-edge (``germ''
of edge) $e$ adjacent to $v$ the monomial $b_e$; we
associate $0$ to each leg. We denote by $\boldsymbol{b}_v$
the resulting collection of size $n_v$. If some edge $e$ is
a loop joining $v$ to itself, $b_e$ would be present in
$\boldsymbol{b}_v$ twice; if an edge $e$ joins $v$ to a
distinct vertex, $b_e$ would be present in
$\boldsymbol{b}_v$ once; all the other entries of
$\boldsymbol{b}_v$ correspond to legs; they are represented
by zeroes. To each vertex $v\in E(\Gamma)$ we associate the
polynomial $N_{g_v,n_v}(\boldsymbol{b}_v)$, where $N_{g,v}$
is defined in~\eqref{eq:N:g:n}. We associate to the stable
graph $\Gamma$ the polynomial obtained as the product
$\prod b_e$ over all edges $e\in E(\Graph)$ multiplied by
the product $\prod N_{g_v,n_v}(\boldsymbol{b}_v)$ over all
$v\in V(\Graph)$. We define $P_\Gamma$ as follows:
\begin{multline}
\label{eq:P:Gamma}
P_\Gamma(\boldsymbol{b})
=
\frac{2^{6g-5+2n} \cdot (4g-4+n)!}{(6g-7+2n)!}\cdot
\\
\frac{1}{2^{|V(\Graph)|-1}} \cdot
\frac{1}{|\operatorname{Aut}(\Graph)|}
\cdot
\prod_{e\in E(\Graph)}b_e\cdot
\prod_{v\in V(\Graph)}
N_{g_v,n_v}(\boldsymbol{b}_v)
\,.
\end{multline}

Tables in Appendices~\ref{a:2:0}
and~\ref{a:1:2} list polynomials associated to all stable
graphs in $\cG_{2,0}$ and in $\cG_{1,2}$.

\begin{Theorem}
\label{th:volume}
The Masur--Veech volume $\Vol \cQ_{g,n}$ of the moduli
space of meromorphic quadratic differentials with $n$ simple poles has
the following value:
\begin{equation}
\label{eq:square:tiled:volume}
\Vol \cQ_{g,n}
= \sum_{\Graph \in \cG_{g,n}} \Vol(\Gamma)\,,
\end{equation}
where the contribution of an individual stable graph
$\Gamma$ has the form
\begin{equation}
\label{eq:volume:contribution:of:stable:graph}
\Vol(\Gamma)=\cZ(P_\Gamma)
\end{equation}
\end{Theorem}

\begin{Remark}
The contribution~\eqref{eq:volume:contribution:of:stable:graph}
of any individual stable graph
has the following natural interpretation.
We have seen that stable graphs $\Gamma$ in $\cG_{g,n}$
are in natural bijective correspondence
with $\Mod_{g,n}$-orbits of \textit{reduced} multicurves
$\gamma_{\mathit{reduced}}=\gamma_1+\gamma_2+\dots$,
where simple closed curves $\gamma_i$ and $\gamma_j$
are not isotopic for any $i\neq j$.
Let $\Gamma\in\cG_{g,n}$, let $k=|V(\Gamma)|$, let
$\gamma_{\mathit{reduced}}=\gamma_1+\dots+\gamma_k$ be the reduced multicurve
associated to $\Gamma$. Let $\gamma_{\boldsymbol{H}}
=\gamma(\Gamma,\boldsymbol{H})=
H_1\gamma_1+\dots H_k\gamma_k$, where
$\boldsymbol{H}=(H_1,\dots,H_k)\in\N^k$.
We have
\begin{equation}
\label{eq:Vol:Gamma}
\Vol(\Gamma)=
\sum_{H\in\N^k}
\Vol\big(\gamma(\Gamma,\boldsymbol{H})\big)
=
\sum_{H\in\N^k}
\Vol(\gamma_{\boldsymbol{H}})\,,
\end{equation}
where the contribution
$\Vol\big(\gamma(\Gamma,\boldsymbol{H})\big)$ of
square-tiled surfaces with the horizontal cylinder
decomposition of type $\gamma(\Gamma,\boldsymbol{H})$
to $\Vol\cQ_{g,n}$ is given
by the formula:
\begin{equation}
\label{eq:contribution:of:gamma:to:volume}
\Vol\big(\gamma(\Gamma,\boldsymbol{H})\big)
=\Vol(\gamma_{\boldsymbol{H}})
=\cY(\boldsymbol{H})(P_\Gamma)\,.
\end{equation}
\end{Remark}

Table~\ref{tab:2:0} bellow illustrates computation of
polynomials $P_\Gamma$ and of the contributions
$\Vol(\Gamma)$ to the Masur--Veech volume $\Vol\cQ_{g,n}$
in the particular case of $(g,n)=(2,0)$. To make the
calculation traceable, we follow the structure of
formula~\eqref{eq:P:Gamma}.
The first numerical factor $128/5$ represents the factor
$\frac{2^{6g-5+2n} \cdot (4g-4+n)!}{(6g-7+2n)!}$. It is
common for all stable graphs in $\cG_{2,0}$. The second
numerical factor in the first line of each calculation
in Table~\ref{tab:2:0} is $1/2^{|V(\Graph)|-1}$, where $|V(\Graph)|$ is the
number of vertices of the corresponding stable graph
$\Gamma$ (equivalently ---
the number of connected components of the complement to the
associated reduced multicurve).
The third numerical factor is
$\cfrac{1}{|\Aut(\Graph)|}$. Recall that neither
vertices nor edges of $\Graph$ are labeled. We first evaluate the
order of the corresponding automorphism group
$|\Aut(\Graph)|$ (this group respects the decoration
of the graph), and only then we associate to edges of $\Graph$
variables $b_1, \dots, b_k$ in an arbitrary way.


\begin{table}
$$
\hspace*{20pt}
\begin{array}{llr}


\includegraphics{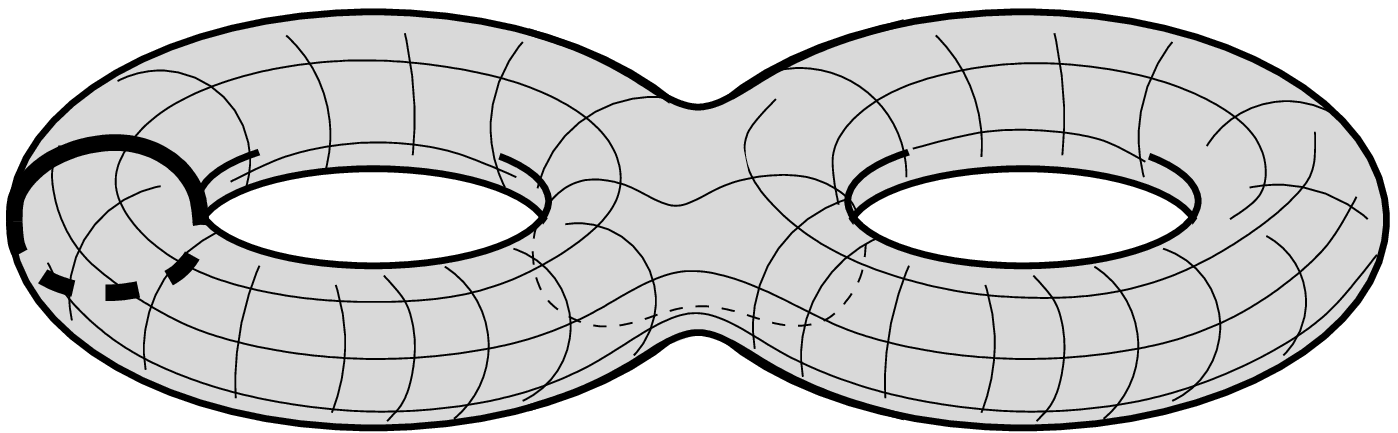}
\begin{picture}(0,0)(0,0)
\put(-24,0){$b_1$}
\end{picture}
\hspace*{68pt} 
\vspace*{5pt}

&
\frac{128}{5}\cdot
1\cdot
\frac{1}{2}\cdot
b_1 \cdot
N_{1,2}(b_1,b_1)=

&
\\


\includegraphics{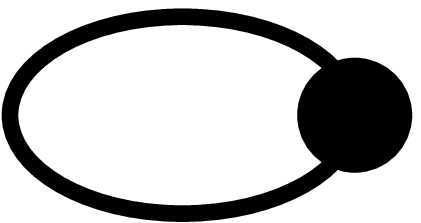}
\begin{picture}(0,0)(0,0)
\put(-10,-10.5){$b_1$}
\put(30,-10.5){$1$}
\end{picture}
\vspace*{5pt}

&
\hspace*{7pt}
=\frac{64}{5}\cdot b_1 \left(\frac{1}{384}(2b_1^2)(2b_1^2)\right)
=\frac{2}{15}\cdot b_1^5\xmapsto{\cZ}
&
\frac{2}{15}\cdot \big(5!\cdot \zeta(6)\big)
\\


&&
=\frac{16}{945}\cdot\pi^6

\vspace*{5pt}\\\hline&&\\

\includegraphics{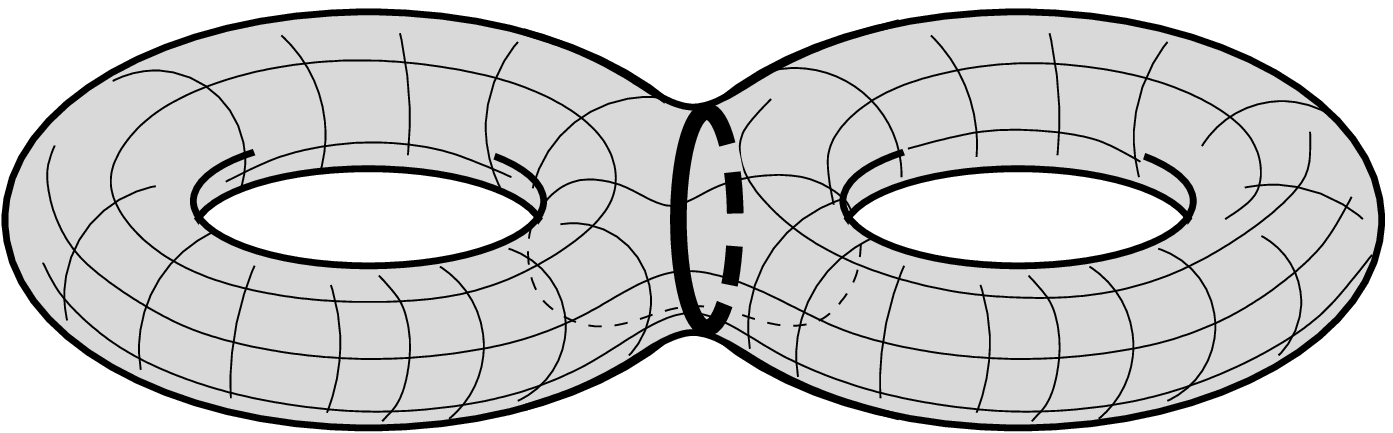}
\begin{picture}(0,0)(0,0)
\put(23,-14){$b_1$}
\end{picture}
\vspace*{5pt}

&
\frac{128}{5}\cdot
\frac{1}{2}\cdot
\frac{1}{2}\cdot
b_1 \cdot N_{1,1}(b_1)\cdot N_{1,1}(b_1)=

&
\\


\includegraphics{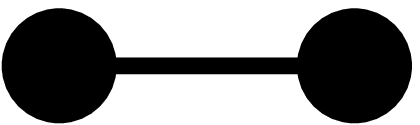}
\begin{picture}(0,0)(0,0)
\put(23,-20){$b_1$}
\put(6,-14){$1$}
\put(41,-14){$1$}
\end{picture}
\vspace*{5pt}

&
\hspace*{6pt}

=\frac{32}{5}\cdot b_1\cdot \left(\frac{1}{48}b_1^2\right)\left(\frac{1}{48}b_1^2\right)
=\frac{1}{360}\cdot b_1^5\xmapsto{\cZ}
&

\frac{1}{360}\cdot \big(5!\cdot \zeta(6)\big)

\\

&&

=\frac{1}{2835}\cdot\pi^6

\vspace*{5pt}\\\hline&&\\

\includegraphics{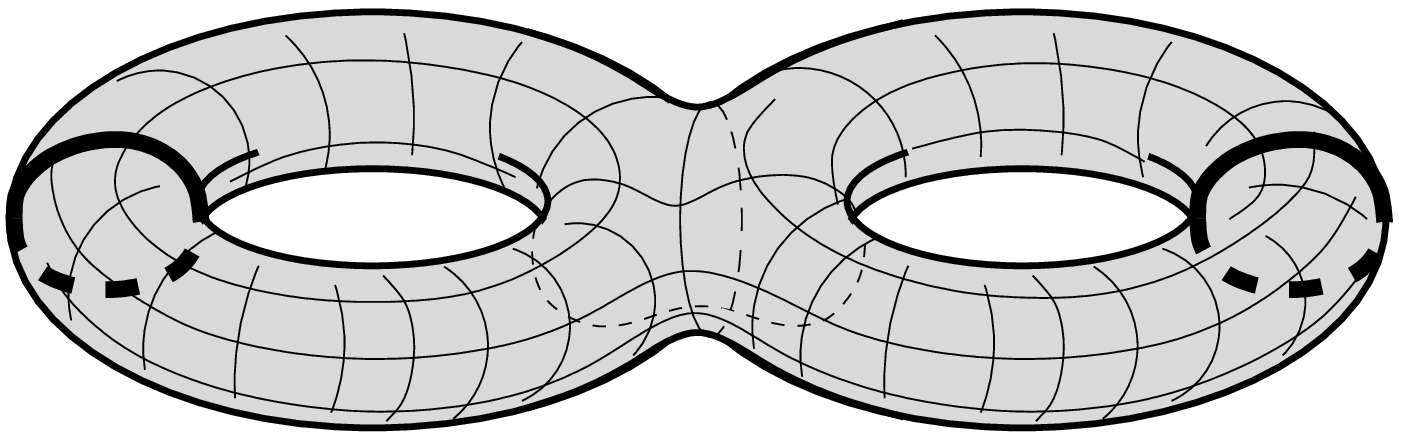}
\begin{picture}(0,0)(0,0)
\put(-24,0){$b_1$}
\put(66,0){$b_2$}
\end{picture}
\vspace*{5pt}

&
\frac{128}{5}\cdot
1 \cdot
\frac{1}{8}\cdot
b_1 b_2 \cdot
N_{0,4}(b_1,b_1,b_2,b_2)=

&

\\


\includegraphics{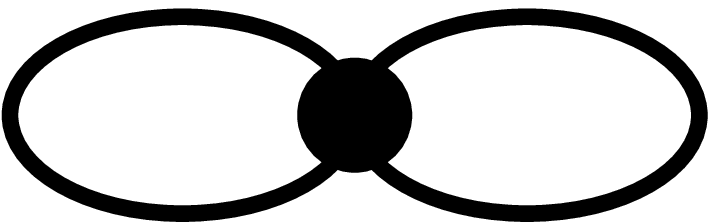}
\begin{picture}(0,0)(0,0)
\put(-9,-19){$b_1$}
\put(52,-19){$b_2$}
\put(23,-30){$0$}
\end{picture}
\vspace*{5pt}

&
\hspace*{20pt}
=\frac{16}{5}\cdot b_1 b_2\cdot\left(\frac{1}{4}(2b_1^2+2b_2^2)\right)=

&\\

&

\hspace*{80pt}
=\frac{8}{5}(b_1^3 b_2+b_1 b_2^3)\xmapsto{\cZ}

&

\frac{8}{5}\!\cdot\! 2\!\cdot\! 3!\zeta(4)\!\cdot\!1!\zeta(2)

\vspace*{5pt}\\ &&

=\frac{8}{225}\cdot \pi^6

\vspace*{5pt}\\\hline&&\\

\includegraphics{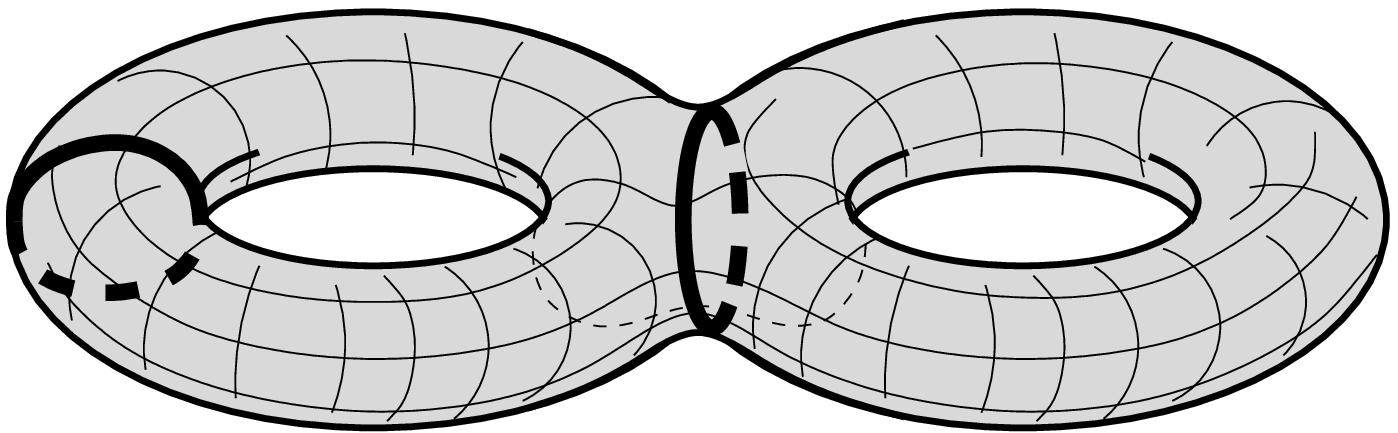}
\begin{picture}(0,0)(0,0)
\put(-24,0){$b_1$}
\put(23,-14){$b_2$}
\end{picture}
\vspace*{5pt}

&

\frac{128}{5}\!\cdot\!
\frac{1}{2}\!\cdot\!
\frac{1}{2}\!\cdot\!
b_1 b_2\!\cdot\! N_{0,3}(b_1,b_1,b_2)\!\cdot\! N_{1,1}(b_2)

&

\\


\includegraphics{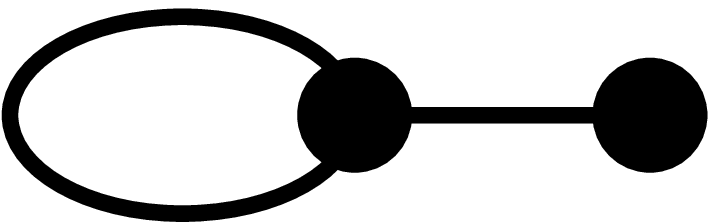}
\begin{picture}(0,0)(0,0)
\put(-19,-10){$b_1$}
\put(23,-16){$b_2$}
\put(6,-10){$0$}
\put(41,-10){$1$}
\end{picture}
\vspace*{5pt}

&
=\frac{32}{5}\cdot b_1 b_2\cdot\big(1\big) \cdot \left(\frac{1}{48} b_2^2\right)
=\frac{2}{15}\cdot b_1 b_2^3\xmapsto{\cZ}
&
\frac{2}{15}\!\cdot\! 1!\zeta(2)\!\cdot\!3!\zeta(4)
\\


&&
=\frac{1}{675}\cdot\pi^6

\vspace*{5pt}\\\hline&&\\

\includegraphics{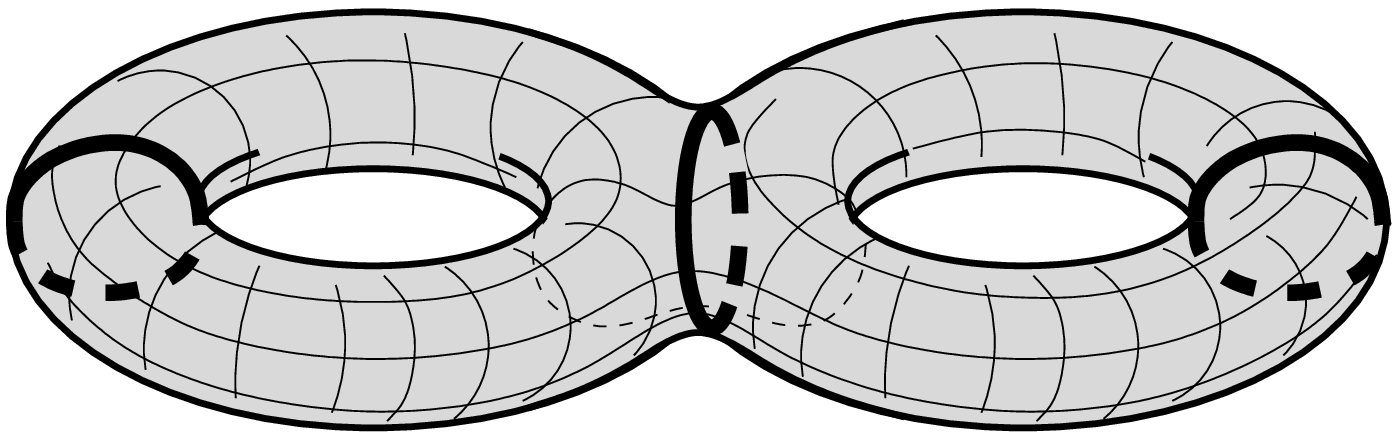}
\begin{picture}(0,0)(0,0)
\put(-24,0){$b_1$}
\put(23,-14){$b_2$}
\put(66,0){$b_3$}
\end{picture}
\vspace*{5pt}

&
\frac{128}{5}\!\cdot\!
\frac{1}{2}\!\cdot\!
\frac{1}{8}\!\cdot\!
b_1 b_2 b_3 \!\cdot\!
N_{0,3}(b_1,b_1,b_2)\cdot

&

\\


\includegraphics{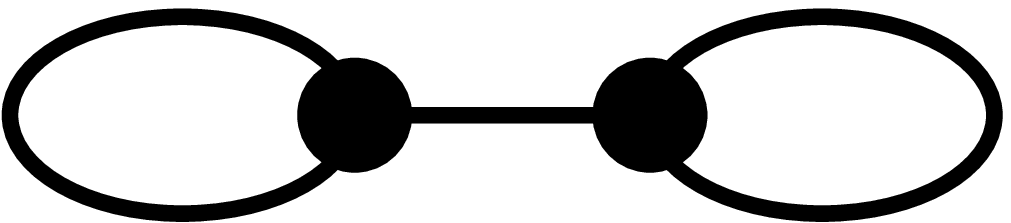}
\begin{picture}(0,0)(0,0)
\put(-18,-10){$b_1$}
\put(23,-16){$b_2$}
\put(63,-10){$b_3$}
\put(5.5,-10){$0$}
\put(42,-10){$0$}
\end{picture}
\vspace*{5pt}

&

\cdot N_{0,3}(b_2,b_3,b_3)
=\frac{8}{5}\!\cdot\! b_1 b_2 b_3\!\cdot\! (1) \!\cdot\! (1)
\xmapsto{\cZ}

&
\frac{8}{5}\cdot \left(1!\,\zeta(2)\right)^3
\\


&&
=\frac{1}{135}\cdot\pi^6

\vspace*{5pt}\\\hline&&\\

\includegraphics{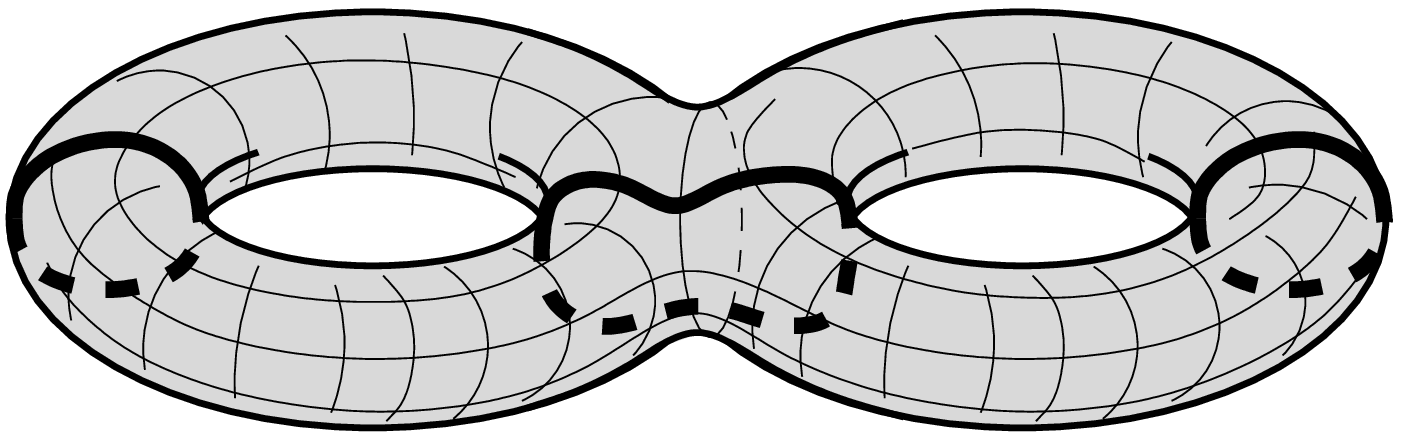}
\begin{picture}(0,0)(0,0)
\put(-24,0){$b_1$}
\put(20,12){$b_2$}
\put(66,0){$b_3$}
\end{picture}
\vspace*{5pt}

&

\frac{128}{5}\!\cdot\!
\frac{1}{2}\!\cdot\!
\frac{1}{12}\!\cdot\!
b_1 b_2 b_3 \!\cdot\!
N_{0,3}(b_1,b_1,b_2)\cdot

&

\\

\includegraphics{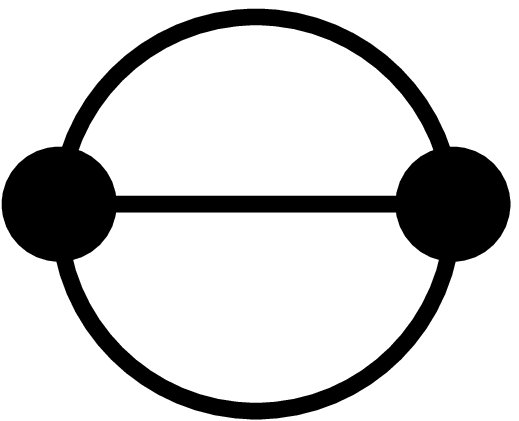}
\begin{picture}(0,0)(0,0)
\put(4,-16){$b_1$}
\put(27.5,-16){$b_2$}
\put(43,-16){$b_3$}
\put(31,2){$0$}
\put(31,-35){$0$}
\end{picture}
\vspace*{5pt}

&

\cdot N_{0,3}(b_2,b_3,b_3)
=\frac{16}{15}\!\cdot\! b_1 b_2 b_3\!\cdot\! (1) \!\cdot\! (1)
\xmapsto{\cZ}

&
\frac{16}{15}\cdot \left(1!\,\zeta(2)\right)^3
\\


&&
=\frac{2}{405}\cdot\pi^6

\vspace*{20pt}
\\
\end{array}
$$
\caption{
\label{tab:2:0}
Computation of $\Vol\cQ_{2,0}$.
Left column represents
stable graphs $\Gamma\in\cG_{2,0}$
and associated multicurves;
middle column gives
associated polynomials
$P_\Gamma$; right column
provides $\Vol(\Gamma)$.
}
\end{table}

Taking the sum of the six contribution
we obtain the answer:

\begin{multline*}
\Vol \cQ_2=\Vol\cQ(1^4)=
\left(
\left(\frac{16}{945}+\frac{1}{2835}\right)
+
\left(\frac{8}{225}+\frac{1}{675}\right)
+
\left(\frac{1}{135}+\frac{2}{405}\right)
\right) \pi^6
=\\=
\left(\frac{7}{405}+\frac{1}{27}+\frac{1}{81}\right) \pi^6
=
\frac{1}{15} \pi^6\,.
\end{multline*}
It matches the value found in~\cite{Goujard:volumes}
by implementing the method of A.~Eskin and A.~Okounkov~\cite{Eskin:Okounkov:pillowcase}.


\begin{Remark}
In genus $0$, the formula simplifies considerably.
It was conjectured by M.~Konstevich and proved by
J.~Athreya, A.~Eskin and A.~Zorich in~\cite{AEZ:genus:0}
that for all $n \geq 4$
\begin{equation}
\label{eq:vol:genus:0}
\Vol\cQ_{0,n} = \frac{\pi^{2n-6}}{2^{n-5}}\,.
\end{equation}

Note that in genus $0$ all correlators of $\psi$-classes
admit closed explicit expression. Rewriting all polynomials
$N_{0,n_v}(\boldsymbol{b}_v)$ for all stable graphs in
$\cG_{0,n}$ in formula~\eqref{eq:square:tiled:volume} for
$\Vol\cQ_{0,n}$ in terms of the corresponding multinomial
coefficients we get formula originally obtained
in~\cite{AEZ:Dedicata}. A lot of technique from this
article is borrowed from~\cite{AEZ:Dedicata}. Note,
however, that the proof of~\eqref{eq:vol:genus:0} is
indirect and is based on analytic
Riemann--Roch--Hierzebruch formula and on fine comparison
of asymptotics of determinants of flat and hyperbolic
Laplacians as $(X,q)$ in $\cQ_{0,n}$ approaches the
boundary. It is a challenge to
derive~\eqref{eq:vol:genus:0}
from~\eqref{eq:square:tiled:volume} directly.
\end{Remark}

Note also the following feature of
Formula~\eqref{eq:square:tiled:volume} which distinguishes
it from approach of
Eskin--Okounkov~\cite{Eskin:Okounkov:Inventiones},
\cite{Eskin:Okounkov:pillowcase} based on quasimodularity
of certain generating function or from approach of
Chen--M\"oeller--Sauvaget--Zagier~\cite{Chen:Moeller:Sauvaget:Zagier}
based on recurrence relation.
Formula~\eqref{eq:square:tiled:volume} allows to analyze
the contribution of individual stable graphs to
$\Vol\cQ_{g,n}$. In particular, it allows to study
statistics of random square-tiled surfaces. It also implies
the following asymptotic lower bound for the Masur--Veech
volume $\Vol\cQ_g$ and a conjectural asymptotic value.

\begin{Theorem}
\label{th:asymptotic:lower:bound}
The following asymptotic inequality holds
\begin{equation}
\Vol\cQ_g
\ge\sqrt{\frac{2}{3\pi g}}
\cdot\left(\frac{8}{3}\right)^{4g-4}
\cdot\left(1+O\left(\frac{1}{g}\right)\right)
\quad\text{as }g\to+\infty\,.
\end{equation}
\end{Theorem}

\begin{Conjecture}
\label{conj:Vol:Qg}
The Masur--Veech volume of the moduli space of holomorphic
quadratic differentials has the following large genus
asymptotics:
\begin{equation}
\label{eq:Vol:Qg}
\Vol\cQ_g
\overset{?}{=}
\frac{4}{\pi}
\cdot\left(\frac{8}{3}\right)^{4g-4}
\cdot\left(1+O\left(\frac{1}{g}\right)\right)
\quad\text{as }g\to+\infty\,.
\end{equation}
\end{Conjecture}
Certain finer conjectures are discussed in
Section~\ref{ss:large:genera} below.

\begin{Remark}
\label{rm:Vol:power:of:pi}
By construction, the polynomial $P_\Gamma(\boldsymbol{b})$
associated to a stable graph $\Graph\in\cG_{g,n}$
by expression~\eqref{eq:P:Gamma} is a
homogeneous polynomial of degree $6g-6+2n-|E(\Gamma)|$.
Moreover, each variable $b_e$ in each monomial appears with
an odd power. In particular,
Formula~\eqref{eq:square:tiled:volume} implies that the
Masur--Veech volume $\Vol\cQ_{g,n}$ is a rational multiple
of $\pi^{6g-6+2n}$ generalizing analogous result proved by
A.~Eskin and A.~Okounkov for the Masur--Veech volumes of
the strata of Abelian
differentials~\cite{Eskin:Okounkov:Inventiones}. Actually,
the contribution $\Vol(\Gamma)$ of each stable graph
$\Gamma\in\cG_{g,n}$ to $\Vol\cQ_{g,n}$ is already a
rational multiple of $\pi^{6g-6+2n}$. Moreover, using the
refined version of $N_{g,n}(b_1, \ldots, b_n)$ due to
Norbury~\cite{Norbury} expressing the counting functions of
ribbon graphs as quasi-polynomials in $b_i$, one can even
show that the generating series of square-tiled surfaces
corresponding a given stable graph is a quasi-modular form.
This result develops the results of
Eskin-Okounkov~\cite{Eskin:Okounkov:pillowcase} that says
that in each stratum, the generating series for the count
of pillowcase covers (in the sense of A.~Eskin and A.~Okounkov)
is a quasimodular form and the analogous result of
Ph.~Engel~\cite{Engel1}, \cite{Engel2} for the count of square-tiled
surfaces.
\end{Remark}

\begin{Remark}
\label{rm:ref:to:Engel}
As it was already mentioned in Remark~\ref{rm:N:g:n:V:g:n},
up to a normalization constant given by the explicit
formula~\eqref{eq:Vgn:through:Ngn} and depending only on
$g$ and $n$, the polynomial $N_{g,n}(b_1,\dots,b_n)$
coincides with the top homogeneous part of Mirzakhani's
volume polynomial $V_{g,n}(b_1,\dots,b_n)$ providing the
Weil--Petersson volume of the moduli space of bordered
Riemann
surfaces~\cite{Mirzakhani:simple:geodesics:and:volumes}.
Note that the classical Weil--Petersson volume of
$\cM_{g,n}$ corresponds to the constant term of
$V_{g,n}(b_1,\dots,b_n)$ when the lengths of all boundary
components are equal to zero. To compute the Masur--Veech
volume we use the top homogeneous parts of polynomials
$V_{g,n}(b_1,\dots,b_n)$. In this sense we use Mirzakhani's
volume polynomials in the opposite regime when the lengths
$b_i$ of all boundary components tend to infinity.
\end{Remark}

\begin{Remark}
\label{rm:to:complete}
Formulae~\eqref{eq:P:Gamma}--\eqref{eq:volume:contribution:of:stable:graph}
admit a generalization allowing to express Masur--Veech
volumes of those strata $\cQ(d_1,\dots, d_n)$, for which
all zeroes have odd degrees $d_i$, in terms of intersection
numbers of $\psi$-classes with the combinatorial cycles in
$\overline{\cM}_{g,n}$ associated to the strata (denoted by
$W_{m_*,n}$ in \cite{Arbarello:Cornalba}, where $m_*$ is the sequence of
multiplicities of the zeroes). In this more general case
the formula requires additional correction subtracting the
contribution of those quadratic differentials which
degenerate to squares of globally defined Abelian
differentials. This generalization is a work in progress.
\end{Remark}

\begin{Remark}
Note that the contribution of square-tiled surfaces having
fixed number of cylinders to the Masur--Veech volume of
more general strata of quadratic differentials might have a
much more sophisticated arithmetic nature.
In~\cite{DGZZ:one:cylinder:Yoccoz:volume} we study in
detail the contribution of square-tiled surfaces having a
single maximal horizontal cylinder to the Masur--Veech
volume of any stratum of Abelian or quadratic
differentials.
\end{Remark}

\subsection{Siegel--Veech constants}
\label{ss:intro:Siegel:Veech:constants}
We now turn to a formula for the Siegel--Veech constants of
$\cQ_{g,n}$. We first recall the definition of
Siegel--Veech constants that involves the flat geometry of
quadratic differentials.

Let $(C,q)$ be a meromorphic quadratic differential in
$\cQ_{g,n}$. The differential $q$ naturally defines a
Riemannian metric $|q|$. The metric is flat; it has conical
singularities exactly at the zeros and poles of $q$. This
metric allows to define geodesics and we say that a
geodesic is \textit{regular} if it does not pass through
the singularities of $q$. Closed regular flat geodesics
appear in families composed of parallel closed geodesics of
the same length. Each such family fills a maximal flat
cylinder $\mathit{cyl}$ having a conical singularity
(possibly the same) at each of the two boundary components.
By the \textit{width} (or, sometimes, by a
\textit{perimeter}) $w$ of a cylinder, we call the length
of a closed geodesic in the corresponding family. The
\textit{height} $h$ of a maximal cylinder is defined as the
distance between the boundary components. In particular,
the flat area of the cylinder is the product $h\cdot w$.

For any given flat surface $S=(C,q)$ and any $L\in\R$, the
number of maximal cylinders in $S$ filled with regular closed
geodesics of  bounded  length  $w(cyl)\le L$ is finite.
Thus, for any $S=(C,q)$ and any $L>0$ the following quantity is well-defined:
\begin{equation}
\label{eq:N:area}
N_{\mathit{area}}(S,L):=
\frac{1}{\Area(S)}
\sum_{\substack{
\mathit{cyl}\subset S\\
w(\mathit{cyl})\le L}}
\Area(cyl)\,.
\end{equation}

We already discussed in Section~\ref{ss:MV:volume} that the
Masur--Veech volume element $\dVolMV$ in $\cQ_{g,n}$
induces the canonical volume element $\dVolMV_1$ on any
noncritical level hypersurface of any real-valued function
$f$. In particular, it induces the canonical volume element
$\dVolMV_1$ on the level hypersurface
$\cQ^{\Area=a}_{g,n}$. It follows from the independent
results of H.~Masur~\cite{Masur:82} and
W.~Veech~\cite{Veech:Gauss:measures}, that the induced
volume $\Vol_1\cQ^{\Area=a}_{g,n}$ is finite.
The following theorem is a special case
of the fundamental result of W.~Veech, \cite{Veech:Siegel}
developed by Y.~Vorobets in~\cite{Vorobets}.

\begin{NNTheorem}[W.~Veech; Ya.~Vorobets]
Let $(g,n)$ be a pair of nonnegative integers such that
$2g+n>3$. There exists a strictly positive constant
$\carea(\cQ_{g,n})$ such that for any strictly
positive numbers
$a$ and $L$ the following holds:
\begin{equation}
\label{eq:SV:constant:definition}
\frac{a}{\pi L^2}\int_{\cQ^{\Area=a}_{g,n}}
N_{\mathit{area}}(S,L)\,d\Vol_1(S)
=
\Vol_1\cQ^{\Area=a}_{g,n}
\ \cdot\ \carea(\cQ_{g,n})\,.
\end{equation}
\end{NNTheorem}
This formula is called the \textit{Siegel--Veech formula}, and the
corresponding  constant $\carea(\cQ_{g,n})$ is called the
\textit{Siegel--Veech constant}.

Eskin and Masur~\cite{Eskin:Masur} proved
that for almost all $S=(C,q)$ in $\cQ_{g,n}$
(with respect to the Masur--Veech measure)
\begin{equation}
\label{eq:SV:asymptotics}
\lim_{L\to+\infty}
\Area(S) \cdot \frac{N_{\mathit{area}}(S,L)}{\pi L^2}=
\carea(\cQ_{g,n})\,.
\end{equation}

\begin{Remark}
Beyond its geometrical relevance, let us mention that the
area Siegel--Veech constant is the most important
ingredient in the Eskin-Kontsevich-Zorich formula for the
sum of the Lyapunov exponents of the Hodge bundle along the
Teichm\"uller geodesic flow~\cite{Eskin:Kontsevich:Zorich}.
\end{Remark}

An edge of a connected graph is called a \textit{bridge} if
the operation of removing this edge breaks the graph into
two connected components.
We define the following function $\separatingQ: E(\Graph)
\to \{\tfrac{1}{2}, 1\}$ on the set of edges of
any connected graph $\Gamma$:
$$
\separatingQ(e)=
\begin{cases}
\tfrac{1}{2}&\text{if the edge $e$ is a bridge\,,}\\
1  &\text{otherwise}
\end{cases}\,.
$$

We define the following operator $\partial_{\Graph}$ on
polynomials $P$ in variables $b_e$ associated to the edges
of stable graphs $\Graph\in\cG_{g,n}$. For every
$e\in E(\Gamma)$ let
\begin{equation}
\label{eq:operator:D:e}
\partial^e_{\Graph} P
:=\separatingQ(e)\, b_e\, \left.\frac{\partial P}{\partial b_e}\right|_{b_e=0}\,,
\end{equation}
and let
\begin{equation}
\label{eq:operator:D}
\partial_{\Graph} P
:=
\sum_{e \in E(\Graph)}
\partial^e_{\Graph} P\,.
\end{equation}

\begin{Theorem}
\label{th:carea}
Let $g,n$ be nonnegative integers satisfying $2g+n\ge 4$.
The Siegel--Veech constant $\carea(\cQ_{g,n})$
satisfies the following relation:
\begin{equation}
\label{eq:carea}
\Vol(\cQ_{g,n})
 \cdot
\carea(\cQ_{g,n})
=
\frac{3}{\pi^2}
\cdot
\sum_{\Graph \in \cG_{g,n}}
\cZ\left(\partial_{\Graph}
P_\Gamma\right)\,.
\end{equation}
\end{Theorem}

As an illustration of the above Theorem we compute
$\carea(\cQ_{2,0})$ and $\carea(\cQ_{1,2})$ in
appendices~\ref{a:2:0} and~\ref{a:1:2} respectively.

\subsection{Masur--Veech Volumes and Siegel--Veech constants}
\label{ss:intro:Siegel:Veech:Masur:Veech}
The Siegel--Veech constant can be expressed in terms of the
Masur--Veech volumes of certain boundary strata. The
corresponding formula for the strata of Abelian
differentials was obtained in~\cite{Eskin:Masur:Zorich};
for the strata of quadratic differentials it was obtained
in~\cite{Goujard:carea}. Before presenting
a reformulation of Corollary~1 in~\cite{Goujard:carea}
concerning the principal stratum $\cQ(1^{4g-4+n}, -1^n)$ in
$\cQ_{g,n}$ we introduce the following conventions.

We define by convention
\begin{eqnarray*}
 \Vol\cQ_{0,3} &=& 4\langle\tau_0^3\rangle=4\,,
 \\
 \Vol\cQ_{1,1} &=&\frac{2\pi^2}{3}\langle\tau_0^3\rangle=\frac{2\pi^2}{3}\,.
\end{eqnarray*}
We also define by convention the following stable graphs:

\[\cG(0,3):=\left\lbrace\begin{array}{c}
\includegraphics{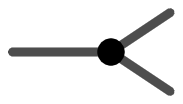}
\begin{picture}(0,0)(0,0)
\put(25,-15){$0$}
\end{picture}
\hspace{60pt}
\vspace{10pt}
\end{array}\right\rbrace,\]
and
\[\cG(1,1):=\left\lbrace\begin{array}{c}
\includegraphics{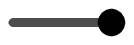}
\begin{picture}(0,0)(0,0)
\put(30,-15){$1$}
\end{picture}
\vspace{10pt}
\hspace{40pt}
\end{array},
\begin{array}{c}
\includegraphics{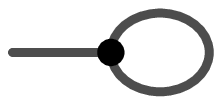}
\begin{picture}(0,0)(0,0)
\put(25,-15){$0$}
\end{picture}
\vspace{10pt}
\hspace{60pt}
\end{array}\right\rbrace.\]
(We did not label the legs since each of the above graphs
admits unique labeling of the legs up to a symmetry of the
graph.) Under such conventions, the values of
$\Vol\cQ_{0,3}$ and of $\cQ_{1,1}$ correspond to
Formula~\eqref{eq:square:tiled:volume}.

Another convention concerns the value of the following ratio
which would be used in the formula below. We define
$$
\left.\frac{(2n_i-7)!}{(n_i-4)!}\right|_{n_i=3}:=
\lim_{n\to 3}\frac{\Gamma(2n-6)}{\Gamma(n-3)}=\frac{1}{2}\,.
$$

\begin{NNTheorem}[\cite{Goujard:carea}]
Let $g$ be a strictly positive integer, and $n$ nonnegative integer.
When $g=1$ we assume that $n\ge 2$. Under the above conventions
the following formula is valid:
\begin{align}
\label{eq:carea:Elise}
\carea(\cQ_{g,n}) \cdot \Vol(\cQ_{g,n})=
\frac{1}{8}
 \sum_{\substack{g_1+g_2=g \\ n_1+n_2=n+2 \\
 g_i\geq 0, n_i\geq 1, \dprinc_i\geq 0}}
 \frac{\zeroes!}{\zeroes_1!\zeroes_2!}\frac{n!}{(n_1-1)!(n_2-1)!}\cdot\\
\notag
\cdot\frac{(\dprinc_1-1)!(\dprinc_2-1)!}{(\dprinc-1)!}
\Vol(\cQ_{g_1, n_1})\times\Vol(\cQ_{g_2,n_2})
 +\frac{\zeroes!}{(\zeroes-2)!}\frac{(\dprinc-3)!}{(\dprinc-1)!}
 \Vol(\cQ_{g-1,n+2})\,.
\end{align}
Here $d=\dim_{\C}\cQ_{g,n}=6g-6+2n$, $d_i=6g_i-6+2n_i$,
$\ell=4g-4+n$, $\zeroes_i=4g_i-4+n_i$\,.

For $g=0$ and any integer $n$ satisfying $n\ge 4$
the following formula is valid:
\begin{multline}
\label{eq:carea:g0:Elise}
\carea(\cQ_{0,n}) \cdot \Vol(\cQ_{0,n})=
\frac{1}{8}
 \sum_{\substack{ n_1+n_2=n+2 \\  n_i\geq 3}}
 \frac{(n-4)!}{(n_1-4)!(n_2-4)!}
\cdot
\\
\cdot
 \frac{n!}{(n_1-1)!(n_2-1)!}
\cdot\frac{(2n_1-7)!(2n_2-7)!}{(2n-7)!}
\Vol(\cQ_{0, n_1})\times\Vol(\cQ_{0,n_2})
\end{multline}
\end{NNTheorem}

Note that the expressions appearing in the right-hand sides
of~\eqref{eq:carea}, \eqref{eq:carea:Elise}
and~\eqref{eq:carea:g0:Elise} can be seen as polynomials in
correlators. More precisely, in the definition of $N_{g_v,
n_v}(\boldsymbol{b}_v)$ one can keep the
correlators
$\langle\psi_1^{d_1} \ldots \psi_k^{d_k}\rangle
=\langle\tau_{d_1} \ldots \tau_{d_k}\rangle$
in~\eqref{eq:c:subscript:d}
without evaluation. We extend the operators $\cZ$
and $\partial_\Gamma$
to polynomials in the variables $b_e$ and
in ``unevaluated'' correlators
by linearity.
For example, under such convention one gets
\[
\Vol(\cQ_{0,5})
=
\frac{\pi^2}{9}
\left( 5 \langle \tau_0^3 \tau_1\rangle\, \langle \tau_0^3 \rangle
+ 4 \langle \tau_0^3 \rangle^3 \right)
\quad \text{and} \quad
\frac{\pi^2}{3} \carea(\cQ_{0,5}) \Vol(\cQ_{0,5}) = \frac{5}{9} \langle \tau_0^3 \rangle^3.
\]
More values for volumes and Siegel--Veech constats are
presented in Tables~\ref{tab:Vol:as:polynomials}
and~\ref{tab:SV:as:polynomials}
in Appendix~\ref{a:tables}.

Viewed in this way, the right-hand sides
of~\eqref{eq:carea} and of~\eqref{eq:carea:Elise} in the
case of $g\ge 1$ (respectively, the right-hand sides
of~\eqref{eq:carea} and of~\eqref{eq:carea:g0:Elise}
in the case $g=0$)
provide identities between polynomials in intersection numbers.
We show that these identities are, actually, trivial.

\begin{Theorem}
\label{th:same:SV}
The right-hand sides of~\eqref{eq:carea} and
of~\eqref{eq:carea:Elise} for $g\ge 1$
(respectively, the right-hand sides of~\eqref{eq:carea} and
of~\eqref{eq:carea:g0:Elise} for $g=0$) considered
as polynomials in intersection numbers of $\psi$-classes
coincide.
\end{Theorem}

Theorem~\ref{th:same:SV} is proved in
Section~\ref{eq:proof:of:coincidence:of:two:SV:expressions}.

\subsection{Frequencies of multicurves (after M.~Mirzakhani)}
\label{ss:Frequencies:of:simple:closed:curves}

We will say that two integral multicurves on the same
smooth surface of genus $g$ with $n$ punctures
\textit{have the same topological type}
if they belong to the same orbit of the mapping class group
$\Mod_{g,n}$.

We change now flat setting to hyperbolic setting. Following
M.~Mirzakhani, given an integral multicurve $\gamma$ in
$\cML_{g,n}(\Z)$ and a hyperbolic surface $X\in\cT_{g,n}$
consider the function $s_X(L,\gamma)$ counting the number
of simple closed geodesic multicurves on $X$ of length at
most $L$ of the same topological type as $\gamma$.
M.~Mirzakhani proves
in~\cite{Mirzakhani:grouth:of:simple:geodesics} the
following Theorem.
\begin{NNTheorem}[M.~Mirzakhani]
For any rational multi-curve $\gamma$ and
any hyperbolic surface $X\in\cT_{g,n}$,
$$
s_X(L,\gamma)\sim B(X)\cdot\frac{c(\gamma)}{b_{g,n}}\cdot L^{6g-6+2n}\,,
$$
as $L\to+\infty$.
\end{NNTheorem}

The coefficient in this asymptotic formula has beautiful
structure. All information about the hyperbolic metric $X$
is carried by the factor $B(X)$. Note that this
factor does not depend on the multicurve $\gamma$.
The information about $\gamma$ is carried by the
coefficient $c(\gamma)$ which in turn does not depend on the
hyperbolic metric,
but only on the topological type of $\gamma$, i.e., it is one
and the same for all multicurves in the orbit
$[\gamma]=\Mod_{g,n}\cdot\gamma$ of $\gamma$ under the mapping class
group.

The factor $B(X)$ has the following geometric meaning.
Consider the unit ball
$B_X=\{\gamma\in\cML_{g,n}\,|\,\ell_X(\gamma)\le 1\}$
defined by means of the length function $\ell_X$. The
factor $B(X)$ is the Thurston's measure of $B_X$:
$$
B(X)=\mu_{\mathrm{Th}}(B_X)\,.
$$

The factor $b_{g,n}$ depends only on $g$ and $n$. It is
defined as the average of $B(X)$ over $\cM_{g,n}$ viewed as
the moduli space of hyperbolic metrics, where the average
is taken with respect to the Weil--Petersson volume form on
$\cM_{g,n}$:
\begin{equation}
\label{eq:b:g:n}
b_{g,n}=\int_{\cM_{g,n}} B(X)\,dX\,.
\end{equation}

Mirzakhani showed that
\begin{equation}
\label{eq:b:g:n:as:sum:of:c:gamma}
b_{g,n}=\sum_{[\gamma]\in\mathcal{O}} c(\gamma)\,,
\end{equation}
where the sum of $c(\gamma)$ taken with respect to
representatives $[\gamma]$ of all orbits $\mathcal{O}$ of
the mapping class group $\Mod_{g,n}$ in $\cML_{g,n}(\Z)$.
This allows to interpret the ratio
$\tfrac{c(\gamma)}{b_{g,n}}$ as the probability to get a
multicurve of type $\gamma$ taking a ``large random''
multicurve (in the same sense as the probability that
coordinates of a ``random'' point in $\Z^2$ are coprime
equals $\tfrac{6}{\pi^2}$). More precisely, M.~Mirzakhani
showed that the asymptotic frequency
$\frac{c(\gamma)}{b_{g,n}}$ represents the density of the
orbit $\Mod_{g,n}\cdot\gamma$ inside the set of all
integral simple closed multicurves $\cML_{g,n}(\Z)$. This
density is analogous to the density $\frac{6}{\pi^2}$ of
integral points with coprime coordinates in $\Z^2$
represented by the $\SLZ$-orbit of the vector $(1,0)$.

\begin{figure}[htb]
   %
   %
\includegraphics{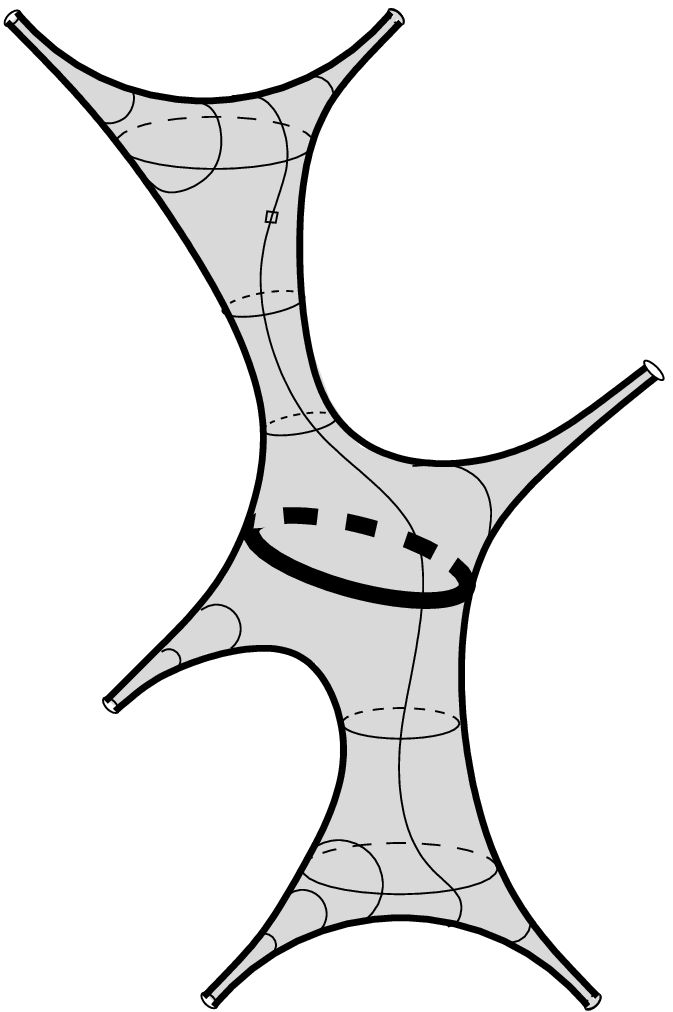}

\includegraphics{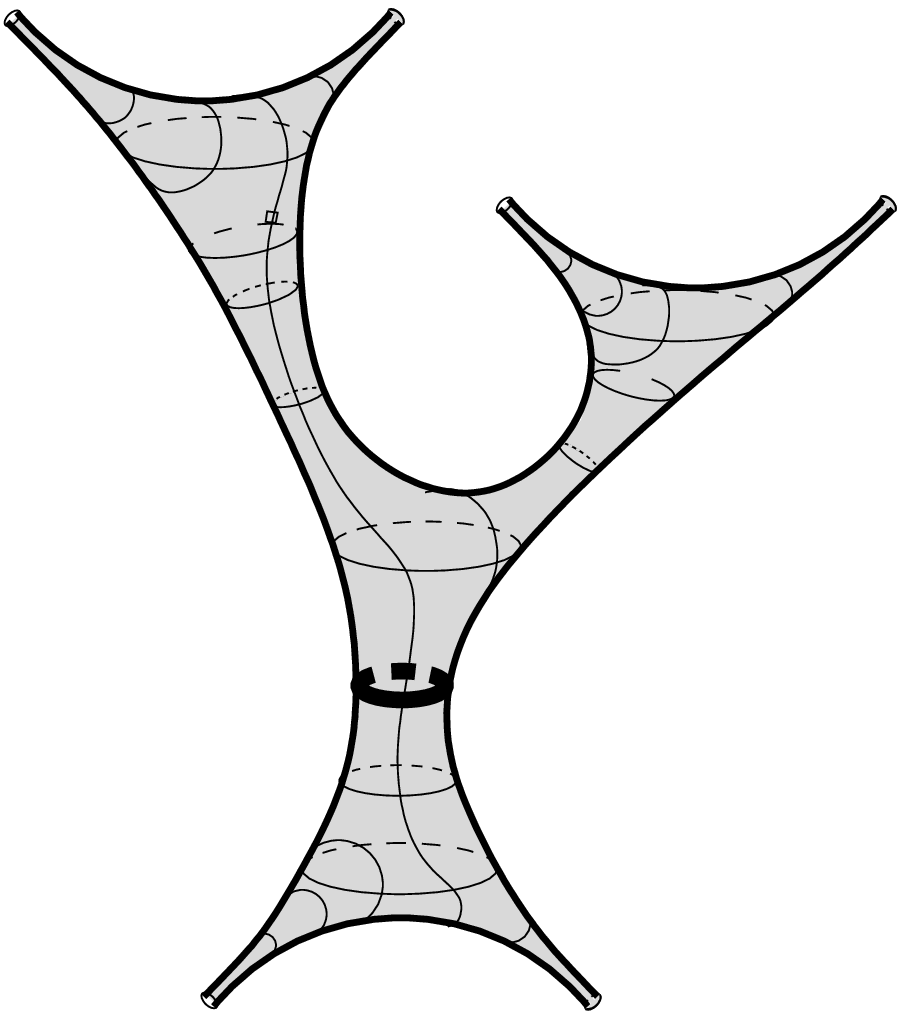}

\vspace{110pt}
\caption{
\label{fig:six:punctured:spheres}
Simple closed curves on a six-punctured sphere}
\end{figure}

M.~Mirzakhani found an explicit expression for the
coefficient $c(\gamma)$ and for the global normalization
constant $b_{g,n}$ in terms of the intersection numbers of
$\psi$-classes.

\begin{Example}
\label{ex:c:gamma:for:six:punctured:spheres} For any
hyperbolic metric $X$ on a sphere with $6$ cusps as in
Figure~\ref{fig:six:punctured:spheres}, a long simple
closed geodesic separates the cusps into groups of $3+3$
cusps with probability $\frac{4}{7}$ and into $2+4$ cusps
with probability $\frac{3}{7}$ (see (2) on page~123
in~\cite{Mirzakhani:grouth:of:simple:geodesics} for
calculation).
\end{Example}

\begin{Remark}
These values we confirmed experimentally in 2017 by M.~Bell
and S.~Schleimer. They were also confirmed by more implicit
independent computer experiment by V.~Delecroix.
\end{Remark}

\subsection{Frequencies of square-tiled surfaces of fixed
combinatorial type}
\label{ss:Frequencies:of:square:tiled:surfaces}

The following Theorem bridges flat and hyperbolic
count.

\begin{Theorem}
\label{th:our:density:equals:Mirzakhani:density}
For any integral multicurve $\gamma\in\cML_{g,n}(\Z)$,
the volume contribution $\Vol(\gamma)$
to the Masur--Veech volume $\Vol\cQ_{g,n}$
coincides with the Mirzakhani's
asymptotic frequency $c(\gamma)$ of simple
closed geodesic multicurves of topological type $\gamma$
up to the explicit factor $const_{g,n}$
depending only on $g$ and $n$:
\begin{equation}
\label{eq:Vol:gamma:c:gamma}
\Vol(\gamma)
=const_{g,n}\cdot c(\gamma)\,,
\end{equation}
where
\begin{equation}
\label{eq:const:g:n}
const_{g,n}
=2\cdot(6g-6+2n)\cdot
(4g-4+n)!\cdot 2^{4g-3+n}\cdot
\end{equation}
\end{Theorem}
Theorem~\ref{th:our:density:equals:Mirzakhani:density}
is proved in Section~\ref{s:comparison:with:Mirzakhani}.

\begin{Example}
\label{ex:Vol:gamma:for:six:punctured:spheres} A
one-cylinder square-tiled surface in the moduli space
$\cQ_{0,6}$ can have $3$ simple poles on each of the two
boundary component of the maximal horizontal cylinder or
can have $2$ simple poles on one boundary component and $4$
simple poles on the other boundary component. The
asymptotic frequency of
square-tiled surfaces of the first type is
$\frac{4}{7}$ and the asymptotic frequency of the
square-tiled surfaces of second type
is $\frac{3}{7}$; compare to
Example~\ref{ex:c:gamma:for:six:punctured:spheres}.
\end{Example}

Combining~\eqref{eq:Vol:Q:as:sum:of:Vol:gamma},
\eqref{eq:b:g:n:as:sum:of:c:gamma},
\eqref{eq:Vol:gamma:c:gamma} and~\eqref{eq:const:g:n} we
get the following immediate Corollary:

\begin{Corollary}
\label{cor:Vol:as:b:g:n}
For any admissible pair of nonnegative integers $(g,n)$,
the Masur--Veech volume $\Vol\cQ_{g,n}$ and the average
Thurston measure of a unit ball $b_{g,n}$ are related as
follows:

\begin{equation}
\label{eq:Vol:g:n:b:g:n}
\Vol\cQ_{g,n}
=2\cdot(6g-6+2n)\cdot
(4g-4+n)!\cdot 2^{4g-3+n}\cdot
b_{g,n}\,.
\end{equation}
\end{Corollary}

\begin{Remark}
In Theorem~1.4 in~\cite{Mirzakhani:earthquake}
M.~Mirzakhani established the relation
$$
\Vol\cQ_{g}
=
const_{g}\cdot
b_{g}\,,
$$
where $b_{g}$ is computed in Theorem~5.3
in~\cite{Mirzakhani:grouth:of:simple:geodesics}.
Our main formula~\eqref{eq:contribution:of:gamma:to:volume}
has basically the same structure as Mirzakhani's formula for $b_{g,n}$.
We provide a detailed comparison of these two formulae in
Section~\ref{s:comparison:with:Mirzakhani}.

However, Mirzakhani but does not give any close formula for
the value of the normalization constant $const_g$. This
constant was recently computed by
F.~Arana--Herrera~\cite{Arana:Herrera} and by L.~Monin and
I.~Telpukhovkiy~\cite{Monin:Telpukhovskiy} simultaneously
and independently of us by different methods.
\end{Remark}

Formula~\eqref{eq:square:tiled:volume} from
Theorem~\ref{th:volume} below
allows to compute $\Vol\cQ_{g,n}$ for
all sufficiently small values of $(g,n)$.
Corollary~\ref{cor:Vol:as:b:g:n}
immediately provides explicit values of $b_{g,n}$
for all such pairs.

When $g=0$ the value $\Vol\cQ_{0,n}$ admits closed
formula~\eqref{eq:vol:genus:0} obtained
in~\cite{AEZ:genus:0}. Corollary~\ref{cor:Vol:as:b:g:n}
translates this formula into the following explicit
expression for $b_{0,n}$.

\begin{Corollary}
The quantity $b_{0,n}$ defined in~\eqref{eq:b:g:n}
has the following value:
\begin{equation}
b_{0,n}=
\frac{1}{(n-3)!}
\cdot
\left(\frac{\pi}{2}\right)^{2(n-3)}\,.
\end{equation}
\end{Corollary}
By Stirling formula we get the following asymptotics for large
$n$:
\begin{equation}
b_{0,n}
\sim\frac{1}{\sqrt{2\pi n}}\cdot
\left(\frac{\pi^2 e}{4n}\right)^{n-3}\quad\text{as}\ n\to+\infty\,.
\end{equation}

\subsection{Statistical geometry of square-tiled surfaces}
\label{ss:Combinatorial:geometry:of:random:st:surfaces}

Theorem~\ref{th:volume} provides a detailed description of
statistical geometric properties of square-tiled surfaces
in $\cQ_{g,n}$ tiled with large number of squares in the
same spirit as the result~\cite[Theorem
 1.2]{Mirzakhani:statistics} of M.~Mirzakhani describing
statistics of lengths of simple closed geodesics in random
pants decomposition. More precisely, Mirzakhani fixes a
reduced multicurve $\gamma=\gamma_1+\dots+\gamma_{3g-3}$
decomposing the surface of genus $g$ into pair of pants and
considers its $\Mod_g$-orbit. For any hyperbolic metric,
M.~Mirzakhani describes the asymptotic distribution of
(normalized) lengths of simple closed geodesics represented
by the components $\gamma_i$, $i=1,\dots,3g-3$, of the
multicurve $\gamma$.

Our result concerns, in particular, the asymptotic
statistics of (normalized) perimeters of a random
square-tiled surface corresponding to a given stable graph
$\Gamma$ and tiled with large number of squares. The
resulting statistics disclose the geometric meaning of
coefficients of the polynomials $P_\Gamma$ associated to a
stable graph $\Gamma$ appearing in our formulae for the
Masur--Veech volumes as in Theorem~\ref{th:volume} and for
the Siegel-Veech constants as in Theorem~\ref{th:carea}.

As we have seen in
Section~\ref{ss:Frequencies:of:square:tiled:surfaces}
asymptotic statistical properties of random square-tiled
surfaces can be translated into asymptotic statistical
properties of geodesic multicurves on random hyperbolic
surfaces and vice versa. This general correspondence
translates the results mentioned above into analogs of a
mean version of Theorem 1.2 in~\cite{Mirzakhani:statistics},
in the sense that we obtain the average of her
statistics averaging over all hyperbolic surfaces in
$\cM_g$, where the average is computed using the
Weil--Petersson measure on $\cM_g$.

Let us define the operator
$\cZt(\boldsymbol{x},\boldsymbol{H})$ on polynomials $\Q[b_1, \ldots, b_k]$ as
follows. We define it on monomials
\[
\cZt(\boldsymbol{x},\boldsymbol{H}) \left( \prod_{i=1}^k b_i^{m_i} \right)
=
\prod_{i=1}^k \frac{x_i^{m_i}}{H_i^{m_i+1}}
\]
and extend it by $\Q$-linearity. We denote by
$$
\Delta^k :=
\{(x_1, \ldots, x_k)\,\vert\, x_i \geq 0;\, x_1 + \ldots + x_k \leq
1\}
$$
the standard $k$-dimensional simplex. The operator $\cZt$ generalizes
both $\cY(\boldsymbol{H})$ and $\cZ$ in the sense that we can recover
$\cY(\boldsymbol{H})$ and $\cZ$ by integration
\begin{align}
\label{eq:recover:Z}
\cZ(P) &= \sum_{H_i \geq 1} \int_{\Delta_k}
\cZt(\boldsymbol{x},\boldsymbol{H})(P)\, dx_1 \ldots dx_k\,.
\\
\label{eq:recover:Y}
\cY(\boldsymbol{H}) (P) &= \int_{\Delta_k}
\cZt(\boldsymbol{x},\boldsymbol{H})(P)\, dx_1 \ldots dx_k\,,
\end{align}

To any square-tiled surface $S$ in $\cST_{g,n}(N)$ we associate
the following data
\[
\phi_{N}(S) := \left(\Gamma, \boldsymbol{H},
\frac{\boldsymbol{b}}{N}\right)
\in \cG_{g,n} \times \N^k \times \R^k\,,
\]
where $\Gamma$ is the stable graph associated to the
horizontal cylinder decomposition,
$k$ is the number of maximal horizontal cylinders
(i.e., number of edges of $\Gamma$),
and
$\boldsymbol{H}=(H_1, \ldots, H_k)$ and
$\boldsymbol{b} = (b_1, \ldots, b_k)$
are respectively the heights and perimeters
of the maximal horizontal cylinders
measured in those units, in which the square of the tiling
has unit sides.
Considering the set
$\cST_{g,n}(2N)$ of all square-tiled
surfaces in $\cQ_{g,n}$
tiled with at most $2N$ squares (compare
to~\eqref{eq:Vol:sq:tiled}) we obtain a measure of finite
mass on $\Gamma \times \N^k \times \R^k$
\[
\tilde\mu_{g,n,N} := 2 (6g-6+2n)\cdot \frac{1}{N^d}
\sum_{S \in \cST_{g,n}(2N)} \frac{1}{|\Aut(S)|}
\,\delta_{\phi(S)}\,,
\]
where $\delta_y$ is the Dirac mass at $y \in \cG_{g,n}
\times \N^k \times \R^k$.

Similarly, for each stable graph $\Gamma\in\cG_{g,n}$, each
$\boldsymbol{H}\in\N^k$, and each $N\in\N$ we can define the
following measure $\mu^{\gamma(\Gamma,\boldsymbol{H})}_{g,n,N}$ on the simplex
$\Delta^k$:
\[
\mu^{\gamma(\Gamma,\boldsymbol{H})}_{g,n,N}
:= 2 (6g-6+2n)\cdot \frac{1}{N^d}
\sum_{S \in \cST_{g,n}(2N)\cap\cSTgn(\Gamma,\boldsymbol{H})} \frac{1}{|\Aut(S)|}
\,\delta_{\phi(S)}\,.
\]
Here $\cSTgn(\Gamma,\boldsymbol{H})$ is the set of square-tiled surfaces
associated to the stable graph $\Gamma$ and having the
vector of heights $\boldsymbol{H}$.

Finally, we can desintegrate the discrete
part $(\Gamma,h)$ and obtain the decomposition
\[
\mu_{g,n,N} = \sum_{\Gamma, \boldsymbol{H}}
\mu^{\gamma(\Gamma,\boldsymbol{H})}_{g,n,N}\,.
\]

Fix $g \geq 0$ and $n \geq 0$ so that $2 g + n \geq 4$. Let
$\tilde\mu_{g,n,N}$, $\mu^{\gamma(\Gamma,\boldsymbol{H})}_{g,n,N}$
and $\mu_{g,n,N}$ be the measures defined above.

\begin{Theorem}
\label{thm:statistics}
We have the following weak convergence of
measures:
\begin{equation}
\label{eq:measure:tilde:mu}
(\tilde\mu_{g,n,N})_{N} \to \sum_{\Gamma,\boldsymbol{H}}
\cZt(\boldsymbol{x},\boldsymbol{H})(P_\Gamma)\,
\delta_\Gamma \otimes \delta_h \otimes d x
\quad\text{as }N\to+\infty\,.
\end{equation}

For each stable graph $\Gamma\in\cG_{g,n}$ and each
$\boldsymbol{H}\in\N^k$ we have:
\begin{equation}
\label{eq:measure:mu:Gamma:H}
\left(\mu^{\gamma(\Gamma,\boldsymbol{H})}_{g,n,N}\right)_N \to
\cZt(\boldsymbol{x},\boldsymbol{H})(P_\Gamma)\, dx
\quad\text{as }N\to+\infty\,.
\end{equation}
Here $P_\Gamma = P_\Gamma(b_1, \ldots, b_k)$
is the global polynomial associated to the stable graph
$\Gamma$ by Formula~\eqref{eq:P:Gamma} and $dx$ is the
Lebesgue measure on the simplex $\Delta^k$.

Similarly, we have the following weak convergence of measures:
\begin{equation}
\label{eq:measure:mu}
(\mu_{g,n,N})_{N} \to \sum_{\Gamma,\boldsymbol{H}}
\cZt(x,\boldsymbol{H})(P_\Gamma)\,d x\,.
\end{equation}
\end{Theorem}

Theorem~\ref{thm:statistics}
is proved in Section~\ref{eq:conditions:on:lengths}.

Comparing~\eqref{eq:recover:Z}
and~\eqref{eq:recover:Y}
with respectively~\eqref{eq:Vol:Gamma}
and~\eqref{eq:contribution:of:gamma:to:volume}
we recover the following global normalizations
of the resulting measures:
\begin{align*}
\int_{\Delta^k}&
\cZt(x,\boldsymbol{H})(P_\Gamma)\, dx
=\Vol(\Gamma,\boldsymbol{H})
\\
\int_{\Delta^k}&
\sum_{\Gamma,\boldsymbol{H}}
\cZt(x,\boldsymbol{H})(P_\Gamma)\,d x\,,
=
\sum_{\Gamma,\boldsymbol{H}}
\int_{\Delta^k}
\cZt(x,\boldsymbol{H})(P_\Gamma)\,d x
=
\Vol\cQ_{g,n}
\end{align*}

The above theorem allows us to describe statistics of
random-square tiled surfaces. For example, we can compute
the asymptotic probability that a random square-tiled
surface tiled with a large number of squares corresponds a
given stable graph $\Gamma$. Considering only square-tiled
surfaces associated to a given stable graph $\Gamma$, we
can compute asymptotic distributions of the heights
$\boldsymbol{H}$ of the maximal horizontal cylinders and
asymptotic distribution of their areas normalized by the
area of the surface. We can also compute asymptotic
statistics of perimeters of the cylinders under appropriate
normalization; for example statistics of the ratios of any
two perimeters. Note, that for the ratios of length
variables, the unit of measurement becomes irrelevant, in
particular,
$$
\frac{H_i}{H_j}=\frac{h_i}{h_j}
\qquad\text{and}\quad
\frac{b_i}{b_j}=\frac{w_i}{w_j}\,.
$$

We will use the notation $\E_\Gamma$
(respectively $\E_{\Gamma,H}$) to
denote the asymptotic expectation values of quantities
evaluated on square-tiled surfaces with given cylinder
decomposition associated to $\Gamma$ (respectively
associated to $\Gamma$ and given heights $H$).

Let us consider several simple examples.
Consider the following stable graph $\Phi$ in $\cG_{2,0}$
and the associated reduced multicurve:

\includegraphics{genus_two_graph_22.eps}
\includegraphics{genus_two_22.eps}
\begin{picture}(0,0)(-48,1)
\put(-19,-10){$b_1$}
\put(23,-16){$b_2$}
\put(6,-10){$0$}
\put(41,-10){$1$}
\put(-40,-10){$\Phi$}
\put(117,-10){$b_1$}
\put(164,-23){$b_2$}
\end{picture}
   %

\noindent
It was computed in Table~\ref{tab:2:0} that
$\Vol(\Phi)=\tfrac{1}{675}\cdot\pi^6$ and
$\Vol\cQ_{2,0}=\tfrac{1}{15}\cdot\pi^6$.
Thus, a random square-tiled surface in $\cQ_2$
(tiled with very large number of squares)
corresponds to the stable graph $\Phi$
with (asymptotic) probability $\tfrac{1}{45}$.

We have also computed in Table~\ref{tab:2:0}
the polynomial
$P_\Phi = \tfrac{2}{15} b_1 b_2^3$.
For any given $\boldsymbol{H}= (H_1, H_2)$ we get
\[
\E_{\Phi,\boldsymbol{H}}\left(\frac{b_1}{b_2}\right)
\ =\
\frac{2! \cdot 2!}{1! \cdot 3!}\ \frac{H_2}{H_1}
\ =\
\frac{2}{3} \cdot \frac{H_2}{H_1}\,.
\]
In other words, if we impose to a square-tiled surface
``of type $\Phi$'' to have cylinders of the same height,
then the perimeter $b_2$ of the second cylinder is in average $\frac{2}{3}$ shorter
than the perimeter $b_1$ of the first cylinder. However, if
we impose a large height to the first cylinder,
its perimeter becomes proportionally short in the average,
which is quite natural.

What might seem somehow counterintuitive is that
if we do not fix $H$, we obtain
\[
\E_\Phi\left(\frac{b_1}{b_2}\right)
= \frac{2}{3} \cdot \frac{\zeta(3)^2}{\zeta(2) \zeta(4)}
\approx 3.90238\,,
\quad \text{while} \quad
\E_\Phi\left(\frac{b_2}{b_1}\right) = +\infty\,.
\]

Now consider the following graph $\Petal_2(2)$ with two edges:

\includegraphics{genus_two_graph_21.eps}
\begin{picture}(0,0)(-40,0)
\put(1,-19){$b_1$}
\put(62,-19){$b_2$}
\put(33,-30){$0$}
\put(-30,-19){$\Petal_2(2)$}
\end{picture}

\includegraphics{genus_two_21.eps}
\begin{picture}(0,0)(-189,9)
\put(-24,0){$b_1$}
\put(66,0){$b_2$}
\end{picture}
\vspace*{30pt}

\noindent
It was computed in Table~\ref{tab:2:0} that
$P_{\Petal_2(2)} = \tfrac{8}{5} (b_1^3 b_2 + b_1 b_2^3)$. Then
for any given $\boldsymbol{H} = (H_1, H_2)$ we have
\[
\E_{\Gamma,\boldsymbol{H}}\left(\frac{b_1}{b_2}\right) =
\frac{\frac{4!}{H_1^5 H_2} + \frac{2!\cdot 2!}{H_1^3 H_2^3}}
{\frac{3!}{H_1^4 H_2^2} + \frac{3!}{H_1^2 H_2^4}}
=
\frac{2 H_2 \left(H_1^2+6 H_2^2\right)}{3 H_1
   \left(H_1^2+H_2^2\right)}\,.
\]
In particular, in the symmetric case, when $H_1=H_2$, we have
$$
\E_{\Gamma,\boldsymbol{H}}\left(\frac{b_1}{b_2}\right) =
\E_{\Gamma,\boldsymbol{H}}\left(\frac{b_2}{b_1}\right) =
\frac{7}{3}\,.
$$

Note also, that we get for free the averaged version of
\cite[Theorem~1.2]{Mirzakhani:statistics}. Namely, when the
stable graph $\Gamma\in\cG_{g,0}$ has maximal possible
number $3g-3$ of vertices (i.e., when the corresponding
multicurve provides a pants decomposition of the surface),
equation~\eqref{eq:P:Gamma} for $P_\Gamma$
takes the following form:
$$
P_\Gamma(\boldsymbol{b})
=(\text{numerical factor})\cdot
b_1\dots b_{3g-3}\cdot\prod_{v\in V(\Gamma)}
N_{0,3}(\boldsymbol{b}_v)\,.
$$
Since $N_{0,3}=1$ identically, we conclude that for
$H_1=H_2=\dots=H_{3g-3}$ the density function of lengths
statistics is the product $b_1\cdot\dots\cdot b_{3g-3}$ up
to a constant normalization factor. Mirzakhani proves in
\cite[Theorem~1.2]{Mirzakhani:statistics} that the same
asymptotic length statistics is valid for any individual
hyperbolic surface in $\cM_g$ (and not only in average, as
we do).

We complete this section considering two examples
describing statistics of heights of the maximal
horizontal cylinders of a random square-tiled surfaces
(equivalently, statistics of weights of a random
integer multicurve in $\cML_{g,n}$. We start with the following
elementary Lemma.

\begin{Lemma}
\label{lm:height:of:one:cylinder}
Consider a random square-tiled surface in $\cQ_{g,n}$
having a single maximal horizontal cylinder. The asymptotic
probability that this cylinder is represented by a single band
of squares (i.e. that $H_1=1$) equals
$\cfrac{1}{\zeta(6g-6+2n)}$.
\end{Lemma}
\begin{proof}
When a stable graph
$\Gamma\in\cG_{g,n}$ has a single edge $b_1$,
Formula~\eqref{eq:P:Gamma} gives
$P_\Gamma=(\text{numerical factor})
\cdot b_1^{6g-7+2n}$,
and the Lemma follows.
\end{proof}

In terms of multicurves this means that a random
single-component integral multicurve $n\gamma$ in $\cML(\Z)$ (where
$n\in\N$ and $\gamma$ is a simple closed curve) is reduced
(i.e. $n=1$) with asymptotic probability
$\frac{1}{\zeta(6g-6+2n)}$.

Note that $\zeta(x)$ tends to $1$ exponentially rapidly as
the real-valued argument $x$ grows. Thus, our result
implies, that when at least one of $g$ or $n$ is large
enough, a random one-cylinder square-tiled surface is tiled
with a single horizontal band of squares with a very large
probability, and a random single-component integral
multicurve is just a simple closed curve with
a very large probability.

The polynomial $P_\Gamma$ enables to compute analogous
probabilities for any given stable graph $\Gamma$. For
example, a random square-tiled surface in $\cQ_2$
associated to the stable graph $\Phi$, considered earlier
in this section, has both cylinders of height $H_1=H_2=1$
with probability
$$
\frac{\cY(1,1)(P_\Phi)}{\cZ(P_\Phi)}
=\frac{1}{\zeta(2)\zeta(4)}
=\frac{540}{\pi^6}\approx 0.561687\,.
$$
A random square-tiled surface in $\cQ_2$ associated to the
stable graph $\Petal_2(2)$, considered earlier in this
section, has heights $H_i$ of both horizontal cylinders
bounded by $2$ with probability
\begin{multline*}
\frac{\cY(1,1)(P_{\Petal_2(2)})
+\cY(1,2)(P_{\Petal_2(2)})
+\cY(2,1)(P_{\Petal_2(2)})
+\cY(2,2)(P_{\Petal_2(2)})}{\cZ(P_{\Petal_2(2)})}
\ =\\=\
\frac{2+\frac{2}{16}+\frac{2}{4}+\frac{2}{64}}{2\zeta(2)\zeta(4)}
=\frac{\frac{85}{64}}{\zeta(2)\zeta(4)}
\approx 0.745991\,.
\end{multline*}

\subsection{Random square-tiled surfaces
and multicurves in large genera}
\label{ss:large:genera}

To describe statistical geometry of square-tiled surfaces
in $\cQ_{g,n}$ for any given couple $(g,n)$, one has to
consider all stable graphs $\Gamma$ in $\cQ_{g,n}$ and
apply the technique developed in
Section~\ref{ss:Combinatorial:geometry:of:random:st:surfaces}
to each graph. However, the number of stable graphs grows
rapidly as at least one of $g$ or $n$ grows. In this
section we show how can one overcome the problem of growing
multitude of stable graphs and we describe large genus
asymptotics of statistical geometry of square-tiled
surfaces and of geodesic multicurves. For simplicity we let
$n=0$ everywhere throughout this section. In terms of
square-tiled surfaces this means that the surface does not
have conical singularities of angle $\pi$. In terms of
hyperbolic multicurves this means that the underlying
hyperbolic surface does not have cusps.

We start with the simplest case of
one-cylinder square-tiled surfaces, or, equivalently,
with the case of simple closed curves.

\begin{Theorem}
\label{th:separating:over:non:separating}
The frequency of separating simple closed geodesics
on a closed hyperbolic surface of large genus $g$ is exponentially
small with respect to the frequency of non-separating
simple closed geodesics:
\begin{equation}
\label{eq:log:sep:over:non:sep}
\frac{c(\gamma_{sep})}{c(\gamma_{nonsep})}
\sim
\sqrt{\frac{2}{3\pi g}}\cdot\frac{1}{4^g}\quad\text{as }g\to+\infty\,,
\end{equation}
\end{Theorem}
Here and below the equivalence means that the ratio of the two
expressions tends to $1$ as $g\to+\infty$.

Theorem~\ref{th:separating:over:non:separating} is proved
in Section~\ref{ss:simple:closed:geodesics}. The proof is
based on the large genus asymptotic formulae for
$2$-correlators
$\langle \psi_1^{d_1}\psi_2^{d_2}\rangle$
uniform for all partitions $d_1+d_2=3g-1$.
This formula is
obtained in Section~\ref{s:2:correlators} using results
of~\cite{Zograf:2:correlators}.
\smallskip

\noindent
\textbf{Conjectural uniform large genus asymptotics of correlators.}
Conjecturally, analogous uniform large genus asymptotic
formulae are also valid for $n$-correlators $\langle
\psi_1^{d_1}\dots\psi_k^{d_n}\rangle$ for any fixed $n$ and
even for $n$ growing logarithmically with respect to $g$.
We discuss the corresponding conjectures in more details in
Appendix~\ref{s:Conjectural:asymptotics:of:correlators}. We
show that the following weaker form of this conjectures is
already sufficient to describe the statistical geometry of
square-tiled surfaces in $\cQ_g$ and of integer multicurves
in $\cML_g$ for large genera $g$.

Denote by $\Pi(3g-3-k,k)$ the set of partitions
$D_1+\dots+D_k=3g-3-k$ of a positive integer $3g-3-k$ into
$k$ nonnegative integers. For any $D\in\Pi(3g-3-k,k)$
we define
\begin{multline}
\Psi(D,k):=
\frac{(g-k)!(3g-3-k)!}{(6g-2k-5)!}
\cdot
\frac{24^{g-k}}{2^{6g-6-k}}
\cdot
\prod_{j=1}^k \frac{(2D_j+2)!}{D_j!}
\cdot\\
\cdot
\int_{\overline{\cM}_{g-k,2k}}
(\psi_1+\psi_2)^{D_1}\dots(\psi_{2k-1}+\psi_{2k})^{D_k}\,.
\end{multline}
By $[x]$ we denote the integer part of $x\in\R$.
For any $C>0$ we define the following two quantities:
\begin{align}
\psi_{min}(g,C)&:=
\min_{1\le k\le [C\log(g)]}
\
\min_{D\in\Pi(3g-3-k,k)}
\Psi(D,k)\,,
\\
\psi_{max}(g,C)&:=
\max_{1\le k\le [C\log(g)]}
\
\max_{D\in\Pi(3g-3-k,k)}
\Psi(D,k)\,.
\end{align}
\begin{Conjecture}
\label{conj:introduction:sum:of:correlators}
For any $C<2$ we have
\begin{equation}
\label{eq:min:max:psi:g:C}
\lim_{g\to+\infty} \psi_{min}(g,C)
=\lim_{g\to+\infty} \psi_{min}(g,C)
=1\,.\end{equation}
\end{Conjecture}
\begin{Remark}
For what follows it would be sufficient to prove
Conjecture~\ref{conj:introduction:sum:of:correlators} for
any particular $C>\tfrac{1}{2}$. Moreover, for the purposes
of this article a weaker conjecture would be sufficient:
we use only certain weighted sums of
correlators as above and, essentially, only for $k$ in the interval
$\big[(\tfrac{1}{2}-\epsilon)\log(g);
\big[(\tfrac{1}{2}+\epsilon)\log(g)]$. Thus, for $k$ in the
complement of this range less restrictive estimates for
correlators would also work.
\end{Remark}
\smallskip

\noindent
\textbf{Two multiple harmonic sums.~}
Consider the following two sums
\begin{align}
\label{eq:multiple:harmonic:sum:def}
H_k(m)=\sum_{j_1+\dots+j_k=m}&
\frac{1}{j_1\cdot j_2\cdots j_k}\,,
\\
\label{eq:multiple:harmonic:zeta:sum:def}
Z_k(m)=\sum_{j_1+\dots+j_k=m}&
\frac{\zeta(2j_1)\cdots\zeta(2j_k)}
{j_1\cdot j_2\cdots j_k}\,.
\end{align}
Their asymptotic expansions are studied in detail
in Section~\ref{s:Two:harmonic:sums}.
Preparing the text we realized that we do not have
a proof of the appropriate
uniform version of the corresponding
asymptotic extension,
so we state it as a Conjecture.
This
Conjecture is a strengthened version of
Theorem~\ref{th:multiple:harmonic:zeta:sum:full:expansion}
from Appendix~\ref{ss:fine:asymptotic:expansions}.

We define in~\eqref{eq:def:A:j} (respectively
in~\eqref{eq:def:B:j}) certain specific sequences of
numbers $A_0, A_1, \dots$ (respectively $B_0, B_1, \dots$)
used in the Conjecture below.

\begin{Conjecture}
\label{conj:uniform:bound:for:error:term}
For any $C$ in the interval $]0;2[$ and for any integer $k$
satisfying $1\le k\le C\log(m)$, the following asymptotic
expansions are valid:
\begin{align}
\label{eq:multiple:harmonic:sum:full:expansion:log}
H_k(m)
&=\frac{k!}{m}\cdot
\left(\sum_{j=0}^{k-1}
\frac{A_j}{(k-1-j)!}
\cdot\big(\log m\big)^{k-1-j}
+\epsilon_k^H(m)\right)
\,,
\\
\label{eq:multiple:harmonic:zeta:sum:full:expansion:log}
Z_k(m)
&=\frac{k!}{m}\cdot
\left(\sum_{j=0}^{k-1}
\frac{B_j}{(k-1-j)!}
\cdot\big(\log m\big)^{k-1-j}
+\epsilon_k^Z(m)\right)
\,,
\end{align}
where
\begin{align}
\label{eq:max:epsilon:H}
\lim_{m\to+\infty}
\frac{1}{\sqrt{m}}\cdot
\sum_{k=1}^{[C\log(m)]}\frac{\epsilon_k^H(m)}{k!\cdot 2^{k-1}}=0\,.
\\
\label{eq:max:epsilon:Z}
\lim_{m\to+\infty}
\frac{1}{\sqrt{m}}\cdot
\sum_{k=1}^{[C\log(m)]}\frac{\epsilon_k^Z(m)}{k!\cdot 2^{k-1}}=0\,.
\end{align}
\end{Conjecture}

\begin{Remark}
In
Theorem~\ref{th:multiple:harmonic:zeta:sum:full:expansion}
from Appendix~\ref{ss:fine:asymptotic:expansions} we prove
that for any fixed $k$ we have $\epsilon_k^H(m)=o(1)$ and
$\epsilon_k^Z(m)=o(1)$ as $g\to+\infty$.
Conjecture~\ref{conj:uniform:bound:for:error:term} claims
that convergence is uniform for all $k\in[1;C\log(m)]$ in
the sense of~\eqref{eq:max:epsilon:H} and~\eqref{eq:max:epsilon:Z}.
\end{Remark}

\noindent\textbf{Geometry of random square-tiled surfaces
and of random multicurves in large genera.~}
Recall that stable graphs $\Gamma\in\cG_g=\cG_{g.0}$ are in
the natural correspondence with $\Mod_g$-orbits of reduced
integral multicurves
$\gamma_{\mathit{reduced}}=\gamma_1+\gamma_2+\dots+\gamma_k$,
where $\gamma_i$ and $\gamma_j$ are pairwise non-isotopic
for $i\neq j$. We say that an integral multicurve
$\gamma=a_1\gamma_1+\dots+a_k\gamma_k$ corresponds to the
stable graph $\Gamma$ if the corresponding reduced
multicurve
$\gamma_{\mathit{reduced}}=\gamma_1+\gamma_2+\dots+\gamma_k$
does. Recall also that for any stable graph $\Gamma$ the
number of vertices of $\Gamma$ is exactly the number of
connected components of the complement
$S_g\setminus\gamma_{reduced}$ (i.e. the number of pieces
in which $\gamma_{reduced}$ chops the topological surface
$S_g$), see
Section~\ref{ss:Square:tiled:surfaces:and:associated:multicurves}.

We use the notion ``\textit{random integral multicurve}'' in $\cML_g(\Z)$
and  ``\textit{random square-tiled surface}'' (tiled with large number
of squares) in the sense described in Section~\ref{ss:Combinatorial:geometry:of:random:st:surfaces}.

\begin{CondTheorem}
\label{cond:th:does:not:separate}
Conjectures~\ref{conj:Vol:Qg},
\ref{conj:introduction:sum:of:correlators}
and~\ref{conj:uniform:bound:for:error:term} imply that a
random integral multicurve $\gamma\in\cML_g(\Z)$ does not
separate the surface $S_g$ (i.e. that
$S_g\setminus\gamma_{reduced}$ is connected) with
probability, which tends to $1$ when $g\to+\infty$.

Equivalently, the same conjectures imply that
all conical points of a random square-tiled
surface in $\cQ_g$ belong to the same horizontal (and to
the same vertical) layer with probability, which tends to
$1$ as $g\to+\infty$.
\end{CondTheorem}
\begin{proof}
For any positive integer $k$ within bounds $1\le k\le g$,
denote by $\Petal_k(g)$ the stable graph in $\cG_g$ having
single vertex decorated with the integer $g-k$, and having
$k$ edges (loops). Theorem~\ref{th:volume} combined with
Conjecture~\ref{conj:introduction:sum:of:correlators} lead
to a simple expression for $\Vol(\Petal_k(g))$ in terms of the
sum~\eqref{eq:multiple:harmonic:zeta:sum:def}, see
Theorem~\ref{th:bounds:for:Vol:Gamma:k:g} and
Corollary~\ref{cor:Vol:g:k} in
Appendix~\ref{s:volume:contribution:from:graphs:with:single:vertex}.

By Conditional Theorem~\ref{cond:th:sum:over:one:cylinder:graphs}
in Appendix~\ref{ss:sum:over:single:vertex:graphs},
Conjectures~\ref{conj:introduction:sum:of:correlators}
and~\ref{conj:uniform:bound:for:error:term} imply together
that
$$
\sum_{k=1}^{[C\log(m)]} \Vol(\Petal_k(g))
\sim
\frac{4}{\pi}
\cdot\left(\frac{8}{3}\right)^{4g-4}
\quad\text{as }g\to+\infty\,.
$$

On the other hand, by Conjecture~\ref{conj:Vol:Qg}
$$
\sum_{\Gamma\in\cG_g} \Vol(\Gamma)
\sim
\frac{4}{\pi}
\cdot\left(\frac{8}{3}\right)^{4g-4}
\quad\text{as }g\to+\infty\,,
$$
which means that asymptotically, as $g\to+\infty$,
the total contribution to the Masur--Veech volume
$\Vol\cQ_g$ of all stable graphs in $\cG_g$
which are different from
$\Petal_k(g)$ with $k\le C\log(g)$,
tends to zero. The latter observation
completes the proof.
\end{proof}

\noindent\textbf{Number of cylinders of a random square-tiled surface.
Number of primitive components of a random multicurve.~}
Consider the number $k$ of maximal horizontal cylinders of
a random square-tiled surface in $\cQ_g$ (equivalently, the
number of components of the reduced multicurve
$\gamma_{reduced}=\gamma_1+\dots+\gamma_k$ associated to a
random integral multicurve $\gamma\in\cML_g(\Z)$) as an
integer random variable with values in $\N$ (where, by
definition, the probability that $k>3g-3$ is equal to
zero). Consider the corresponding probability distribution
$\cP$.

Suppose that there exists a constant $C$ in the interval
$]\tfrac{1}{2};2[$ for which both
Conjectures~\ref{conj:introduction:sum:of:correlators}
and~\ref{conj:uniform:bound:for:error:term} are valid. Let
us consider now only square-tiled surfaces (integral
multicurves) in $\cQ_g$ which are associated to one of the
graphs $\Petal_k(g)$ with $k$ in the range $1\le k\le
C\log(g)$. Consider corresponding random square-tiled
surfaces (equivalently, random multicurves) and consider
the number of maximal horizontal cylinders (equivalently,
the number of components of the associated reduced
multicurve) as a random variable. Consider the resulting
conditional probability distribution $\cP_{cond}$.

We prove in Conditional Theorem~\ref{th:Poisson} in
Appendix~\ref{ss:Approximation:by:Poisson} that
Conjectures~\ref{conj:Vol:Qg},
\ref{conj:introduction:sum:of:correlators}
and~\ref{conj:uniform:bound:for:error:term} imply that both
probability distributions $\cP_{cond}$ and $\cP$ converge
in total variation to the Poisson distribution with
parameter
\begin{equation*}
\lambda(g)=\frac{\log(6g-6)+\gamma}{2}+(\log2-1)\,,
\end{equation*}
when $g\to\infty$.
(Here we add $1$ to the Poisson random
variable, so that it takes values in $\{1,2,3,\dots\}$ and
not in $\{0,1,2,\dots\}$ as usually.)

Note that the Poisson distribution with parameter $\lambda$
slowly tends to the normal distribution when
$\lambda\to+\infty$. Thus, in plain terms, the above
observation implies that the number $k$ of cylinders of a
random square-tiled surface in $\cQ_g$ (equivalently, the
number $k$ of components of a reduced multicurve
$\gamma_{reduced}=\gamma_1+\dots+\gamma_k$ associated to a
random integral multicurve $\gamma\in\cML_g(\Z)$ on a
surface of large genus $g$) is located within
bounds:
$$
\frac{\log(g)}{2}-3\sqrt{\frac{\log(g)}{2}}
\le k \le
\frac{\log(g)}{2}+3\sqrt{\frac{\log(g)}{2}}
$$
with probability
greater than $0.98$, when $g$ is large enough.
\smallskip

\noindent\textbf{Heights of cylinders of a random
square-tiled surface of large genus. Weights of a random
integer multicurve.~}
We complete this Section with one more Conjecture.

\begin{Conjecture}
\label{conj:one:vertex:dominates:for:fixed:k}
For any fixed positive integer $k$ the contribution
$\Vol(\Petal_k(g))$ of square-tiled surfaces represented by
the graph $\Petal_k(g)$ to $\Vol\cQ(1^{4g-4})$ is uniformly
dominating the total contribution of all other $k$-cylinder
surfaces as $g\to+\infty$.
\end{Conjecture}

\begin{Remark}
For $k=1$
Conjecture~\ref{conj:one:vertex:dominates:for:fixed:k}
instantly follows from Theorem~\ref{th:separating:over:non:separating}.
\end{Remark}

The following Conditional Theorem generalizes
Lemma~\ref{lm:height:of:one:cylinder} from
Section~\ref{ss:Combinatorial:geometry:of:random:st:surfaces}.

\begin{CondTheorem}
\label{cond:th:height:for:fixed:k}
Conjecture~\ref{conj:introduction:sum:of:correlators}
implies that for any fixed $k$,
for which
Conjecture~\ref{conj:one:vertex:dominates:for:fixed:k}
is valid,
all maximal horizontal cylinders
of a random $k$-cylinder
square-tiled surface in $\cQ_g$
have unit height
$H_i=1$, for $i=1,\dots,k$, with probability which tends to
$1$ as $g\to+\infty$ (equivalently, all weights $a_i$,
$i=1,\dots,k$, of a random $k$-component integral
multicurve $a_1\gamma_1+\dots+a_k\gamma_k$ in $\cML_g(\Z)$
are equal to $1$ with probability which tends to $1$ as
$g\to+\infty$).
\end{CondTheorem}

Our last Conditional Theorem shows that
the latter property is not uniform.

\begin{CondTheorem}
\label{cond:th:glovbal:contribution:of:height:1}
Conjectures~\ref{conj:Vol:Qg},
\ref{conj:introduction:sum:of:correlators}
and~\ref{conj:uniform:bound:for:error:term} imply that all
maximal horizontal cylinders of a random square-tiled
surface in $\cQ_g$ have unit heights with probability which
tends to $\tfrac{\sqrt{2}}{2}$ as $g\to+\infty$
(equivalently, the probability, that a random integral
multicurve $\gamma\in\cML_g(\Z)$ is reduced, tends to
$\tfrac{\sqrt{2}}{2}$ as $g\to+\infty$).
\end{CondTheorem}

\subsection{Structure of the paper}

In Section~\ref{s:proofs} we prove Theorem~\ref{th:volume}
stated in Section~\ref{ss:intro:Masur:Veech:volumes}
providing the formula for the Masur--Veech volume
$\Vol\cQ_{g,n}$ and Theorem~\ref{th:carea} stated in
Section~\ref{ss:intro:Siegel:Veech:constants} providing the
formula for area Siegel--Veech constant
$\carea(\cQ_{g,n})$.

In Section~\ref{s:comparison:with:Mirzakhani} we compare
our Formula~\ref{eq:square:tiled:volume} for
$\Vol\cQ_{g,n}$ with Mirzakhani's formula for $b_{g,n}$ and
our Formula~\eqref{eq:volume:contribution:of:stable:graph}
for $\Vol(\Gamma)$ for a stable graph $\Gamma$ with
Mirzakhani's formula for the associated $c(\gamma)$ for the
associated multicurve $\gamma$. We elaborate translation
between two languages and prove
Theorem~\ref{th:our:density:equals:Mirzakhani:density}
stated in
Section~\ref{ss:Frequencies:of:square:tiled:surfaces}
evaluating the normalization constant~\eqref{eq:const:g:n}
between the corresponding quantities.

In Section~\ref{s:2:correlators} we elaborate a uniform
asymptotic formula~\eqref{eq:a:g:k:difference} for
$2$-correlators $\langle\tau_{d_1}\tau_{d_2}\rangle$, which
has independent interest. We apply it to computation of
asymptotic frequencies $c(\gamma_{sep})$ and
$c(\gamma_{nonsep})$ of separating and of non-separating
simple closed hyperbolic geodesics on a hyperbolic surface
of large genus $g$ thus proving
Theorem~\ref{th:separating:over:non:separating} stated in
Section~\ref{ss:large:genera}.

For the sake of completeness, we present a detailed formal
definition of a stable graph in
Appendix~\ref{s:stable:graphs}

Appendix~\ref{s:explicit:calculations} provides examples of
explicit calculations of the Masur--Veech volume
$\Vol\cQ_{g,n}$ and of the Siegel--Veech constant
$\carea(\cQ_{g,n})$ for small $g$ and $n$.

Appendix~\ref{a:tables} presents tables of
$\Vol\cQ_{g,n}$ and of $\carea(\cQ_{g,n})$ for small $g$
and $n$.

In Appendix~\ref{s:Two:harmonic:sums} we elaborate
asymptotic expansions for the multiple harmonic
sums~\eqref{eq:multiple:harmonic:sum:def}
and~\eqref{eq:multiple:harmonic:zeta:sum:def}. In
particular, the results of this Section provide evidence
for Conjecture~\ref{conj:uniform:bound:for:error:term}
from Section~\ref{ss:large:genera}.

In Appendix~\ref{s:Conjectural:asymptotics:of:correlators}
we present conjectural uniform asymptotic formula for
correlators generalizing
Conjecture~\ref{conj:introduction:sum:of:correlators} from
Section~\ref{ss:large:genera}.

In
Appendix~\ref{s:volume:contribution:from:graphs:with:single:vertex}
we compute the asymptotic contribution $\Vol(\Petal_k(g))$
of the stable graph $\Petal_k(g))$ with a single vertex and
with $k$ loops to the Masur--Veech volume $\Vol\cQ_g$ for
large genera $g$ conditionally to
Conjecture~\ref{conj:introduction:sum:of:correlators} and
prove the Conditional Theorems stated in
Section~\ref{ss:large:genera} describing the asymptotic
geometry of random square-tiled surfaces of large genus and
of random integral multicurves on a surface of large genus.
\medskip

\noindent\textbf{Acknowledgements.}
   %
Numerous results of this paper were directly or indirectly
inspired by beautiful ideas of Maryam~Mirzakhani. Working
on this paper we had a constant feeling that we are
following her steps. This concerns in particular the
relation between Masur--Veech volume $\Vol\cQ_{g,n}$ and
frequencies of hyperbolic multicurves and relation between
large genus asymptotics of the volumes of moduli spaces and
intersection numbers of $\psi$-classes.

A.~Eskin planned to use Kontsevich formula for computation
of Masur--Veech volumes before invention of the the
approach of
Eskin--Okounkov~\cite{Eskin:Okounkov:Inventiones}. We thank
him for useful conversations and for indicating to us that
our technique of evaluation of the Masur--Veech volumes
admits generalization to strata with only odd zeroes, see
Remark~\ref{rm:to:complete}.

We are extremely grateful to A.~Aggarwal who explained
us a conceptual method of computation of multiple
harmonic sums developed in Appendix~\ref{s:Two:harmonic:sums}.

We thank F.~Arana--Herrera, S.~Barazer, C.~Ball, G.~Borot,
E.~Duryev, M.~Liu, L.~Monin, B.~Petri, K.~Rafi,
I.~Telpukhovky, S.~Wolpert, A.~Wright, D.~Zagier for useful
discussions. We are grateful to M.~Kazarian for his
computer code evaluating intersection numbers which we used
on the early stage of this project. We highly appreciate
computer experiments of M.~Bell and S.~Schleimer which
provided computer evidence independent of theoretical
predictions of frequencies of multicurves on surfaces of
low genera.

We are grateful to MPIM in Bonn, where considerable part of
this work was performed,
and MSRI in Berkeley for providing us with friendly and
stimulating environment.

\section{Proofs of the volume formulae}
\label{s:proofs}

We start this section by recalling the necessary background
and normalization conventiones that are used in the
subsequent sections of the paper.

\subsection{The principal stratum and Masur--Veech measure}
\label{ss:background:strata}
In this section we recall the canonical construction of the
Masur--Veech measure on $\cQ_{g,n}$ and its link with the
integral structure given by the square-tiled surfaces.

Consider a compact nonsingular complex curve $C$ of genus
$g$ endowed with a meromorphic quadratic differential $q$
with $\zeroes=4g-4+n$ simple zeroes and with $n$ simple
poles. Any such pair $(C,q)$ defines a canonical ramified
double cover $\pi:\hat C\to C$ such that $\pi^\ast
q=\hat\omega^2$, where $\hat\omega$ is an Abelian
differential $\widehat\omega$ on the double cover $\widehat
C$. The ramification points of $\pi$ are exactly the zeroes
and poles of $q$. The double cover $\widehat C$ is endowed
with the canonical involution $\iota$ interchanging the two
preimages of every regular point of the cover. The stratum
$\cQ(1^{\zeroes},-1^n)$ of such differentials is modelled
on the subspace of the relative cohomology of the double
cover $\widehat C$, antiinvariant with respect to the
involution $\iota$. This antiinvariant subspace is denoted
by $H^1_-(\widehat C,\{\widehat P_1,\dots,\widehat
P_{\zeroes}\};\C)$, where $\{\widehat P_1,\dots,\widehat
P_{\zeroes}\}$ are zeroes of the induced Abelian
differential $\widehat\omega$. The stratum $\cQ(1^\zeroes,
-1^n)$ is open and dense in $\cQ_{g,n}$ and its complement
in $\cQ_{g,n}$ has positive codimension. In what follows we
always work with this stratum.

We define a lattice in $H^1_-(\widehat
S,\{\widehat{P}_1,\ldots,\widehat{P}_{\zeroes}\};\C)$ as
the subset of those linear forms which take values in
$\Z\oplus i\Z$ on $H^-_1(\widehat S,\{\widehat
P_1,\dots,\widehat P_{\zeroes}\};\Z)$. The integer points
in $\cQ_{g,n}$ are exactly those quadratic differentials
for which the associated flat surface with the metric $|q|$
can be tiled with $1/2 \times 1/2$ squares. In this way the
integer points in $\cQ_{g,n}$ are represented by
\textit{square-tiled surface} as defined in
Section~\ref{ss:Square:tiled:surfaces:and:associated:multicurves}.

We define the Masur--Veech volume
element $d\!\Vol$ on $\cQ(1^{\zeroes},-1^n)$ as the linear volume
element in the vector space
$H^1_-(S,\{\widehat{P}_1,\ldots,\widehat{P}_{\zeroes}\};\C{})$ normalized in such a
way that the fundamental domain of the above lattice has unit
volume.

By construction, the volume element $d\!\Vol^{symplectic}$
in $\cQ_{g,n}$ induced by the canonical symplectic
structure considered in Section~\ref{ss:MV:volume} and the
linear volume element $d\!\Vol^{period}$ in period
coordinates defined in this section
belong to the same Lebesgue measure class. It was proved by
H.~Masur in~\cite{Masur:Hamiltonian} that the Teichm\"uller
flow is Hamiltonian, in particular, that
$d\!\Vol^{symplectic}$ is preserved by the the
Teichm\"uller flow. By the results of
H.~Masur~\cite{Masur:82} and
W.~Veech~\cite{Veech:Gauss:measures}, the volume element
$d\!\Vol^{period}$ is also preserved by the Teichm\"uller
flow. Ergodicity of the Teichm\"uller flow now implies that
the two volume forms are pointwise proportional with
constant coefficient.

We postpone evaluation of this constant factor to another
paper. Throughout this paper we consider the normalization
of the Masur--Veech volume element
$d\!\Vol=d\!\Vol^{period}$ as defined in the current
section and then define $\Vol\cQ_{g,n}$ by
means of~\eqref{eq:def:Vol:Q:g:n} and of
Convention~\ref{conv:MV:volume:notation}. This definition
incorporates the conventions on the choice of the lattice,
on the choice of the level of the area function, and the
convention on the dimensional constant. We
follow~\cite{AEZ:Dedicata}, \cite{AEZ:genus:0},
\cite{Goujard:volumes},
\cite{DGZZ:meanders:and:equidistribution},
\cite{DGZZ:one:cylinder:Yoccoz:volume} in the choice of
these conventions; see \S~4.1 in~\cite{AEZ:genus:0} and
Appendix~A in\cite{DGZZ:meanders:and:equidistribution} for
the arguments in favour of this normalization.

\subsection{Jenkins--Strebel differentials and stable graphs}
\label{ss:Jenkins:Strebel}
A quadratic differential $q$ in $\cQ_{g,n}$ is called
\textit{Jenkins-Strebel} if its horizontal foliation
contains only closed leaves. Any Jenkins--Strebel
differential can be decomposed into maximal horizontal
cylinders with zeroes and simple poles located on the
boundaries of these cylinders. We  call these boundaries
\textit{singular layers}. Each singular layer defines a
metric ribbon graph representing an oriented surface with
boundary. When the quadratic differential belongs to the
principal stratum $\cQ(1^l,-1^n)$, the ribbon graph has
vertices of valence three at simple zeroes of $q$, vertices
of valence one at simple poles of $q$ and no vertices of
any other valency. Throughout this paper we always assume
that the quadratic differential $q$ belongs to the
principal stratum.

Every ribbon graph $\ribbongraph$ considered as an oriented surface
with boundary has certain genus $g(\ribbongraph)$, certain number
$n(\ribbongraph)$ of boundary components, and certain number
$p(\ribbongraph)$ of univalent vertices often called \textit{leaves}
of the graph. The number $m(\ribbongraph)$ of trivalent vertices can
be expressed through these quantities as
\begin{equation}
\label{eq:m:g:p:n}
m(\ribbongraph)=4g(\ribbongraph)-4+2n(\ribbongraph)+p(\ribbongraph)
\end{equation}

\begin{figure}[htb]
  %
\includegraphics{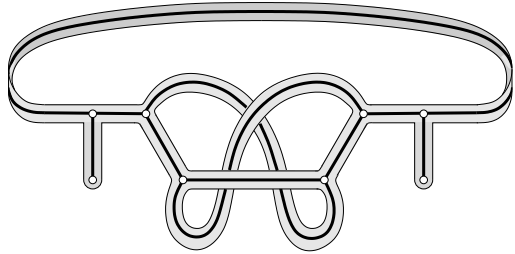}
\vspace{85pt}
\caption{
\label{fig:ribbon:graph}
This ribbon graph $\ribbongraph$ has
genus $g(\ribbongraph)=1$; it has $n(\ribbongraph)=2$ boundary components
and $p(\ribbongraph)=2$ univalent vertices.
}
\end{figure}

To every Jenkins--Strebel differential $q$ as above we
associate a \textit{stable graph} $\Graph =(V, H, \alpha,
\iota, \boldsymbol{g}, L)$ in the same way as we did it in
Section~\ref{ss:Square:tiled:surfaces:and:associated:multicurves}
in the particular case when Jenkins--Strebel differential
$q$ represents a square-tiled surface. (A formal definition
of a stable graph can be found, for example,
in~\cite{Okounkov:Pandharipande}; we reproduce it in
Appendix~\ref{s:stable:graphs} for completeness.) We recall
briefly the construction of $\Gamma$.

The vertices of $\Graph$ encode the singular layers. The
set of all vertices (singular layers) is denoted by $V$. A
tubular neighborhood of a singular layer $v$ is a surface
with boundary, which has some genus $g_v$. The
\textit{genus decoration} $\boldsymbol{g} = (g_v)_v$
associates to each $v\in V$ the nonnegative integer $g_v$.
Any maximal horizontal cylinder of $q$ has two boundary
components which are canonically identified with
appropriate boundary components of tubular neighborhoods of
appropriate singular layers $v_i, v_j$ (where $v_i$ and
$v_j$ coincide when the cylinder goes from the singular layer
to itself). In this way, each maximal horizontal
cylinder defines an edge of $\Graph$ joining the
boundary layers $v_i, v_j$. Finally, simple poles of
$q$ are encoded by the \textit{legs} of $\Graph$. By
convention the $n$ simple poles are labeled, so the legs of
$\Gamma$ inherit the labeling $L$.
Relation~\eqref{eq:m:g:p:n}
implies the stability condition $2 g_v - 2 + n_v > 0$ for
every vertex $v$ of $\Graph$.

\begin{Remark}
Consider the following trivial stable graph: it has a
unique vertex decorated by the integer $g$; it has $n$
legs; it has no edges. Such graph does not correspond to
any Strebel differential. However, being considered as an
element of the set $\cG_{g,n}$ of all stable graphs, it
provides zero contribution when we formally apply
formula~\eqref{eq:square:tiled:volume} defining the
Masur--Veech volume $\Vol\cQ_{g,n}$ in
Theorem~\ref{th:volume}. It also
provides zero contribution when we formally apply
formula~\eqref{eq:carea} defining the Siegel--Veech
constant $\carea(\cQ_{g,n})$ in
Theorem~\ref{th:carea}.
\end{Remark}

\subsection{Conditions on the lengths of the waist curves of the cylinders}
\label{eq:conditions:on:lengths}

Having a square-tiled surface and its associated stable
graph $\Graph$, we denote by $k=|E(\Graph)|$ the number of
maximal cylinders filled with closed horizontal
trajectories. Denote by $w_1, \dots, w_k$ the lengths of
the waist curves of these cylinders. Since every edge of
any singular layer $v$ is followed by the boundary of the
corresponding ribbon graph twice, the sum of the lengths of
all boundary components of each singular layer $V$ is
integer (and not only half-integer).

Let $\Graph$ be a stable graph and let us consider the collection of linear
forms $l_v=\sum_{e\in E_v(\Graph)}w_e$ in variables $w_1,\dots, w_k$, where
$k=|E(\Graph)|$, $v$ runs over the vertices $V(\Graph)$, and $E_v(\Graph)$ is
the set of edges adjacent to the vertex $v$ (ignoring legs).
It is immediate to see that the vector space
spanned by all such linear forms has dimension $|V(\Graph)|-1$.

Let us make a change of variables passing from half-integer
to integer parameters
$b_i:=2w_i$ where $i=1,\dots,k$. Consider the integer sublattice
$\mathbb{L}^k\subset\Z^k$ defined by the linear relations
\begin{equation}
\label{eq:sublattice:L}
l_{v}(b_1,\dots,b_k)=\sum_{e\in E_v(\Graph)}b_{e}=0\ (\operatorname{mod} 2)
\end{equation}
for all vertices $v\in V(\Graph)$. By the above
remark, the sublattice $\mathbb{L}^k$ has index $2^{|V(\Graph)|-1}$
in $\Z^k$. We summarize the observations of this section in the following
criterion.

\begin{Corollary}
\label{cor:criterion}
A collection of strictly positive numbers
$w_1,\dots,w_{k}$,
where $w_i\in\tfrac{1}{2}\N$ for $i=1,\dots,k$,
corresponds to a square-tiled surface
realized by a stable graph $\Graph\in\cG_{g,n}$ if and only if $k=|E(\Graph)|$ and the
corresponding vector $\boldsymbol{b}=2\boldsymbol{w}$ belongs to the sublattice
$\mathbb{L}^{k}$. This sublattice has index
$|\Z^{k}:\mathbb{L}^{k}|=2^{|V(\Graph)|-1}$
in the integer lattice $\Z^k$.
\end{Corollary}

We complete this section with a generalization
of Lemma~3.7 in~\cite{AEZ:Dedicata} which would be used in
the proof of our main
formula~\eqref{eq:volume:contribution:of:stable:graph} for
the Masur--Veech volume $\Vol\cQ_{g,n}$.

\begin{Lemma}
\label{lm:evaluation:for:monomial}
Let $\mathbb{L}^k$ be a sublattice of finite index
$|\Z^k:\mathbb{L}^k|$ in the integer lattice $\Z^k$
and let $m_1,\dots,m_k\in\N$ be any positive integers.
The following formula holds
\begin{multline}
\label{eq:3.7}
\lim_{N\to+\infty}
\frac{1}{N^{|m|+k}}
\sum_{\substack{\boldsymbol{b}\cdot\boldsymbol{H}\le N\\b_i, H_i\in\N\\ \boldsymbol{b}\in\mathbb{L}^k}}
b_1^{m_1}\cdots b_k^{m_k}
\ =\\=\
\frac{1}{|\Z^k:\mathbb{L}^k|}\cdot
\frac{1}{(|m|+k)!}
\cdot
\prod_{i=1}^k \Bigg(m_i!\cdot \zeta(m_i+1)\Bigg)
\ =\\=\
\frac{1}{|\Z^k:\mathbb{L}^k|}\cdot
\frac{1}{(|m|+k)!}
\cdot
\cZ(b_1^{m_1}\cdots b_k^{m_k})
\,.
\end{multline}

Moreover, the sum in $\boldsymbol{H}$ and the limit commute:
$$
\lim_{N\to+\infty}
\frac{1}{N^{|m|+k}}
\sum_{\substack{\boldsymbol{b}\cdot\boldsymbol{H}\le N\\b_i, H_i\in\N\\ \boldsymbol{b}\in\mathbb{L}^k}}
b_1^{m_1}\cdots b_k^{m_k}
\ =\
\sum_{\substack{\boldsymbol{H}\\H_i \in \N}}
\lim_{N\to+\infty}
\frac{1}{N^{|m|+k}}
\sum_{\substack{\boldsymbol{b}\cdot\boldsymbol{H}\le N\\b_i \in\N\\ \boldsymbol{b}\in\mathbb{L}^k}}
b_1^{m_1}\cdots b_k^{m_k}
$$
and we have
$$
\lim_{N\to+\infty}
\frac{1}{N^{|m|+k}}
\sum_{\substack{\boldsymbol{b}\cdot\boldsymbol{H}\le N\\b_i \in\N\\ \boldsymbol{b}\in\mathbb{L}^k}}
b_1^{m_1}\cdots b_k^{m_k}
\ =\
\frac{1}{|\Z^k:\mathbb{L}^k|}\cdot
\frac{1}{(|m|+k)!}
\cdot
\cY(\boldsymbol{H})(b_1^{m_1}\cdots b_k^{m_k})\,.
$$
\end{Lemma}
\begin{proof}
The limit with omitted restriction
$\boldsymbol{b}\in\mathbb{L}^k$ is computed in Lemma~3.7
in~\cite{AEZ:Dedicata}. Note that the inversion of sum and
limits here is valid by virtue of the dominated convergence
theorem. More precisely, under the substitution $x_i =
\frac{b_i H_i}{N}$ the sum approximates the integral
from below.

The restriction $\boldsymbol{b}\in\mathbb{L}^k$ rescales
the volume element in the corresponding integral sum by the
index $|\Z^k:\mathbb{L}^k|$ of the sublattice $\mathbb{L}^k$
in $\Z^k$ which
produces the extra factor $|\Z^k:\mathbb{L}^k|^{-1}$.

Finally, the last equality in~\eqref{eq:3.7}
and the last equality in the last line
of the assertion of Lemma~\ref{lm:evaluation:for:monomial}
are just the
definitions~\eqref{eq:cZ} and~\eqref{eq:cV} of $\cZ$
and of $\cY$ respectively.
\end{proof}

The above Lemma admits an immediate generalization in terms
of densities. The Lemma below provides a proof of
Theorem~\ref{thm:statistics}.
We consider the relation~\eqref{eq:measure:mu};
the other statements of
Theorem~\ref{thm:statistics}
are proved analogously.

\begin{Lemma}
\label{lm:general:evaluation:for:monomial}
Let $F: \Delta^k \times \N^k \to \R$ be a continuous
function integrable with respect to the density
$\cZt(b_1^{m_1} \ldots b_k^{m_k}) \delta_{\boldsymbol{H}}
d\boldsymbol{x}$ defined in~\eqref{eq:measure:mu}.
Then
\begin{multline*}
\lim_{N\to+\infty}
\frac{1}{N^{|m|+k}}
\sum_{\substack{\boldsymbol{b}\cdot\boldsymbol{H}\le 2N
  \\b_i, H_i\in\N
  \\ \boldsymbol{b}\in\mathbb{L}^k}}
F\left( \left(\frac{b_1 H_1}{2N},\dots,\frac{b_k H_k}{2N}\right),
\boldsymbol{H}\right)
\,b_1^{m_1}\cdots b_k^{m_k}
\ =\\= \
\frac{1}{|\Z^k:\mathbb{L}^k|}\cdot
\sum_{\boldsymbol{H}}
\int_{\Delta^k} F(\boldsymbol{x}, \boldsymbol{H})\,
\cZt(b_1^{m_1}\cdots b_k^{m_k})\, d\boldsymbol{x}\,.
\end{multline*}
\end{Lemma}

\subsection{Counting trivalent metric ribbon graphs with leaves}
\label{ssec:trivalent:graphs:with:leaves}
We need the following elementary generalization of the Theorem of M.~Kontsevich
stated in section~\ref{ss:intro:volume:polynomials} allowing to our metric
ribbon graph have univalent vertices (\textit{leaves}) in addition
to trivalent vertices.

We use the letter $p$ to denote the number of leaves.
Consider a collection of positive integers $b_1,\dots, b_n$
such that $\sum_{i=1}^n b_i$ is even. Similarly to
$\cN_{g,n}(b_1, \ldots, b_n)$ defined in
section~\ref{ss:intro:volume:polynomials}, let us denote by
$\cN_{g,n,p}(b_1,\dots,b_n)$ the weighted count of
connected metric ribbon graphs $G$  of genus $g$ with $n$
labeled boundary components of integer lengths
$b_1,\dots,b_n$ and $p$ univalent vertices. In other words
$$
\cN_{g,n,p}(b_1,\dots,b_n):=\sum_{G \in \cR_{g,n,p}} \frac{1}{|\Aut(G)|} N_G(b_1,\dots,b_n).
$$
where $\cR_{g,n,p}$ is the set of equivalence classes of
ribbon graphs of genus $g$, $n$ boundary components and $p$
univalent vertices. The counting function $\cN_{g,n,p}$
generalizes Kontsevich polynomials.

\begin{Proposition}
\label{pr:count:with:leaves}
Consider $n$-tuples of large positive integers
$b_1,\dots,b_n$ such that
$\sum_{i=1}^n b_i$ is even. The following relation holds:
\begin{align}
\cN_{g,n,p}(b_1,\dots,b_n)
&= \cN_{g,n+p,0}(b_1,\dots,b_n,\underbrace{0,\dots,0}_p) \\
&= N_{g,n+p}(b_1,\dots,b_n,\underbrace{0,\dots,0}_p)
+\text{lower order terms}
\,,
\end{align}
where the Kontsevich polynomials $N_{g,n}$ are defined by formula~\eqref{eq:N:g:n}.
\end{Proposition}

The proof of Proposition~\ref{pr:count:with:leaves}
is the combination of the following two Lemmas.

\begin{NNLemma}
   %
Suppose that for some $g,n,p$ the leading term of
$\cN_{g,n,p}(b_1,\dots,b_n)$ is a homogeneous polynomial
$N^{top}_{g,n,\poles}(b_1,\dots, b_n)$ in $b_1,\dots,b_n$.
Then the leading term of $\cN_{g,n,\poles+1}(b_1,\dots,
b_n)$ is also a homogeneous polynomial
$N^{top}_{g,n,\poles+1}(b_1,\dots, b_n)$ in
$b_1,\dots,b_n$. Moreover, it satisfies the relation
\begin{equation}
\label{eq:g:m:k}
N^{top}_{g,n,\poles+1}=\frac{1}{2}\cdot\cI(N^{top}_{g,n,\poles})\,,
\end{equation}
where  $\cI=\sum_{i=1}^n  \cI_{b_i}$, and operators  $\cI_{b_i}$
are defined on monomials  by
$$
\cI_{b_i}(b_1^{j_1}\dots b_{i-1}^{j_{i-1}}
\cdot b_i^{j_i}\cdot
b_{i+1}^{j_{i+1}}  \dots b_n^{j_n})=
b_1^{j_1}\dots b_{i-1}^{j_{i-1}} \cdot
\frac{b_i^{j+2}}{j+2}\cdot
b_{i+1}^{j_{i+1}}  \dots b_n^{j_n}
$$
and  are  extended  to arbitrary polynomials by linearity.
\end{NNLemma}
\begin{proof}
This Lemma mimics Lemma 3.5 in~\cite{AEZ:Dedicata}. Formally
speaking, in~\cite{AEZ:Dedicata} the corresponding statement
is formulated only for $g=0$, but it is immediate to see that the
inductive proof is applicable without any changes to any genus as
soon as the base of induction, corresponding to $\poles=0$ is valid.
\end{proof}

\begin{NNLemma}
Polynomials $N_{g,n}$
defined by equation~\eqref{eq:N:g:n} satisfy the relations:
\begin{equation}
\label{eq:g:m:plus:k}
N_{g,n+\poles+1}(b_1,\dots,b_n,0,\dots,0)=
\frac{1}{2}\cdot\cI(N_{g,n+\poles}(b_1,\dots,b_n,0,\dots,0))\,.
\end{equation}
\end{NNLemma}
\begin{proof}
Using the explicit expressions of $N_{g,n+\poles+1}$ and of
$N_{g,n+\poles}$ in terms of $\psi$-classes
we immediately see that the above
relation is equivalent to the following one:
\begin{multline*}
\langle \psi_1^{d_1}\psi_2^{d_2}\dots \psi_n^{d_n} \rangle
=\\=
\langle \psi_1^{d_1-1}\psi_2^{d_2}\dots \psi_n^{d_n} \rangle
+
\langle \psi_1^{d_1}\psi_2^{d_2-1}\dots \psi_n^{d_n} \rangle
+\dots+
\langle \psi_1^{d_1}\psi_2^{d_2}\dots \psi_n^{d_n-1} \rangle
\end{multline*}
which is nothing else but the string equation for a monomial $M$
in $\tau$-correlators:
$$
\langle \tau_0 M \rangle = \sum_{i=1}^\infty \langle \tau_{i-1} \frac{\partial M}{\partial \tau_i} \rangle\,.
$$
\end{proof}

\begin{proof}[Proof of Proposition~\ref{pr:count:with:leaves}]
For $\poles=0$ (when there are no poles at all) the statement
corresponds to the original Theorem of Kontsevich stated in
section~\ref{ss:intro:volume:polynomials}. We use this as the base of
induction in $p$ for any fixed pair $(g,n)$. It remains to notice that
equations~\eqref{eq:g:m:k} and~\eqref{eq:g:m:plus:k} recursively
define the corresponding polynomials for any $(g,n,p)$ starting from
$(g,n,0)$.
\end{proof}

\subsection{Proof of the volume formula}

A square-tiled surface corresponding to a fixed stable graph
$\Graph$ can be described by three groups of
parameters. Parameters in different groups can
be varied independently. Parameters in the first group are
responsible for the lengths of horizontal saddle connections. In this
group we fix only the lengths $w_1, \dots, w_k$ of the waist curves
of the cylinders filled with closed horizontal trajectories, where $k$
is the number of edges in $\Graph$. This
leaves certain freedom for the choice of the lengths of horizontal
saddle connections. The criterion of admissibility of a given
collection $\boldsymbol{w}=(w_1,\dots,w_k)$ is given by
Corollary~\ref{cor:criterion}. The count for the number of choices of
the lengths of all individual saddle connections for a fixed choice
of $\boldsymbol{w}$ is given in
Proposition~\ref{pr:count:with:leaves}.

There are no restrictions on the choice of strictly positives integer
or half-integer heights $h_1,\dots, h_k$ of the cylinders.

Having chosen the widths $w_1, \dots, w_k$ of all maximal cylinders
(i.e. the lengths of the closed horizontal trajectories) and the
heights $h_1, \dots, h_k$ of the cylinders, the flat area of the
entire surface is already uniquely determined as the sum
$\boldsymbol{w}\cdot\boldsymbol{h}=w_1 h_1+\dots +w_k h_k$ of flat areas of
individual cylinders.

However, when the lengths of all horizontal saddle connections and
the heights $h_i$ of all cylinders are fixed, there is still a freedom
in the third independent group of parameters. Namely, we can twist
each cylinder by some twist $\phi_i\in\frac{1}{2}\N$ before attaching
it to the layer. Applying, if necessary, appropriate Dehn twist we
can assume that $0\le\phi_i<w_i$, where $w_i$ is the perimeter
(length of the waist curve) of the corresponding cylinder. Thus, for
any choice of lengths of horizontal saddle connections realizing some
square-tiled surface with the stable graph $\Graph$ and for any
choice $h_1,\dots, h_k$ of heights of the cylinders we get
$(2w_1)\cdot\ldots\cdot(2w_k)$ square-tiled surfaces sharing the same
lengths of the horizontal saddle connections and same heights of the
cylinders.

In Proposition~\ref{pr:count:with:leaves} we assume that the lengths
of the edges of the metric ribbon graph are integer. Clearly, if we
allow these lengths to be also half-integer, we get
$N_{g,n+p}\big(2w_1,\dots,2w_n,0,\dots,0\big)$ as the leading
term of the new count. The realizability condition
$2\boldsymbol{w}\in\mathbb{L}^k$ from Corollary~\ref{cor:criterion}
translates as the compatibility condition of the parity of the sum of
the lengths of the boundary components of each individual connected
ribbon graph as in Proposition~\ref{pr:count:with:leaves}.

We are ready to write a formula for the leading term in
the number of all square-tiled surfaces tiled with at most $2N$
squares represented by the stable graph $\Graph$ when the integer
bound $N$ becomes sufficiently large:
\begin{multline*}
\card(\cSTgn(2N)\cap\cSTgn(\Gamma))
\sim\\
\sim
(4g-4+n)!\cdot
\frac{1}{|\operatorname{Aut}(\Graph)|}\cdot
\sum_{\substack{\boldsymbol{w}\cdot\boldsymbol{h}\le N/2\\w_i, h_i\in\frac{1}{2}\N\\ 2\boldsymbol{w}\in\mathbb{L}^k}}
(2w_1) \cdots (2w_k)\cdot
\prod_{v\in V(\Graph)}
N_{g_v,n_v}(2\boldsymbol{w}_v)\,,
\end{multline*}
Notations in the above expression mimic notations
in~\eqref{eq:square:tiled:volume}, namely $k=|E(\Graph)|$,
and $\boldsymbol{w}_v$ is defined analogously to
$\boldsymbol{b}_v$ in~\eqref{eq:square:tiled:volume}.
The factor $l!=(4g-4+n)!$ represents the number of ways
to label the $l$ trivalent vertices of the ribbon graphs,
which correspond to $l$ simple zeroes of the corresponding
Strebel quadratic differential $q$. Note that by convention
the univalent vertices (leaves) (corresponding to simple poles
of $q$ and also to $n$ marked points) are already labeled.

Making a change of variables $b_h:=(2w_h)\in\N$ and
$H_i:=(2h_i)\in\N$ we can rewrite the above expression as
\begin{multline}
\label{eq:last}
\card(\cSTgn(2N)\cap\cSTgn(\Gamma))
\sim\\
\sim
(4g-4+n)!\cdot
\frac{1}{|\operatorname{Aut}(\Graph)|}\cdot
\sum_{\substack{\boldsymbol{b}\cdot\boldsymbol{H}\le 2N\\b_i, H_i\in\N\\ \boldsymbol{b}\in\mathbb{L}^k}}
b_1 \cdots b_k
\cdot
\prod_{v\in V(\Graph)}
N_{g_v,n_v}(\boldsymbol{b}_v)\,.
\end{multline}

The expression above is a homogeneous polynomial of degree
$6g-6+2n-k$. For any individual monomial the corresponding
sum was evaluated in
Lemma~\ref{lm:evaluation:for:monomial}. It remains to adapt
formula~\eqref{eq:3.7} to our specific context.

By Corollary~\ref{cor:criterion} the sublattice
$\mathbb{L}^k$ in the above formula has index
$2^{|V(\Graph)|-1}=|\Z^k:\mathbb{L}^k|$. The corresponding
factor $1/2^{|V(\Graph)|-1}$ appears
as the first factor in the second line of
definition~\eqref{eq:P:Gamma}
of $P_\Gamma(\boldsymbol{b})$.

The degree of the homogeneous polynomial denoted by $|m|$ in
formula~\eqref{eq:3.7} equals in our case to $6g-6+2n-k$, so
$|m|+k=6g-6+2n=\dim_{\C}\cQ(1^\zeroes,-1^n)=\dprinc$. Note also that
in~\eqref{eq:3.7} we perform the summation under the condition
$\boldsymbol{b}\cdot\boldsymbol{H}\le N$ while in the above formula
we sum over the region $\boldsymbol{b}\cdot\boldsymbol{H}\le 2N$.
This provides extra factor $2^{\dprinc}$.
Finally, passing from
$\card(\cSTgn(2N)\cap\cSTgn(\Gamma))$
in~\eqref{eq:last} to $\Vol\cQ_{g,n}$
by~\eqref{eq:Vol:sq:tiled}
we introduce the extra factor
$2(6g-6+2n)$.
The resulting product factor
$$
2(6g-6+2n)\cdot\frac{2^{\dprinc}}{\dprinc !} \cdot (4g-4+n)!
=
\frac{2^{6g-5+2n} \cdot (4g-4+n)!}{(6g-7+2n)!}
$$
is the factor in the first line of
definition~\eqref{eq:P:Gamma} of
$P_\Gamma(\boldsymbol{b})$.

Theorem~\ref{th:volume} is proved.

\subsection{Yet another expression for the Siegel--Veech constant}

For any square-tiled surface $S$ define the following quantity. Suppose that
$S$ has $k$ maximal cylinders filled with closed horizontal trajectories.
Denote as usual by $w_i$ the length of the closed horizontal trajectory (length of the
waist curve) of the $i$-th cylinder and denote by $h_i$ its height.
Define
$$
M(S):=
\sum_{i=1}^k
\frac{h_{i}}{w_{i}}\,.
$$

The $\GL$-orbit of any square-tiled surface $S$ is a closed
invariant submanifold $\cL(S)$ in the ambient stratum of
quadratic (or Abelian) differentials. In the same way in
which we defined in
Section~\ref{ss:intro:Siegel:Veech:constants} the area
Siegel--Veech constant $\carea(\cQ_{g,n})$ for $\cQ_{g,n}$,
we can define the area Siegel--Veech constant $\carea(\cL(S))$
for $\cL(S)$. It satisfies, in particularly, analogs
of~\eqref{eq:SV:constant:definition}
and~\eqref{eq:SV:asymptotics}. Theorem 4 in~\cite{Eskin:Kontsevich:Zorich} proves
the following assertion.

\begin{NNTheorem}[\cite{Eskin:Kontsevich:Zorich}]
For   any   connected   square-tiled   surface $S$, the
Siegel--Veech constant $c_{\mathit{area}}(\cL(S))$ has  the
following value:
\begin{equation}
\label{eq:SVconstant:for:square:tiled}
c_{\mathit{area}}(\nu_1)=
\cfrac{3}{\pi^2}\cdot
\cfrac{1}{\card(\SLZ\cdot S)}\
\sum_{S_i\in\SLZ\cdot S} M(S)
\end{equation}
\end{NNTheorem}

Recall that for   any   connected   component of any  stratum of
Abelian differentials or of a stratum of meromorphic quadratic
differentials with at most simple poles we denote by $\cST(N)$ the
number of square-tiled surfaces in this stratum tiled with at most
$N$ squares. Define now the quantity:
$$
\cD\cST(N):=\sum_{\substack{
\text{square-tiled surfaces $S$ tiled}\\
\text{with at most $N$ squares}}} M(S)\,.
$$

The latter quantity has the following geometric interpretation.
Consider all square-tiled surfaces obtained from square-tiled
surfaces as above by cutting exactly one cylinder $\cyl_i$ along the
closed horizontal trajectory at the level $h$, where $0\le h< h_i$
and $h$ is integer in the case of Abelian differentials and
half-integer in the case of quadratic differentials. In other words,
we do not chop the squares along the cut. The above quantity
enumerates bordered square-tiled surfaces obtained in this way.
Indeed, we loose the twist parameter $w_i$ (correspondingly $2w_i$)
along the cylinder which is now cut open, but we gain the new height
parameter $h_i$ (correspondingly $2h_i$) responsible for the level of
the cut.

As a corollary of the above theorem, D.~Chen and A.~Eskin
proved the following
Theorem (see Appendix A in~\cite{Chen}):

\begin{NNTheorem}[D.~Chen, A.~Eskin]
For   any   connected   component of any  stratum
of  Abelian differentials
and of any stratum of meromorphic quadratic differentials with at most simple poles,
the Siegel--Veech
constant $c_{\mathit{area}}$ has  the
following value:
\begin{equation}
\label{eq:SVconstant:for:stratum}
c_{\mathit{area}}=
\cfrac{3}{\pi^2}\cdot
\lim_{N\to\infty}
\frac{\cD\cST(2N)}
{\cST(2N)}
\end{equation}
\end{NNTheorem}
(Formally speaking, the original Theorem is proved only for the components of the strata
of Abelian differentials, but the proof can be easily extended to the
strata of quadratic differentials.)

\subsection{Proof of the formula for the area Siegel--Veech constant}
We have already evaluated the denominator
in~\eqref{eq:SVconstant:for:stratum}. Evaluating the numerator
following the lines of the initial computation we reduce the problem to
the evaluation of the sum~\eqref{eq:last} counted with the weight
$\frac{h_i}{w_i}=\frac{H_i}{b_i}$, for each $i=1,\dots,k$,
where we use the notations of formula~\eqref{eq:last}:
\begin{equation}
\label{eq:last:carea:1}
\sum_{\substack{\boldsymbol{b}\cdot\boldsymbol{H}\le 2N\\b_i, H_i\in\N\\ \boldsymbol{b}\in\mathbb{L}^k}}
\left(b_1 \cdots b_k
\cdot
\prod_{v\in V(\Graph)}N_{g_v,n_v}(\boldsymbol{b}_v)\right)\cdot\frac{H_i}{b_i}
\end{equation}

Denote by $P(b_1,\dots,b_k)$ the homogeneous polynomial
$\prod N_{g_v,n_v}(\boldsymbol{b}_v)$
in the formula above.
It is easy to see that the condition
$b_i H_i<\boldsymbol{b}\cdot\boldsymbol{H}\le 2N$ implies that
the contribution of any monomial of $P(b_1,\dots,b_k)$
containing the variable $b_i$ to the above sum
is of order $o(N^{\dprinc})$, so it does not contribute to the limit~\eqref{eq:SVconstant:for:stratum}.
Thus, up to lower order terms the sum~\eqref{eq:last:carea:1}
coincides with the sum
\begin{equation}
\label{eq:last:carea:2}
\sum_{\substack{\boldsymbol{b}\cdot\boldsymbol{H}\le 2N\\b_i, H_i\in\N\\ \boldsymbol{b}\in\mathbb{L}^k}}
b_1 \cdots b_{i-1}\cdot H_i\cdot b_{i+1}\cdots b_k
\cdot
P(b_1,\dots,b_{i-1},0,b_{i+1},\dots, b_k)\,.
\end{equation}
It is sufficient to interchange the notations $b_i$ and $H_i$
to see that
\begin{multline*}
\sum_{\substack{\boldsymbol{b}\cdot\boldsymbol{H}\le 2N\\b_i, H_i\in\N}}
b_1 \cdots b_{i-1}\cdot H_i\cdot b_{i+1}\cdots b_k
\cdot
P(b_1,\dots,b_{i-1},0,b_{i+1},\dots, b_k)
=\\=
\sum_{\substack{\boldsymbol{b}\cdot\boldsymbol{H}\le 2N\\b_i, H_i\in\N}}
b_1 \cdots b_k
\cdot
P(b_1,\dots,b_{i-1},0,b_{i+1},\dots, b_k)\,.
\end{multline*}
We already know how to evaluate the latter sum, so it remains to
study the impact of the extra condition $\boldsymbol{b}\in\mathbb{L}^k$
present in the sum~\eqref{eq:last:carea:2}.

Recall the strategy of evaluation of the sum~\eqref{eq:last:carea:2}
(see the proof of analogous Lemma~3.7 in~\cite{AEZ:Dedicata} for
reduction to integral sums and the proof of
Lemma~\ref{lm:evaluation:for:monomial} for the impact of the
sublattice condition). Variables $H_1,\dots, H_{i-1}, b_i,
H_{i+1},\dots, H_k$ are considered as parameters. For each collection
of such parameters we evaluate the corresponding integral sum over a
simplex in the $k$-dimensional space with coordinates
$b_1,\dots,b_{i-1},H_i, b_{i+1},\dots, b_k$. After that we perform
summation with respect to parameters $H_1,\dots, H_{i-1}, b_i,
H_{i+1},\dots, H_k$,

When the edge of the graph $\Graph$ corresponding to the
variable $b_i$ is a \textit{bridge} (i.e. when this edge is
separating), the \textit{parameter} $b_i$ is always
even. The space of integration now has coordinates
$b_1,\dots,b_{i-1},H_i, b_{i+1},\dots, b_k$; the sublattice
$\mathbb{L}^k$ in it is defined by the system of
equations~\eqref{eq:sublattice:L} where we let $b_i=0$.
Such sublattice has index $2^{|V(\Graph)|-2}$ and not
$2^{|V(\Graph)|-1}$ as before. Thus, on the level of
integration we gain factor $2$ with respect to the initial
count. However, since the \textit{parameter} $b_i$ is now
always even, evaluating the corresponding sum with respect
to possible values of this parameter we get the sum
$$
\frac{1}{2^2}+\frac{1}{4^2}+\frac{1}{6^2}\dots=
\frac{1}{4}\cdot \zeta(2)
$$
instead of the original sum
$$
\frac{1}{1^2}+\frac{1}{2^2}+\frac{1}{3^2}\dots=\zeta(2)\,.
$$
Thus, when $b_i$ corresponds to a bridge
(i.e. to a separating edge), we get
\begin{multline*}
\sum_{\substack{\boldsymbol{b}\cdot\boldsymbol{H}\le 2N\\b_i, H_i\in\N\\ \boldsymbol{b}\in\mathbb{L}^k}}
b_1 \cdots b_{i-1}\cdot H_i\cdot b_{i+1}\cdots b_k
\cdot
P(b_1,\dots,b_{i-1},0,b_{i+1},\dots, b_k)
=\\=
\frac{1}{2}\cdot\sum_{\substack{\boldsymbol{b}\cdot\boldsymbol{H}\le 2N\\b_i, H_i\in\N\\ \boldsymbol{b}\in\mathbb{L}^k}}
b_1 \cdots b_k
\cdot
P(b_1,\dots,b_{i-1},0,b_{i+1},\dots, b_k)\,.
\end{multline*}

When $b_i$ corresponds to a non separating edge, the \textit{parameter}
$b_i$ in the sum~\eqref{eq:last:carea:2} can take even and odd values.
The new space of integration has coordinates
$b_1,\dots,b_{i-1},H_i, b_{i+1},\dots, b_k$;
where the sublattice in it is defined by the
system of equations~\eqref{eq:sublattice:L}
in which we substitute $b_i=0$ or $b_i=1$
depending on the parity of the value of the parameter $b_i$.
The sublattice is linear in
the first case and affine in the second case.
Such sublattice has index $2^{|V(\Graph)|-1}$
as before. Thus, when $b_i$ corresponds to a non-separating edge, we get
\begin{multline*}
\sum_{\substack{\boldsymbol{b}\cdot\boldsymbol{H}\le 2N\\b_i, H_i\in\N\\ \boldsymbol{b}\in\mathbb{L}^k}}
b_1 \cdots b_{i-1}\cdot H_i\cdot b_{i+1} b_k
\cdot
P(b_1,\dots,b_{i-1},0,b_{i+1},\dots, b_k)
=\\=
\sum_{\substack{\boldsymbol{b}\cdot\boldsymbol{H}\le 2N\\b_i, H_i\in\N\\ \boldsymbol{b}\in\mathbb{L}^k}}
b_1 \cdots b_k
\cdot
P(b_1,\dots,b_{i-1},0,b_{i+1},\dots, b_k)\,.
\end{multline*}

We have proved that
\begin{multline*}
\sum_{\substack{\boldsymbol{b}\cdot\boldsymbol{H}\le 2N\\b_i, H_i\in\N\\ \boldsymbol{b}\in\mathbb{L}^k}}
\left(b_1 \cdots b_k
\cdot \prod_{v\in V(\Graph)}
N_{g_v,n_v}(\boldsymbol{b}_v)
\right)
\cdot\left(\sum_{i=1}^k\frac{H_i}{b_i}\right)
=\\=
\sum_{\substack{\boldsymbol{b}\cdot\boldsymbol{H}\le 2N\\b_i, H_i\in\N\\ \boldsymbol{b}\in\mathbb{L}^k}}
b_1 \cdots b_k\cdot\cD_{\Graph}
\left(
\prod_{v\in V(\Graph)}
N_{g_v,n_v}(\boldsymbol{b}_v)
\right)
+\text{lower order terms}\,,
\end{multline*}
where operator $\cD_{\Graph}$ is defined in
formula~\eqref{eq:operator:D}. Applying to the latter sum the same
technique as in the end of the proof of Theorem~\ref{th:volume} we
complete the proof of Theorem~\ref{th:carea}.

\subsection{Equivalence of two expressions for the Siegel--Veech constant}
\label{eq:proof:of:coincidence:of:two:SV:expressions}

In this section we prove Theorem~\ref{th:same:SV}.

We start the proof by establishing a natural correspondence
between summands of the two expressions. For any stable
graph $\Gamma\in\cG_{g,n}$ and and any edge $e$ of $\Gamma$
we define a combinatorial surgery
$\Cut_e\Graph$ of $\Gamma$. We describe it separately
in the case when $e$ is a bridge (i.e. a separating edge),
and when it is not.

We start with the case when $e$ is a bridge. Cut the edge
$e$ transforming it into two legs. Assign index $1$ to one
of the resulting graphs, and index $2$ to the remaining
one. We do not modify the genus decoration of the vertices.
The set of vertices $V(\Gamma)$ gets naturally partitioned
into two complementary subsets $V=V_1\sqcup V_2$. Define
$g_i=\sum_{v\in V_i} g_v$ for $i=1,2$. Similarly, the $n$
original legs are partitioned into $n_1$ legs which get to
$\Gamma_1$ and into $n_2$ legs which get to $\Gamma_2$. For
$i=1,2$, relabel the $n_i$ legs of $\Gamma_i$ to the
consecutive labels $1,2,\dots,n_i$ preserving the order
of labels. Assign the label $n_i+1$ to the new leg of
$\Gamma_i$ created during the surgery. The stability
condition $2 g_v - 2 + n_v > 0$, which is valid for every
vertex $v$ of $\Graph$, implies that we get two stable
graphs $\Gamma_i\in\cG_{g_i,n_i+1}$.

The only ambiguity in this construction is the choice of
the label ($1$ or $2$) for one of the components $\Gamma_i$
of the graph $\Gamma$ with removed bridge $e$. In general,
there are two choices, except the case when there is a
symmetry of $\Gamma$ acting on the edge $e$ as a flip (i.e.
a symmetry which sends $e$ to itself exchanging its two
ends).

Note that the surgery is reversible in the following sense.
Given two stable graphs $\Gamma_i\in\cG_{g_i,n_i+1}$ we can
glue the endpoint of the leg with index $n_1+1$ of
$\Gamma_1$ to the endpoint of the leg with index $n_2+1$ of
$\Gamma_2$ creating a connected graph with $n=n_1+n_2$ legs
and with an extra bridge joining $\Gamma_1$ to $\Gamma_2$.
The only ambiguity in this construction is in relabeling
the $n=n_1+n_2$ legs to a consecutive list $(1,2,\dots,n)$;
there are $\binom{n}{n_1}$ ways to do it.

We describe now the surgery $\Cut_e\Gamma$ in the
remaining case when the edge $e$ of $\Gamma$ is not a
bridge (i.e. is not separating). Cutting such edge we
transform it into two legs. We keep the same labels for the
preexisting legs and we associate labels $n+1$ and $n+2$ to
the two created legs. In general, there are two ways to do
that, except when there is a symmetry of $\Gamma$ acting on
the edge $e$ as a flip (i.e. a symmetry which sends $e$ to
itself exchanging its two ends). We get a stable graph
$\Gamma'\in\cG_{g-1,n+2}$.

The inverse operation (applicable to stable graphs with at
least two legs) consists in gluing the two legs of higher
index together transforming them into an edge and keeping
the same labeling for the other legs. Once again we do not
modify the genus decoration of the vertices.

Note that the operator $\partial_\Gamma$ is defined
in~\eqref{eq:operator:D} as a sum
$\sum \partial^e_{\Graph}$
over edges $e\in E(\Gamma)$ of a
stable graph $\Gamma$. Thus, our key sum $\sum_{\Graph \in
\cG_{g,n}} \cZ\left(\partial_{\Graph} P_\Gamma\right)$ in
the right-hand side of~\eqref{eq:carea} in
Theorem~\ref{th:carea} can be seen as the sum over all
pairs $(\Gamma,e)$, where $\Gamma\in\cG_{g,n}$ and $e\in
E(\Gamma)$. We show below that for every such pair
$(\Gamma,e)$, the corresponding term of the resulting sum
has simple expression in terms of the product
$\cZ(P_{\Gamma_1})\cZ(P_{\Gamma_2})$ when $e$ is a bridge
and in terms of $\cZ(P_{\Gamma'})$ when $e$ is not a
bridge, where $\Gamma_1\sqcup\Gamma_2$ (respectively $\Gamma'$)
are the stable graphs obtained under applying the surgery
$\Cut_e\Gamma$.

Having a stable graph $\Graph$ we associate to it the
polynomial
$$
\Pi_{\Graph}(\boldsymbol{b}):=
\prod_{v\in V(\Graph)}
N_{g_v,n_v}(\boldsymbol{b}_v)\,.
$$
By definition~\eqref{eq:P:Gamma} of
$P_{\Graph}(\boldsymbol{b})$ we have
$$
P_{\Graph}(\boldsymbol{b})
=(\text{combinatorial factor})
\cdot\left(\prod_{e\in E(\Gamma)} b_e\right)
\cdot\Pi_{\Graph}(\boldsymbol{b})\,.
$$
A pair $(\Graph,e_0)$ provides a nonzero contribution
to the sum in the right-hand side of~\eqref{eq:carea}
if and only the term
$\partial^{e_0}_{\Graph}=\chi_{\Graph}(e_0)b_{e_0}
\left.\frac{\partial}{\partial b_{e_0}}\right|_{b_{e_0}=0}$,
in the operator $\partial_{\Graph}$
applied to
$\left(\prod_{e\in E(\Gamma)} b_e\right)
\cdot\Pi_{\Graph}(\boldsymbol{b})$
does not identically vanish
(see given by~\eqref{eq:operator:D:e}).
The latter is
equivalent to the condition that the polynomial
$\left.\Pi_{\Graph}\right|_{b_{e_0}=0}$ does not identically
vanish.

If the edge $e_0$ is a bridge, consider the stable
graphs $\Gamma_1, \Gamma_2$ obtained under the surgery
$\Cut_e\Gamma$. The polynomial
$\left.\Pi_{\Graph}\right|_{b_{e_0}=0}$
splits naturally into the product:
$$
\left.\Pi_{\Graph}\right|_{b_j=0}
=\Pi_{\Graph_1}\Pi_{\Graph_2}\,,
$$
so when $e_0$ is a bridge,
and when
$\left.\Pi_{\Graph}\right|_{b_{e_0}=0}$ does not identically
vanish, we get
\begin{align}
\label{eq:separating}
\cZ\left(\partial_{\Graph}^{e_0}\left( \prod_e b_e \cdot  \Pi_{\Graph}\right)\right)
&=\cZ\left(\frac{1}{2}\cdot \prod_e b_e
 \cdot  \left.\Pi_{\Graph}\right|_{b_{e_0}=0}\right)
\\
&=\cZ\left(\frac{1}{2}\cdot b_{e_0}\cdot
 \prod_{e\in E(\Graph_1)}b_e\cdot \Pi_{\Graph_1}
 \cdot \prod_{e\in E(\Graph_2)}b_e\cdot \Pi_{\Graph_2}\right)\notag\\
&=\frac{1}{2}\cdot\frac{\pi^2}{3}
 \cdot\cZ\left(\prod_{e\in E(\Graph_1)}b_e\cdot \Pi_{\Graph_1}\right)
 \cdot\cZ\left(\prod_{e\in E(\Graph_2)}b_e\cdot \Pi_{\Graph_2}\right),\notag
\end{align}

If the edge $e_0$ is
not a bridge
and
$\left.\Pi_{\Graph}\right|_{b_{e_0}=0}$ does not identically
vanish, we get
$\left.\Pi_{\Graph}\right|_{b_{e_0}=0}=\Pi_{\Graph'}$
so
\begin{multline}
\label{eq:non:separating}
\cZ\left(\partial_{\Graph}^{e_0}
\left( \prod_e b_e \cdot  \Pi_{\Graph}\right)\right)
=\\=
\cZ\left(b_{e_0}\cdot\prod_{e\in E(\Graph')}b_e\cdot \Pi_{\Graph'}\right)
=\frac{\pi^2}{3}\cdot\cZ\left(\prod_{e\in E(\Graph')}b_e\cdot \Pi_{\Graph'}\right)\,.
\end{multline}

Rewrite the right-hand side
of~\eqref{eq:carea} as
$$
\frac{3}{\pi^2}
\cdot
\sum_{\Graph \in \cG_{g,n}}
\cZ\left(\partial_{\Graph}
P_\Gamma\right)
=
\frac{3}{\pi^2}
\cdot
\sum_{\Graph \in \cG_{g,n}}\sum_{e\in E(\Gamma)}
\cZ\left(\partial^e_{\Graph}
P_\Gamma\right)
$$
Apply~\eqref{eq:separating}
and~\eqref{eq:non:separating}
to the resulting sum, keeping the
intersection numbers non evaluated
and cancel out the factors
$\frac{3}{\pi^2}$ and
$\frac{\pi^2}{3}$.

Suppose now that $g\ge 1$ (the consideration in the case
$g=0$ is completely analogous). Consider the sum in the
right-hand side of~\eqref{eq:carea:Elise} and replace
$\Vol\cQ_{g_i,n_i}$ for $i=1,2$ and $\Vol\cQ_{g-1,n+2}$ in
it by the corresponding sums~\eqref{eq:square:tiled:volume}
over $\cG_{g_1,n_1}\times\cG_{g_2,n_2}$ and $\cG_{g-1,n+2}$
respectively (where we keep the intersection numbers non
evaluated. It is easy to see, that we get term-by-term the
same sum as above. Theorem~\ref{th:same:SV} is proved.


\section{Comparison with Mirzakhani formula for $b_{g,n}$}
\label{s:comparison:with:Mirzakhani}

Consider a pair $(X,\lambda)$, where $X$ is a hyperbolic
metric $X$ on a closed smooth surface $S_g$ of genus $g$ and
$\lambda$ --- a measured lamination on $S_g$.
In the paper~\cite{Mirzakhani:earthquake} M.~Mirzakhani
associates to almost any such pair $(X,\lambda)$ a unique
holomorphic quadratic differential $q=F(\lambda,X)$ on the
complex curve $C=C(X)$ corresponding to the hyperbolic
metric $X$.

Consider the measure $\mu_{\mathrm{WP}}$ on the moduli
space $\cM_g$ coming from the Weil--Petersson volume
element and consider Thurston measure $\mu_{\mathrm{Th}}$
on the space of measured laminations $\cM\cL_g$. Using
ergodicity arguments, M.~Mirzakhani proves
in~\cite{Mirzakhani:earthquake} that the pushforward
measure
\begin{equation}
\label{eq:F:ast:WP:times:Th}
\mu_g:=F_\ast(\mu_{\mathrm{WP}}\otimes\mu_{\mathrm{Th}})
\end{equation}
on $\cQ_g$ under the map
$$
F: \cM_g\times\cM\cL_g \to \cQ_g
$$
is proportional to the Masur--Veech measure.
Equation~\eqref{eq:F:ast:WP:times:Th} should be considered
as the definition of Mirzakhani's normalization of the
Masur---Veech measure on $\cQ_g$.

\begin{Remark}
Note that under such implicit definition of the
Masur--Veech measure it is not clear at all why its density
should be constant in period coordinates.
Definition~\eqref{eq:F:ast:WP:times:Th} does not provide
any distinguished lattice in period coordinates either.
Thus, though we know, by results of Mirzakhani, that the
density Masur--Veech measure defined
by~\eqref{eq:F:ast:WP:times:Th} differs from the
Masur--Veech volume element defined in
section~\ref{ss:background:strata} by a constant numerical
factor which depends only on $g$, evaluation of this factor
is not quite straightforward. It is performed in
Theorem~\ref{th:factor} below.
\end{Remark}

M.~Mirzakhani proves in~\cite{Mirzakhani:earthquake} that
the map $F$ identifies the hyperbolic length
$\ell_\lambda(X)$ of the measured lamination $\lambda$
evaluated in the hyperbolic metric $X$ with the norm
$\|q(\lambda,X)\|=\int_C |q|$ of the quadratic differential
$q$ (defined as the area of the flat surface associated to
the pair $(C,q)$):
\begin{equation}
\label{eq:norm:q:length:lambda}
\int_{C(X)} |q(\lambda,X)|=\ell_\lambda(X)\,.
\end{equation}
Given a hyperbolic metric $X$ in $\cM_g$ define a
``unit ball'' $B_X\subset \cM\cL_g$ as
$$
B_X:= \{\lambda \in \cML: \ell_X(\lambda) \leq 1\}\,.
$$
Relation~\eqref{eq:norm:q:length:lambda} implies that as
the image of $F$ restricted to the total space of the
bundle of ``unit balls'' $B_X$ over $\cM_g$ one gets the
total space $\cQ^{\le 1}_g$ of the bundle of ``unit balls''
in $\cQ_g$, where
\begin{equation}
\label{eq:ball:of:radius:1:in:Q:g:n}
\cQ^{\Area\le 1}_g
=
\left\{(C,q)\in\cQ_g,\,|\, \Area(C,q) \le 1\right\}\,.
\end{equation}
We have seen in Section~\ref{ss:MV:volume} that the real
hypersurface
\begin{equation}
\label{eq:sphere:of:radius:1:in:Q:g:n}
\cQ^{\Area=1}_g
=
\left\{(C,q)\in\cQ_{g,n}\,|\, \Area(C,q) = 1\right\}\,.
\end{equation}
can be seen as the unit cotangent bundle to $\cM_g$.
In notations of Mirzakhani it is denoted as $\cQ^1\cM_g$.
M.~Mirzakhani defines the Masur--Veech volume of $\cQ_g$ as
\begin{equation}
\label{eq:Vol:Q:1:Mirzakhani:definition}
\Vol\cQ^1\cM_g=
\Vol_{\mathrm{Mir}}\cQ_g:=\mu_g(\cQ^{\Area\le 1}_g)\,.
\end{equation}
The above observations imply the following formula
for $\Vol_{\mathrm{Mir}}\cQ_g$
(see Theorem~1.4 in~\cite{Mirzakhani:earthquake}):
\begin{NNTheorem}[Mirzakhani]
The Masur--Veech volume of the moduli space of holomorphic
quadratic differentials on complex curves of genus $g$
defined by~\eqref{eq:Vol:Q:1:Mirzakhani:definition}
under normalizations~\eqref{eq:F:ast:WP:times:Th}
satisfies the following relation
\begin{equation}
\label{eq:MV:Volume:as:average:Thurston:measure}
\Vol\cQ^1\cM_g=
\Vol_{\mathrm{Mir}}\cQ_g = \int_{\cM_g} B(X)\,dX\,,
\end{equation}
where $B(X) := \mu_{\mathrm{Th}}(B_X)$ is the Thurston measure of the
unit ball $B_X$ in the space of measured geodesic
laminations with respect to the length function defined by
the hyperbolic metric $X$.
\end{NNTheorem}

The integral on the right-hand side of
formula~\eqref{eq:MV:Volume:as:average:Thurston:measure}
is explicitely evaluated by Mirzakhani in the paper~\cite{Mirzakhani:grouth:of:simple:geodesics}.
Combined with our formula~\eqref{eq:square:tiled:volume}
for the Masur--Veech volume $\Vol\cQ_g$ of the moduli
space of quadratic differentials in normalization of
section~\ref{ss:background:strata} we get the following
result.

\begin{Theorem}
\label{th:factor}
The pushforward
$F_\ast(\mu_{\mathrm{WP}}\otimes\mu_{\mathrm{Th}})$ of the
product measure $\mu_{\mathrm{WP}}\otimes\mu_{\mathrm{Th}}$
on $\cM_g\times\cM\cL_g$ to $\cQ_g$ under the map $F$
admits a density at almost all points of $\cQ_g$. In the
period coordinates of the principal stratum $\cQ(1^{4g-4})$
in $\cQ_g$ this density coincides with the Masur--Veech
volume element $d\Vol_{\mathrm{MV}}$ defined in
section~\ref{ss:background:strata} divided by the factor
$(4g-4)!\cdot 2^{4g-3}$:
$$
F_\ast(\mu_{\mathrm{WP}}\otimes\mu_{\mathrm{Th}})=
\frac{d\Vol_{\mathrm{MV}}}{(4g-4)!\cdot 2^{4g-3}}
$$

The Masur--Veech volume $\Vol\cQ^1\cM_g$ in
Mirzakhani's normalization~\eqref{eq:MV:Volume:as:average:Thurston:measure}
and the Masur--Veech volume $\Vol\cQ_g$
in normalization of formula~\eqref{eq:square:tiled:volume}
are related by the following factor:
$$
\Vol\cQ_g
=\big((12g-12)\cdot
(4g-4)!\cdot 2^{4g-3}\big)\cdot
\Vol\cQ^1\cM_g\,.
$$
\end{Theorem}

Theorem~\ref{th:factor} is obtained as an immediate
corollary of the more general result presented
in Theorem~\ref{th:our:density:equals:Mirzakhani:density}.

\subsection{Mirzakhani's formulae for the Masur--Veech volume of $\cQ_g$
and for asymptotic frequencies of simple closed geodesic multicurves}

Integral~\eqref{eq:MV:Volume:as:average:Thurston:measure}
is a particular case a more general integral
\begin{equation}
\label{eq:bgn:as:integral}
b_{g,n}:=\int_{\overline{\cM}_{g,n}} B(X)\,dX
\end{equation}
introduced in Theorem~1.2
in~\cite{Mirzakhani:grouth:of:simple:geodesics}. The latter
integral is computed in Theorem~5.3 on page 118
in~\cite{Mirzakhani:grouth:of:simple:geodesics}.
To reproduce the corresponding formula and closely related
formula for for the asymptotic frequencies $c(\gamma)$
of simple closed geodesic multicurves $\gamma$
of fixed topological type we need to remind the
notations from~\cite{Mirzakhani:grouth:of:simple:geodesics}.
\medskip

\noindent
\textbf{Simple closed multicurves.}
Depending on the context, we denote by the same symbol
$\gamma$ a collection of disjoint, essential, nonperipheral
simple closed curves, no two of which are in the same
homotopy class; a disjoint union of such curves; the
corresponding primitive multicurve; and the corresponding
orbit in the space $\cM\cL_{g,n}$ under the action of the
mapping class group $\Mod_{g,n}$,
$$
\gamma
=\gamma_1\sqcup\dots\sqcup\gamma_k
=\gamma_1+\dots+\gamma_k
=(\gamma_1,\dots,\gamma_k)\,.
$$
(M.~Mirzakhani uses
in~\cite{Mirzakhani:grouth:of:simple:geodesics} symbols
$\gamma$, $\hat\gamma$ and $\tilde\gamma$ for these objects
depending on the context.) To every such multicurve
$\gamma$ M.~Mirzakhani associates
in~\cite{Mirzakhani:grouth:of:simple:geodesics} a
collection of quantities $N(\gamma)$, $M(\gamma)$,
$\Sym(\gamma)$, $b_\gamma, \Vol_{\mathrm{WP}}(\cM_{g,n}(\gamma,x))$
involved in the formula for $b_{g,n}$. For the sake of
completeness, and to simplify formulae comparison we
reproduce the definitions of these quantities; see the
original paper~\cite{Mirzakhani:grouth:of:simple:geodesics}
of M.~Mirzakhani for details.
\medskip

\noindent
\textbf{Collection $\boldsymbol{\cS_{g,n}}$ of all topological types of primitive multicurves.}
Recall that by $S_{g,n}$ we denote a smooth orientable
topological surface of genus $g$ with $n$ punctures.
Consider the finite set $\mathcal{S}_{g,n}$ defined as
\begin{equation}
\label{eq:S:cal:g:n}
\cS_{g,n}:=
\{\gamma\,|\,\gamma\ \text{is a union of simple closed curves on}
\ S_{g,n}\}/\Mod_{g,n}\,.
\end{equation}
(see formula~(5.4) on page~118
in~\cite{Mirzakhani:grouth:of:simple:geodesics}). It is
immediate to see that $\cS_{g,n}$ is an a canonical
bijection with the set $\cG_{g,n}$ of stable graphs defined
in section~\ref{ss:intro:Masur:Veech:volumes},
\begin{equation}
\label{eq:cS:cG}
\cS_{g,n}\simeq \cG_{g,n}\,.
\end{equation}

\noindent
\textbf{Symmetries $\boldsymbol{\Stab(\gamma)}$ and $\boldsymbol{N(\gamma)}$ of a primitive multicurve.}
For any set $A$ of homotopy classes of simple closed curves
on $S_{g,n}$, Mirzakhani defines $\Stab(A)$ as
$$
\Stab(A):=\{g\in\Mod{g,n}\,|\, g\cdot A=A\}\subset\Mod_{g,n}\,.
$$
Having a multicurve $\gamma$ on $S_{g,n}$ as above,
Mirzakhani defines
$$
\Sym(\gamma):=\Stab(\gamma)/\cap_{i=1}^k\Stab(\gamma_i)\,,
$$
(see the beginning of section~4 on page~112
of~\cite{Mirzakhani:grouth:of:simple:geodesics} for both
definitions). For any single connected simple closed curve
$\gamma_i$ one has $|\Sym(\gamma_i)|=1$.

For each connected simple closed curve
$\gamma_i$ define
$\Stab_0(\gamma_i)\subset\Stab(\gamma_i)$ as
the subgroup consisting of elements which
preserve the orientation of $\gamma_i$.
Define $N(\gamma)$ as
$$
N(\gamma):=\left|
\bigcap_{i=1}^k\Stab(\gamma_i)/
\bigcap_{i=1}^k\Stab_0(\gamma_i)\right|\,.
$$
(see page~113
of~\cite{Mirzakhani:grouth:of:simple:geodesics}).

Consider a stable graph $\Gamma(\gamma)$ associated
to a primitive multicurve $\gamma$.
It follows from definitions of $\Stab(\gamma)$, $N(\gamma)$
and $\Aut(\gamma)$ that
\begin{equation}
\label{eq:Sym:N:Aut}
|\Aut(\Gamma(\gamma))|=|\Sym(\gamma)|\cdot N(\gamma)\,.
\end{equation}

\noindent
\textbf{Number $\boldsymbol{M(\gamma)}$ of one-handles.}
Consider now the closed surface
\begin{equation}
\label{eq:decomposition:of:S:g:n}
S_{g,n}(\gamma)=\bigsqcup_{j=1}^s S_{g_j,n_j}
\end{equation}
obtained from $S_{g,n}$ by cutting along all
$\gamma_1,\dots,\gamma_k$. Here $S_{g_1,n_1},\dots,
S_{g_s,n_s}$ are the connected components of the resulting
surface $S_{g,n}(\gamma)$. Define
$$
M(\gamma):=|\{i\,|\, \gamma_i\
\text{separates off a one-handle from}\ S_{g,n}\}|\,,
$$
(see this formula in the statement of Theorem~4.1 on
page~114 of~\cite{Mirzakhani:grouth:of:simple:geodesics}).
By definition, a ``one-handle'' is a surface of genus one
with one boundary component, i.e., a surface of type
$S_{1,1}$.

\begin{Remark}
\label{rm:one:handle} There is one very particular case,
when the index $i$ in the definition of $M(\gamma)$ should
be counted with multiplicity $2$. Namely, when $\gamma$ is
a connected separating simple closed curve on a surface
$S_2=S_{2,0}$, this single simple closed curve separates
off \textit{two} one-handles at once, and thus should be
counted with multiplicity $2$. In other words, the quantity
$M(\gamma)$ might be defined as
\begin{equation}
\label{eq:M:gamma}
M(\gamma):=\text{number of surfaces of type }\ S_{1,1}\
\text{in decomposition~\eqref{eq:decomposition:of:S:g:n}}
\end{equation}
without any exceptions and multiplicities.
\end{Remark}
\medskip

\noindent
\textbf{Volume polynomials $\boldsymbol{V_{g,n}}$.}
In~\cite{Mirzakhani:simple:geodesics:and:volumes}
and~\cite{Mirzakhani:volumes:and:intersection:theory}
proves the following statement, which we state
following Theorems~4.2 and~4.3
in~\cite{Mirzakhani:grouth:of:simple:geodesics}.

\begin{NNTheorem}[Mirzakhani]
The Weil--Petersson volume
$\Vol_{\textrm{WP}}\cM_{g,n}(b)$
of the moduli space
of bordered
hyperbolic surfaces of genus $g$ with hyperbolic
boundary components of lengths $b_1,\dots,b_n$
is a polynomial $V_{g,n}(b_1,\dots,b_n)$
in even powers of $b_1,\dots,b_n$; that is,
\begin{equation*}
\Vol_{\mathrm{WP}}\cM_{g,n}(b)
=V_{g,n}(b)=\sum_{\substack{\alpha\\|\alpha|\le 3g-3+n}} C_\alpha\cdot b^{2\alpha}\,,
\end{equation*}
where $C_\alpha>0$ lies in $\pi^{6g-6+2n-2|\alpha|}\cdot\Q$.
The coefficient $C_\alpha$ is given by
\begin{equation*}
C_\alpha=\frac{1}{2^{|\alpha|}\cdot\alpha!\cdot(3g-3+n-|\alpha|)!}
\int_{\overline{\cM}_{g,n}}
\psi_1^{\alpha_1}\cdots\psi_n^{\alpha_n}
\cdot\omega^{3g-3+n-|\alpha|}\,,
\end{equation*}
where $\omega$ is the Weil--Petersson symplectic form,
$\alpha!=\prod_{i=1}^n \alpha_i!$, and $|\alpha|=\sum_{i=1}^n\alpha_i$.
\end{NNTheorem}

\begin{Remark}[M.~Kazarian]
\label{rm:Kazarian}
For one particular pair $(g,n)$, namely
for $(g,n)=(1,1)$, Mirzakhani's normalization
\begin{equation}
\label{eq:Vol:M11:Mirzakhani}
\Vol^{\mathit{Mirzakhani}}_{\textrm{WP}}\cM_{1,1}(b)
=\frac{1}{24}(b^2+4\pi^2)\,,
\end{equation}
(see equation~(4.5) on page~116
in~\cite{Mirzakhani:grouth:of:simple:geodesics}
or~\cite{Mirzakhani:volumes:and:intersection:theory}) is
twice bigger than the normalization in many other papers.
Topologically $\overline{\cM}_{1,1}$ is homeomorphic to
$\CP$. However, since every elliptic curve with a marked
point admits an involution, the fundamental class of the
orbifold $\overline{\cM}_{1,1}$ used in the integration
equals
$[\overline{\cM}_{1,1}]=\frac{1}{2}\left[\CP\right]\in
H_2(\CP)$ which gives the value
\begin{equation}
\label{eq:Vol:M11:Kazarian}
\Vol_{\textrm{WP}}\cM_{1,1}(b)
=\frac{1}{48}(b^2+4\pi^2)\,.
\end{equation}

\end{Remark}

Denote by $V^{top}_{g,n}(b)$ the homogeneous part of this
highest degree $6g-6+2n$ of $V_{g,n}(b)$. It follows from
the definition of the volume polynomial $V_{g,n}(b)$ that
\begin{equation}
\label{eq:Vgn}
V^{top}_{g,n}(b)=\sum_{|\alpha|= 3g-3+n} C_\alpha\cdot b^{2\alpha}\,,
\end{equation}
where $C_\alpha$ is given by
\begin{equation}
\label{eq:C:alpha}
C_\alpha=\frac{1}{2^{3g-3+n}\cdot\alpha!}
\int_{\overline{\cM}_{g,n}}
\psi_1^{\alpha_1}\cdots\psi_n^{\alpha_n}\,.
\end{equation}
Comparing the definition of $V^{top}_{g,n}(b)$
with the definition~\eqref{eq:N:g:n}--\eqref{eq:correlator}
of the polynomial $N_{g,n}(b)$ we get the following result.
\begin{Lemma}
The homogeneous parts of top degree of Mirzakhani's
volume polynomial $V_{g,n}$ and of Kontsevich's polynomial
$N_{g,n}$ coincide up to the constant factor:
\begin{equation}
\label{eq:Vgn:through:Ngn}
V^{top}_{g,n}(b)
=\begin{cases}
\hspace*{12pt}
2^{2g-3+n}\cdot N_{g,n}(b)
&\text{for}\ (g,n)\neq(1,1)\,;
\\
2\cdot 2^{2g-3+n}\cdot N_{g,n}(b)
&\text{for}\ (g,n)=(1,1)\,.
\end{cases}
\end{equation}
\end{Lemma}
In the exceptional case $(g,n)=(1,1)$ the polynomial
$V^{top}_{1,1}(b)=\frac{1}{24} b^2$ used by Mirzakhani is
twice larger than $N_{1,1}(b)=\frac{1}{48} b^2$. The origin
of this extra factor $2$ is explained in
Remark~\ref{rm:Kazarian}.
\medskip

\noindent \textbf{Volume polynomial
$\boldsymbol{\Vol_{\mathrm{WP}}(\cM_{g,n}(\gamma,x))}$
associated to a multicurve $\boldsymbol{\gamma}$.}
Assuming that the initial surface $S_{g,n}$
is endowed with a hyperbolic metric, and
that the simple closed curves $\gamma_1,\dots,\gamma_k$
are realized by simple closed hyperbolic geodesics
of hyperbolic lengths $x_1=\ell_{\gamma_1}(X),\dots,
x_k=\ell_{\gamma_k}(X)$,
the boundary $\partial S_{g,n}(\gamma)$
gets $k$ pairs of distinguished boundary
components of lengths $x_1,\dots,x_k$.
We assume that the $n$ marked points
of $S_{g,n}$ are represented by hyperbolic cusps,
i.e. by hyperbolic boundary components of zero length.
Denote by $(g_i,n_i)$ the genus and the number of boundary
components (including cusps) of each the surface $S_j$,
where $j=1,\dots,s$. To each boundary component
of each surface with boundary
$S_j$ we have assigned a length variable which is
equal to $x_i$ if the boundary component comes from
the cut along $\gamma_i$, or is equal to $0$ if the boundary
component comes from a cusp (one of the $n$ marked points).
Consider the corresponding Mirzakhani--Weil--Petersson
volume polynomial $V_{g_j,n_j}(x)$
and define
\begin{equation}
\label{eq:Vol:Mgn:gamma:x}
\Vol_{\mathrm{WP}}(\cM_{g,n}(\gamma,x))
:=\frac{1}{N(\gamma)}
\prod_{j=1}^s V_{g_j,n_j}(x)\,,
\end{equation}
(see formula~(4.1) on page~113 of~\cite{Mirzakhani:grouth:of:simple:geodesics}).
Denote by $(2d_1,\dots,2d_k)_\gamma$ the coefficient
of $x^{2d_1}\cdots x^{2d_k}$ in this polynomial and
let
\begin{equation}
\label{eq:b:gamma:5:3}
b_\gamma(2d_1,\dots,2d_k):= (2d_1,\dots,2d_k)_\gamma
\frac{\prod_{i=1}^k(2d_i+1)!}{(6g-6+2n)!}\,,
\end{equation}
(see equation~(5.3) on page~118 in~\cite{Mirzakhani:grouth:of:simple:geodesics}).

Now everything is ready to state Mirzakhani's result and to
prove Theorem~\ref{th:our:density:equals:Mirzakhani:density} and
Theorem~\ref{th:factor}.
\medskip

\noindent
\textbf{Mirzakhani's results and proof of comparison theorems.}
Let $\gamma=\gamma_1\sqcup\dots\sqcup\gamma_k$ be a
primitive simple closed curve as above. Let
$\gamma_a=\sum_{i=1}^k a_i\gamma_i$, where $a_i\in\N$ for
$i=1,\dots,k$. (To follow original notations of Mirzakhani,
we denotes integer weights of components of a multicurve by
$a_i$ and not by $H_i$ as before.)

\begin{NNTheorem}[Mirzakhani~\cite{Mirzakhani:grouth:of:simple:geodesics}]
The frequency $c(\gamma_a)$ of a multi-curve
$\gamma_a=\sum_{i=1}^k a_i\gamma_i$ is equal to
\begin{equation}
\label{eq:c:gamma}
c(\gamma_a)
=\frac{2^{-M(\gamma)}}{|\Sym(\gamma)|}
\cdot\sum_{\substack{d\\|d|=3g-3+n-k}}
\frac{b_\gamma(2d_1,\dots,2d_k)}{a_1^{2d_1+2}\dots a_k^{2d_k+2}}\,.
\end{equation}
The expression $b_{g,n}$ given by the
integral~\eqref{eq:bgn:as:integral} can be represented as
the sum
\begin{equation}
\label{eq:bgn:as:sum:B:gamma}
b_{g,n}=\sum_{\gamma\in\cS_{g,n}} B_\gamma\,,
\end{equation}
where for $\gamma=\bigsqcup_{i=1}^k \gamma_i$
one has
\begin{equation}
\label{eq:B:gamma}
B_\gamma
=\frac{2^{-M(\gamma)}}{|\Sym(\gamma)|}
\cdot\sum_{|d|=3g-3+n-k}
b_\gamma(2d_1,\dots,2d_k)\prod_{i=1}^k\zeta(2s_i+2)\,.
\end{equation}
\end{NNTheorem}

\begin{proof}[Proof of Theorem~\ref{th:our:density:equals:Mirzakhani:density}]
Let $\gamma=\gamma_1\sqcup\dots\gamma_k$
be a primitive simple closed curve in $\cS_{g,n}$.
Let $\Gamma$ be the associated decorated graph.
The sets $\cS_{g,n}$ and $\cG_{g,n}$ are
in the canonical one-to-one correspondence.

Defining $M(\gamma)$ as in~\eqref{eq:M:gamma} we see that
$M(\gamma)$ is the number of terms $V_{1,1}$ in the
product~\eqref{eq:Vol:Mgn:gamma:x}, i.e. number of indices
$j$ in the range $\{1,\dots,s\}$ such that
$(g_j,n_j)=(1,1)$. Thus, the factor $2^{-M(\gamma)}$ in
equations~\eqref{eq:c:gamma} and~\eqref{eq:B:gamma}
compensates the difference in
normalizations~\eqref{eq:Vol:M11:Mirzakhani}
and~\eqref{eq:Vol:M11:Kazarian}, see
Remark~\ref{rm:Kazarian}. Hence,
applying~\eqref{eq:Vgn:through:Ngn} to the right-hand side
of~\eqref{eq:Vol:Mgn:gamma:x} we get
$$
2^{-M(\gamma)}\prod_{j=1}^s V_{g_j,n_j}(x)
=
\left(\prod_{j=1}^s N_{g_j,n_j}(x)\right)\cdot
\left(\prod_{j=1}^s 2^{2g_j-3+n_j}\right)\,.
$$

Having a decorated graph $\Gamma$ one has
\begin{multline*}
g=1-\chi(\Gamma)+\sum_{v_j\in V(\Gamma)} g(v_j)
=1-|V(\Gamma)|+|E(\Gamma)|+\sum_{j=1}^s g_j
=\\=
1+\sum_{j=1}^s \left(g_j -1 +\frac{n_j}{2}\right)
-\frac{(\text{total number of half-edges (legs)})}{2}
\,.
\end{multline*}
(By convention $E(\Gamma)$ denotes the set of ``true''
edges of $\Gamma$ versus $n$ half-edges (legs)
corresponding to $n$ marked points on $S_{g,n}$.) Hence,
$$
\sum_{j=1}^s (2g_j-3+n_j)
=2g-2+n
-|V(\Gamma)|
=2g-3+n
-(|V(\Gamma)|-1)\,.
$$

Note also that by definition
$$
|\Aut(\Gamma)|=|\Sym(\gamma)|\cdot N(\gamma)\,.
$$
We have proved that
\begin{multline}
\label{eq:two:polynomials}
\frac{2^{-M(\gamma)}}{|\Sym(\gamma)|\cdot N(\gamma)}
\prod_{j=1}^s V_{g_j,n_j}(x)
=\\=
2^{2g+n-3}\cdot
\left(\frac{1}{2^{|V(\Graph)|-1}} \cdot
\frac{1}{|\operatorname{Aut}(\Graph)|}
\cdot
\prod_{v\in V(\Graph)}
N_{g_v,n_v}(x)\right)\,.
\end{multline}

Let us prove now
relation~\eqref{eq:Vol:gamma:c:gamma}.
Expressions~\eqref{eq:b:gamma:5:3} and~\eqref{eq:c:gamma}
represent $c(\gamma_a)$ as the sum of terms constructed
using the polynomial in the left-hand side
of~\eqref{eq:two:polynomials}.
Expression~\eqref{eq:contribution:of:gamma:to:volume} represents
$\Vol\big(\gamma(\Gamma,\boldsymbol{a})\big)$ as the same sum of the analogous terms
constructed using the proportional polynomial in the
brackets on the right-hand side
of~\eqref{eq:two:polynomials}. The only difference between
the two somes comes from the global normalization factors
shared by all terms of the sums. Namely,
expression~\eqref{eq:b:gamma:5:3} has extra $(6g-6+2n)!$ in
the denominator, while the first line of
expression~\eqref{eq:P:Gamma} has the extra
global factor $\cfrac{2^{6g-5+2n} \cdot
(4g-4+n)!}{(6g-7+2n)!}$. Taking into consideration the
coefficient of proportionality $2^{2g-3}$ relating the two
polynomials in~\eqref{eq:two:polynomials} we get
$$
(6g-6+2n)!\cdot c(\gamma_a)=2^{2g+n-3}\cdot
\left(\cfrac{2^{6g-5+2n} \cdot (4g-4+n)!}{(6g-7+2n)!}\right)^{-1}
\cdot\Vol\big(\gamma(\Gamma,\boldsymbol{a})\big)\,,
$$
and~\eqref{eq:Vol:gamma:c:gamma} follows.

It remains to notice that by definitions~\eqref{eq:B:gamma}
and~\eqref{eq:c:gamma} of $B_\gamma$ and $c(\gamma_a)$
respectively one has
$$
B_\gamma=\sum_{a\in\N^k} c(\gamma_a)\,.
$$
Similarly, by Definition~\eqref{eq:Vol:Gamma}
of $\Vol\Gamma$ one has
$$
\Vol\Gamma=\sum_{a\in\N^k} \Vol\big(\gamma(\Gamma,\boldsymbol{a})\big)\,.
$$
Since the coefficient of proportionality
$const_{g,n}$
between $\Vol\big(\gamma(\Gamma,\boldsymbol{a})\big)$
and $c(\gamma_a)$ in~\eqref{eq:Vol:gamma:c:gamma}
is common for all $a\in\N^k$,
relation~\eqref{eq:Vol:gamma:c:gamma}
for the individual terms
implies analogous relation
$$
\Vol\Gamma=const_{g,n}\cdot B_\gamma
$$
for the above sums of all terms
over all $\boldsymbol{a}\in\N^k$
and for the corresponding sums
$$
\Vol\cQ_{g,n}=const_{g,n}\cdot B_\gamma
$$
(see~\eqref{eq:square:tiled:volume}
and~\eqref{eq:bgn:as:sum:B:gamma}) over all stable graphs
$\Gamma\in\cG_{g,n}$.
Theorem~\ref{th:our:density:equals:Mirzakhani:density},
Corollary~\ref{cor:Vol:as:b:g:n} and
Theorem~\ref{th:factor} are proved.
\end{proof}


\section{Large genus asymptotics for frequencies
of simple closed curves and of one-cylinder
square-tiled surfaces}
\label{s:2:correlators}

\subsection{Universal bounds for $2$-correlators}

Following Witten~\cite{Witten} define
\begin{equation}
\label{eq:tau:correlator}
\langle \tau_{d_1}\dots\tau_{d_n}\rangle=\int_{\overline{\cM}_{g,n}}\psi^{d_1}\dots\psi^{d_n}\,,
\end{equation}
where $d_1+\dots+d_n=3g-3+n$.

Consider the following normalization of the $2$-correlators
$\langle\tau_k\tau_{3g-1-k}\rangle_g$
introduced in~\cite{Zograf:2:correlators}:
\begin{equation}
\label{eq:a:g:k}
a_{g,k}
=\frac{(2k+1)!!\cdot(6g-1-2k)!!}{(6g-1)!!}
\cdot 24^g\cdot g!
\cdot \langle\tau_k\tau_{3g-1-k}\rangle_g\,.
\end{equation}
By~(7) in~\cite{Zograf:2:correlators} under such
normalization the differences of $2$-correlators admit the
following explicit expression:
\begin{multline}
\label{eq:a:g:k:difference}
a_{g,k+1}-a_{g,k}
=\\=
\cfrac{(6g-3-2k)!!}{(6g-1)!!}
\cdot
\begin{cases}
\cfrac{(6j-1)!!}{j!}
\cdot\cfrac{(g-1)!}{(g-j)!}
\cdot(g-2j)\,,
\quad&\text{ for }k=3j-1\,,
\\ \ & \ \\
-2\cdot\cfrac{(6j+1)!!}{j!}
\cdot\cfrac{(g-1)!}{(g-1-j)!}\,,
\quad&\text{ for }k=3j\,,
\\ \ & \ \\
2\cdot\cfrac{(6j+3)!!}{j!}
\cdot\cfrac{(g-1)!}{(g-1-j)!}\,,
\quad&\text{ for }k=3j+1\,,
\end{cases}
\end{multline}
where $k=0,1,\dots,\left[\cfrac{3g-1}{2}\right]-1$, and
$a_{g,0}=1$.

\begin{Proposition}
\label{pr:main:bounds}
For all $g\in\N$ and for all integer $k$ in the range
$\{2,3,\dots,3g-3\}$ the following bounds are valid:
\begin{equation}
\label{eq:main:bounds}
1-\frac{2}{6g-1}=a_{g,1} =a_{g,3g-2}
< a_{g,k}
< a_{g,0}=a_{g,3g-1}=1\,.
\end{equation}
\end{Proposition}
\smallskip

\noindent\textbf{Structure of the proof}
The expression $a_{g,1}=1-\frac{2}{6g-1}$ immediately
follows from the fact that $a_{g,0}=1$ and recursive
relations~\eqref{eq:a:g:k:difference}. For $g=1$
bounds~\eqref{eq:a:g:k:difference} are trivial. The
symmetry $a_{g,3g-1-k}=a_{g,k}$ allows us to confine $k$ to
the range $\{2,3,\dots,\left[\frac{3g-1}{2}\right]\}$.

Using recursive relations~\eqref{eq:a:g:k:difference} we
evaluate $a_{g,k}$ explicitly for $k=2,\dots,5$ and prove
in Lemma~\ref{lm:k:up:to:5}
bounds~\eqref{eq:a:g:k:difference} for these small values
of $k$. In genera $g=2,3,4$, the expression
$\left[\frac{3g-1}{2}\right]$ is bounded by $5$ which
implies bounds~\eqref{eq:a:g:k:difference} for any $k$ when
$g=2,3,4$. From this point we always assume that $g\ge 5$
and $k$ is in the range
$\{6,\dots,\left[\frac{3g-1}{2}\right]\}$.

We start the
main part of the proof by rewriting the recurrence
relations~\eqref{eq:a:g:k:difference} in a form
convenient for estimates. Namely, we introduce the
function
\begin{equation}
\label{eq:R}
R(g,j)=\frac{\binom{3g}{3j} \binom{g}{j}}{\binom{6g}{6j}}
\end{equation}
and express the right-hand side of each of the recurrence
relations~\eqref{eq:a:g:k:difference} as a product
$R\cdot\frac{P_i}{Q}$, where $P_i$, $i=1,2,3$, and $Q$
are explicit polynomials in $g$ and $j$. In
Lemma~\ref{lm:P:Q} we show that for any $g$ the absolute
value of each of the rational functions $P_i/Q$ on the
range of $j$ corresponding to $k\in
\{6,\dots,\left[\frac{3g-1}{2}\right]\}$ is bounded from
above by $1$. In Lemma~\ref{lm:R} we show that for any
fixed $g$ and $0\le j\le \left[\frac{g-1}{2}\right]$, the
expression $R(g,j)$ is monotonously decreasing as a
function of $j$. Combining these two lemmas we obtain in
Lemma~\ref{lm:difference:bounded:by:jmax:R2} the estimate
$-R(g,2)\cdot\frac{g-3}{2} \le a_{g,k}-a_{g,5} \le
R(g,2)\cdot\frac{3g-11}{3}$ valid for all $g$ and $k$ under
consideration.

We use explicit expressions for rational functions $\elow$
and $\eup$ in $a_{g,5}=1-\frac{2}{6g-1}+\elow =1-\eup$
obtained in Lemma~\ref{lm:k:up:to:5} to prove in
Lemma~\ref{lm:R2:jmax:smaller:than:epsilon} that for
$g\ge5$ the inequalities $R(g,2)\cdot\frac{g-3}{3}<\elow$
and $R(g,2)\cdot\frac{3g-11}{3} <\eup$ hold which completes
the proof.

\subsection{Small values of $k$ and $g$}
Recall that $a_{g,0}=1$. Recursive
relations~\eqref{eq:a:g:k:difference} provide the following
first several terms $a_{g,k}$ for $k=1,2,3,4,5$, where for
each $k$ in this range we indicate the smallest value of
$g$ starting from which the corresponding equality holds:
\begin{align}
\label{eq:a:g:0}
a_{g,0}&=1
&\text{ for }g\ge 1\,,
\\
\label{eq:a:g:1}
a_{g,1}&=1-\frac{2}{6g-1}
&\text{ for }g\ge 1\,,
\\
\notag
a_{g,2}
&=1-\frac{12(g-1)}{(6g-1)(6g-3)}
&\text{ for }g\ge 2\,,
\\
\notag
a_{g,3}
&=1-\frac{3(24g^2-49g+30)}{(6g-1)(6g-3)(6g-5)}
&\text{ for }g\ge 3\,,
\\
\notag
a_{g,4}
&=1-\frac{2}{6g-1}+\frac{9(g-2)(34g-35)}{(6g-1)(6g-3)(6g-5)(6g-7)}
&\text{ for }g\ge 3\,,
\\
\label{eq:epsilon:lower:init}
a_{g,5}
&=1-\frac{2}{6g-1}
+\frac{27(68g^3-308g^2+519g-280)}{(6g\!-\!1)(6g\!-\!3)(6g\!-\!5)(6g\!-\!7)(6g\!-\!9)}
=\\
\label{eq:epsilon:upper:init}
&=1-\frac{9 (g-2) (288 g^3-780 g^2+1012 g-525)}{(6g-1)(6g-3)(6g-5)(6g-7)(6g-9)}
&\text{ for }g\ge 4\,.
\end{align}

\begin{Lemma}
\label{lm:k:up:to:5}
For all $g\in\N$ and for all $k\in\N$ satisfying
$2\le k\le \min\left(5,\left[\frac{3g-1}{2}\right]\right)$
the relations~\eqref{eq:main:bounds} are valid:
\begin{equation*}
1-\frac{2}{6g-1}=a_{g,1} < a_{g,k} < a_{g,0}=1\,.
\end{equation*}
\end{Lemma}
\begin{proof}

By~\eqref{eq:a:g:0}
and~\eqref{eq:a:g:1} the terms $a_{g,0}$ and $a_{g,1}$ indeed have
values as claimed in the statement of
Lemma~\ref{lm:k:up:to:5}.

It follows from recurrence
relations~\eqref{eq:a:g:k:difference} that for $k$
satisfying $0\le k\le \left[\frac{3g-1}{2}\right]-1$ we
have $a_{g,k+1}<a_{g,k}$ if and only if $k\equiv
0\,(\operatorname{mod} 3)$ and we have $a_{g,k+1}>a_{g,k}$
for the remaining $k$ in this range. In particular, for
$g\ge 2$ the difference $a_{g,2}-a_{g,1}$ is strictly
positive, which implies the desired lower
bound~\eqref{eq:main:bounds} for $a_{g,2}$ when $g\ge 2$.
The explicit expression for $a_{g,2}$ when $g\ge 2$ implies
the desired strict upper bound~\eqref{eq:main:bounds}.

Recursive relations~\eqref{eq:a:g:k:difference} imply that
for $g\ge 3$ the difference $a_{g,3}-a_{g,2}$ is strictly
positive. Since $a_{g,2}$ satisfies the desired lower
bound~\eqref{eq:main:bounds}, the term $a_{g,3}$ also does.
The quadratic polynomial $(24g^2-49g+30)$ in the numerator
of the explicit expression for $a_{g,3}$ admits only
strictly positive values, which implies the upper
bound~\eqref{eq:main:bounds} for $a_{g,3}$ when $g\ge 3$.

Recursive relations~\eqref{eq:a:g:k:difference} imply that
for $g\ge 3$ the difference $a_{g,4}-a_{g,3}$ is strictly
negative. Since $a_{g,3}$ satisfies the desired upper
bound~\eqref{eq:main:bounds}, the term $a_{g,4}$ also does.
The explicit expression for $a_{g,4}$ implies the lower
bound~\eqref{eq:main:bounds} for $a_{g,4}$ when $g\ge 3$.

Finally, recursive relations~\eqref{eq:a:g:k:difference}
imply that for $g\ge 4$ the difference $a_{g,5}-a_{g,4}$ is
strictly positive, which implies the desired lower
bound~\eqref{eq:main:bounds} for $a_{g,5}$ when $g\ge 4$.
It remains to verify that the polynomial $(288 g^3-780
g^2+1012 g-525)$ in the explicit
expression~\eqref{eq:epsilon:upper:init} for $a_{g,5}$
attains only strictly positive values for $g\ge 4$ to prove
the desired upper bound~\eqref{eq:main:bounds} for
$a_{g,5}$. For $g\ge 4$ we have:
\begin{multline*}
288 g^3-780 g^2+1012 g-525
>\\>
288 g^3 - 864 g^2 + 864 g  -576
=288((g-1)^3-1)>0\,.
\end{multline*}
\end{proof}

\begin{Corollary}
\label{cor:g:up:to:4}
For any $g$ in $\{1,2,3,4\}$ and for all $k\in\N$
satisfying $2\le k\le 3g-3$, the desired
bounds~\eqref{eq:main:bounds} are valid:
\begin{equation*}
1-\frac{2}{6g-1}=a_{g,1} < a_{g,k} < a_{g,0}=1\,.
\end{equation*}
\end{Corollary}
\begin{proof}
Recall that the symmetry $a_{g,k}=a_{g,3g-1-k}$ allows to
limit $k$ to the range $2\le k\le
\left[\frac{3g-1}{2}\right]$. Thus, for $g\le 4$ the
largest possible value of $k$ satisfying $k\le
\left[\frac{3g-1}{2}\right]$ equals to $\left[\frac{3\cdot
4-1}{2}\right]=5$. The proof of
bounds~\eqref{eq:main:bounds} for $k\le 5$ and any $g$ is
already completed in Lemma~\ref{lm:k:up:to:5}.
\end{proof}

\subsection{Alternative form of recurrence relations}

We start by extracting the common factor in
relations~\eqref{eq:a:g:k:difference} and by simplifying
it.

Define the following polynomial in $g$ and $j$:
\begin{equation}
\label{eq:Q}
Q(g,j)=g(6g-6j-1)(6g-6j-3)\,,
\end{equation}
Rewriting double factorials in terms of factorials
and rearranging we get
\begin{multline*}
\cfrac{(6g-6j-5)!!}{(6g-1)!!}\cdot
\cfrac{(6j-1)!!\cdot(g-1)!}{j!\,(g-j)!}
=\\=
\left(\frac{(6g-6j-5)!}{(3g-3j-3)!\,2^{3g-3j-3}}\right)
\left(\frac{(3g-1)!\,2^{3j-1}}{(6g-1)!}\right)
\left(\frac{(6j-1)!}{(3j-1)!\,2^{3j-1}}\right)
\frac{(g-1)!}{j!\,(g-j)!}
=\\=
8\cdot
\left(\frac{(6j-1)!\,(6g-6j-5)!}{(6g-1)!}\right)
\left(\frac{(3g-1)!}{(3j-1)!\,(3g-3j-3)!}\right)
\left(\frac{(g-1)!}{j!\,(g-j)!}\right)
=\\=
8\cdot
\left(\frac{(6j)!\,(6g-6j)!}{(6g)!}\right)
\left(\frac{(3g)!}{(3j)!\,(3g-3j)!}\right)
\left(\frac{g!}{j!\,(g-j)!}\right)
\cdot\\ \cdot
\frac{6g}{(6j)\cdot(6g-6j)(6g-6j-1)(6g-6j-2)(6g-6j-3)(6g-6j-4)}
\cdot\\ \cdot
\frac{(3j)\cdot(3g-3j)(3g-3j-1)(3g-3j-2)}{3g}
\cdot\frac{1}{g}
=\\=
\frac{\binom{3g}{3j} \binom{g}{j}}{\binom{6g}{6j}}
\cdot
\frac{1}{g\cdot(6g-6j-1)(6g-6j-3)}
=R(g,j)\cdot\frac{1}{Q(g,j)}\,.
\end{multline*}
   %
   %
Defining the following polynomials in $g$ and $j$:
\begin{align}
\label{eq:p1}
P_1(g,j)&=(6g-6j-1)(6g-6j-3)(g-2j)\,,\\
\label{eq:p2}
P_2(g,j)&=-2(6g-6j-3)(6j+1)(g-j)\,,\\
\label{eq:p3}
P_3(g,j)&=2(6j+1)(6j+3)(g-j)\,,
\end{align}
we can express the recurrence
relations~\eqref{eq:a:g:k:difference} as
\begin{align}
\label{eq:R:P1:Q}
a_{g,3j}-a_{g,3j-1} &= R(g,j)\cdot \frac{P_1(g,j)}{Q(g,j)}\,,
&\text{ where }3\le 3j\le\left[\frac{3g-1}{2}\right]\,,
\\
\label{eq:R:P2:Q}
a_{g,3j+1}-a_{g,3j} &= R(g,j)\cdot \frac{P_2(g,j)}{Q(g,j)}\,,
&\text{ where }1\le 3j+1\le\left[\frac{3g-1}{2}\right]\,,
\\
\label{eq:R:P3:Q}
a_{g,3j+2}-a_{g,3j+1} &= R(g,j)\cdot \frac{P_3(g,j)}{Q(g,j)}\,,
&\text{ where }2\le 3j+2\le\left[\frac{3g-1}{2}\right]\,.
\end{align}

\begin{Lemma}
\label{lm:P:Q}
For any $g\in\N$ the following bounds are valid
\begin{align}
\label{eq:R:P1:Q:bound}
0<\frac{P_1(g,j)}{Q(g,j)}<1\,,
&&\text{ where }3\le 3j\le\left[\frac{3g-1}{2}\right]\,,
\\
\label{eq:R:P2:Q:bound}
-1<\frac{P_2(g,j)}{Q(g,j)}<0\,,
&&\text{ where }1\le 3j+1\le\left[\frac{3g-1}{2}\right]\,,
\\
\label{eq:R:P3:Q:bound}
0<\frac{P_3(g,j)}{Q(g,j)}<1\,,
&&\text{ where }2\le 3j+2\le\left[\frac{3g-1}{2}\right]\,.
\end{align}
\end{Lemma}
\begin{proof}
The bounds $3\le 3j\le\left[\frac{3g-1}{2}\right]$
in~\eqref{eq:R:P1:Q:bound} imply that $0 < j <
\frac{g}{2}$. Dividing expression~\eqref{eq:p1} for
$P_1(g,j)$ by expression~\eqref{eq:Q} for $Q(g,j)$ we get
$$
\frac{P_1(g,j)}{Q(g,j)}=\frac{g-2j}{g}\,.
$$
Clearly,
$$
0<\frac{g-2j}{g}<1
$$
for any $g\in\N$ and for all $j$ satisfying $0 < j <
\frac{g}{2}$.

The bounds $1\le 3j+1\le\left[\frac{3g-1}{2}\right]$
in~\eqref{eq:R:P2:Q:bound} imply that $0\le j
<\frac{g}{2}$. Dividing expression~\eqref{eq:p2} for
$P_2(g,j)$ by expression~\eqref{eq:Q} for $Q(g,j)$ we get
\begin{multline*}
\frac{P_2(g,j)}{Q(g,j)}
=-2\cdot\frac{6j+1}{6g-6j-1}\cdot\frac{g-j}{g}
=-\frac{6g-6j}{6g-6j-1}\cdot\frac{2j+\frac{1}{3}}{g}
=\\=
-\left(1+\frac{1}{6g-6j-1}\right)
\cdot\left(\frac{2j+\frac{1}{3}}{g}\right)\,.
\end{multline*}
Since $0\le 2j \le g-1$, the latter expression is always
strictly negative for this range of $j$. Both factors in the brackets
in the latter expression are monotonously growing
on this range of $j$, so the maximum of the absolute value
of the product is attained at $2j=g-1$. We get
\begin{multline}
\label{eq:P2:over:Q}
0<\left|\frac{P_2(g,j)}{Q(g,j)}\right|
\le\left(1+\frac{1}{6g-(3g-3)-1}\right)
\cdot\left(\frac{g-\frac{2}{3}}{g}\right)
=\\=
\left(1+\frac{1}{3g+2}\right)
\cdot\left(1-\frac{1}{\frac{3}{2}g}\right)<1
\quad\text{ for }g\in\N\
\text{ and }0\le j < \frac{g}{2}\,.
\end{multline}

The bounds $2\le 3j+2\le\left[\frac{3g-1}{2}\right]$
in~\eqref{eq:R:P3:Q:bound} imply that $0\le j<\frac{g}{2}$
and that $6j\le 3g-5$. Dividing expression~\eqref{eq:p3}
for $P_3(g,j)$ by expression~\eqref{eq:Q} for $Q(g,j)$ we
get
$$
\frac{P_3(g,j)}{Q(g,j)}
=
2\cdot\frac{(6j+1)(6j+3)(g-j)}{g(6g-6j-1)(6g-6j-3)}
=\left|\frac{P_2(g,j)}{Q(g,j)}\right|
\cdot\frac{6j+3}{6g-6j-3}\,.
$$
The expression $\frac{6j+3}{6g-6j-3}$ is strictly positive
and is monotonously growing on the range of $j$ under
consideration, so it attains its maximum on the largest
possible value of $j$. Since $6j\le 3g-5$ we get
$$
0<\frac{6j+3}{6g-6j-3}\le\frac{3g-2}{3g+2}<1
\quad\text{for }
2\le 3j+2\le\left[\frac{3g-1}{2}\right]
\text{ and }g\in\N\,.
$$
Combined with~\eqref{eq:P2:over:Q} this proves the desired
bounds~\eqref{eq:R:P3:Q:bound}
which completes the proof of Lemma~\ref{lm:P:Q}.
\end{proof}

\begin{Lemma}
\label{lm:R}
For any fixed value of $g\in\N$, the expression $R(g,j)$
considered as a function of $j$ is strictly monotonously
decreasing on the range
$\left\{0,1,\dots,\left[\frac{g-1}{2}\right]\right\}$ of the
argument $j$.
\end{Lemma}
\begin{proof}
It is immediate to verify that
\begin{multline}
\label{eq:ratio:of:R}
R(g,j+1)/R(g,j)=
\frac{(6j+5)(6j+3)(6j+1)}{(6g-6j-1)(6g-6j-3)(6g-6j-5)}
\cdot\frac{g-j}{j+1}
=\\=
\frac{j+5/6}{j+1}
\cdot\frac{6j+1}{6g-6j-5}
\cdot\frac{6j+3}{6g-6j-3}
\cdot\frac{6g-6j}{6g-6j-1}
=\\=
\left(1-\frac{1}{6j+6}\right)
\cdot\frac{6j+1}{6g-6j-5}
\cdot\frac{6j+3}{6g-6j-3}
\cdot\left(1+\frac{1}{6g-6j-1}\right)
\end{multline}
For any fixed $g\in\N$ each of the four terms in the last
line of the above expression is
strictly monotonously increasing as
a function of $j$ on the range
$\{0,1,...\left[\frac{g-1}{2}\right]\}$.
It is immediate to verify that
when $2j=g-1$, the product of four terms
in the last line of the above expression
is identically equal to $1$ for all $g\in\N$,
and the Lemma follows.
\end{proof}

\begin{Lemma}
\label{lm:difference:bounded:by:jmax:R2}
For any $g\in\N$ and for any
integer $k$ in the range
$6\le k\le \left[\frac{3g-1}{2}\right]$
the following bounds hold:
\begin{equation}
\label{eq:agk:minus:ag5}
-R(g,2)\cdot\frac{g-3}{2}
\le a_{g,k}-a_{g,5}
\le R(g,2)\cdot\frac{3g-11}{3}\,.
\end{equation}
\end{Lemma}
\begin{proof}
We can assume that $g\ge 5$; otherwise the range of $k$ is
empty and the statement is vacuous.

Consider the sequence
$\left\{a_{g,5},a_{g,6},\dots,a_{g,k}\right\}$, and denote
by $n_+(k)$ the number of entries $m$ in the set
$\{5,6,\dots,k-1\}$, for which the inequality
$a_{g,m+1}>a_{g,m}$ holds. Similarly, denote by $n_-(k)$
the number of entries in the same set for which inequality
$a_{g,m+1}<a_{g,m}$ holds.

Combining the recurrence relations in the
form~\eqref{eq:R:P1:Q}--\eqref{eq:R:P3:Q},
bounds~\eqref{eq:R:P1:Q:bound}--\eqref{eq:R:P3:Q:bound} and
Lemma~\ref{lm:R} we conclude that
$$
-R(g,2)\cdot n_-(k)
\le a_{g,k}-a_{g,5}
\le R(g,2)\cdot n_+(k)\,.
$$
It remains to translate the restriction $k\le
\left[\frac{3g-1}{2}\right]$ into upper bounds for $n_+(k)$
and $n_-(k)$ as functions of $g$.

Recall that it follows from recurrence
relations~\eqref{eq:a:g:k:difference} that for $m$
satisfying $0\le m\le \left[\frac{3g-1}{2}\right]-1$ we
have $a_{g,m+1}<a_{g,m}$ if and only if $m\equiv
0\,(\operatorname{mod} 3)$ and we have $a_{g,m+1}>a_{g,m}$
for the remaining $m$ in this range. This implies that
\begin{align*}
n_+&=2(j-2)
&n_-&=j-2\,,
&\text{when }k=3j-1\,,
\\
n_+&=2(j-2)+1
&n_-&=j-2\,,
&\text{when }k=3j\,,\hspace*{17pt}
\\
n_+&=2(j-2)+1
&n_-&=j-1\,,
&\text{when }k=3j+1\,.
\end{align*}
In all these cases we have
\begin{align*}
n_+(k) &\le \frac{2k-10}{3}\,,
\\
n_-(k) &\le \frac{k-4}{3}\,.
\end{align*}
By assumption $k\le \left[\frac{3g-1}{2}\right]$,
so the latter bounds imply that
\begin{align*}
n_+(k) &\le \frac{3g-11}{3}
\\
n_-(k) &\le \frac{g-3}{2}\,.
\end{align*}
and~\eqref{eq:agk:minus:ag5} follows.
\end{proof}

We assume that $g\ge 5$ and
$k\in\{6,\dots,\left[\frac{3g-1}{2}\right]\}$. Define
\begin{align}
\label{eq:epsilon:lower:def}
\elow&=\frac{27(68g^3-308g^2+519g-280)}{(6g-1)(6g-3)(6g-5)(6g-7)(6g-9)}\,,
\\
\label{eq:epsilon:upper:def}
\eup&=\frac{9 (g-2) (288 g^3-780 g^2+1012 g-525)}{(6g-1)(6g-3)(6g-5)(6g-7)(6g-9)}\,.
\end{align}
In these notations
expressions~\eqref{eq:epsilon:lower:init}
and~\eqref{eq:epsilon:upper:init} for $a_{g,5}$ can be
written as
\begin{equation}
\label{eq:a:g:5:in:terms:of:epsilon}
a_{g,5}=1-\frac{2}{6g-1}+\elow =1-\eup\,.
\end{equation}

\begin{Lemma}
\label{lm:R2:jmax:smaller:than:epsilon}
For any integer $g$ satisfying $g\ge 5$
the following strict inequalities are valid:
\begin{align*}
R(g,2)&\cdot\frac{g-3}{2}\quad   <\elow
\\
R(g,2)&\cdot\frac{3g-11}{3} <\eup\,.
\end{align*}
\end{Lemma}
\begin{proof}
The proof is a straightforward calculation.

First note that all the quantities $R(g,2), (g-3), (3g-11),
\elow, \eup$ are strictly positive for $g\ge 5$, where
strict positivity of $\elow$ and of $\eup$ was proved in
Lemma~\ref{lm:k:up:to:5}. Thus, it is sufficient to prove
that
\begin{align}
\label{eq:low}
\frac{2\,\elow}{R(g,2)\cdot(g-3)} > 1
\quad\text{for }g\ge 5\,,
\\
\label{eq:up}
\frac{3\,\eup}{R(g,2)\cdot(3g-11)} > 1
\quad\text{for }g\ge 5\,.
\end{align}

Applying definition~\eqref{eq:R} to evaluate $R(g,2)$ and
cancelling out common factors in the numerator and in the
denominator of the resulting expression we get
\begin{equation}
\label{eq:R:g:2}
R(g,2)=
\frac{10395}{2}
\cdot\frac{g(g-1)}{(6g-1)(6g-3)(6g-5)(6g-7)(6g-9)(6g-11)}\,.
\end{equation}

Plug expression~\eqref{eq:epsilon:lower:def} for $\elow$
and the above expression~\eqref{eq:R:g:2} for $R(g,2)$ into
the left-hand side of~\eqref{eq:low} and cancel out the
common factors in the numerator and in the denominator of
the resulting expression. Applying polynomial division with
remainder to the resulting numerator and denominator we get
\begin{multline*}
\frac{2\,\elow}{R(g,2)\cdot(g-3)}=
2\cdot 27\cdot\frac{2}{10395}
\cdot\frac{(68g^3-308g^2+519g-280)(6g-11)}{g(g-1)(g-3)}
=\\=\frac{4}{385}
\cdot\left(
408 g -964 + \frac{1422 g^2- 4497 g + 3080}{g(g-1)(g-3)}
\right)\,.
\end{multline*}
It is immediate to check that $(1422 g^2- 4497 g + 3080)$
is positive for $g\ge 5$. It remains to note that for $g\ge
5$ we have $(408g-964)\ge (408\cdot 5-964)=1076> 385/4$ which completes the
proof of~\eqref{eq:low}.

Performing analogous manipulations we get
\begin{multline*}
\frac{3\,\eup}{R(g,2)\cdot(3g-11)}
=\\=
3\cdot 9\cdot\frac{2}{10395}
\cdot\frac{(g-2) (288 g^3-780 g^2+1012 g-525)(6g-11)}{g(g-1)(3g-11)}
=\\=
\frac{2}{385}
\cdot\left(
576 g^2 - 1080 g + 2964
+\frac{9790 g^2 + 1735 g - 11550}{g(g-1)(3g-11)}
\right)\,.
\end{multline*}
It is immediate to check that $(9790 g^2 + 1735 g - 11550)$
is positive for $g\ge 5$ as well as the denominator of the
corresponding fraction. It remains to note that the
function $(576 g^2 - 1080 g + 2964)$ is monotonously
growing on the interval $[5;+\infty[$, so for $g\ge 5$ we
get:
$$
576 g^2 - 1080 g + 2964\ge 11964
\ge 576\cdot 5^2 - 1080\cdot 5 + 2964= 11964
> 385/2\,,
$$
which completes the proof of~\eqref{eq:up}.
\end{proof}

\begin{proof}[Proof of Proposition~\ref{pr:main:bounds}]
For small genera, $g=1,2,3,4$,
Proposition~\ref{pr:main:bounds} was proved in
Corollary~\ref{cor:g:up:to:4}.

For genera $g\ge 5$ and $k=2,3,4,5$,
Proposition~\ref{pr:main:bounds} was proved in
Lemma~\ref{lm:k:up:to:5}. The symmetry
$a_{g,k}=a_{g,3g-1-k}$ implies
Proposition~\ref{pr:main:bounds} for symmetric values of
$k$.

For $g\ge 5$ and $k$ in the range $6\le k\le
\left[\frac{3g-1}{2}\right]$
Proposition~\ref{pr:main:bounds} immediately follows from
combination of
Lemma~\ref{lm:difference:bounded:by:jmax:R2},
expression~\eqref{eq:a:g:5:in:terms:of:epsilon} for
$a_{g,5}$ and Lemma~\ref{lm:R2:jmax:smaller:than:epsilon}.
The symmetry $a_{g,k}=a_{g,3g-1-k}$ implies
Proposition~\ref{pr:main:bounds} for symmetric values of
$k$.
\end{proof}

\subsection{Asymptotic behavior of normalized $2$-correlators
in large genera}

In this section we describe briefly behavior of $a_{g,k}$
for $g\gg 1$. More detailed discussion would be given in a
separate (and more general) paper.

When $g\to +\infty$ and $j$ remains bounded,
expressions~\eqref{eq:p1}--\eqref{eq:p3} and~\eqref{eq:Q}
for polynomials $P_i(g,j)$, $i=1,\dots,4$, and $Q(g,j)$
respectively imply that
\begin{align*}
\frac{P_1(g,j)}{Q(g,j)}
&=1-(2j)\cdot\frac{1}{g}\,,
\\
\frac{P_2(g,j)}{Q(g,j)}
&=-\left(2j+\frac{1}{3}\right)\cdot\frac{1}{g}+o\left(\frac{1}{g}\right)\,,
\\
\frac{P_3(g,j)}{Q(g,j)}
&=\left(2j+\frac{1}{3}\right)\left(j+\frac{1}{2}\right)\cdot\frac{1}{g^2}+o\left(\frac{1}{g^2}\right)\,,
\end{align*}
as $g\to+\infty$. In
particular, for $g\gg1$ we see that for small values of $j$
the ratio $\frac{P_1(g,j)}{Q(g,j)}$ is close to $1$, while
the ratio $\frac{P_2(g,j)}{Q(g,j)}$ is of the order
$\frac{1}{g}$ and the ratio $\frac{P_3(g,j)}{Q(g,j)}$ is of
the order $\frac{1}{g^2}$. Thus, assuming that $g\gg 1$,
and taking consecutive terms $\left\{a_{g,3j-1},\,a_{g,3j},\,a_{g,3j+1},\,a_{g,3j+1}\right\}$
with $j\ll g$
we observe certain increment from $a_{g,3j-1}$ to $a_{g,3j}$,
much smaller decrement from $a_{g,3j}$ to $a_{g,3j+1}$
and very small increment from $a_{g,3j+1}$ to $a_{g,3j+2}$.

For any fixed $g\gg 1$ and $j\ll g$ the quantity $R(g,j)$
defined by~\eqref{eq:R} is very rapidly decreasing
as $j$ grows. We conclude from expression~\eqref{eq:ratio:of:R}
that for bounded $j$ and $g\to+\infty$ one has
$$
R(g,j+1)=R(g,j)\cdot\frac{(j+\frac{5}{6})(j+\frac{3}{6})(j+\frac{1}{6})}{(j+1)}
\cdot\frac{1}{g^2}\cdot(1+o(1))\,.
$$
Since $R(g,0)=1$ we get the following expressions for $j=0,1,2,3$:
\begin{align*}
R(g,0)&=1\,,
\\
R(g,1)&
=\frac{(0+\frac{5}{6})(0+\frac{3}{6})(0+\frac{1}{6})}{(0+1)}
\cdot\frac{1}{g^2}\cdot\big(1+o(1)\big)
&
=\frac{5}{72}\cdot\frac{1}{g^2}\cdot\big(1+o(1)\big)\,,
\\
R(g,2)&
=\frac{(1+\frac{5}{6})(1+\frac{3}{6})(1+\frac{1}{6})}{(1+1)}
\cdot\frac{5}{72}\cdot\frac{1}{g^4}\cdot\big(1+o(1)\big)
&
=\frac{385}{3456}\cdot\frac{1}{g^4}\cdot\big(1+o(1)\big)\,,
\\
R(g,3)&
=\frac{(2+\frac{5}{6})(2+\frac{3}{6})(2+\frac{1}{6})}{(2+1)}
\cdot\frac{385}{3456}\cdot\frac{1}{g^6}\cdot\big(1+o(1)\big)
\hspace*{-2.8pt}
&
\hspace*{-5.1pt}
=\frac{425425}{746496}\cdot\frac{1}{g^6}\cdot\big(1+o(1)\big)\,.
\end{align*}
Lemma~\ref{lm:difference:bounded:by:jmax:R2}
admits the following immediate generalization:
\begin{Lemma}
\label{lm:difference:bounded:by:Rj}
For any $g\in\N$ and for any
integer $k$ in the range
$3j\le k\le \left[\frac{3g-1}{2}\right]$
the following bounds hold:
\begin{equation}
\label{eq:agk:minus:ag:j}
-R(g,j)\cdot\frac{g-j-1}{2}
\le a_{g,k}-a_{g,3j-1}
\le R(g,j)\cdot\frac{3g-2j-7}{3}\,.
\end{equation}
\end{Lemma}
Thus, having found the asymptotic expansion (when
$g\to+\infty$) in $\frac{1}{g}$ up to the term
$\frac{1}{g^{2j-2}}$ for some $a_{g,3j-1}$, we get the same
asymptotic expansion up to the term $\frac{1}{g^{2j-2}}$
for all $a_{g,k}$ with $k$ satisfying $3j-1\le k\le 3g-3j$.

Finally, no matter whether $g=2j$ or $g=2j+1$ one easily
derives from Stirling formula that
$$
R(g,j)\approx\frac{1}{2^{2g-1}}\cdot\frac{1}{\sqrt{\pi g}}\,.
$$
Morally, when $g\gg1$ and the index $k$ is located
sufficiently far from the extremities of the range
$\{0,1,\dots,3g-1\}$, the values of $a_{g,k}$ become,
basically, indistinguishable.

\begin{figure}[hbt]
\includegraphics{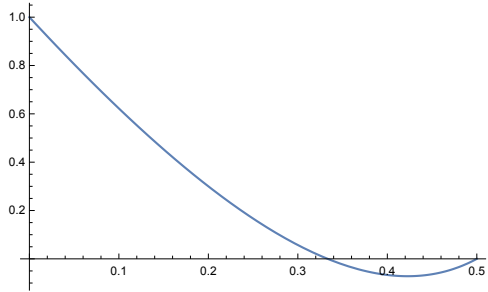}
\vspace{90bp}
\caption{
\label{fig:graph}
Graph of $f(x)$}
\end{figure}

It is curious to note that the sequence
$\{a_{g,2},\,a_{g,5},\dots,a_{g,3j_{max}-1}\}$, where
$j_{max}$ is the maximum integer satisfying
$3j_{max}-1\le\left[\frac{3g-1}{2}\right]$, is not
monotonously increasing for $g\ge 17$. Let $x=\frac{j}{g}$.
Our bounds on $j$ imply that $0<x<\frac{1}{2}$.
For any fixed $g\gg1$ define
$$
f(x)=f\left(\frac{j}{g}\right)=\frac{(P_1(g,j)+P_2(g,j)+P_3(g,j))}{Q(g,j)}\,.
$$
For large values of $g$ the graph of $f(x)$ has the form as
in Figure~\ref{fig:graph}, so up to some
point the function $f$ remains positive and the sequence
$a_{g,2},\,a_{g,5},\dots$ monotonously increases, but then
it attains its maximum and very slowly monotonously
decreases up to $a_{g,3j_{max}-1}$.


\subsection{Asymptotic volume contribution of one-cylinder
square-tiled surfaces.}
\label{s:simple:closed:geodesics:asymptotics}

In this section we compute the large genus asymptotics for
the contributions of the stable graphs having a single edge
to the Masur--Veech volume $\Vol\cQ_g$ of the moduli space
of holomorphic quadratic differentials on complex curves of
genus $g$. In other words, we compute asymptotic
contributions of one-cylinder square-tiled surfaces to the
Masur--Veech volume of the principal stratum
$\cQ(1^{4g-4})$ of holomorphic quadratic differentials in
large genus $g$. These contributions obviously provide
lower bounds for $\Vol\cQ_g$.

Denote by $\Graph_1(g)$ the stable graph having a single
vertex and having a single edge and such that the vertex is
decorated with integer $g-1$
(see~Figure~\ref{fig:non:separating}). The single edges
forms a loop, so this stable graphs represents a surface of
genus $g$ with no marked points. This stable graph encodes
the orbit of a simple closed nonseparating curve on a
surface of genus $g$. This stable graph represents also the
boundary divisor
$\delta_{0,g}\subset\overline{\cM}_g-\cM_g$ for which the
generic stable curve has single irreducible component of
genus $g-1$ with single node.

\begin{figure}[htb]
\includegraphics{genus_two_graph_11.eps}
\begin{picture}(0,0)(158,0) 
\put(44,-14){$g-1$}
\end{picture}
\includegraphics{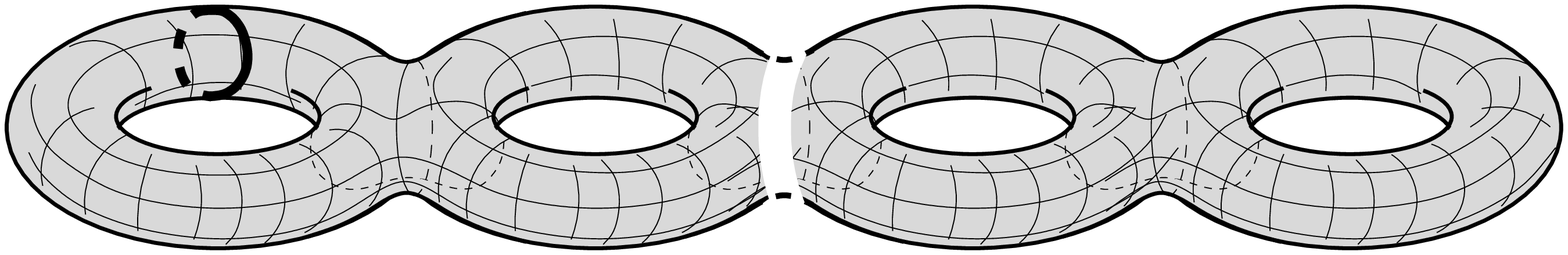}
\begin{picture}(0,0)(155,0) 
\put(75,-26){$\underbrace{\rule{175pt}{0pt}}_g$}
\put(280,-29){{\scriptsize geometric}}
\put(283,-36){{\scriptsize genus $g$}}
\end{picture}
\includegraphics{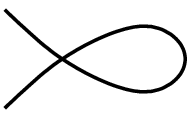}
\vspace*{40pt}
\caption{
\label{fig:one:vertex:one:loop}
Graph $\Graph_1(g)$ on the left
represents non-separating simple closed curves on a surface
of genus $g$ (as in the middle) and the boundary divisor
$\delta_{0,g}$ of irreducible stable curves with single node
as on the right.}
\label{fig:non:separating}
\end{figure}

\begin{Theorem}
\label{th:contribution:Gamma:1} The contribution of the
graph $\Graph_1(g)$ having single vertex decorated with
label $g-1$ and single edge to the Masur--Veech volume
$\Vol\cQ_g$ of the moduli space $\cQ_g$ of holomorphic
quadratic differentials on complex curves of genus $g$ has
the following asymptotics:
\begin{equation}
\label{eq:asymptotic:contribution:Gamma:1}
\Vol\Graph_1(g)
=\sqrt{\frac{2}{3\pi g}}
\cdot\left(\frac{8}{3}\right)^{4g-4}
\cdot\left(1+O\left(\frac{1}{g}\right)\right)
\quad\text{as }g\to+\infty\,.
\end{equation}
\end{Theorem}
\begin{proof}
It would be slightly more convenient to prove
formula~\eqref{eq:asymptotic:contribution:Gamma:1} in genus
$g+1$. Assign variable $b_1$ to the only edge of the graph.
Formula~\eqref{eq:volume:contribution:of:stable:graph} from
Theorem~\ref{th:volume}
applied to the graph $\Graph_1(g+1)$ gives
\begin{multline}
\label{eq:contribution:graph:G1:initial}
\left(\frac{(4g)!}{(6g)!}\cdot 2^{6g}\cdot 12g\right) \cdot
\frac{1}{2}\cdot 1\cdot
b_1 \cdot
N_{g,2}(b_1,b_1)
\xmapsto{\cZ}
\\
\xmapsto{\ \cZ\ }
\Vol\Graph_1(g+1)
=(4g)!\cdot2^{g+2}\cdot\zeta(6g)
\cdot\sum_{d_1+d_2=3g-1}
\frac{\langle\psi_1^{d_1}\psi_2^{d_2}\rangle}{d_1!\cdot d_2!}
=\\=
(4g)!\cdot2^{g+2}\cdot\zeta(6g)
\cdot\sum_{k=0}^{3g-1}
\frac{\langle\tau_k\tau_{3g-1-k}\rangle_g}{k!\cdot(3g-1-k)!}
\,.
\end{multline}

Now pass to the normalization~\eqref{eq:a:g:k}
of the correlators
$\langle\tau_k\tau_{3g-1-k}\rangle_g$:
\begin{equation*}
a_{g,k}
=\frac{(2k+1)!!\cdot(6g-1-2k)!!}{(6g-1)!!}
\cdot 24^g\cdot g!
\cdot \langle\tau_k\tau_{3g-1-k}\rangle_g\,.
\end{equation*}
By Proposition~\ref{pr:main:bounds}, the $2$-correlators
admit the following uniform bounds under such
normalization:
$$
1-\frac{2}{6g-1}
\le a_{g,k}
\le 1\,,\quad\text{for }k=0,1,\dots,3g-1\,.
$$

Rewriting the expression~\eqref{eq:contribution:graph:G1:initial}
for $\Vol\Graph_1(g+1)$ in terms of $a_{g,k}$
we get
\begin{multline}
\label{eq:contribution:graph:G1:a:g:k:init}
\Vol\Graph_1(g+1)/\zeta(6g)
=\\=
(4g)!\cdot2^{g+2}
\cdot\frac{(6g-1)!!}{24^g\cdot g!}
\cdot\sum_{k=0}^{3g-1}
\frac{a_{g,k}}
{k!\cdot(2k+1)!!\cdot(3g-1-k)!\cdot(6g-1-2k)!!}\,.
\end{multline}
Passing from double factorials to factorials,
\begin{align*}
(6g-1)!! &= \frac{(6g)!}{(3g)!\cdot 2^{3g}}\\
(2k+1)!! &= \frac{(2k+1)!}{k!\cdot 2^k}\\
(6g-1-2k)!! &=\frac{(6g-1-2k)!}{(3g-1-k)!\cdot 2^{3g-k-1}}\,,
\end{align*}
we rewrite and simplify expression~\eqref{eq:contribution:graph:G1:a:g:k:init}
as follows:
\begin{multline}
\label{eq:contribution:graph:G1:a:g:k:intermediate}
\Vol\Graph_1(g+1)/\zeta(6g)
=\\=
(4g)!\cdot 2^{g+2}
\cdot\frac{1}{3^g\cdot 2^{3g}}
\cdot\frac{1}{g!\cdot(3g)!}
\cdot\frac{1}{2}
\sum_{k=0}^{3g-1}
\binom{6g}{2k+1}\cdot a_{g,k}
=\\=
\frac{(4g)}{g!\cdot(3g)!}
\cdot\frac{1}{3^g\cdot 2^{2g}}
\cdot 2\sum_{k=0}^{3g-1}
\binom{6g}{2k+1}\cdot a_{g,k}\,.
\end{multline}
Combining the classical identity
$$
\sum_{k=0}^{3g-1} \binom{6g}{2k+1}= 2^{6g-1}
$$
with bounds~\eqref{eq:main:bounds}
we get the following bounds for $\Vol\Graph_1(g+1)$:
\begin{equation}
\label{eq:contribution:graph:G1:bounds}
\binom{4g}{g}
\cdot\left(\frac{2^4}{3}\right)^g
\cdot\left(1-\frac{2}{6g-1}\right)
\le
\Vol\Graph_1(g+1)/\zeta(6g)
\le
\binom{4g}{g}
\cdot\left(\frac{2^4}{3}\right)^g
\,.
\end{equation}
Note that
$$
\zeta(6g)\to 1\quad\text{as }g\to+\infty
$$
and the convergence is exponentially fast.

Applying Stirling's formula to factorials
in the binomial coefficient
$\binom{4g}{g}$ in the latter expression
and simplifying we get
\begin{equation}
\label{eq:binom:4g:g}
\binom{4g}{g}
=\sqrt{\frac{2}{3\pi g}}
\cdot\left(\frac{2^8}{3^3}\right)^g
\cdot\left(1+O\left(\frac{1}{g}\right)\right)
\quad\text{as }g\to+\infty\,.
\end{equation}
Combining the latter equality with
bounds~\eqref{eq:contribution:graph:G1:bounds}
we get the desired formula~\eqref{eq:asymptotic:contribution:Gamma:1}
in genus $g+1$:
$$
\Vol\Graph_1(g+1)
=\sqrt{\frac{2}{3\pi g}}
\cdot\left(\frac{8}{3}\right)^{4g}
\cdot\left(1+O\left(\frac{1}{g}\right)\right)
\quad\text{as }g\to+\infty\,.
$$
\end{proof}

\begin{Remark}
\label{rm:potential:detalisation}
Actually, we have very good control of asymptotic
expansions of correlators $a_{g,k}$ in powers of
$\frac{1}{g}$, see~\cite{}, so it would not be difficult to
specify several terms of the asymptotic expansion of
$O\left(\frac{1}{g}\right)$ in
formula~\eqref{eq:asymptotic:contribution:Gamma:1}. We do
not do it only because we do not currently see any specific
need for a more precise expression.
\end{Remark}

Theorem~\ref{th:contribution:Gamma:1} immediately implies
Theorem~\ref{th:asymptotic:lower:bound}.
\bigskip

We proceed with the remaining graph having single edge.
This time it has two vertices joined by the edge as in
Figure~\ref{fig:separating}. The two vertices are decorated
with strictly positive integers $g_1,g_2\in\N$ such that
$g_1+g_2=g$ (see~Figure~\ref{fig:non:separating}). Without
loss of generality we may assume that $g_1\le g_2$. This
stable graph encodes the orbit of a simple closed curve
separating the compact surface of genus $g$ without
punctures into subsurfaces of genera $g_1$ and $g_2$. This
stable graph represents also the boundary divisor
$\delta_{g_1,g}\subset\overline{\cM}_g-\cM_g$ for which a generic stable
curve has two irreducible components of genera $g_1$ and
$g_2$ respectively joined at a single node.

\begin{figure}[htb]
\includegraphics{genus_two_graph_12.eps}
\begin{picture}(0,0)(145,-5) 
\put(-2,-8){$g_1$}
\put(31,-8){$g_2$}
\end{picture}
\includegraphics{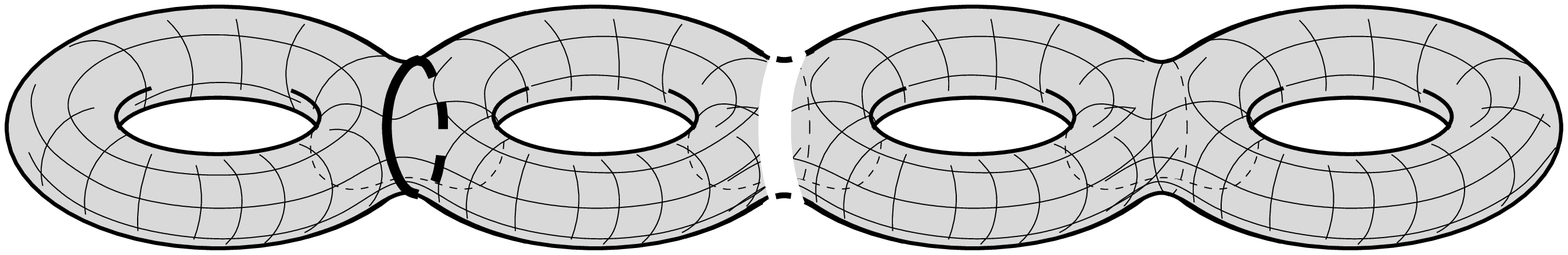}
\begin{picture}(0,0)(155,5) 
\put(75,-26){$\underbrace{\rule{80pt}{0pt}}_{g_1}$}
\put(175,-26){$\underbrace{\rule{80pt}{0pt}}_{g_2}$}
\put(280,-30){{\scriptsize $g_1$}}
\put(303,-30){{\scriptsize $g_2$}}
\end{picture}
\includegraphics{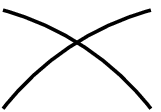}
\vspace*{40pt}
\caption{
\label{fig:two:verticies:one:edge}
Graph $\Separating(g_1,g_2)$
represents simple closed curves
on a surface of genus $g$ separating the surface
into subsurfaces of genera $g_1$ and $g_2$, where $g_1+g_2=g$
(in the middle). The same graph represents the
boundary divisor
$\delta_{g_1,g}\subset\overline{\cM}_g-\cM_g$ for which a generic stable
curve has two irreducible components of genera $g_1$ and
$g_2$ respectively joined at a single node (as on the right).
}
\label{fig:separating}
\end{figure}

\begin{Proposition}
\label{pr::contribution:Gamma:1}
The contribution $\Vol(\Separating(g_1,g-g_1))$ of the
graph $\Separating(g_1,g-g_1)$ having a single edge joining
two vertices decorated with labels $g_1, g-g_1$
respectively, to the Masur--Veech volume $\Vol\cQ_g$ of the
moduli space $\cQ_g$ of holomorphic quadratic differentials
has the following value:
\begin{equation}
\label{eq:asymptotic:contribution:segment}
\Vol(\Separating(g_1,g-g_1))
=\frac{4\cdot\zeta(6g-6)}{|\operatorname{Aut}(\Separating(g_1,g-g_1))|}
\cdot \binom{4g-4}{g}
\cdot\frac{1}{12^g}
\cdot \binom{g}{g_1}
\cdot \binom{3g-4}{3g_1-2}
\,.
\end{equation}
\end{Proposition}
\begin{proof}
Let $g_2=g-g_1$.
The contribution of the graph $\Separating(g_1,g_2)$ is given by
formula~\eqref{eq:square:tiled:volume} from
Theorem~\ref{th:volume}:
\begin{equation}
\label{eq:Vol:Graph:g1:g2}
\Vol(\Separating(g_1,g_2))
=\frac{2^{6g-5} \cdot (4g-4)!}{(6g-7)!}\cdot \\
\frac{1}{2} \cdot
\frac{1}{|\operatorname{Aut}(\Graph(g_1,g_2))|}
\cdot
\cZ\big(b\cdot
N_{g_1,1}(b)\cdot N_{g_2,1}(b)
\big)\,,
\end{equation}
where
\begin{equation}
\label{eq:Aut}
|\operatorname{Aut}(\Graph(g_1,g_2))|=
\begin{cases}
2,&\text{when }g_1=g_2\,,\\
1,&\text{otherwise}\,.
\end{cases}
\end{equation}

By the result of E.~Witten~\cite{Witten} one has
the following closed expression for $1$-correlators:
$$
\langle \psi_1^{3g-2} \rangle
=
\int_{\overline{\cM}_{g,1}} \psi_1^{3g-2}
=
\frac{1}{24^g\cdot g!}\,.
$$

Applying definitions~\eqref{eq:N:g:n}
and~\eqref{eq:c:subscript:d} to $N_{g,1}(b)$
and using the above expression
for $\langle \psi_1^{3g-2} \rangle$
we get the following close formula
for the polynomial $N_{g,1}(b)$:
\begin{multline}
N_{g,1}(b)=c_{3g-2} b^{2(3g-2)}=
\frac{1}{2^{5g-6+2}\cdot(3g-2)!}
\cdot\langle \psi_1^{3g-2} \rangle \cdot b^{6g-4}
=\\=
\frac{1}{2^{8g-4}\cdot 3^g\cdot g!\cdot (3g-2)!}
\cdot b^{6g-4}\,.
\end{multline}

Using definition~\eqref{eq:cZ} of $\cZ$
and the assumption $g_1+g_2=g$
we can
now develop the rightmost factor
in~\eqref{eq:cZ} as follows
\begin{multline*}
\cZ\big(b\cdot N_{g_1,1}(b)\cdot N_{g_2,1}(b)\big)
=\\=
\cZ\left(
\frac{1}{2^{8g_1-4}\!\cdot\! 2^{8g_2-4}}
\!\cdot\!\frac{1}{3^{g_1}\!\cdot\! 3^{g_2}}
\!\cdot\! \frac{1}{g_1!\, g_2!}
\!\cdot\! \frac{1}{(3g_1-2)! (3g_2-2)!}
\!\cdot\! b\cdot b^{6g_1-4}\cdot b^{6g_2-4}\right)
=\\=
\frac{1}{2^{8g-8}}
\cdot\frac{1}{3^g}
\cdot \frac{1}{g_1!\, (g-g_1)!}
\cdot \frac{1}{(3g_1-2)! (3g-3g_1-2)!}
\cdot(6g-7)!\cdot\zeta(6g-6)
\,.
\end{multline*}
Plugging this expression into~\eqref{eq:Vol:Graph:g1:g2};
multiplying and dividing by $g!$ and by $(3g-4)!$ to
pass to binomial coefficients
and simplifying we get the desired
formula~\eqref{eq:asymptotic:contribution:segment}.
\end{proof}

We complete this section with computation of cumulative
contribution of the graphs $\Separating(g_1,g-g_1)$ to
$\Vol\cQ_g$ coming from all
$g_1=1,\dots,\left[\frac{g}{2}\right]$.

\begin{Proposition}
\label{pr:sum:of:contributions:of:separating:graphs}
The following asymptotic relation is valid
\begin{equation}
\label{eq:sum:of:contributions:of:separating:graphs}
\sum_{g_1=1}^{\left[\frac{g}{2}\right]}
\Vol(\Separating(g_1,g-g_1))
\sim
\frac{2}{3\pi g}\cdot\frac{1}{4^g}
\cdot\left(\frac{8}{3}\right)^{4g-4}\,,
\end{equation}
where the equivalence means that the ratio of the two expressions
tends to $1$ as $g\to\infty$.
\end{Proposition}
\begin{proof}
Applying formula~\eqref{eq:Vol:Graph:g1:g2};
making the summation index range from $1$ to $g-1$
(instead of up to $\left[\frac{g}{2}\right]$
and taking into consideration that when $g=2g_1$,
the term $\Vol\Separating(g_1,g_1)$ has
$|\operatorname{Aut}(\Separating(g_1,g_1))|=2$
(see equation~\eqref{eq:Aut}) we get
the following expression for the numerator
of the latter expression:
$$
\sum_{g_1=1}^{\left[\frac{g}{2}\right]}
\Vol(\Separating(g_1,g-g_1))
=
2\cdot\zeta(6g-6)
\cdot \binom{4g-4}{g}
\cdot\frac{1}{12^g}
\cdot
\sum_{g_1=1}^{g-1}
\binom{g}{g_1}
\cdot \binom{3g-4}{3g_1-2}
\,.
$$

The zeta value $\zeta(6g-6)$ tends to $1$ exponentially fast
when $g\to+\infty$. Stirling Formula provides the follwing
asymptotic value of the binomial coefficient
$$
\binom{4g-4}{g}
\sim
\frac{1}{\sqrt{6\pi g}}
\cdot 3^g\cdot\left(\frac{4}{3}\right)^{4g-4}\,.
$$
Thus, to complete the proof of formula~\eqref{eq:log:sep:over:non:sep}
and, thus, of Theorem~\ref{th:separating:over:non:separating},
it remains to prove the Lemma below.
\end{proof}

\begin{Lemma}
\label{lm:sum:of:products:of:two:binomials}
The following asymptotic formula is valid:
\begin{equation}
\label{eq:sum:of:products:of:binomials:asymptotics}
S(g)=\sum_{g_1=1}^{g-1}
\binom{g}{g_1}
\binom{3g-4}{3g_1-2}
\sim
\frac{1}{\sqrt{6\pi g}}\cdot 2^{4g-4}
\end{equation}
as $g\to+\infty$.
\end{Lemma}
\begin{proof}
The probability density of the normal distribution
$\cN(\mu,\sigma^2)$ is given by the function
$$
f(x\,|\,\mu,\sigma^2)=
\frac{1}{\sqrt{2\pi\sigma^2}}
\cdot e^{-\tfrac{(x-\mu)^2}{2\sigma^2}}\,.
$$

Let $g$ be a large positive integer. After normalization by
$2^g$ the distribution of binomial coefficients
$\binom{g}{k}$, where $k=0,1,\dots,g$, tends to the normal
distribution $\cN\!\!\left(\tfrac{g}{2},\tfrac{g}{4}\right)$ as
$g\to+\infty$.

Let $m$ be any positive integer which we use as a fixed
parameter. The normalized distribution of binomial
coefficients $\binom{m\cdot g}{k}$, where $k=0,1,\dots,
m\cdot g$, tends to the normal distribution
$\cN\!\!\left(\tfrac{m\cdot g}{2},\tfrac{m\cdot
g}{4}\right)$ as $g\to\infty$. Hence, the normalized
distribution of binomial coefficients $\binom{m\cdot
g}{m\cdot k}$, where $k=0,1,\dots, g$, tends to the normal
distribution
$\cN\!\!\left(\tfrac{g}{2},\tfrac{g}{4m}\right)$ as
$g\to+\infty$. In particular, letting $m=3$ we see that the normalized distribution
of binomial coefficients $\binom{3g-4}{3k-2}$, where
$k=1,\dots,g-1$, tends to the normal distribution
$\cN\!\!\left(\tfrac{g}{2},\tfrac{g}{12}-\tfrac{1}{9}\right)$.

The normalized distribution of the product of two normal
distributions $\cN(\mu,\sigma_1^2)$ and
$\cN(\mu,\sigma_1^2)$ sharing the same mean $\mu$ is the
normal distribution
$\cN\!\!\left(\mu,\tfrac{\sigma_1^2\sigma_2^2}{\sigma_1^2+\sigma_2^2}\right)$.
Hence, after normalization by $S(g)$, the product
$\binom{g}{g_1}\binom{3g-4}{3g_1-2}$, of the two binomial
distributions, where $g_1=1,\dots,g-1$, tends to the normal
distribution $\cN(g/2,\sigma^2)$, where
$$
\sigma^2=
\frac
{\tfrac{g}{4}\left(\tfrac{g}{12}-\tfrac{1}{9}\right)}
{\tfrac{g}{4}+\left(\tfrac{g}{12}-\tfrac{1}{9}\right)}
\sim
\frac{g}{16}\,,\quad\text{as } g\to+\infty\,.
$$
Thus, the asymptotic value of the sum $S(g)$ can be
computed as the value of the product distribution
$\binom{g}{g_1}\binom{3g-4}{3g_1-2}$ at $\mu=g/2$
multiplied by $\sqrt{2\pi\sigma^2}$, that is
\begin{equation}
\label{eq:value:at:mean:times:sqrt:2:pi:sigma}
S(g)\sim
\binom{g}{\left[\frac{g}{2}\right]}
\cdot\binom{3g-4}{3\left[\frac{g}{2}\right]-2}
\cdot\sqrt{2\pi \tfrac{g}{16}}\,.
\end{equation}

From Stirling formula we get
$$
\binom{2m}{m}\sim \frac{2^{2m}}{\sqrt{\pi m}}
\quad\text{and}\quad
\binom{2m+1}{m}\sim \frac{2^{2m+1}}{\sqrt{\pi m}}
\,.
$$
Applying these asymptotic formula to each of the binomial
coefficients
in~\eqref{eq:value:at:mean:times:sqrt:2:pi:sigma} for even
and odd values of $g$ respectively we get
$$
S(g)
\sim\frac{2^g}{\sqrt{\pi \tfrac{g}{2}}}
\cdot\frac{2^{3g-4}}{\sqrt{\pi \tfrac{3g-4}{2}}}
\cdot\frac{\sqrt{\pi g}}{\sqrt{8}}
\sim
\frac{1}{\sqrt{6\pi g}}\cdot 2^{4g-4}\,.
$$

\end{proof}

\begin{NNRemark}
   %
The hypergeometric sum $S(g)$ on the left hand side
of~\eqref{eq:sum:of:products:of:binomials:asymptotics}
satisfies the following recursive relation obtained
applying Zeilberger's algorithm:
\begin{multline*}
S(g+2)=
2\cdot\frac{(324g^4+432 g^3+123 g^2-49 g-8)}{(6g-1)(3g+4)(3g-1)(g+1)}\cdot S(g+1)
+\\+
36\cdot\frac{(6g+5)(4g-1)(4g-3)}{(6g-1)(3g+4)(3g-1)}\cdot S(g)\,.
\end{multline*}
\end{NNRemark}

The asymptotic
formula~\eqref{eq:sum:of:products:of:binomials:asymptotics}
from Lemma~\ref{lm:sum:of:products:of:two:binomials} admits
the following immediate generalization.

\begin{Proposition}
\label{prop:sum:of:products:of:two:binomials}
Let $a_i\in\N$, $b_i,c_i\in\Z$, where $i=1,\dots,p$;
$r_j\in\N$, $s_j,t_j\in\Z$, where $j=1,\dots,q$,
be any collection of integer parameters satisfying
the condition
$$
(a_1+\dots+a_p)-(r_1+\dots+r_q)>0\,.
$$
Let $n\in\N$. The following asymptotic formula is valid:
\begin{multline}
\label{eq:sum:of:products:and:ratios:of:binomials:asymptotics}
\sum_{k=0}^{n}
\frac
{
\begin{pmatrix}a_1 n+b_1\\a_1 k+c_1\end{pmatrix}
\cdots
\begin{pmatrix}a_p n+b_p\\a_p k+c_p\end{pmatrix}
}{
\begin{pmatrix}r_1 n+s_1\\r_1 k+t_1\end{pmatrix}
\cdots
\begin{pmatrix}r_q n+s_q\\r_q k+t_q\end{pmatrix}
}
\sim\\
\sim
\rule{0pt}{20pt}
\left(\frac{\pi n}{2}\right)^{\tfrac{q-p+1}{2}}
\cdot 2^{(\sum_i a_i -\sum_j r_j)n+(\sum_i b_i - \sum_j s_j)}
\cdot\left(\frac{\prod_j r_j}{\prod_i a_i}\right)^{\tfrac{1}{2}}
\cdot\frac{1}{\left(\sum_i a_i-\sum_j r_j\right)^{\tfrac{1}{2}}}
\end{multline}
as $n\to+\infty$.
\end{Proposition}
By convention, the sum is evaluated only for those $k$ in
the range $0,\dots,n$ for which all binomial coefficients
in the numerator and in the numerator are well-defined,
the sum $\sum_{j=1}^q r_j$ is equal to zero and
the product $\prod_{j=1}^q r_j$ is equal to one
when $q=0$.

\subsection{Frequencies of simple closed geodesics}
\label{ss:simple:closed:geodesics}

In this section we use the setting and the notations as in
the Theorem~6.1 of M.~Mirzakhani~\cite{Mirzakhani:grouth:of:simple:geodesics}
reproduced at the end of Section~\ref{ss:Square:tiled:surfaces:and:associated:multicurves}.

Recall that equivalence classes of smooth simple closed
curves on an oriented surface of genus $g$ without boundary
or punctures, where curves are considered up to a
diffeomorphism are classified as follows. The curve can be
separating or non-separating. All non-separating curves
as in Figure~\ref{fig:non:separating}
belong to the same class; we denote the corresponding
frequency by $c(\gamma_{nonsep,g})$.

Separating simple closed curves are classified by the
genera $g_1,g_2$ of components in which the curve separates
the surface; see Figure~\ref{fig:separating}. Here
$g_1+g_2=g$; $g_1,g_2\ge 1$, and pairs $g_1,g_2$ and
$g_2,g_1$ correspond to the same equivalence class.
Denote a simple closed curve of this type
by $\gamma_{g_1,g_2}$.

Recall that reduced simple closed multicurves are in the
natural bijective correspondence with stable graphs. The
stable graph corresponding to $\gamma_{g_1,g_2}$ is,
clearly, $\Separating(g_1,g_2)$.

Recall that the volume contribution $\Vol(\Separating(g_1,g_2))$
comes from all one-cylinder square-tiled surfaces of genus
$g=g_1+g_2$ such that the waist curve of the single cylinder
separates the surface into two surfaces of genera $g_1$ and
$g_2$ respectively. This single horizontal cylinder can be composed of
$a=1,2,\dots$ horizontal bands of squares.
The contribution $\Vol(\Separating(g_1,g_2))$
is the sum of contributions $\Vol(a\cdot\gamma_{g_1,g_2})$
of square-tiled surfaces having fixed value $a\in\N$,
$$
\Vol(\Separating(g_1,g_2))
=\sum_{a=1}^{+\infty} \Vol(a\cdot\gamma_{g_1,g_2})\,.
$$
Recall also, that
$$
\Vol(a\cdot\gamma_{g_1,g_2})
=\frac{1}{a^d} \Vol(\gamma_{g_1,g_2})\,,
$$
where $d=6g-6=\dim\cQ_g$, which implies that
$$
\label{eq:Vol:Gamma:g1:g2:1}
\Vol(\Separating(g_1,g_2))
=\zeta(6g-6)\cdot\Vol(\gamma_{g_1,g_2})\,.
$$
Combining the latter relation with
formula~\eqref{eq:Vol:gamma:c:gamma} from
Theorem~\ref{th:our:density:equals:Mirzakhani:density} we get
$$
\frac{\Vol(\Separating(g_1,g_2))}{\zeta(6g-6)}
=\Vol(\gamma_{g_1,g_2})
=2\cdot(6g-6)\cdot
(4g-4)!\cdot 2^{4g-3}\cdot
c(\gamma_{g_1,g_2})\,.
$$
Applying
expression~\eqref{eq:asymptotic:contribution:segment} for
$\Vol\Graph(g_1,g_2)$ we get the following formula:
\begin{multline}
\label{eq:Mirzakhani:separating}
c(\gamma_{g_1,g-g_1})=
\frac{1}{|\Aut\Separating(g_1,g-g_1)|}
\cdot\\
\cdot\frac{1}{
2^{3g-4}
\cdot 24^g\cdot
g_1!\cdot (g-g_1)!\cdot (3g_1-2)!\cdot(3(g-g_1)-2)!\cdot (6g-6)
}\,.
\end{multline}

In this way we reprove the formula for the frequency of
simple closed separating geodesics first proved by
M.~Mirzakhani (see page~124
in~\cite{Mirzakhani:grouth:of:simple:geodesics}).

\begin{Remark}
The formula on page~124
in~\cite{Mirzakhani:grouth:of:simple:geodesics} contains
two misprints: the power in the first factor in the
denominator is indicated as $2^{3g-2}$ while it should be
read as $2^{3g-4}$ and the fifth factor is indicated as
$(3g-2)!$ while it should be read as $(3i-2)!$.
Indeed, following Mirzakhani's calculation we have to use
formula~(5.5)
from~\cite{Mirzakhani:grouth:of:simple:geodesics} for
$c(\gamma)$. In notations of this formula applied to our
particular $\gamma$ we have $n=0$, $k=1$, $a_1=1$,
$s_1=3g-4$. Mirzakhani assumes for simplicity that $g>2i>2$,
which implies that $M(\gamma)=0$ and
$|\operatorname{Sym}(\gamma)|=1$
(and implies that $|\Aut\Separating(g_1,g-g_1)|=1$ in notations
of Formula~\eqref{eq:Mirzakhani:separating} above).
Thus, applying
formula~(5.5)
from~\cite{Mirzakhani:grouth:of:simple:geodesics} she gets
$$
c(\gamma)=b_\Gamma(2\cdot(3g-4))
=(2\cdot(3g-4))_\Gamma\cdot\frac{(2\cdot(3g-4)+1)!}{(6g-6)!}\,.
$$
To compute $(2\cdot(3g-4))_\Gamma$ we compute following
Mirzakhani the product of the coefficients of the leading
terms of the polynomials $V_{i,1}(x)$ and $V_{g-i,1}(x)$
(see the bottom of page~122
in~\cite{Mirzakhani:grouth:of:simple:geodesics}). In this
way we get
\begin{multline*}
\frac{1}{(3i-2)!\cdot i!\cdot 24^i\cdot 2^{3i-2}}
\cdot
\frac{1}{(3(g-i)-2)!\cdot (g-i)!\cdot 2^{3(g-i)-2}}
=\\=
\frac{1}{2^{3g-4}\cdot 24^g\cdot i!\cdot (g-i)!\cdot (3i-2)!\cdot (3(g-i)-2)!}\,.
\end{multline*}
(compare to~\eqref{eq:Mirzakhani:separating} replacing
$g_1$ with $i$). The remaining last factor $(6g-6)$ in the
denominator of the formula of Mirzakhani comes from
$\frac{(2\cdot(3g-4)+1)!}{(6g-6)!}=\frac{1}{6g-6}$.
\end{Remark}

Denote
by $c(\gamma_{sep,g})$ the sum of the frequencies
$c(\gamma_{g_1,g_2})$ over all equivalence classes
of separating curves, i.e.
over all unordered pairs $(g_1,g_2)$ satisfying
$g_1+g_2=g$; $g_1,g_2\ge 1$.
We are now ready to prove
Theorem~\ref{th:separating:over:non:separating}.

\begin{proof}[Proof of Theorem~\ref{th:separating:over:non:separating}]
By Theorem~\ref{th:our:density:equals:Mirzakhani:density}
we have
$$
\frac{c(\gamma_{sep,g})}{c(\gamma_{nonsep,g})}=
\frac{\sum_{g_1=1}^{\left[\frac{g}{2}\right]}
\Vol(\Separating(g_1,g-g_1))}
{\Vol\Graph_1(g)}\,.
$$

Plugging the asymptotic
values~\eqref{eq:asymptotic:contribution:Gamma:1}
and~\eqref{eq:sum:of:contributions:of:separating:graphs}
respectively in the denominator and numerator of the ratio
on the right hand side we obtain the desired asymptotic
value for the ratio on the left hand side.
\end{proof}

In the table below we present the exact (first line) and
approximate (second line) values of the ratio
$\tfrac{c(\gamma_{sep,g})}{c(\gamma_{nonsep,g})}$ of the
two frequencies in small genera and the value given by the
asymptotic formula~\eqref{eq:log:sep:over:non:sep} (third
line).

$$
\begin{array}{c|c|c|c|c||c}
g&2&3&4&5&11\\
\hline &&&&&\\
[-\halfbls]
\text{\small Exact}
&\frac{1}{48}
&\frac{5}{1776}
&\frac{605}{790992}
&\frac{4697}{27201408}
&\frac{166833285883}{5360555755385245488}
\\ &&&&& \\
[-\halfbls]
\hline &&&&&\\
[-\halfbls]
\text{\small Approximate}
& \scriptstyle 2.08\cdot 10^{-2}
& \scriptstyle 2.82\cdot 10^{-3}
& \scriptstyle 7.65\cdot 10^{-4}
& \scriptstyle 1.73\cdot 10^{-4}
& \scriptstyle 3.11\cdot 10^{-8}
\\ &&&&& \\
[-\halfbls]
\hline &&&&&\\
[-\halfbls]
\text{\small Asymp. formula}
& \scriptstyle 2.03\cdot 10^{-2}
& \scriptstyle 4.16\cdot 10^{-3}
& \scriptstyle 9.00\cdot 10^{-4}
& \scriptstyle 2.01\cdot 10^{-4}
& \scriptstyle 3.31\cdot 10^{-8}
\end{array}
$$


\appendix

\section{Stable graphs: formal definition}
\label{s:stable:graphs}

We now introduce the definition of a stable graph.
Let $S$ be an hyperbolic surface of genus $g$ and
$n$ (numbered) marked points. Let $\gamma$ be a multicurve on $S$ (that is
a finite union of disjoint geodesics) the surface $S \setminus \gamma$
is a union of surfaces with boundaries (and marked points). The dual
graph $\Graph$ to this decomposition is constructed as follows:
\begin{itemize}
\item each subsurface gives rise to a vertex $v$ of $\Graph$ carrying
some genus $g_v$, some marked points (the \textit{legs}) and a
certain number of boundary components,
\item each component $\gamma'$ of $\gamma$ gives rise to an edge between the
two components on the two sides of $\gamma'$.
\end{itemize}
See the figures in the tables of Appendix~\ref{a:2:0} for the
correspondence between multicurves and stable graphs in genus 2.

We now introduce the more formal following definition.
\begin{Definition}
\label{def:stable:graph}
A \textit{stable graph} for $\cM_{g,n}$ is a 6-tuple
$\Graph = (V, H, \alpha, \iota, \mathbf{g}, L)$ where
\begin{itemize}
\item $V$ is a finite set of \textit{vertices}
\item $H$ is a finite set of \textit{half-edges}
\item $\iota: H \to H$ is an involution. The fixed
points of $\iota$ are called the \textit{legs}
and the 2-cycles the \textit{edges} of $\Graph$.
\item $\alpha: H \to V$ is a map that attaches
an half-edge to a vertex. The number of half-edges
at a given vertex $v$ is denoted by $n_v := |\alpha^{-1}(v)|$.
\item The graph is connected: for each pair of vertices
$(u,v)$ there exists a sequence of
half edges $(h_1, h'_1, h_2, h'_2, \ldots, h_k, h'_k)$
so that $\iota(h_i) = h'_i$, $u = \alpha(h_1)$,
$v = \alpha(h'_k)$ and $\alpha(h'_i) = \alpha(h_{i+1})$..
\item $\mathbf{g} = \{g_v\}_{v \in V}$ is a set
of non-negative integers, one at each vertex,
called the \textit{genus decoration}.
\item $L$ is a bijection from the set of legs to
$\{1,\ldots,n\}$.
\item The \textit{genus} $g(\Graph)$ of $\Graph$ should be equal to $g$ where
\[
g(\Graph) = h^1(\Graph) + \sum_{v \in V} g_v
\]
where $h^1(\Graph)$ is the first betti number of the graph.
\item the following \textit{stability condition} holds at each vertex $v$ of $\Graph$
\[
2 g_v - 2 + n_v > 0.
\]
\end{itemize}
\end{Definition}
Given a stable graph $\Graph$, we associate an underlying graph whose vertex
set is $V$ and each 2 cycle $(h,h')$ of $\iota$ gives an edge
attached to $\alpha(h)$ and $\alpha(h')$. We denote $E = E(\Graph)$ these
edges. Such graph might have multiple edges and loops. The additional
information of the stable graph is the genus decoration $\mathbf{g}$
and the $n$ legs.

Two stable graphs $\Graph = (V,H,\mathbf{v},\iota,L,\mathbf{g})$ and
$\Graph'=(V',H',\mathbf{v'},\iota,L',\mathbf{g'})$ are isomorphic if there exists
two bijections $\phi: V \to V'$ and $\psi: H \to H'$ that preserve
edges, legs and markings that is
\[
\iota \circ \psi = \iota' \circ \psi
\quad
L'(\psi(h)) = L(h)
\quad
g'_{\phi(v)} = g_v.
\]
The definition of stable graphs and automorphisms mean that
neither vertices nor edges of the decorated graph
$\Graph$ are labeled, but each leg is numbered by $L$.

We denote by $\cG_{g,n}$ the set of isomorphism classes of stable graphs
with given genus $g$ and number of legs $n$.

As we already mentioned, each stable graph
in $\cG_{g,n}$ corresponds to a bounary cycle of the
Deligne-Mumford compactification $\overline{\cM_{g,n}}$.
More precisely, each vertex $v$ of the graph corresponds to
the component of a nodal curve (of genus $g_v$ and contain
the marked points corresponding to the legs attached at
this vertex) and each edge of $E(\Graph)$ represents a node
together with its number in $\{1,2,\ldots,n\}$.
(See survey~\cite{Vakil} for excellent introduction in this
subject and for beautiful illustrations.)
Hence, the
unique stable graph with no edge and $n$ legs in
$\cG_{g,n}$ corresponds to the component $\cM_{g,n}$ (the
smooth curves).

\section{Examples of explicit calculations}
\label{s:explicit:calculations}

\subsection{Holomorphic quadratic differentials in genus two.}
\label{a:2:0}
We start by evaluation of the Siegel--Veech constant
$\carea(Q_2)$. Note that certain graphs do not
contribute at all to $\carea$.

$$
\hspace*{20pt}
\begin{array}{ll|l}

\includegraphics{genus_two_21.eps}
\begin{picture}(0,0)(0,0)
\put(-24,0){$b_1$}
\put(66,0){$b_2$}
\end{picture}
\hspace*{67pt}

&
\frac{128}{5}\cdot
1 \cdot
\frac{1}{8}\cdot
b_1 b_2 \cdot
N_{0,4}(b_1,b_1,b_2,b_2)=

&

(1+1)\cdot

\\


\includegraphics{genus_two_graph_21.eps}
\begin{picture}(0,0)(0,-10)
\put(-9,-19){$b_1$}
\put(52,-19){$b_2$}
\put(23,-30){$0$}
\end{picture}
  %

&
\rule{0pt}{12pt}
=\frac{16}{5}\cdot b_1 b_2\cdot\left(\frac{1}{4}(2b_1^2+2b_2^2)\right)

&

\cdot\!\frac{8}{5}\!\cdot\! 3!\zeta(4)\!\cdot\!1!\zeta(2)

\\

&

\hspace*{45pt}
\xmapsto{\partial_{\Graph}}\frac{8}{5}(1\cdot b_1 b_2^3+1\cdot b_1^3 b_2)\xmapsto{\cZ}

&

=1\cdot\frac{8}{225}\cdot \pi^6

\\
\vspace*{-18pt}\\
&&\\
\hline
&&\\


\includegraphics{genus_two_22.eps}
\begin{picture}(0,0)(0,0)
\put(-24,0){$b_1$}
\put(23,-14){$b_2$}
\end{picture}
\rule{0pt}{12pt}

&

\frac{128}{5}\!\cdot\!
\frac{1}{2}\!\cdot\!
\frac{1}{2}\!\cdot\!
b_1 b_2\!\cdot\! N_{0,3}(b_1,b_1,b_2)\!\cdot\! N_{1,1}(b_2)

&
1\!\cdot\!
\frac{2}{15}\!\cdot\! 1!\zeta(2)\!\cdot\!3!\zeta(4)

\\


\includegraphics{genus_two_graph_22.eps}
\begin{picture}(0,0)(0,0)
\put(-19,-10){$b_1$}
\put(23,-16){$b_2$}
\put(6,-10){$0$}
\put(41,-10){$1$}
\end{picture}

&
\rule{0pt}{15pt}
=\frac{32}{5}\!\cdot\! b_1 b_2 \!\cdot\! 1
\!\cdot\! \left(\frac{1}{48} b_2^2\right)
\xmapsto{\partial_{\Graph}}
\!\frac{2}{15}\!\cdot\! 1\!\cdot\!  b_1 b_2^3\xmapsto{\cZ}
&
=1\cdot\frac{1}{675}\cdot\pi^6

\\&&\\
\vspace*{-18pt}\\
&&\\
\hline
&&\\


\includegraphics{genus_two_31.eps}
\begin{picture}(0,0)(0,0)
\put(-24,0){$b_1$}
\put(23,-14){$b_2$}
\put(66,0){$b_3$}
\end{picture}

&
\frac{128}{5}\!\cdot\!
\frac{1}{2}\!\cdot\!
\frac{1}{8}\!\cdot\!
b_1 b_2 b_3 \!\cdot\!
N_{0,3}(b_1,b_1,b_2)\cdot

&

(1+\frac{1}{2}+1)\cdot

\\


\includegraphics{genus_two_graph_31.eps}
\begin{picture}(0,0)(0,5)
\put(-18,-10){$b_1$}
\put(23,-16){$b_2$}
\put(63,-10){$b_3$}
\put(5.5,-10){$0$}
\put(42,-10){$0$}
\end{picture}

\rule{0pt}{12pt}

&

\cdot N_{0,3}(b_2,b_3,b_3)
=\frac{8}{5}\!\cdot\!
{b_1} {b_2} {b_3}
\!\cdot\! (1) \!\cdot\! (1)

&
\cdot\frac{8}{5}\cdot \left(1!\,\zeta(2)\right)^3
\\


&
\rule{0pt}{15pt}
\xmapsto{\partial_{\Graph}}
\!\frac{8}{5}\!\left(1\!\cdot\! b_1 b_2 b_3
\!+\!\frac{1}{2} b_1 b_2 b_3
\!+\! 1\!\cdot\! b_1 b_2 b_3\right)\!\!\xmapsto{\cZ}\!
\hspace*{-3pt}
&
=\frac{5}{2}\cdot\frac{1}{135}\cdot\pi^6

\\
\vspace*{-18pt}\\
&&\\
\hline
&&\\


\includegraphics{genus_two_32.eps}
\begin{picture}(0,0)(0,0)
\put(-24,0){$b_1$}
\put(20,12){$b_2$}
\put(66,0){$b_3$}
\end{picture}

&

\frac{128}{5}\!\cdot\!
\frac{1}{2}\!\cdot\!
\frac{1}{12}\!\cdot\!
b_1 b_2 b_3 \!\cdot\!
N_{0,3}(b_1,b_1,b_2)\cdot

&

(1+1+1)\cdot

\\

\includegraphics{genus_two_graph_32.eps}
\begin{picture}(0,0)(0,0)
\put(4,-16){$b_1$}
\put(27.5,-16){$b_2$}
\put(43,-16){$b_3$}
\put(31,2){$0$}
\put(31,-35){$0$}
\end{picture}

&

\rule{0pt}{15pt}

\cdot N_{0,3}(b_2,b_3,b_3)
=\frac{16}{15}\!\cdot\! {b_1 b_2 b_3}\!\cdot\! (1) \!\cdot\! (1)

&
\frac{16}{15}\!\cdot\! \left(1!\,\zeta(2)\right)^3
\\


&
\rule{0pt}{15pt}

\xmapsto{\partial_{\Graph}}
\!\frac{8}{5}\!\left(1\!\cdot\! b_1 b_2 b_3
\!+\!\frac{1}{2} b_1 b_2 b_3
\!+\! 1\!\cdot\! b_1 b_2 b_3\right)\!\!\xmapsto{\cZ}\!
\hspace*{-3pt}
&
=3\cdot\frac{2}{405}\cdot\pi^6
\end{array}
$$
\smallskip

Taking the sum of the four contributions we obtain:
$$
\left(
\left(1\cdot\frac{8}{225}+1\cdot\frac{1}{675}\right)
+
\left(\frac{5}{2}\cdot\frac{1}{135}+3\cdot\frac{2}{405}\right)
\right)\cdot \pi^6
=
\left(\frac{1}{27}+\frac{1}{30}\right)\cdot \pi^6
=
\frac{19}{270}\cdot \pi^6\,.
$$
Dividing by $\Vol\cQ(1^4)=\cfrac{1}{15}\cdot\pi^6$ we get the answer which
matches the value found in~\cite{Goujard:carea}:
$
\cfrac{\pi^2}{3}\cdot c_{area}(\cQ(1^4))=\left(\cfrac{19}{270}\pi^6\right):
\left(\cfrac{1}{15}\pi^6\right)=\cfrac{19}{18}\,.
$

The computation of the
Masur--Veech volume $\Vol\cQ_2$ was presented in
Table~\ref{tab:2:0}
in
Section~\ref{ss:intro:Masur:Veech:volumes}.
The first two graphs in Table~\ref{tab:2:0} represent the
contribution to the volume $\Vol\cQ(1^4)$ of square-tiled
surfaces having single maximal cylinder. The resulting
contribution
$
\frac{7}{405} \pi^6 = \frac{49}{3} \zeta(6)
$
was found in Appendix~C
in~\cite{DGZZ:initial:arXiv:one:cylinder:with:numerics} by
completely different technique.

The third and the fourth graph together represent the contribution
$
\frac{1}{27}\pi^6
$
of square-tiled surfaces having two maximal cylinders. The last two
graphs --- the contribution
$
\frac{1}{81}\pi^6
$
of square-tiled surfaces having three maximal cylinders.

Normalizing the contribution of $1,2,3$-cylinder square-tiled
surfaces by the entire volume $\Vol\cQ(1^4)$ of the stratum we get
the quantity $p_k(\cQ(1^4))$ which can be interpreted as the
``probability'' for a ``random'' square-tiled surface in the stratum
$\cQ(1^4)$ to have exactly $k$ horizontal cylinders. These same quantities
$p_k$ provide ``probabilities'' of getting a $k$-band generalized
interval exchange transformation (linear involution) taking a
``random'' generalized interval exchange transformation in the Rauzy
class representing the stratum $\cQ(1^4)$ (see section~3.2
in~\cite{DGZZ:initial:arXiv:one:cylinder:with:numerics}
for details). The latter quantities are particularly
simple to evaluate in numerical experiments. The resulting
proportions
$$
\big(p_1(\cQ(1^4)),p_2(\cQ(1^4)),p_3(\cQ(1^4))\big)=
\left(\frac{7}{405},\frac{1}{27},\frac{1}{81}\right):\frac{1}{15}=
\left(\frac{7}{27},\frac{15}{27},\frac{5}{27}\right)
$$
match the numerical experiments obtained earlier
in Appendix~C
of~\cite{DGZZ:initial:arXiv:one:cylinder:with:numerics}.

\subsection{Holomorphic quadratic differentials in genus three.}
\label{a:3:0}

In genus three there are already $41$  different decorated ribbon
graphs. We do not provide the graph-by-graph calculation as we
did in genus two but only the contributions of $k$-cylinder
square-tiled surfaces for all possible values $k=1,\dots,6$ of
cylinders.

$$
\begin{array}{|c|c|c|c|}
\hline &&&\\
[-\halfbls]
\text{Number of}       &\text{Number of}          &\text{Contribution}               &\text{Relative}          \\
\text{cylinders } k    &\text{graphs } \Graph    &\text{to the volume}              &\text{contribution } p_k \\
&&& \\
[-\halfbls]
\hline
&&&\\
[-\halfbls]
1                      &           2              &\frac{94667}{126299250}\cdot\pi^{12}  &\frac{757336}{3493125}   \\
&&& \\
[-\halfbls]
\hline
&&&\\
[-\halfbls]
2                      &           5              &\frac{150749}{108256500}\cdot\pi^{12} &\frac{4220972}{10479375} \\
&&& \\
[-\halfbls]
\hline
&&&\\
[-\halfbls]
3                      &           9              &\frac{84481}{86605200}\cdot\pi^{12}   &\frac{591367}{2095875}   \\
&&& \\
[-\halfbls]
\hline
&&&\\
[-\halfbls]
4                      &           12             &\frac{5989}{21651300}\cdot\pi^{12}    &\frac{167692}{2095875}   \\
&&& \\
[-\halfbls]
\hline
&&&\\
[-\halfbls]
5                      &           8              &\frac{1}{17820}\cdot\pi^{12}          &\frac{28}{1725}          \\
&&& \\
[-\halfbls]
\hline
&&&\\
[-\halfbls]
6                      &           5              &\frac{1}{144342}\cdot\pi^{12}         &\frac{56}{27945}         \\
&&& \\
\hline
\end{array}
$$

The resulting contribution of $1$-cylinder surfaces was confirmed by
the alternative combinatorial study of the Rauzy class of the stratum
$\cQ(1^8)$ (see section~3.2
in~\cite{DGZZ:initial:arXiv:one:cylinder:with:numerics}.
The approximate values
$p_k(\cQ(1^8))$ were confirmed by numerical experiments with
statistics of $k$-band generalized interval exchange transformations.

Taking the sum of all contributions we get the volume of the moduli
space $\cQ_3$ of holomorphic quadratic differentials in genus $3$:
$
\Vol\cQ_3=\cfrac{115}{33264}\cdot \pi^{12}\,.
$

\subsection{Meromorphic quadratic differentials in genus one.}
\label{a:1:2}

In this section we apply formula~\eqref{eq:square:tiled:volume} to compute
the Masur--Veech volume $\Vol\cQ(1^2,-1^2)$ of the moduli space
$\cQ_{1,2}$ of meromorphic quadratic differentials in genus $g=1$
with two simple poles $p=2$.

We use the same convention on the order of numerical factors
in every first line of the middle column as in section~\ref{a:2:0}.
Namely, for $(g,p)=(1,2)$ we have
$
\zeroes!\cdot 2\dprinc
\cdot
\frac{2^{\dprinc}}{\dprinc!}
=\frac{32}{5}\,,
$
which is the first factor. The second factor is $1/2^{|V(\Graph)|-1}$.

The third factor is $|\Aut(\Graph)|^{-1}$. We remind
that the vertices and edges of $\Graph$ are not labeled while
the two legs are labeled. An automorphism of $\Graph$ preserves
the decoration of vertices and the labeling of the legs. For example,
the graph $\Graph$ in the second line does not have any nontrivial
automorphisms.

\begin{table}[hbt]
$$
\hspace*{20pt}
\begin{array}{llr}


\includegraphics{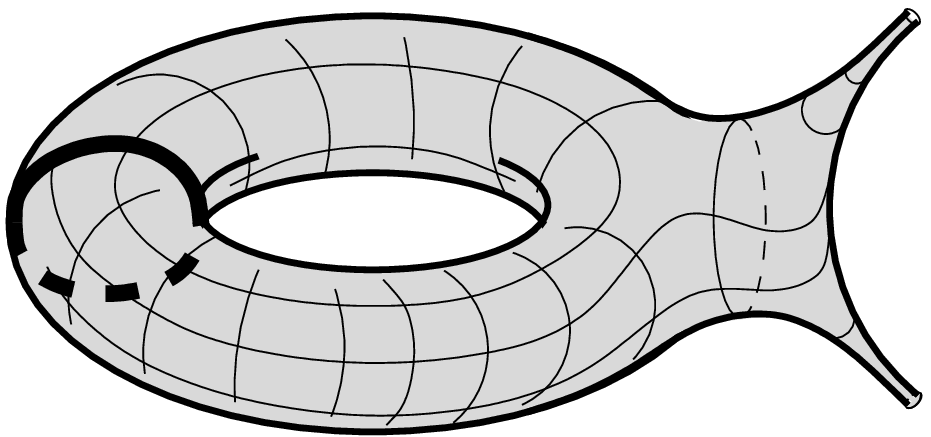}
\begin{picture}(0,0)(0,0)
\put(-9,0){$b_1$}
\end{picture}
\hspace*{45pt} 
\vspace*{10pt}

&

\frac{32}{3}\cdot
1\cdot
\frac{1}{2}\cdot
b_1 \cdot
N_{0,4}(b_1,b_1,0,0) =

&

\\ 
\includegraphics{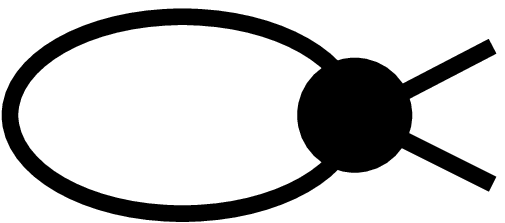}
\begin{picture}(0,0)(0,0)
\put(12,0){$b_1$}
\put(40,-9){$0$}
\end{picture}
\vspace*{10pt}

&

\hspace*{17pt}
=
\frac{16}{3}\cdot
b_1 \cdot \left(\frac{1}{4}(2b_1^2)\right)
\hspace*{5pt} =\hspace*{5pt}
\frac{8}{3}\cdot b_1^3\
\hspace*{1pt}
\xmapsto{\ \cZ\ }

&

\hspace*{-20pt}
\frac{8}{3}\cdot 3!\cdot\zeta(4)
=
\frac{8}{45}\pi^4

\\\hline&&\\


\includegraphics{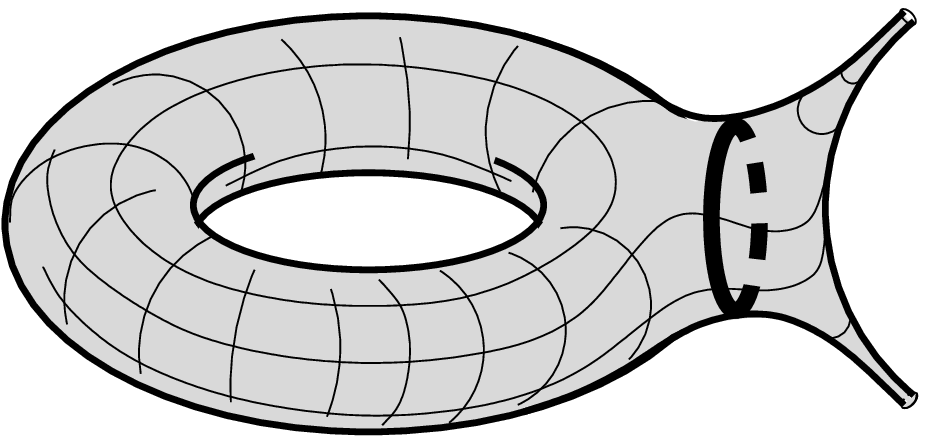}
\begin{picture}(0,0)(0,0)
\put(38,12){$b_1$}
\end{picture}
\vspace*{10pt}

&

\frac{32}{3}\cdot
\frac{1}{2}\cdot
1\cdot
b_1 \cdot
N_{1,1}(b_1)\cdot N_{0,3}(0,0,b_1) =

&

\\ 

\includegraphics{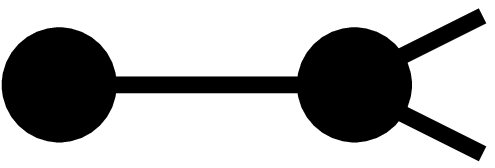}
\begin{picture}(0,0)(0,0)
\put(29,5){$b_1$}
\put(22,-9){$1$}
\put(39.5,-9){$0$}
\end{picture}
\vspace*{10pt}

&

\hspace*{17pt}
= \frac{16}{3}\cdot b_1 \cdot \left(\frac{1}{48} b_1^2\right) \cdot (1)
=\frac{1}{9}\cdot b_1^3\ \xmapsto{\ \cZ\ }

&

\hspace*{-20pt}
\frac{1}{9}\cdot 3!\cdot \zeta(4) = \frac{1}{135}\cdot\pi^4

\\\hline&&\\


\includegraphics{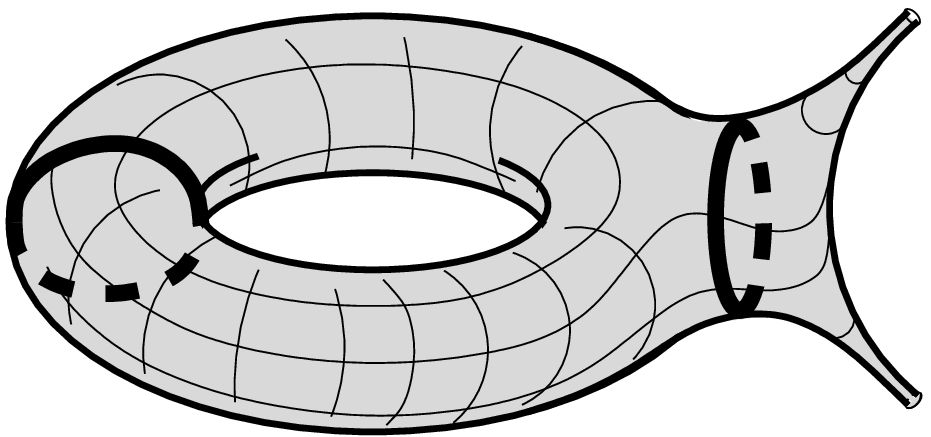}
\begin{picture}(0,0)(0,0)
\put(-9,0){$b_1$}
\put(38,12){$b_2$}
\end{picture}
\vspace*{10pt}

&

\frac{32}{3}\cdot
\frac{1}{2}\cdot
\frac{1}{2}\cdot
b_1 b_2\cdot
N_{0,3}(b_1,b_1,b_2)\cdot N_{0,3}(b_1,0,0)

&

\\ 

\includegraphics{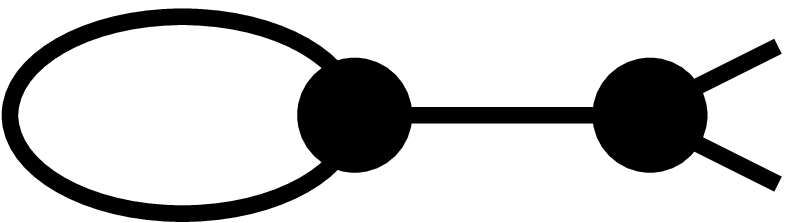}
\begin{picture}(0,0)(0,0)
\put(-5,0){$b_1$}
\put(29,5){$b_2$}
\put(22,-9){$0$}
\put(39.5,-9){$0$}
\end{picture}
\vspace*{10pt}

&

\hspace*{17pt}
= \frac{8}{3}\cdot
b_1 b_2\cdot (1)\cdot (1)=
\frac{8}{3}\cdot b_1 b_2
\
\hspace*{1pt}
\xmapsto{\ \cZ\ }

&

\hspace*{-20pt}
\frac{8}{3}\cdot \big(\zeta(2)\big)^2
=
\frac{2}{27}\pi^4

\\\hline&&\\


\includegraphics{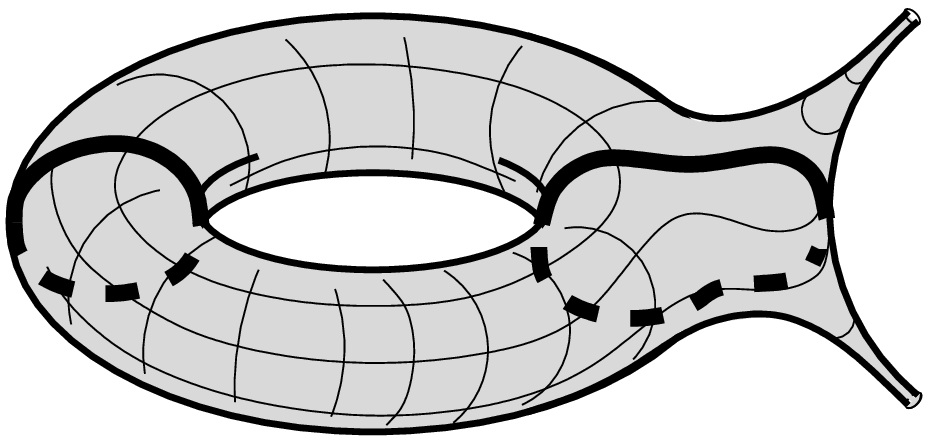}
\begin{picture}(0,0)(0,0)
\put(-9,0){$b_1$}
\put(38,12){$b_2$}
\end{picture}
\vspace*{15pt}

&

\frac{32}{3}\cdot
\frac{1}{2}\cdot
\frac{1}{2}\cdot
b_1 b_2\cdot
N_{0,3}(0,b_1,b_2)\cdot N_{0,3}(b_1,b_2,0)

&

\\ 

\includegraphics{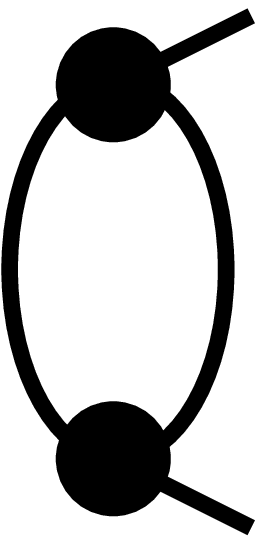}
\begin{picture}(0,0)(0,0)
\put(15,0){$b_1$}
\put(41,0){$b_2$}
\put(29.5,1.5){$0$}
\put(29.5,-20){$0$}
\end{picture}
\vspace*{20pt}

&

\hspace*{17pt}
= \frac{8}{3}\cdot
b_1 b_2\cdot (1)\cdot (1)=
\frac{8}{3}\cdot b_1 b_2

\
\hspace*{1pt}
\xmapsto{\ \cZ\ }

&

\hspace*{-20pt}
\frac{8}{3}\cdot \big(\zeta(2)\big)^2
=
\frac{2}{27}\pi^4
\end{array}
$$
\caption{
\label{tab:1:2}
Computation of $\Vol\cQ_{1,2}$.
Left column represents
stable graphs $\Gamma\in\cG_{1,2}$
and associated multicurves;
middle column gives
associated polynomials
$P_\Gamma$; right column
provides $\Vol(\Gamma)$.
}
\end{table}

The resulting value
\begin{equation*}
\Vol\cQ_{1,2}=
\left(\left(\frac{8}{45}+\frac{1}{135}\right)+\left(\frac{2}{27}+\frac{2}{27}\right)\right)\cdot\pi^4
=
\left(\frac{5}{27}+\frac{4}{27}\right)\cdot\pi^4
=\cfrac{\pi^4}{3}
\end{equation*}
matches the one found in~\cite{Goujard:volumes}. The
contribution of $1$-cylinder square-tiled surfaces and the
proportion $5:4$ between $1$-cylinder and $2$-cylinder
contributions match the corresponding quantities found in
Appendix~C
in~\cite{DGZZ:initial:arXiv:one:cylinder:with:numerics}.


Now we evaluate the Siegel--Veech constant $\carea(Q(1^4))$.

$$
\hspace*{5pt}
\begin{array}{llr}


\includegraphics{genus_one_21.eps}
\begin{picture}(0,0)(0,0)
\put(-9,0){$b_1$}
\put(38,12){$b_2$}
\end{picture}
\vspace*{10pt}
\hspace*{43pt}
&

\frac{32}{3}\cdot
\frac{1}{2}\cdot
\frac{1}{2}\cdot
b_1 b_2\cdot
N_{0,3}(b_1,b_1,b_2)\cdot N_{0,3}(b_1,0,0)

&

\\ 

\includegraphics{genus_one_graph_21.eps}
\begin{picture}(0,0)(0,0)
\put(-5,0){$b_1$}
\put(29,5){$b_2$}
\put(22,-9){$0$}
\put(39.5,-9){$0$}
\end{picture}
\vspace*{10pt}

&

= \frac{8}{3}\cdot
b_1 b_2 \cdot 1\cdot 1=
\frac{8}{3}\cdot  b_1 b_2
\
\xmapsto{\cZ\circ\partial_{\Graph}}

&

\hspace*{-45pt}
\left(1 + \frac{1}{2}\right)\cdot
\frac{8}{3}\cdot \big(\zeta(2)\big)^2
=
\frac{3}{2}\cdot
\frac{2}{27}\pi^4

\\\hline&&\\


\includegraphics{genus_one_22.eps}
\begin{picture}(0,0)(0,0)
\put(-9,0){$b_1$}
\put(38,12){$b_2$}
\end{picture}
\vspace*{15pt}

&

\frac{32}{3}\cdot
\frac{1}{2}\cdot
\frac{1}{2}\cdot
b_1 b_2\cdot
N_{0,3}(0,b_1,b_2)\cdot N_{0,3}(b_1,b_2,0)

&

\\ 

\includegraphics{genus_one_graph_22.eps}
\begin{picture}(0,0)(0,0)
\put(15,0){$b_1$}
\put(41,0){$b_2$}
\put(29.5,1.5){$0$}
\put(29.5,-20){$0$}
\end{picture}
\vspace*{20pt}

&

= \frac{8}{3}\cdot
b_1 b_2\cdot 1\cdot 1=
\frac{8}{3}\cdot b_1 b_2

\
\hspace*{1pt}
\xmapsto{\cZ\circ\partial_{\Graph}}

&

\hspace*{-45pt}
(1+1)\cdot
\frac{8}{3}\cdot \big(\zeta(2)\big)^2
=
2\cdot
\frac{2}{27}\pi^4
\end{array}
$$

Taking the sum of the two contributions we obtain the answer:

$$
\left(\frac{3}{2}\cdot\frac{2}{27}+2\cdot\frac{2}{27}\right)\cdot\pi^4
=
\frac{7}{27}\cdot \pi^4
\,.
$$
Dividing by $\Vol\cQ(1^2,-1^2)=\frac{\pi^4}{3}$ we get the answer which
matches the value found in~\cite{Goujard:carea}.
$$
\frac{\pi^2}{3}\cdot c_{area}(\cQ(1^2,-1^2))=
\left(\frac{7}{27}\cdot\pi^4\right):
\left(\frac{1}{3}\cdot\pi^4\right)=\frac{7}{9}\,.
$$
   %


\section{Tables of volumes and of Siegel--Veech constants}
\label{a:tables}

In this appendix we present numerical data for
$\cQ_{g,n}$ corresponding to small values of $g$ and $n$.
We apply formulae~\eqref{eq:square:tiled:volume}
and~\eqref{eq:carea} respectively to express in
Tables~\ref{table:vol:int:nb} and~\ref{table:SV:int:nb} the
Masur--Veech volumes $\Vol(\cQ_{g,n})$ and the Siegel--Veech
constants $\carea(\cQ_{g,n})$ as polynomials in the
intersection numbers.

Recall that applying formula~\eqref{eq:carea:Elise},
one can express the
Siegel--Veech constant $\carea$ in terms of the volumes of the
principal boundary strata. Table~\ref{table:SV:int:nb} provides the
corresponding explicit expressions for low-dimensional strata.

Finally, Table~\ref{table:vol:SV} gathers the numerical values of the
volumes, of the Siegel--Veech constants and of the sums of the
Lyapunov exponents of the low-dimensional strata
$\cQ(1^\zeroes,-1^p)$. (Recall that the stratum
$\cQ(1^\zeroes,-1^p)$ is open and dense in $\cQ_{g,n}$,
where $p=n$ and $l=4g-4+n$.)
By formula~2.3
in~\cite{Eskin:Kontsevich:Zorich}, these quantities are related as
follows:
$$
\Lambda^+=\frac{1}{24}\left(\frac{5}{3}\ell-3 p\right) +\frac{\pi^2}{3}\carea(\cQ(1^\zeroes,-1^p))\,.
$$
The sum $\Lambda^+=\lambda_1^+ +\dots+ \lambda_g^+$ of the top $g$
Lyapunov exponents of the stratum $\cQ(1^\zeroes,-1^p)$
admits experimental evaluation. The values
of $\Lambda^+$ based on the values of $\carea(\cQ(1^\zeroes,-1^p))$
computed in this paper match the approximate values obtained in these
numerical experiments.

One has $\Lambda^+=0$ in genus zero, so in genus zero
the Siegel--Veech constant
admits a simple closed formula. By formula~1.1 in~\cite{AEZ:genus:0} one has
$$
\Vol\cQ(1^{p-4},-1^p)=2\pi^2\left(\frac{\pi^2}{2}\right)^{p-4}\,.
$$
We get the same expressions in the entries of Table~\ref{table:vol:SV}
corresponding to genus $0$ computed by formulae~\eqref{eq:square:tiled:volume}
and~\eqref{eq:carea} of the current paper.


\begin{table}[ht]
\renewcommand{\arraystretch}{1.5}
\hspace*{-6pt}
$\begin{array}{|c|c|c|}
\hline
g&p & \Vol\cQ(1^{4g-4+p} ,-1^p) \\
\hline
0&4 & 12 \zeta(2)\langle \tau_0^3\rangle^2 \\
 \cline{3-3 }
 && 2\pi^2 \langle \tau_0^3\rangle^2\\
 \hline
0 & 5 & 40\zeta(4)\langle \tau_0^3\tau_1\rangle \langle \tau_0^3\rangle + 20\zeta(2)^2 \langle \tau_0^3\rangle^3\\
 \cline{3-3 }
&& \frac{\pi^4}{9}(4\langle \tau_0^3\tau_1\rangle \langle \tau_0^3\rangle+ 5 \langle \tau_0^3\rangle^3)\\
\hline
0 & 6 &  20\zeta(6)(4\langle \tau_0^3\tau_1\rangle^2+3\langle \tau_0^4\tau_2\rangle \langle \tau_0^3\rangle)+120\zeta(4)\zeta(2)\langle \tau_0^3\tau_1\rangle \langle \tau_0^3\rangle^2 +28\zeta(2)^3 \langle \tau_0^3\rangle^4\\
   \cline{3-3}
 & &\pi^6(\frac{16}{189}\langle \tau_0^3\tau_1\rangle^2+\frac{4}{63}\langle \tau_0^4\tau_2\rangle \langle \tau_0^3\rangle+\frac{2}{9}\langle \tau_0^3\tau_1\rangle \langle \tau_0^3\rangle^2 +\frac{7}{54}\langle \tau_0^3\rangle^4)\\
  \hline

1 & 2 & 16\zeta(4)(\langle \tau_1\rangle \langle \tau_0^3\rangle+\langle \tau_0^3\tau_1\rangle)+\frac{16}{3}\zeta(2)^2 \langle \tau_0^3\rangle^2\\
   \cline{3-3}
 & & \frac{4\pi^4}{135}\big(6(\langle \tau_1\rangle \langle \tau_0^3\rangle+\langle \tau_0^3\tau_1\rangle)+5 \langle \tau_0^3\rangle^2\big)\\
  \hline
1 & 3 & 24\zeta(6)(\langle \tau_0^4\tau_2\rangle+\langle \tau_0^3\tau_1^2\rangle+2\langle \tau_1\rangle\langle \tau_0^3\tau_1\rangle+3\langle \tau_0\tau_2\rangle \langle \tau_0^3\rangle)\\
 & &+\frac{24}{5}\zeta(2)\zeta(4)(8\langle \tau_0^3\tau_1\rangle \langle \tau_0^3\rangle+3\langle \tau_1\rangle \langle \tau_0^3\rangle^2)+\frac{32}{5}\zeta(2)^3 \langle \tau_0^3\rangle^3\\
   \cline{3-3 }
 & & \frac{\pi^6}{45}\big(\frac{8}{7}(\langle \tau_0^4\tau_2\rangle+\langle \tau_0^3\tau_1^2\rangle+2\langle \tau_1\rangle\langle \tau_0^3\tau_1\rangle+3\langle \tau_0\tau_2\rangle \langle \tau_0^3\rangle) \\
 &&
  +\frac{16}{5}\langle \tau_0^3\tau_1\rangle \langle \tau_0^3\rangle+\frac{6}{5}\langle \tau_1\rangle \langle \tau_0^3\rangle^2+\frac{4}{3} \langle \tau_0^3\rangle^3\big)\\
  \hline
2&0 & 192\zeta(6)(\langle \tau_0\tau_2\rangle
    \!+\!\langle \tau_1^2\rangle
    \!+\!\langle \tau_1\rangle^2)
    \!+\!\frac{96}{5}\zeta(2)\zeta(4)(\langle \tau_1\rangle \langle \tau_0^3\rangle
    \!+\!\langle \tau_0^3\tau_1\rangle)
    \!+\!\frac{8}{3}\zeta(2)^3 \langle \tau_0^3\rangle^2\\
   \cline{3-3 }
 & & \frac{64\pi^6}{315}(\langle \tau_0\tau_2\rangle +\langle \tau_1^2\rangle+\langle \tau_1\rangle^2)+\frac{8\pi^6}{225}(\langle \tau_1\rangle \langle \tau_0^3\rangle+\langle \tau_0^3\tau_1\rangle)+\frac{\pi^6}{81} \langle \tau_0^3\rangle^2\\
  \hline
2 & 1
    & 320\zeta(8)(\langle\tau_0^2 \tau_3\rangle+3\langle\tau_0\tau_1\tau_2\rangle
      + 3\langle \tau_1\rangle\langle \tau_0\tau_2\rangle)
      + \frac{160}{7}\zeta(6)\zeta(2)(\langle \tau_0^4\tau_2\rangle+\langle \tau_0^3\tau_1^2\rangle+\\
 & &\hspace*{-5pt}
 \langle \tau_0\tau_2\rangle\langle \tau_0^3\rangle
 \!+\!2\langle \tau_0^3\tau_1\rangle\langle \tau_1\rangle
 \!+\!2\langle \tau_0\tau_2\rangle)
  \!+\!\frac{96}{4}\zeta(4)^2 (\langle \tau_0^3\tau_1^2\rangle
  \!+\!2\langle \tau_0^3\tau_1\rangle\langle \tau_1\rangle
  \!+\! \langle \tau_1\rangle^2\langle \tau_0^3\rangle
  +\!
  \\
 & &\langle \tau_1^2\rangle)+\frac{16}{7}\zeta(4)\zeta(2)^2(9\langle \tau_0^3\tau_1\rangle\langle \tau_0^3\rangle +4\langle \tau_1\rangle\langle \tau_0^3\rangle^2)+\frac{40}{21}\zeta(2)^4\langle \tau_0^3\rangle^3\\
  \cline{3-3 }
   & & \frac{32\pi^8}{945}(\langle\tau_0^2 \tau_3\rangle+3\langle\tau_0\tau_1\tau_2\rangle+ 3\langle \tau_1\rangle\langle \tau_0\tau_2\rangle) + \frac{16\pi^8}{3969}(\langle \tau_0^4\tau_2\rangle+\langle \tau_0^3\tau_1^2\rangle+\\
 & &\langle \tau_0\tau_2\rangle\langle \tau_0^3\rangle
    \!+\!2\langle \tau_0^3\tau_1\rangle\langle \tau_1\rangle+2\langle \tau_0\tau_2\rangle)
    \!+\!\frac{2\pi^8}{675} (\langle \tau_0^3\tau_1^2\rangle
    \!+\!2\langle \tau_0^3\tau_1\rangle\langle \tau_1\rangle
    \!+\! \langle \tau_1\rangle^2\langle \tau_0^3\rangle +\\
 & &\langle \tau_1^2\rangle)+\frac{2}{2835}(9\langle \tau_0^3\tau_1\rangle\langle \tau_0^3\rangle +4\langle \tau_1\rangle\langle \tau_0^3\rangle^2)+\frac{5}{3402}\langle \tau_0^3\rangle^3\\
  \hline
\end{array}$
\vspace{0.5cm}
\caption{
\label{tab:Vol:as:polynomials}
Volumes of low-dimensional strata as polynomials in intersection numbers\label{table:vol:int:nb}}
\end{table}

\newpage

\begin{table}[ht]
\renewcommand{\arraystretch}{1.5}
$\begin{array}{|c|c|c|}
\hline
g&p & \Vol\cQ(1^{4g-4+p},1^p) \cdot\frac{\pi^2}{3}\cdot \carea\cQ(1^{4g-4+p},1^p) \\
\hline
0 &4 & \frac{\pi^2}{3}\cdot \frac{3}{16}(\Vol\cQ_{0,3})^2\\
\cline{3-3}
 &&6 \zeta(2) \langle \tau_0^3\rangle^2\\
 \cline{3-3}
 && \pi^2 \langle \tau_0^3\rangle^2\\
 \hline
0 & 5 & \frac{\pi^2}{3}\cdot\frac{5}{24} \Vol\cQ_{0,4}\Vol\cQ_{0,3}\\
\cline{3-3}
& &20 \zeta(2)^2 \langle \tau_0^3\rangle^3\\
\cline{3-3}
&& \frac{5\pi^4}{9} \langle \tau_0^3\rangle^3\\
\hline
0 & 6 & \frac{\pi^2}{3}\cdot\Big(\frac{1}{24}(\Vol\cQ_{0,4})^2+\frac{3}{16}\Vol\cQ_{0,5}\Vol\cQ_{0,3}\Big)\\
\cline{3-3}
  && 60\zeta(2)\zeta(4)\langle \tau_0^3\tau_1\rangle \langle \tau_0^3\rangle^2+42\zeta(2)^3 \langle \tau_0^3\rangle^4 \\
\cline{3-3}
 & &\frac{\pi^6}{9}\langle \tau_0^3\tau_1\rangle \langle \tau_0^3\rangle^2+\frac{7\pi^6}{36} \langle \tau_0^3\rangle^4\\
  \hline
1 & 2 & \frac{\pi^2}{3}\cdot\Big(\frac{1}{24}\Vol\cQ_{1,1}\Vol\cQ_{0,3} +\frac{1}{3}\Vol\cQ_{0,4}\Big)\\
\cline{3-3}
 & & \frac{28}{3}\zeta(2)^2 \langle \tau_0^3\rangle^2\\
\cline{3-3}
  && \frac{7\pi^4}{27} \langle \tau_0^3\rangle^2\\
  \hline
1 & 3 & \frac{\pi^2}{3}\cdot\Big(\frac{3}{10}\Vol\cQ_{0,5}+\frac{1}{80}\Vol\cQ_{1,1}\Vol\cQ_{0,4}  +\frac{9}{160}\Vol\cQ_{1,2}\Vol\cQ_{0,3}\Big)\\
\cline{3-3}
 & & \zeta(2)\zeta(4)(\frac{156}{5}\langle \tau_0^3\tau_1\rangle \langle \tau_0^3\rangle+\frac{36}{5}\langle \tau_1\rangle \langle \tau_0^3\rangle^2)+\frac{78}{5}\zeta(2)^3 \langle \tau_0^3\rangle^3 \\
\cline{3-3}
  && \frac{13\pi^6}{225}\langle \tau_0^3\tau_1\rangle \langle \tau_0^3\rangle+\frac{\pi^6}{75}\langle \tau_1\rangle \langle \tau_0^3\rangle^2+\frac{13\pi^6}{180}\langle \tau_0^3\rangle^3\\
\hline
2 & 0 & \frac{\pi^2}{3}\cdot\Big(\frac{1}{40}(\Vol\cQ_{1,1})^2+\frac{3}{5}\Vol\cQ_{1,2}\Big)\\
\cline{3-3}
  && \frac{96}{5}\zeta(2)\zeta(4)(\langle \tau_1\rangle \langle \tau_0^3\rangle+\langle \tau_0^3\tau_1\rangle)+\frac{36}{5}\zeta(2)^3 \langle \tau_0^3\rangle^2\\
\cline{3-3}
  && \frac{8\pi^6}{225}(\langle \tau_1\rangle\langle \tau_0^3\rangle+\langle \tau_0^3\tau_1\rangle)+\frac{\pi^6}{30}\langle \tau_0^3\rangle^2\\
  \hline
2 & 1 &\frac{\pi^2}{3}\cdot\Big( \frac{10}{21}\Vol\cQ_{1,3} + \frac{1}{56}\Vol\cQ_{1,1}\Vol\cQ_{1,2}\Big)\\
\cline{3-3}
   &&
   \frac{160}{7}\zeta(6)\zeta(2)(\langle \tau_0^4\tau_2\rangle+\langle \tau_0^3\tau_1^2\rangle+
   2\langle \tau_0^3\tau_1\rangle\langle \tau_1\rangle+3\langle \tau_0\tau_2\rangle\langle \tau_0^3\rangle )+ \\
   &&\zeta(4)\zeta(2)^2(\frac{272}{7}\langle \tau_0^3\tau_1\rangle\langle \tau_0^3\rangle +16\langle \tau_1\rangle\langle \tau_0^3\rangle^2)+\frac{48}{7}\zeta(2)^4\langle \tau_0^3\rangle^3\\
\cline{3-3}
   &&
   \frac{16\pi^8}{3969}(\langle \tau_0^4\tau_2\rangle+\langle \tau_0^3\tau_1^2\rangle+
   2\langle \tau_0^3\tau_1\rangle\langle \tau_1\rangle+3\langle \tau_0\tau_2\rangle\langle \tau_0^3\rangle )+ \\
   &&\frac{34\pi^8}{2835}\langle \tau_0^3\rangle +\frac{2\pi^8}{405}\langle \tau_1\rangle\langle \tau_0^3\rangle^2+\frac{\pi^8}{189}\langle \tau_0^3\rangle^3\\
   \hline
\end{array}$
\vspace{0.5cm}
\caption{
\label{tab:SV:as:polynomials}
Siegel--Veech constants of low-dimensional strata in terms of volumes of principal boundary strata, and as polynomials in intersection numbers\label{table:SV:int:nb}}
\end{table}

\begin{table}[ht]
 \renewcommand{\arraystretch}{1.2}
$\begin{array}{|c|c|c|c|c|c|c|}
\hline
g& p & \textrm{Stratum} &\Vol & {\pi^2}/{3}\cdot c_{area} & \Lambda^+ & \Lambda^- \\
\hline
0 & 5 & \cQ(1, -1^5) & \pi^4 &{5}/{9} &  0 & {4}/{3} \\
\hline
0 & 6 & \cQ(1^2, -1^6) & {1}/{2}\cdot\pi^6 & {11}/{18} & 0 & {5}/{3} \\
\hline
0 & 7 & \cQ(1^3, -1^7) & {1}/{4}\cdot\pi^8 & 2/3 & 0 & 2\\
\hline
1 & 2 & \cQ(1^2, -1^2) & {1}/{3}\cdot\pi^4 & {7}/{9} & {2}/{3} & {4}/{3}\\
\hline
1 & 3 & \cQ(1^3, -1^3) & {11}/{60}\cdot\pi^6 & {47}/{66} & {6}/{11} & {17}/{11}\\
\hline
1 & 4 & \cQ(1^4, -1^4) & {1}/{10}\cdot\pi^8 & {44}/{63} & {10}/{21} & {38}/{21}\\
\hline
1 & 5 & \cQ(1^5, -1^5) & {163}/{3024}\cdot\pi^{10} & {2075}/{2934} & {70}/{163} & {1025}/{489}\\
\hline
2 & 0 & \cQ(1^4) & {1}/{15}\cdot\pi^6 & {19}/{18} & {4}/{3} & {5}/{3} \\
\hline
2 & 1 & \cQ(1^5, -1) & {29}/{840}\cdot\pi^8 & {230}/{261} & {32}/{29} & {154}/{87}\\
\hline
2 & 2 & \cQ(1^6, -1^2) & {337}/{18144}\cdot\pi^{10} & {8131}/{10110} & {1636}/{1685} & {3321}/{1685}\\
\hline
3 & 0 & \cQ(1^8) & {115}/{33264}\cdot\pi^{12} & {24199}/{25875}  & {4286}/{2875}  & {18608}/{8625} \\
\hline
4 & 0 & \cQ(1^{12}) & \pi^{18}\cdot{2106241}/ & {283794163}/ &{91179048}/ &{143835073}/\\
& & & {11548293120} & {315936150} & {52656025} & {52656025} \\
\hline
\end{array}$
\vspace{0.5cm}
\caption{
\label{tab:Vol:Q:g:n}
Numerical values of volumes, of Siegel--Veech constants and of
sums of Lyapunov exponents for low-dimensional strata\label{table:vol:SV}}
\end{table}

\newpage


\section{Two multiple harmonic sums}
\label{s:Two:harmonic:sums}

In this section we compute asymptotic expansions of
multiple harmonic sum and its closed relative which are
needed in the study of large genus asymptotics of
Masur--Veech volume $\Vol\cQ_g$. Both sums seem to us very
natural, but we failed to find corresponding asymptotic
expansions in the literature (with exception for relation
(21) in Theorem~VI.2 in~\cite{Flajolet:Sedgewick} which is
very close to what we need).

We state all necessary results in
Sections~\ref{s:Basic:asymptotic:expansions}
and~\ref{ss:fine:asymptotic:expansions} for the basic
asymptotics and for fine asymptotics respectively. The
proofs are postponed to
Sections~\ref{ss:sum:in:one:variable}
and~\ref{ss:sum:in:multiple:variable} respectively.

\subsection{Basic asymptotic expansions}
\label{s:Basic:asymptotic:expansions}

It is a classical result that the harmonic sum
has the following asymptotic expansion:
\begin{equation}
\label{eq:harmoinc:sum}
\sum_{k=1}^{m}\,
\frac{1}{k}
=\log(m)+\gamma+\frac{1}{2m}
+o\left(\frac{1}{m}\right)
\quad\text{as }m\to+\infty\,,\,,
\end{equation}
where $\gamma= 0.5772...$ is the Euler--Mascheroni constant.

We need the asymptotic expansion for the
following analog of the harmonic sum.

\begin{Lemma}
\label{lm:zeta:harmonic:sum}
The following asymptotic formula holds:
\begin{equation}
\label{eq:sum:1:value}
\sum_{k=1}^m
\frac{\zeta(2k)}{k}
=\log(2m)+\gamma+O\left(\frac{1}{m}\right)
\quad\text{as }m\to+\infty\,,
\end{equation}
where $\gamma$ is the Euler--Mascheroni constant.
\end{Lemma}

We also need asymptotic expansions for the multivariable
analogs of the two above sums. We separate the case of the
sum over two variables and the case of three and more variables.

\begin{Lemma}
\label{lm:double:harmonic:sum}
The following asymptotic formulae hold as $m\to+\infty$:
\begin{align}
\label{eq:double:harmonic:sum}
\sum_{k=1}^{m-1}&
\frac{1}{k\cdot(m-k)}
=
\frac{2}{m}\left(\log m+ \gamma
+O\left(\frac{1}{m}\right)\right)\,,
\\
\label{eq:double:harmonic:zeta:sum}
\sum_{j=k}^{m-1}&
\frac{\zeta(2k)\cdot\zeta\big(2(m-k)\big)}{k\cdot(m-k)}
=
\frac{2}{m}\left(\log(2m) + \gamma
+O\left(\frac{1}{m}\right)\right)\,,
\end{align}
\end{Lemma}

\begin{Theorem}
\label{th:multiple:harmonic:sum}
For any $k\ge 3$
the following asymptotic formulae hold as $m\to+\infty$:
\begin{align}
\label{eq:multiple:harmonic:sum}
\sum_{j_1+\dots+j_k=m}&\
\frac{1}{j_1\cdot j_2\cdots j_k}
=
\frac{k}{m}\left(\big(\log m+ \gamma\big)^{k-1}
+O\left(\log^{k-3} m\right)\right)\,,
\\
\label{eq:multiple:harmonic:zeta:sum}
\sum_{j_1+\dots+j_k=m}&\
\frac{\zeta(2j_1)\cdots\zeta(2j_k)}
{j_1\cdot j_2\cdots j_k}
=
\frac{k}{m}\left(\big(\log(2m)+ \gamma\big)^{k-1}
+O\left(\log^{k-3} m\right)\right)\,,
\end{align}
where in the case $k=3$ the expressions
$O\left(\log^{k-3}m\right)$ should be
interpreted as terms which are
uniformly bounded in $m$.
\end{Theorem}

\subsection{Fine asymptotic expansions}
\label{ss:fine:asymptotic:expansions}

For any $k\in\N$ we shall need the full expansion of the
sums~\eqref{eq:multiple:harmonic:sum}
and~\eqref{eq:multiple:harmonic:zeta:sum} in the powers of
$\log(m)$ up to the term $O(1)$ uniformly bounded in $m$.
We start by defining quantities $A_j, B_j$ which will be
used in the corresponding expansions.

For $j=0,1,2,\dots$ denote
\begin{equation}
\label{eq:c:j:as:integral}
c_j=\frac{1}{j!}\int_0^{+\infty} e^{-u}(\log u)^j\,du\,.
\end{equation}
It is known that
$$
\Gamma^{(n)}(z)=\int_0^{+\infty}(\log t)^n e^{-t} t^{z-1}\, dt
\quad\text{for }n\ge 0, z>0
$$
see, say~\cite[(5.9.19)]{DLMF}. Thus
\begin{equation}
\label{eq:def:c:j}
c_j=\frac{1}{j!}\cdot\Gamma^{(j)}(z)\big\vert_{z=1}
=\frac{1}{j!}\cdot
\left.\frac{d^j \Gamma(z)}{dz^j}\right|_{z=1}
\,.
\end{equation}
For small values of $j$ we get
\begin{align*}
c_0&=1\\
c_1&=-\gamma\\
c_2&=\frac{1}{2!}\big(\gamma^2+\zeta(2)\big)\\
c_3&=-\frac{1}{3!}\big(\gamma^3+3\gamma\cdot\zeta(2)+2\zeta(3)\big)\\
c_4&=\frac{1}{4!}\Big(\gamma^4+6\gamma^2\cdot\zeta(2)+8\gamma\cdot\zeta(3)+\frac{27}{2}\zeta(4)\Big)\\
c_5&=-\frac{1}{5!}\Big(\gamma^5+10\gamma^3\cdot\zeta(2)+20\gamma^2\cdot\zeta(3)+\frac{135}{2}\gamma\cdot\zeta(4)+20\zeta(2)\zeta(3)+24\zeta(5)\Big)\,.
\end{align*}

Though the Guess stated below is not needed for the current
paper, we could not resist temptation to formulate it for
the sake of completeness.

\begin{Guess}
Coefficients $c_j$ defined by~\eqref{eq:def:c:j}
satisfy the following relation:
\begin{equation}
\label{eq:lim:c:j}
\lim_{j\to+\infty} c_j-(-1)^j
\overset{?}{=}0\,.
\end{equation}
where convergence is exponentially fast.
\end{Guess}
By definition~\eqref{eq:def:c:j} of
$c_j$, relation~\eqref{eq:lim:c:j} is equivalent to the
following one:
\begin{equation}
\lim_{j\to+\infty}
\left(\frac{1}{j!}
\cdot\Gamma^{(j)}(z)\big\vert_{z=1}-(-1)^j\right)
\overset{?}{=}0\,.
\end{equation}

Define $A_0, A_1, \dots $
recursively by the initial condition $A_0=1$
and by the following recurrence relation:
\begin{equation}
\label{eq:def:A:j}
A_j\cdot c_0+A_{j-1}\cdot c_1+\dots+A_1\cdot c_{j-1}+A_0\cdot c_j
=0\quad\text{for }j\in\N\,.
\end{equation}
For small values of $j$ we get
\begin{align*}
\label{eq:A:0}
A_0&=1\\
A_1&=\gamma\,,\\
A_2&=\frac{\gamma^2}{2!} - \frac{\zeta(2)}{2}\,,\\
A_3&=\frac{\gamma^3}{3!} - \frac{\zeta(2)}{2}\cdot\frac{\gamma}{1!} + \frac{\zeta(3)}{3}\,,\\
A_4&=\frac{\gamma^4}{4!} - \frac{\zeta(2)}{2}\cdot\frac{\gamma^2}{2!}
     +\frac{\zeta(3)}{3}\cdot\frac{\gamma}{1!}\\
     &\notag
     +\left(-\frac{\zeta(4)}{4}+\frac{1}{2}\cdot\left(\frac{\zeta(2)}{2}\right)^2\right)\,.
\end{align*}

Define $B_0, B_1, \dots $
recursively by the initial condition $B_0=1$
and by the following recurrence relation:
\begin{equation}
\label{eq:def:B:j}
B_j\cdot c_0+B_{j-1}\cdot c_1+\dots+B_1\cdot c_{j-1}+B_0\cdot c_j
=\frac{(\log 2)^j}{j!}\,.
\end{equation}
For small values of $j$ we get
\begin{align}
\label{eq:B:0}
B_0&=1\\
\label{eq:B:1}
B_1&=\frac{\big(\log(2) + \gamma\big)}{1!}\,,\\
\notag
B_2&=\frac{\big(\log(2)+\gamma\big)^2}{2!} - \frac{\zeta(2)}{2}\,,\\
\notag
B_3&=\frac{\big(\log(2)+\gamma\big)^3}{3!} - \frac{\zeta(2)}{2}\cdot\frac{\big(\log(2)+\gamma\big)}{1!} + \frac{\zeta(3)}{3}\,,\\
\notag
B_4&=\frac{\big(\log(2)+\gamma\big)^4}{4!} - \frac{\zeta(2)}{2}\cdot\frac{\big(\log(2)+\gamma\big)^2}{2!}
     +\frac{\zeta(3)}{3}\cdot\frac{\big(\log(2)+\gamma\big)}{1!}\\
     &\notag
     +\left(-\frac{\zeta(4)}{4}+\frac{1}{2}\cdot\left(\frac{\zeta(2)}{2}\right)^2\right)\,.
\end{align}

The Proposition below provides generating functions for the
sequences $A_0,A_1,\dots$ and $B_0,B_1,\dots$.

\begin{Proposition}
\label{prop:generating:function:for:A:and:B}
For any $x$ satisfying $|x|<2$ we have
\begin{align}
\label{eq:sum:A:j:x:power:j}
\sum_{j=0}^{+\infty} A_j\cdot x^j
&=\frac{1}{\Gamma(1+x)}
\\
\label{eq:sum:B:j:x:power:j}
\sum_{j=0}^{+\infty} B_j\cdot x^j
&=\frac{2^x}{\Gamma(1+x)}
\,.
\end{align}
\end{Proposition}

\begin{Remark}
The coefficients of the series expansion of $\frac{1}{\Gamma(z)}$
(closely related to the coefficients $A_0, A_1, \dots$)
are given by the following recurrence (see~\cite[(5.7.2)]{DLMF}
Let
$$
\frac{1}{\Gamma(z)}
=\sum_{k=1}^{+\infty} d_k z^k\,.
$$
Then $d_1=1$, $d_2=\gamma$, and
$$
(k-1)d_k=\gamma d_{k-1}-\zeta(2) d_{k-2}+\zeta(3) d_{k-3}-\dots+(-1)^k\zeta(k-1)d_1\,.
$$
\end{Remark}

\begin{Corollary}
\label{cor:properties:of:B}
The following relations are valid:
\begin{align}
\label{eq:sum:A:j:over:2:power:j}
\sum_{j=0}^{+\infty}\frac{A_j}{2^j}
&=\frac{2}{\sqrt{\pi}}\,,
\\
\label{eq:sum:B:j:over:2:power:j}
\sum_{j=0}^{+\infty}\frac{B_j}{2^j}
&=2\sqrt{\frac{2}{\pi}}\,,
\\
\label{eq:A:B:j:tend:fast:to:zero}
\limsup_{j\to+\infty}\sqrt[j]{|A_j|}
&=
\limsup_{j\to+\infty}\sqrt[j]{|B_j|}
=\frac{1}{2}\,,
\\
\label{eq:sum:j:A:j:x:power:j}
\sum_{j=1}^{+\infty}
j\cdot A_j x^{j-1}
&=
-\frac{\psi(1+x)}{\Gamma(1+x)}\,,
\\
\label{eq:sum:j:B:j:x:power:j}
\sum_{j=1}^{+\infty}
j\cdot B_j x^{j-1}
&=
\frac{2^x}{\Gamma(1+x)}
\cdot\big(\log(2)-\psi(1+x)\big)\,,
\\
\label{eq:sum:j:A:j:over:2:power:j}
\sum_{j=1}^{+\infty}\frac{j\cdot A_j}{2^{j-1}}
&=2\cdot\frac{\big(2\log(2)+\gamma-2\big)}{\sqrt{\pi}}\,,
\\
\label{eq:sum:j:B:j:over:2:power:j}
\sum_{j=1}^{+\infty}\frac{j\cdot B_j}{2^{j-1}}
&=2\sqrt{2}\cdot\frac{\big(3\log(2)+\gamma-2\big)}{\sqrt{\pi}}\,,
\end{align}
where $\psi(z)=\cfrac{\Gamma'(z)}{\Gamma(z)}$
is the digamma function.
\end{Corollary}

\begin{Theorem}
\label{th:multiple:harmonic:zeta:sum:full:expansion}
For any $k\in\N$,
the following asymptotics hold as $m\to+\infty$:
\begin{align}
\label{eq:multiple:harmonic:sum:full:expansion}
H_k(m)=\sum_{j_1+\dots+j_k=m}&
\frac{1}{j_1\cdot j_2\cdots j_k}
\, =\\ \notag
&=\frac{k!}{m}\cdot
\left(
\sum_{j=0}^{k-1}
\frac{A_j}{(k-1-j)!}
\cdot\big(\log m\big)^{k-1-j}+o(1)\right)
\,,
\\
\label{eq:multiple:harmonic:zeta:sum:full:expansion}
Z_k(m)=\sum_{j_1+\dots+j_k=m}&
\frac{\zeta(2j_1)\cdots\zeta(2j_k)}
{j_1\cdot j_2\cdots j_k}
\, =\\ \notag
&=\frac{k!}{m}\cdot
\left(
\sum_{j=0}^{k-1}
\frac{B_j}{(k-1-j)!}
\cdot\big(\log m\big)^{k-1-j}+o(1)\right)
\,.
\end{align}
\end{Theorem}

For any $C>0$ define
\begin{align}
\label{eq:S:m:harmonic}
\VSumH(C,m)&=\sum_{k=1}^{[C\log(m)]}
\frac{H_k(m)}{k!\cdot 2^{k-1}}
\,,\quad
&\VSumH(m)&=\sum_{k=1}^m
\frac{H_k(m)}{k!\cdot 2^{k-1}}
\,,\\
\label{eq:S:m:zeta}
\VSumZ(C,m)&=\sum_{k=1}^{[C\log(m)]}
\frac{Z_k(m)}{k!\cdot 2^{k-1}}
\,,\quad
&\VSumZ(m)&=\sum_{k=1}^m
\frac{Z_k(m)}{k!\cdot 2^{k-1}}
\,.
\end{align}

\begin{CondTheorem}
\label{cond:th:volume:coefficient:total:sum}
For any $C$ in the interval $]\tfrac{1}{2};2[$ for which
Conjecture~\ref{conj:uniform:bound:for:error:term} is
valid, one has
\begin{align}
\label{eq:volume:coefficient:H}
\lim_{m\to+\infty} \sqrt{m}\cdot \VSumH(C,m)
=\frac{2}{\sqrt{\pi}}
=\lim_{m\to+\infty} \sqrt{m}\cdot \VSumH(m)
\,,
\\
\label{eq:volume:coefficient:Z}
\lim_{m\to+\infty} \sqrt{m}\cdot \VSumZ(C,m)
=2\sqrt{\frac{2}{\pi}}
=\lim_{m\to+\infty} \sqrt{m}\cdot \VSumZ(m)
\,.
\end{align}
\end{CondTheorem}

\begin{Remark}
Actually, for the sequel we only need the first equation in
each of the two lines above and only for some concrete
$C\in]\tfrac{1}{2};2[$ for which
Conjecture~\ref{conj:introduction:sum:of:correlators} is
valid.
Note also, that for $C<\tfrac{1}{2}$ the above
relations seem to be false.
\end{Remark}

We derive Conditional
Theorem~\ref{cond:th:volume:coefficient:total:sum}
from
Conjecture~\ref{conj:uniform:bound:for:error:term}
at the end of
Section~\ref{ss:sum:in:multiple:variable}.

\subsection{Computation of the sum for single variable}
\label{ss:sum:in:one:variable}

In this Section we provide a straightforward computation
of asymptotic expansion~\eqref{eq:sum:1:value}.

\begin{proof}[Proof of Lemma~\ref{lm:zeta:harmonic:sum}]
We break the sum into three
known ones:
\begin{multline}
\label{eq:sum:1}
\sum_{k=1}^n
\frac{\zeta(2k)}{k}
=
2\sum_{k=1}^n
\frac{\zeta(2k)}{2k}
=
\sum_{j=2}^{2n}
\frac{(\zeta(j)-1)+1+(-1)^j\zeta(j)}{j}
=\\=
\sum_{j=2}^{2n}
\frac{\zeta(j)-1}{j}
\,+\,
\sum_{j=2}^{2n}\,
\frac{1}{j}
\,+\,
\sum_{j=2}^{2n}
\frac{(-1)^j\zeta(j)}{j}\,.
\end{multline}
It was found by Euler that
\begin{align*}
\sum_{j=2}^{+\infty}
\frac{(-1)^j\zeta(j)}{j}&=\gamma\,,
\\
\sum_{j=2}^{+\infty}
\frac{\zeta(j)-1}{j}&=1-\gamma
\end{align*}
(see~(2.2.3)
in the survey~\cite{Lagarias} for the first sum.
For the second sum,
extract~(2.2.11) from~(2.2.5) in the same survey.
Alternatively, see pages 109 and 111--112 in~\cite{Havil}.)

Using exponentially rapid convergence
of $\zeta(j)$ to $1$ when $j$ grows
to estimate the tails of the two above sums
and using the asymptotic expansion
$$
\sum_{j=2}^{2n}\,
\frac{1}{j}
=-1+\log(2n)+\gamma+O\left(\frac{1}{g}\right)
$$
(where the term $-1$ is present since the sum starts from $j=2$), we obtain
the asymptotic expansion~\eqref{eq:sum:1:value} summing
the three terms in the right-hand side
of~\eqref{eq:sum:1}.
\end{proof}

Now we are ready to derive the asymptotic
expansion~\eqref{eq:double:harmonic:zeta:sum}.

\begin{proof}[Proof of Lemma~\ref{lm:double:harmonic:sum}]
We first note that $\zeta(m)$ monotoneously tends to $1$
as $m\to+\infty$ and the convergence is
exponentially fast. Thus, up to a very small error term
$\delta(m)$ we can rewrite our sum
in the following simpler way:
\begin{multline}
\label{eq:coef:2:sum}
\sum_{j=1}^{m-1}
\frac{\zeta(2j)\cdot\zeta(2m-2j)}{j\cdot(m-j)}
=\\=
\frac{\zeta(2)\cdot\zeta(2m-2)}{1\cdot(m-1)}
+\frac{\zeta(4)\cdot\zeta(2m-4)}{2\cdot(m-2)}
+\dots
+\frac{\zeta(2m-4)\cdot\zeta(4)}{(m-2)\cdot 2}
+\frac{\zeta(2m-2)\cdot\zeta(2)}{(m-1)\cdot 1}
=\\=
2\left(
\frac{\zeta(2)}{1\cdot(m-1)}
+\frac{\zeta(4)}{2\cdot(m-2)}
+\dots
+\frac{\zeta\left(2\left[\frac{m}{2}\right]\right)}
{\left[\frac{m}{2}\right]\cdot(m-\left[\frac{m}{2}\right])}
\right)
+\delta(m)
\,,
\end{multline}
where
\begin{equation}
0<\delta(m)\le\zeta(2)\cdot\big(\zeta(m)-1\big)
=o\left(\frac{1}{m}\right)\,.
\end{equation}
We split the latter sum into three sums, which
we evaluate one-by-one:
\begin{multline}
\label{eq:three:sums}
\frac{\zeta(2)}{1\cdot(m-1)}
+\frac{\zeta(4)}{2\cdot(m-2)}
+\dots
+\frac{\zeta\left(2\left[\frac{m}{2}\right]\right)}
{\left[\frac{m}{2}\right]\cdot(m-\left[\frac{m}{2}\right])}
=\\=
\frac{1}{m}\Bigg(
\left(
\frac{\zeta(2)}{1}
+\frac{\zeta(2)-1}{m-1}
+\frac{1}{m-1}
\right)
+
\left(
\frac{\zeta(4)}{2}
+\frac{\zeta(4)-1}{m-2}
+\frac{1}{m-2}
\right)
+\dots+
\\+\dots+
\left(
\frac{\zeta\left(2\left[\frac{m}{2}\right]\right)}
{\left[\frac{m}{2}\right]}
+
\frac{\zeta\left(2\left[\frac{m}{2}\right]\right)-1}
{(m-\left[\frac{m}{2}\right])}
+
\frac{1}
{(m-\left[\frac{m}{2}\right])}
\right)
\Bigg)
=\\=
\frac{1}{m}\left(
\sum_{j=1}^{\left[\frac{m}{2}\right]}
\frac{\zeta(2j)}{j}
+
\sum_{j=1}^{\left[\frac{m}{2}\right]}
\frac{\zeta(2j)-1}{m-j}
+
\sum_{j=1}^{\left[\frac{m}{2}\right]}
\frac{1}{m-j}
\right)\,.
\end{multline}

The first sum was computed in equation~\eqref{eq:sum:1:value}
in Lemma~\ref{lm:zeta:harmonic:sum}:
$$
\sum_{j=1}^{\left[\frac{m}{2}\right]}
\frac{\zeta(2j)}{j}
=\log m+\gamma+O\left(\frac{1}{m}\right)\,.
$$

As $\zeta(2j)$ extremely rapidly converges to $1$ when $j$
grows, the second sum is negligible,
\begin{equation}
\label{eq:sum:2}
\sum_{j=1}^{\left[\frac{m}{2}\right]}
\frac{\zeta(2j)-1}{m-j}
=
O\left(\frac{1}{m}\right)\,.
\end{equation}

The third sum is the difference of two harmonic sums:
\begin{multline}
\label{eq:sum:3}
\sum_{j=1}^{\left[\frac{m}{2}\right]}
\frac{1}{m-j}
=
\sum_{k=1}^{m-1}\, \frac{1}{k}
-
\sum_{k=1}^{m-\left[\frac{m}{2}\right]-1}\frac{1}{k}
=\\=
\left(\log(m-1)+\gamma+O\left(\frac{1}{m}\right)\right)-
\left(\log\left(m-\left[\frac{m}{2}\right]-1\right)
+\gamma
+O\left(\frac{1}{m}\right)\right)
=\\=
\log 2+O\left(\frac{1}{m}\right)
\,.
\end{multline}

Summing up the resulting values of the three sums in the
last line of~\eqref{eq:three:sums}; taking into
consideration the coefficient $\frac{1}{m}$ in front of
these three sums in the last line of~\eqref{eq:three:sums}
and coefficient $2$ in the last line
of~\eqref{eq:coef:2:sum} we
obtain~\eqref{eq:double:harmonic:zeta:sum}.
\end{proof}

For the purposes of the current paper we do not need better
precision in the asymptotic expansion for the
sum~\eqref{eq:sum:1:value} in single variable. However, we
expect that it would not be too difficult to get the
following next term in~\eqref{eq:sum:1:value}.
\begin{Guess}
We expect that the following asymptotic formula holds as $m\to+\infty$:
\begin{equation}
\label{eq:sum:1:value:better}
\sum_{k=1}^n
\frac{\zeta(2k)}{k}
=\log(2n)+\gamma+\frac{1}{2n}+o\left(\frac{1}{n}\right)\,,
\end{equation}
\end{Guess}

Indeed, for the harmonic sum we have
$$
\sum_{j=1}^{2n}\,
\frac{1}{j}
\sim\log(2n)+\gamma+\frac{1}{4n}\,.
$$
The tail of the first sum in our decomposition of the sum
on the left hand side of~\eqref{eq:sum:1:value:better} int
three sums~\eqref{eq:three:sums} is negligible since
$\sum_{j=2}^{+\infty}(\zeta(j)-1)$ is rapidly converging.
The tail of the third sum should behave as the tail of the
alternated harmonic series. For the alternated harmonic
series (it has opposite sign with respect to the ours) one
has
$$
\sum_{j=1}^{m}\,
\frac{(-1)^{j+1}}{j}
=\ln 2-\int_0^1\frac{x^m}{1+x}\,dx\,,
$$
so, we expect that the tail
is $\sim\frac{1}{2m}$. Thus, for our third sum
we expect that the error term should be
$$
\sum_{j=2}^{2n}
\frac{(-1)^j\zeta(j)}{j}
\overset{?}{\sim}\gamma+\frac{1}{4n}
\,,
$$
so we expect the more precise
version~\eqref{eq:sum:1:value:better}
of~\eqref{eq:sum:1:value} which is confirmed by numerics.

\subsection{Computation of the sum for multiple variables}
\label{ss:sum:in:multiple:variable}

In this Section we prove the assertions stated in
Section~\ref{ss:fine:asymptotic:expansions}.

\begin{proof}[Proof of Proposition~\ref{prop:generating:function:for:A:and:B}]
We prove~\eqref{eq:sum:B:j:x:power:j}; the proof
of~\eqref{eq:sum:A:j:x:power:j} is completely analogous
(but slightly simpler).

Multiplying~\eqref{eq:def:B:j}
by $x^j$ and taking the sum over all $j=0,1,2,\dots$
we get
$$
\sum_{j=0}^{+\infty}
\big(
B_j\cdot c_0+B_{j-1}\cdot c_1+\dots+B_1\cdot c_{j-1}+B_0\cdot c_j
\big)x^j
=\sum_{j=0}^{+\infty}
\frac{(\log 2)^j}{j!}x^j\,.
$$
Clearly, the series converges for any $x$ in $\R$
(actually, for any $x$ in $\C$) and we obtain $e^{\log(2)
x}=2^x$ on the right. Note that $\Gamma(1+x)$ is an
analytic function with radius of convergence of the power series
equal to $2$. Thus, we can rearrange the sum on the left
in the following way.
\begin{multline*}
\sum_{j=0}^{+\infty}
\big(
B_j\cdot c_0+B_{j-1}\cdot c_1+\dots+B_1\cdot c_{j-1}+B_0\cdot c_j
\big)x^j
=\\=
\sum_{j=0}^{+\infty}
\big(
(B_j x^j)\cdot (c_0 x^0)
+(B_{j-1} x^{j-1})\cdot (c_1 x^1)
+\dots+\\+\dots
+(B_1 x^1)\cdot (c_{j-1} x^{j-1})
+(B_0 x^0)\cdot (c_j x^j)
\big)
=\\=
B_0 (c_0 x_0+c_1 x^1 + c_2 x^2+\dots)
+ B_1 x^1 (c_0 x_0+c_1 x^1 + c_2 x^2+\dots)
+ \dots
=\\=
(B_0+B_1 x_1 + B_2 x^2 + \dots)\cdot
\left(\sum_{j=0}^{+\infty}
\frac{1}{j!}\cdot\Gamma^{(j)}(x)\big\vert_{x=1}\cdot x^j\right)
=
\left(\sum_{j=0}^{+\infty} B_j\cdot x^j\right)
\cdot\Gamma(1+x)\,,
\end{multline*}
where we used definition~\eqref{eq:def:c:j} of $c_j$
at the transformation next to the last line and the standard
rules of manipulations of the power series within radius of convergence.
\end{proof}

\begin{proof}[Proof of Corollary~\ref{cor:properties:of:B}]
To prove~\eqref{eq:sum:B:j:over:2:power:j} let
$x=\frac{1}{2}$ in~\eqref{eq:sum:B:j:x:power:j} and note
that
\begin{equation}
\label{eq:Gamma:at:3:2}
\Gamma\left(\frac{3}{2}\right)
=\frac{\sqrt{\pi}}{2}\,.
\end{equation}

Recall that zeroes of $\Gamma(z)$ are located at negative
integers. Thus, the analytic functions on the right
of~\eqref{eq:sum:A:j:x:power:j} and of
of~\eqref{eq:sum:B:j:x:power:j} have radius of convergence
$2$ which implies~ \eqref{eq:A:B:j:tend:fast:to:zero}.

Taking derivatives of~\eqref{eq:sum:A:j:x:power:j}
and of~\eqref{eq:sum:B:j:x:power:j}
we get~\eqref{eq:sum:j:A:j:x:power:j}
and~\eqref{eq:sum:j:B:j:x:power:j} respectively.

To prove~\eqref{eq:sum:j:A:j:over:2:power:j}
and~\eqref{eq:sum:j:B:j:over:2:power:j}
it remains
to evaluate respectively~\eqref{eq:sum:j:B:j:x:power:j}
and~\eqref{eq:sum:j:B:j:x:power:j}
at $x=\frac{1}{2}$.
We already have the necessary value~\eqref{eq:Gamma:at:3:2}
of the $\Gamma$ function. The value of the digamma function
at a half-integer point can be evaluated using the following
standard identity:
$$
\psi\left(n+\frac{1}{2}\right)=
-\gamma-2\log(2)+\sum_{k=1}^n \frac{2}{2k-1}\,,
$$
so
\begin{equation}
\label{eq:psi:3:2}
\psi\left(\frac{3}{2}\right)=
-\gamma-2\log(2)+2\,.
\end{equation}
\end{proof}

We are now ready to prove the main
Theorem~\ref{th:multiple:harmonic:zeta:sum:full:expansion}
of this Section. The proof follows the technique
of~\cite{Flajolet:Sedgewick}. In particular,
relation~\eqref{eq:multiple:harmonic:sum:full:expansion}
from
Theorem~\ref{th:multiple:harmonic:zeta:sum:full:expansion}
is analogous to relation (21) in Theorem~VI.2
in~\cite{Flajolet:Sedgewick}, where the notations
should be substituted in the following way to be translated
to the notations of
Theorem~\ref{th:multiple:harmonic:zeta:sum:full:expansion}:
$$
\alpha:=1\,,\quad
\beta:=k-1\,,\quad
n:=m\,,\quad
k:=j\,.
$$

\begin{proof}[Proof of Theorem~\ref{th:multiple:harmonic:zeta:sum:full:expansion}]
We start by proving~\eqref{eq:multiple:harmonic:sum:full:expansion}.
Then, having made several technical adjustments
we prove~\eqref{eq:multiple:harmonic:zeta:sum:full:expansion}
which is slightly more technical.

Let $H_k(m)$ be as
in~\eqref{eq:multiple:harmonic:sum:full:expansion}.
Define
$$
h_k(m)=m\cdot H_k(m)\,.
$$
For each $k\in\N$ we are looking for constants
$\alpha_0,\alpha_1,\alpha_{k-1}$ such that
\begin{multline}
\label{eq:def:of:asymptotic:expansion}
h_k(m)=\\
=\alpha_0\cdot\big(\log(m)\big)^{k-1}
+\alpha_1\cdot\big(\log(m)\big)^{k-2}+\dots
+\alpha_{k-1} +o(1)\quad \text{as }m\to+\infty\,.
\end{multline}

Let
$$
F(t)=\sum_{m=1}^{+\infty}\frac{t^m}{m}=-\log(1-t)\,.
$$
Clearly,
$$
\sum_{m=k}^{+\infty} H_k(m)\cdot t^m
=\big(F(t)\big)^k
$$
Hence,
\begin{multline*}
\sum_{m=k}^{+\infty} h_k(m)\cdot t^{m-1}
=\sum_{m=k}^{+\infty} m\cdot H_k(m)\cdot t^{m-1}
=\\
=\Big(\big(F(t)\big)^k\Big)'
=(-1)^{k-1}\cdot k\cdot\frac{\big(\log(1-t)\big)^{k-1}}{1-t}\,.
\end{multline*}
Applying the substitution $t=1-\tfrac{1}{N}$
to the latter equality we get
\begin{equation}
\label{eq:Riemann:sum:init}
\sum_{m=k}^{+\infty} h_k(m)
\cdot\left(1-\tfrac{1}{N}\right)^{m-1}
=N\cdot k\cdot \big(\log(N)\big)^{k-1}\,.
\end{equation}
We introduce now a new parameter $u$ such that
$m=u\cdot N$. Note that
\begin{equation}
\label{eq:limit:e}
\left(1-\tfrac{1}{N}\right)^{m-1}
=\left(1-\tfrac{1}{N}\right)^{u\cdot N-1}
\sim e^{-u}\quad\text{as }N\to+\infty\,.
\end{equation}
Dividing both sides of
equation~\eqref{eq:Riemann:sum:init} by $N$ we get
on the left hand side the Riemann sum for $h_k$ interpreted
as a function of continuous parameter defined by the
right hand side of the asymptotic
expansion~\eqref{eq:def:of:asymptotic:expansion}:
$$
\frac{1}{N}\cdot
\sum_{m=k}^{+\infty} h_k(m)
\cdot\left(1-\tfrac{1}{N}\right)^{m-1}
\sim
\int_0^{+\infty} h_k(uN)\cdot e^{-u} \,du\quad
\text{as }N\to+\infty\,.
$$
Thus, equation~\eqref{eq:Riemann:sum:init} implies the
following key equality of the integral of the
expansion~\eqref{eq:def:of:asymptotic:expansion}: weighted
by $e^{-u}$ and the term on the right hand side
of~\eqref{eq:Riemann:sum:init} divided ny $N$:
\begin{equation}
\label{eq:integral:equation:on:asymptotic:expansion:A}
\int_0^{+\infty} h_k(uN)\cdot e^{-u} \,du
=k\cdot \big(\log(N)\big)^{k-1}
\quad\text{as }N\to+\infty\,,
\end{equation}
where $h_k(uN)$ should be interpreted as the right hand
side of~\eqref{eq:def:of:asymptotic:expansion}.
Applying~\eqref{eq:def:of:asymptotic:expansion} we get
the following relation:
\begin{multline*}
\int_0^{+\infty}
e^{-u}\cdot
\Big(
\alpha_0\cdot\big(\log(u)+\log(N)\big)^{k-1}
+\alpha_1\cdot\big(\log(u)+\log(N)\big)^{k-2}+\dots
\\
\dots +
\alpha_{k-1}\Big) \,du
=k\cdot \big(\log(N)\big)^{k-1}
 +o(1)
\quad\text{as }N\to+\infty\,.
\end{multline*}

Dividing both sides of the latter relation by $k$,
applying notation~\eqref{eq:c:j:as:integral}
and passing from $\alpha_j$ to $A_j$ defined as
$$
A_j=\frac{\alpha_j}{k\cdot(k-1)\cdots(k-j)}
$$
we can rewrite the latter equation as
\begin{multline*}
A_0\cdot c_0\cdot\big(\log(N)\big)^{k-1}
+\big(A_1\cdot c_0 + A_0\cdot c_1\big)\cdot\big(\log(N)\big)^{k-2}
+\\
+\big(A_2\cdot c_0 + A_1\cdot c_1 + A_0\cdot c_2\big)\cdot\big(\log(N)\big)^{k-2}
+\dots
=\big(\log(N)\big)^{k-1}
 +o(1)\,.
\end{multline*}
Recall that $c_0=\int_0^{+\infty} e^{-u}\,du=1$. Thus, the
latter relation implies that $A_0=1$, and that
$A_1,\dots,A_k$ satisfy the system of
equations~\eqref{eq:def:A:j} which recursively defines
all of them. This completes the proof
of~\eqref{eq:multiple:harmonic:sum:full:expansion}.

The scheme of the proof
of~\eqref{eq:multiple:harmonic:zeta:sum:full:expansion} is
completely analogous. Let $Z_k(m)$ be as
in~\eqref{eq:multiple:harmonic:zeta:sum:full:expansion}.
Define
$$
z_k(m)=m\cdot Z_k(m)\,.
$$
For each $k\in\N$ we are looking for constants
$\beta_0,\beta_1,\beta_{k-1}$ such that
\begin{multline}
\label{eq:def:of:asymptotic:expansion:Z}
z_k(m)=\\
=\beta_0\cdot\big(\log(m)\big)^{k-1}
+\beta_1\cdot\big(\log(m)\big)^{k-2}+\dots
+\beta_{k-1} +o(1)\quad \text{as }m\to+\infty\,.
\end{multline}
Let
\begin{equation}
\label{eq:Phi}
\Phi(t)
=\log\big(\Gamma(1-\sqrt{t})\big)
+\log\big(\Gamma(1-\sqrt{t})\big)
\,.
\end{equation}
\begin{Lemma}
Function $\Phi(t)$ defined in~\eqref{eq:Phi}
has the following analytic expansion:
\begin{equation}
\label{eq:Phi:analytic:expansion}
\Phi(t)=\sum_{m=1}^{+\infty}\frac{\zeta(2m)}{m}\cdot t^m\,.
\end{equation}
\end{Lemma}
\begin{proof}
By Equation~(5.7.3) in~\cite{DLMF} we have
$$
\log\Gamma(1+z)=-\log(1+z)+z\cdot(1-\gamma)
+\sum_{n=2}^{+\infty} \big(\zeta(n)-1\big)\cdot\frac{z^k}{k}\,,
$$
which implies the following asymptotic expansion:
$$
\log\Gamma(1+z)=
-z\cdot\gamma
+\sum_{n=2}^{+\infty} (-1)^n \zeta(n)\cdot\frac{z^n}{n}\,.
$$
Hence,
\begin{multline*}
\log\big(\Gamma(1-\sqrt{t})\big)
+\log\big(\Gamma(1+\sqrt{t})\big)
=2\cdot\sum_{m=1}^{+\infty}
\frac{\zeta(2m)}{2m}\left(\sqrt{t}\right)^{2m}
=\sum_{m=1}^{+\infty}
\frac{\zeta(2m)}{m}\cdot t^m\,.
\end{multline*}
\end{proof}

Clearly,
$$
\sum_{m=k}^{+\infty} Z_k(m)\cdot t^m
=\big(\Phi(t)\big)^k
$$
Hence,
\begin{multline}
\label{eq:generating:series}
\sum_{m=k}^{+\infty} z_k(m)\cdot t^{m-1}
=\sum_{m=k}^{+\infty} m\cdot Z_k(m)\cdot t^{m-1}
=\\
=\Big(\big(\Phi(t)\big)^k\Big)'
=k\cdot\Phi^{k-1}(t)
\cdot\frac{\psi(1+\sqrt{t})-\psi(1-\sqrt{t})}{2\sqrt{t}}\,.
\end{multline}

\begin{Lemma}
The following expansions are valid as $x\to 0_+$:
\begin{align}
\label{eq:expansion:Phi}
\Phi(1-x)
=\log\big(\Gamma(1+\sqrt{1-x})\big)
&+\log\big(\Gamma(1-\sqrt{1-x})\big)=
\\ \notag
&=-\log(x)+\log(2)-\frac{3}{4}\cdot x+o(x)\,.
\\
\label{eq:expansion:psi}
\frac{\psi(1+\sqrt{1-x})-\psi(1-\sqrt{1-x})}{2\sqrt{1-x}}&
=\frac{1}{x}+\frac{3}{4}
+\left(\frac{11}{16}-\frac{\zeta(2)}{2}\right)\cdot x+o(x)\,.
\end{align}
\end{Lemma}
\begin{proof}
We have
$$
1+\sqrt{1-x}=2-\frac{1}{2}\cdot x+o(x)\,.
$$
and
$$
\Gamma(2)=1\,,\qquad \Gamma'(2)=1-\gamma\,,
$$
so
$$
\Gamma\big(1+\sqrt{1-x}\big)=1-\frac{(1-\gamma)}{2}\cdot x + o(x)
$$
and
\begin{equation}
\label{eq:log:Gamma:plus}
\log\Gamma\big(1+\sqrt{1-x}\big)=
-\frac{(1-\gamma)}{2}\cdot x + o(x)\,.
\end{equation}

We also have
$$
\Gamma(z)=\frac{1}{z}-\gamma
+\frac{1}{2}\big(\gamma^2+\zeta(2)\big)\cdot z +o(z)\,,
$$
so
$$
\Gamma\big(1-\sqrt{1-x}\big)=\frac{2}{x}\cdot
\left(1-\frac{2\gamma+1}{4}\cdot x +o(x)\right)\,,
$$
and hence
\begin{equation}
\label{eq:log:Gamma:minus}
\log\Gamma\big(1-\sqrt{1-x}\big)=
\log(2)-\log(x)-\frac{2\gamma+1}{4}\cdot x +o(x)\,.
\end{equation}
Taking the sum of~\eqref{eq:log:Gamma:plus}
and~\eqref{eq:log:Gamma:minus} we
get~\eqref{eq:expansion:Phi}.

Using the following formulae:
\begin{align*}
\psi(z+1)&=\psi(z)+\frac{1}{z}\,,
\\
\psi(z+1)&=-\gamma+\sum_{m=2}^{+\infty}(-1)^m\zeta(m) z^{m-1}
\quad\text{for }|z|<1
\,.
\end{align*}
(see~\cite[(5.7.4)]{DLMF}) we get:
\begin{multline}
\label{eq:psi:plus}
\psi\big(1+\sqrt{1-x}\big)=
\psi\big(\sqrt{1-x}\big)+\frac{1}{\sqrt{1-x}}
=\\=
\psi\left(1-\frac{x}{2}+o(x)\right)
+\left(1-\frac{x}{2}+o(x)\right)^{-1}
=(1-\gamma)+\frac{1-\zeta(2)}{2}\cdot x +o(x)\,.
\end{multline}

We also get
\begin{multline}
\label{eq:psi:minus}
\psi\big(1-\sqrt{1-x}\big)=
\psi\big(2-\sqrt{1-x}\big)-\frac{1}{1-\sqrt{1-x}}
=\\=
\psi\left(2-\left(1-\frac{x}{2}+o(x)\right)\right)
+\frac{2}{x}\cdot\left(1-\frac{x}{4}-\frac{x^2}{8}+o(x^2)\right)^{-1}
=\\=
-\frac{2}{x}
+\left(\frac{1}{2}-\gamma\right)
+\frac{4\zeta(2)+1}{8}\cdot x +o(x)\,.
\end{multline}
and
$$
\frac{1}{2\sqrt{1-x}}
=\frac{1}{2}\cdot\left(
1+\frac{1}{2}\cdot x+\frac{3}{8}\cdot x^2+o(x^2)\right)\,.
$$
Multiplying the sum of~\eqref{eq:psi:minus}
and~\eqref{eq:psi:minus} by the latter expression we
get~\eqref{eq:expansion:psi}.
\end{proof}

Applying the substitution $t=1-\tfrac{1}{N}$
to~\eqref{eq:generating:series} and applying~\eqref{eq:expansion:Phi}
and~\eqref{eq:expansion:psi} to the right hand side
of the resulting expression on the right hand side
we get
\begin{multline}
\label{eq:Riemann:sum:for:Z}
\sum_{m=k}^{+\infty} z_k(m)
\cdot\left(1-\tfrac{1}{N}\right)^{m-1}
=\\=
N\cdot k\cdot
\left(\log(N)+\log(2)-\frac{3}{4}\cdot\frac{1}{N}
+o\left(\frac{1}{N}\right)\right)\cdot
\left(1+\frac{3}{4}\cdot\frac{1}{N}
+o\left(\frac{1}{N}\right)\right)\,.
\end{multline}
We introduce now parameter $u$ such that
$m=u\cdot N$ and note~\eqref{eq:limit:e}.
Dividing both sides of
equation~\eqref{eq:Riemann:sum:for:Z} by $N$ we get
on the left hand side the Riemann sum for $z_k$ interpreted
as a function of continuous parameter defined by the
right hand side of the asymptotic
expansion~\eqref{eq:def:of:asymptotic:expansion:Z}:
$$
\frac{1}{N}\cdot
\sum_{m=k}^{+\infty} z_k(m)
\cdot\left(1-\tfrac{1}{N}\right)^{m-1}
\sim
\int_0^{+\infty} z_k(uN)\cdot e^{-u} \,du\quad
\text{as }N\to+\infty\,.
$$
Thus, equation~\eqref{eq:Riemann:sum:for:Z} implies the
following key equality of the integral of the
expansion~\eqref{eq:def:of:asymptotic:expansion:Z}: weighted
by $e^{-u}$ and the term on the right hand side
of~\eqref{eq:Riemann:sum:for:Z} divided by $N$:
\begin{multline}
\label{eq:integral:equation:on:asymptotic:expansion:B}
\int_0^{+\infty} z_k(uN)\cdot e^{-u} \,du
=\\
=k\cdot\left(\log(N)+\log(2)-\frac{3}{4}\frac{1}{N}
+o\left(\frac{1}{N}\right)\right)^{k-1}
\cdot\left(1+\frac{3}{4}\frac{1}{N}
+o\left(\frac{1}{N}\right)\right)
 +o(1)\,.
\end{multline}
where $z_k(uN)$ should be interpreted as the right hand
side of~\eqref{eq:def:of:asymptotic:expansion}.
Applying~\eqref{eq:def:of:asymptotic:expansion:Z} we get
the following relation:
\begin{multline*}
\int_0^{+\infty}
e^{-u}\cdot
\Big(
\beta_0\cdot\big(\log(u)+\log(N)\big)^{k-1}
+\beta_1\cdot\big(\log(u)+\log(N)\big)^{k-2}+\dots
+\beta_{k-1}\Big) \,du
=\\
=k\cdot\big(\log(N)+\log(2)\big)^{k-1}
+o(1)\,.
\end{multline*}
Identifying the terms containing the same powers
of $\log(N)$ on both sides of the above equality
we get the following system of equations:
\begin{align*}
\left(\beta_0
\cdot\int_0^{+\infty} e^{-u}\,du
\right)
\cdot\big(\log(N)\big)^{k-1}
&=k\cdot\big(\log(N)\big)^{k-1}
\\
\Bigg(\beta_0\cdot(k-1)
\cdot\int_0^{+\infty}\log(u)\cdot e^{-u}\,du
&+
\beta_1
\cdot\int_0^{+\infty} e^{-u}\,du
\Bigg)
\cdot\big(\log(N)\big)^{k-2}
=\\
&=k(k-1)\log(2)
\cdot\big(\log(N)\big)^{k-2}\,,
\end{align*}
etc. The system of equations above should be considered as a
system of equations on coefficients in front of
$\big(\log(N)\big)^{k-1}, \cdot\big(\log(N)\big)^{k-2}, \dots$.

Dividing both sides of the equation corresponding to the
term $\big(\log(N)\big)^{k-1-j}$ by
${k\cdot(k-1)\cdots(k-j)}$, passing to
notation~\eqref{eq:c:j:as:integral} for the integrals and
passing from $\beta_j$ to $B_j$ defined as
$
B_j=\cfrac{\beta_j}{k\cdot(k-1)\cdots(k-j)}
$
we obtain system of equations~\eqref{eq:def:B:j}. Recall
that this system of equations uniquely defines
$B_1,\dots,B_k$ by recursion starting from $B_0=1$. This
completes the proof
of~\eqref{eq:multiple:harmonic:sum:full:expansion}.
\end{proof}

\begin{proof}[Proof of
Conditional~Theorem~\ref{cond:th:volume:coefficient:total:sum}
]
We prove relation~\eqref{eq:volume:coefficient:Z};
the proof of~\eqref{eq:volume:coefficient:H} is
completely analogous.

Suppose that
Conjecture~\ref{conj:uniform:bound:for:error:term}
is valid for some $C\in]\tfrac{1}{2};2[$.
By~\eqref{eq:multiple:harmonic:zeta:sum:full:expansion:log}
for any $K\le C\log(m)$ we have:
\begin{multline}
\label{eq:rearranging:sum}
\sqrt{m}\cdot \sum_{k=1}^K
\frac{1}{k!}\cdot\frac{1}{2^{k-1}}\cdot Z_k(m)
=
\frac{1}{\sqrt{m}}\cdot
\sum_{k=1}^K
\frac{1}{k!}\cdot\frac{1}{2^{k-1}}
\cdot\\ \cdot
\Big(k\cdot B_0\cdot(\log m)^{k-1}+k(k-1)\cdot B_1\cdot(\log m)^{k-2}
+\dots+k!\cdot B_{k-1}+\epsilon_k^Z(m)\Big)
=\\=
\frac{1}{\sqrt{m}}\cdot
\Bigg(
B_0\cdot\sum_{k=1}^K
\frac{1}{(k-1)!}\cdot\left(\frac{\log m}{2}\right)^{k-1}
+
\frac{B_1}{2}\cdot\sum_{k=2}^K
\frac{1}{(k-2)!}\cdot\left(\frac{\log m}{2}\right)^{k-2}
+\\+\dots
+\frac{B_{K-1}}{2^{K-1}}+\sum_{k=1}^K\frac{\epsilon_k^Z(m)}{k!\cdot 2^{k-1}}
\Bigg)
=\\=
\frac{1}{\sqrt{m}}\cdot
\Bigg(
B_0\cdot\sum_{j=0}^{K-1}
\frac{1}{j!}\cdot\left(\frac{\log m}{2}\right)^j
+
\frac{B_1}{2}\cdot\sum_{j=0}^{K-2}
\frac{1}{j!}\cdot\left(\frac{\log m}{2}\right)^j
+\\+
\frac{B_2}{2^2}\cdot\sum_{j=0}^{K-3}
\frac{1}{j!}\cdot\left(\frac{\log m}{2}\right)^j
+\dots+\frac{B_{K-1}}{2^{K-1}}
+\sum_{k=1}^K\frac{\epsilon_k^Z(m)}{k!\cdot 2^{k-1}}
\Bigg)\,.
\end{multline}

It is easy to see that for any $C> \tfrac{1}{2}$
and for any fixed $n\in\N$ we have
$$
\lim_{m\to+\infty}
\frac{1}{\sqrt{m}}
\cdot\sum_{j=0}^{[C\log(m)]-n}
\frac{1}{j!}\cdot\left(\frac{\log m}{2}\right)^j
=
\lim_{m\to+\infty}
\frac{1}{\sqrt{m}}
\cdot
\exp\left(\frac{\log m}{2}\right)
=1\,.
$$
We now let $K=[C\log(m)]$ in~\eqref{eq:rearranging:sum}.
For any fixed $n$ we split the latter expression
in~\eqref{eq:rearranging:sum} into the sum of the first $n$
terms and the sum of the remaining terms.
\begin{multline}
\label{eq:claculation:of:total:sum}
B_0\cdot\Bigg(\frac{1}{\sqrt{m}}\cdot\sum_{j=0}^{K-1}
\frac{1}{j!}\cdot\left(\frac{\log m}{2}\right)^j\Bigg)
+\\+
\frac{B_1}{2}
\cdot\Bigg(\frac{1}{\sqrt{m}}\cdot\sum_{j=0}^{K-2}
\frac{1}{j!}\cdot\left(\frac{\log m}{2}\right)^j\Bigg)
+ \dots +
\frac{B_{n-1}}{2^{n-1}}
\cdot\Bigg(\frac{1}{\sqrt{m}}\cdot\sum_{j=0}^{K-n}
\frac{1}{j!}\cdot\left(\frac{\log m}{2}\right)^j\Bigg)
+\\
+\frac{B_{n}}{2^{n}}
\cdot\Bigg(\frac{1}{\sqrt{m}}\cdot\sum_{j=0}^{K-n-1}
\frac{1}{j!}\cdot\left(\frac{\log m}{2}\right)^j\Bigg)
+\dots+
\frac{1}{\sqrt{m}}\cdot\frac{B_{K-1}}{2^{K-1}}
+\frac{1}{\sqrt{m}}\cdot
\sum_{k=1}^K\frac{\epsilon_k^Z(m)}{k!\cdot 2^{k-1}}
\Bigg)\,.
\end{multline}

For any fixed $n$ the sum of the first $n$ terms tends to
$$
B_0+\frac{B_1}{2}+\dots+\frac{B_{n-1}}{2^{n-1}}
$$
as $m\to+\infty$. Choosing $n$ sufficiently large
we can make this sum arbitrary close to the
sum~\eqref{eq:sum:B:j:over:2:power:j}.

For any $m,L\in\N$
$$
0\le
\sum_{j=0}^{L}
\frac{1}{j!}\cdot\left(\frac{\log m}{2}\right)^j
\le \sqrt{m}\,,
$$
so
$$
\frac{1}{\sqrt{m}}\cdot
\left|\frac{B_n}{2^n}\cdot\sum_{j=0}^{K-(n+1)}
\frac{1}{j!}\cdot\left(\frac{\log m}{2}\right)^j
+\dots+\frac{B_{K-1}}{2^{K-1}}
\right|
\le
\sum_{j=n}^{+\infty}
\frac{|B_n|}{2^n}
\,.
$$
Relation~\eqref{eq:A:B:j:tend:fast:to:zero} thus implies
that choosing $n$ large enough we can make the second part
of the sum~\eqref{eq:claculation:of:total:sum} get
arbitrary close to zero for all sufficiently large $m$.

It remains to note that by~\eqref{eq:max:epsilon:H}
from Conjecture~\ref{conj:uniform:bound:for:error:term}
we have
$$
\lim_{m\to+\infty}
\frac{1}{\sqrt{m}}\cdot
\sum_{k=1}^{[C\log(m)]}\frac{\epsilon_k^Z(m)}{k!\cdot 2^{k-1}}=0\,.
$$
\end{proof}

\begin{proof}[Proof of Theorem~\ref{th:multiple:harmonic:sum}]
We prove below relation~\eqref{eq:multiple:harmonic:zeta:sum};
the proof of~\eqref{eq:multiple:harmonic:sum}
is completely analogous.

Using explicit expressions~\eqref{eq:B:0}
and~\eqref{eq:B:1} for $B_0$ and $B_1$ respectively we get
the two terms of highest degrees in the right-hand side
of~\eqref{eq:multiple:harmonic:zeta:sum:full:expansion}.
Rewriting them as below we
obtain~\eqref{eq:multiple:harmonic:zeta:sum}:
\begin{multline*}
\frac{k!}{m}\cdot\left(
\frac{1}{(k-1)!}\!\cdot\!\big(\log(m)\big)^{k-1}
\!+\!\frac{(\log(2)+\gamma)}{(k-2)!}\cdot\big(\log(m)\big)^{k-2}
\!+\!O\!\left(\big(\log(m)\big)^{k-3}\right)\right)\!=
\\
\frac{k}{m}\cdot\left(
\big(\log(m)\big)^{k-1}
+(k-1)\cdot(\log(2)+\gamma)\cdot\big(\log(m)\big)^{k-2}
+O\left(\big(\log(m)\big)^{k-3}\right)\right)=
\\
=\frac{k}{m}\cdot\left(\big(\log(m)+(\log(2)+\gamma)\big)^{k-1}
+O\left(\big(\log(m)\big)^{k-3}\right)\right)=
\\
=\frac{k}{m}\cdot\left(\big(\log(2m)+\gamma\big)^{k-1}
+O\left(\big(\log(m)\big)^{k-3}\right)\right)\,.
\end{multline*}
\end{proof}


\section{Conjectural asymptotics of normalized correlators}
\label{s:Conjectural:asymptotics:of:correlators}

Let $g$ be a nonnegative integer and $n$ a positive
integer. Let the pair $(g,n)$ be different from $(0,1)$ and
$(0,2)$. Let $d_1,\dots,d_n$ be an ordered partition of $3g
- 3 + n$ into a sum of nonnegative integers,
$|d|=d_1+\dots+d_n=3g-3+n$, let $\boldsymbol{d}$ be a
multiindex $(d_1,\dots,d_n)$. Denote by $\Pi(3g-3+n,n)$
the set of all such partitions.
We use square brackets to denote the following
normalization of correlators in $\psi$-classes:
\begin{multline}
\label{eq:tau:square:brackets}
[\tau_{d_1}\dots\tau_{d_n}]_{g,n}=
\left(\prod_{i=1}^n 2^{2d_i}(2d_i+1)!!\right)\,
\langle \tau_{d_1} \dots \tau_{d_n}\rangle_{g,n}
=\\=
2^{3g-3+n}\cdot
\frac{(2d_1+1)!\cdots(2d_n+1)!}{d_1!\dots d_n!}
\int_{\overline{\cM}_{g,n}} \psi_1^{d_1}\cdots\psi_n^{d_n}\,.
\end{multline}

The following correlators admit close expression:
\begin{align}
\label{eq:max:bracket:tau}
[\tau_0^{n-1}\tau_{3g-3+n}]_{g,n}
&=2^{6g-6+2n}(6g-5+2n)!!
\int_{\overline{\cM}_{g,n}} \psi_n^{3g-3+n}
=\\
\notag
&=2^{6g-6+2n}(6g-5+2n)!!
\cdot\frac{1}{24^g\cdot g!}
=\\
\notag
&=\frac{(6g-5+2n)!}{(3g-3+n)!}
\cdot\frac{2^{3g-3+n}}{24^g\cdot g!}
\,;
\\
\label{eq:other:bracket:tau}
[\tau_1^{n-1}\tau_{3g-2}]_{g,n}
&=2^{6g-6+2n}\cdot 3^{n-1}\cdot(6g-3)!!
\int_{\overline{\cM}_{g,n}} \psi_1\dots\psi_{n-1}\psi_n^{3g-2}
=\\
\notag
&=2^{6g-6+2n}\cdot 3^{n-1}\cdot(6g-3)!!
\cdot\frac{(2g-3+n)!}{24^g\cdot g!(2g-2)!}
=\\
\notag
&=6^{n-1}\cdot\frac{(6g-3)!}{(3g-2)!}
\cdot\frac{(2g-3+n)!}{(2g-2)!}
\cdot\frac{2^{3g-3+n}}{24^g\cdot g!}
\,.
\end{align}

Let $\textbf{D}=(D_1,\dots,D_k)\in\Pi(3g-3+2k,k)$.
Define $c_{g,k}(\boldsymbol{D})$ as
\begin{multline}
\label{eq:sum:of:normalized:correlators}
c_{g,k}(\boldsymbol{D})=
\frac{g!\cdot(3g-3+2k)!}{(6g+4k-5)!}
\cdot\frac{3^g}{2^{3g-6+5k}}\cdot
\\
\cdot
\sum_{d_{1,1}+d_{1,2}=D_1}\dots\sum_{d_{k,1}+d_{k,2}=D_k}
\int_{\overline{\cM}_{g,2k}}
\psi_1^{d_{1,1}}\psi_2^{d_{1,2}}
\dots \psi_{2k-1}^{d_{k,1}}\psi_{2k}^{d_{k,2}}
\cdot\prod_{j=1}^k \cfrac{(2D_j+2)!}{d_{j,1}!\cdot d_{j,2}!}\,,
\end{multline}
which in notations~\eqref{eq:tau:square:brackets} is
equivalent to
\begin{multline}
\label{eq:sum:of:normalized:sq:correlators}
c_{g,k}(\boldsymbol{D})=
\frac{g!\cdot(3g-3+2k)!}{(6g+4k-5)!}
\cdot\frac{3^g}{2^{6g-9+7k}}\cdot
\\
\cdot
\sum_{d_{1,1}+d_{1,2}=D_1}\dots\sum_{d_{k,1}+d_{k,2}=D_k}
[\tau_{d_{1,1}}\tau_{d_{1,2}}
\dots \tau_{d_{k,1}}\tau_{d_{k,2}}]_{g,2k}
\cdot\prod_{j=1}^k \binom{2D_j+2}{2d_{j,1}+1}\,.
\end{multline}

We now state several conjectures on asymptotics of normalized
correlators in $\psi$-classes.
For any fixed pair $g,k\in\N$ denote
\begin{align}
\label{eq:c:max}
c^{max}_{g,k}
=&\max_{\boldsymbol{D}\in\Pi(3g-3+2k,k)}
c_{g,k}(\boldsymbol{D})\,,
\\
\label{eq:c:min}
c^{min}_{g,k}
=&\min_{\boldsymbol{D}\in\Pi(3g-3+2k,k)}
c_{g,k}(\boldsymbol{D})\,.
\end{align}

\begin{Conjecture}
\label{conj:sum:of:correlators:fixed:k}
For any fixed $k\in\N$
\begin{equation}
\label{eq:sum:of:correlators:fixed:k}
\lim_{g\to+\infty} c^{max}_{g,k}
= \lim_{g\to+\infty} c^{min}_{g,k}
=1\,.
\end{equation}
\end{Conjecture}

The next conjecture is more ambitious. It claims that the
limit exists even if we let $k$ grow, but sufficiently
slowly with respect to $g$.
Conjecture~\ref{conj:sum:of:correlators} is just a
reformulation of
Conjecture~\ref{conj:introduction:sum:of:correlators}.
Clearly, Conjecture~\ref{conj:sum:of:correlators}
implies Conjecture~\ref{conj:sum:of:correlators:fixed:k}.

\begin{Conjecture}
\label{conj:sum:of:correlators}
For any fixed $C<2$ one has
\begin{equation}
\label{eq:sum:of:correlators}
\lim_{g\to+\infty}\
\max_{1\le k\le C\log(g)}
c^{max}_{g,k}
=
\lim_{g\to+\infty}
\min_{1\le k\le C\log(g)}
c^{min}_{g,k}
=1\,.
\end{equation}
\end{Conjecture}


In the conjectures above we considered huge weighted
sum of correlators. The next two Conjectures are even more
ambitious since they treat all individual correlators.
Define $\epsilon(\boldsymbol{d})=\epsilon(d_1,\dots,d_n)$
as
\begin{equation}
\label{eq:epsilon:d}
\epsilon(\boldsymbol{d})
=\frac{[\tau_{d_1}\dots\tau_{d_n}]_{g,n}}
{[\tau_0^{n-1}\tau_{3g-3+n}]_{g,n}}-1\,.
\end{equation}
Define
\begin{equation}
\label{eq:epsilon:max:min}
\epsilon^{max}_{g,n}=\max_{\boldsymbol{d}\in\Pi(3g-3+n,n)}
\epsilon(\boldsymbol{d})\,,
\qquad
\epsilon^{min}_{g,n}=\min_{\boldsymbol{d}\in\Pi(3g-3+n,n)}
\epsilon(\boldsymbol{d})
\end{equation}

\begin{Conjecture}
\label{conj:normalized:correlators}
For any fixed $n\in\N$ one has
\begin{equation}
\label{eq:conj:Mirzakhani:Zograf}
\lim_{g\to+\infty} \epsilon^{max}_{g,n}
= \lim_{g\to+\infty} \epsilon^{min}_{g,n}
= 0\,.
\end{equation}
\end{Conjecture}
\begin{Remark}
The idea of normalization~\eqref{eq:tau:square:brackets}
aiming the above conjecture originates from discussions
of M.~Mirzakhani with P.~Zograf.
\end{Remark}

Conjecture~\ref{conj:sum:of:correlators:fixed:k}
follows from
Conjecture~\ref{conj:normalized:correlators}
(in a way in which we derive
Conjecture~\ref{conj:sum:of:correlators}
from
Conjecture~\ref{conj:normalized:correlators:log:g:strong}
below).

For $n=1$ the conjecture is vacuous: the set $\Pi(3g-2,1)$
contains single element, so
$\epsilon^{max}_{g,1}=\epsilon^{min}_{g,1}=0$.
For $n=2$ Conjecture~\ref{conj:normalized:correlators}
follows from the following Theorem giving explicit values
of $\epsilon^{max}_{g,2}$ and of $\epsilon^{min}_{g,2}$.

\begin{Theorem}
For any $g\in\N$ one has
\begin{equation}
\label{eq:epsilon:max:epsilon:min:n:2}
\epsilon^{max}_{g,2}=0\,,
\qquad
\epsilon^{min}_{g,2}=-\frac{2}{6g-1}\,.
\end{equation}
\end{Theorem}
\begin{proof}
Using explicit expression~\eqref{eq:max:bracket:tau}
we get
$$
[\tau_0\tau_{3g-1}]_{g,2}
=2^{6g-2}(6g-1)!!
\cdot\frac{1}{24^g\cdot g!}\,,
$$
so
$$
\frac{[\tau_{k}\tau_{3g-1-k}]_{g,2}}
{[\tau_0\tau_{3g-1}]_{g,2}}
=
\frac{(2k+1)!!\cdot(6g-1-2k)!!}{(6g-1)!!}
\cdot 24^g\cdot g!
\cdot \langle\tau_k\tau_{3g-1-k}\rangle_g
=a_{g,k}\,,
$$
where $a_{g,k}$ were defined in equation~\eqref{eq:a:g:k}.
Equation~\eqref{eq:epsilon:max:epsilon:min:n:2} now
follows immediately from bounds~\eqref{eq:main:bounds}
from Proposition~\ref{pr:main:bounds}.
\end{proof}

Using explicit expression~\eqref{eq:max:bracket:tau} for
$[\tau_0^{n-1}\tau_{3g-3+n}]_{g,n}$ in denominator of the
ratio in~\eqref{eq:epsilon:d} and the second line in
definition~\eqref{eq:tau:square:brackets} for the numerator
of the ratio in~\eqref{eq:epsilon:d} we can express
correlators in the following form:
\begin{multline}
\label{eq:ansatz}
\frac{\langle \psi_1^{d_1} \dots \psi_n^{d_n}\rangle_{g,n}}
{d_1!\dots d_n!}
=\\=
\frac{(6g-5+2n)!}{(2d_1+1)!\cdots(2d_n+1)!}
\cdot\frac{1}{g!\,(3g-3+n)!}
\cdot\frac{1}{24^g}
\cdot\big(1+\epsilon(\boldsymbol{d})\big)\,.
\end{multline}
or, equivalently,
\begin{equation}
\label{eq:ansatz:sq}
[\tau_{d_1} \dots \tau_{d_n}]_{g,n}
=
\frac{2^{n-3}}{3^g}\cdot
\frac{(6g-5+2n)!}{g!\,(3g-3+n)!}
\cdot\big(1+\epsilon(\boldsymbol{d})\big)\,.
\end{equation}
The above relation combined with
Conjecture~\ref{conj:normalized:correlators} explains the
origin of normalization~\eqref{eq:tau:square:brackets}: the
factor on the right-hand side depends only on $g$ and $n$;
the dependence on the individual partition $\boldsymbol{d}$
is reduced to the term $\epsilon(\boldsymbol{d})$ which
conjecturally tends to zero uniformly for all partitions in
$\Pi(3g-3+n,n)$ as genus grows.

Finally, we have the following the most ambitious conjecture.
For any $C<2$ and any $g\in\N$ define
$$
\delta(C,g)=
\max_{1\le n\le C\log(g)}
\max(|\epsilon^{min}_{g,n}|,|\epsilon^{max}_{g,n}|)\,.
$$

\begin{Conjecture}
\label{conj:normalized:correlators:log:g:strong}
For any $C<2$
\begin{equation}
\label{eq:normalized:correlators:log:g:strong}
\lim_{g\to+\infty}\delta(C,g)=0\,.
\end{equation}
\end{Conjecture}

Note that if we do not impose the condition that $n$ is
allowed to grow only sufficiently slowly with respect to
$g$, then the guess does not work already for the ratio of
correlators~\eqref{eq:max:bracket:tau}
and~\eqref{eq:other:bracket:tau}. However, when $n\le C\log
g$, we get the following asymptotics for this particular
ratio. First note that
\begin{multline*}
\frac{[\tau_1^{n-1}\tau_{3g-2}]_{g,n}}
{[\tau_0^{n-1}\tau_{3g-3+n}]_{g,n}}
=
6^{n-1}\cdot
\frac{(6g-3)!}{(6g-5+2n)!}
\frac{(3g-3+n)!}{(3g-2)!}
\cdot\frac{(2g-3+n)!}{(2g-2)!}
=\\=
\left(
\prod_{i=1}^{2n-2}
\left(
1+\frac{i-3}{6g}
\right)
\right)^{-1}
\cdot
\left(
\prod_{i=1}^{n-1}
\left(
1+\frac{i-2}{3g}
\right)
\right)
\cdot
\left(
\prod_{i=1}^{n-1}
\left(
1+\frac{i-2}{2g}
\right)
\right)\,.
\end{multline*}
Passing to logarithms and assuming that $n\ll g$ we get
$$
\log\frac{[\tau_1^{n-1}\tau_{3g-2}]_{g,n}}
{[\tau_0^{n-1}\tau_{3g-3+n}]_{g,n}}
\sim\frac{n^2}{12g}\to 0\quad\text{ as } g\to+\infty\,.
$$

We have seen that Conjecture~\ref{conj:sum:of:correlators}
implies Conjecture~\ref{conj:sum:of:correlators:fixed:k}.
Clearly,
Conjecture~\ref{conj:normalized:correlators:log:g:strong}
implies Conjecture~\ref{conj:normalized:correlators}. We
show now how Conjecture~\ref{conj:sum:of:correlators} (and,
thus, equivalent
Conjecture~\ref{conj:introduction:sum:of:correlators})
follows from
Conjecture~\ref{conj:normalized:correlators:log:g:strong}.
Thus, Conjecture~\ref{conj:normalized:correlators:log:g:strong}
is our strongest conjecture on asymptotics of correlators:
it implies all other ones.

\begin{proof}
We prove how
Conjecture~\ref{conj:sum:of:correlators}
(and, hence, the equivalent
Conjecture~\ref{conj:introduction:sum:of:correlators})
follow from
Conjecture~\ref{conj:normalized:correlators:log:g:strong}.
We prove the upper bound; the lower bound is completely
analogous. Recall the following combinatorial identity:
$$
\sum_{m=0}^{n-1}\binom{2n}{2m+1}=2^{2n-1}\,.
$$

Fix $C<2$.
Combining~\eqref{eq:sum:of:normalized:sq:correlators},
with~\eqref{eq:ansatz:sq}, applying the upper
bound~\eqref{eq:normalized:correlators:log:g:strong} and
using the combinatorial identity as above we get
\begin{multline*}
\limsup_{g\to+\infty}
\max_{\substack{1\le k\le C\log(g)\\
      \boldsymbol{D}\in\Pi(3g-3+2k,k)}}
c_{g,k}(\boldsymbol{D})
=\\=
\limsup_{g\to+\infty}
\max_{\substack{1\le k\le C\log(g)\\
      \boldsymbol{D}\in\Pi(3g-3+2k,k)}}
\frac{g!\cdot(3g-3+2k)!}{(6g+4k-5)!}
\cdot\frac{3^g}{2^{6g-9+7k}}\cdot
\\
\cdot
\sum_{d_{1,1}+d_{1,2}=D_1}\dots\sum_{d_{k,1}+d_{k,2}=D_k}
[\tau_{d_{1,1}}\tau_{d_{1,2}}
\dots \tau_{d_{k,1}}\tau_{d_{k,2}}]_{g,2k}
\cdot\prod_{j=1}^k \binom{2D_j+2}{2d_{j,1}+1}
\le\\\le
\limsup_{g\to+\infty}
\max_{1\le k\le C\log(g)}
\left(
\frac{g!\,(3g-3+2k)!}{(6g+4k-5)!}
\cdot\frac{3^g}{2^{6g-9+7k}}
\right)\cdot\left(
\frac{2^{2k-3}}{3^g}\cdot
\frac{(6g-5+4k)!}{g!\,(3g-3+2k)!}
\right)
\\ \cdot
\prod_{j=1}^k
\left(\sum_{d_{j,1}=0}^{D_j} \binom{2D_j+2}{2d_{j,1}+1}\right)
\cdot
\limsup_{g\to+\infty}
\big(1+\delta(C,g)\big)
=\\=
\limsup_{g\to+\infty}
\max_{1\le k\le C\log(g)}
\frac{1}{2^{6g-6+5k}}\cdot
\left(\prod_{j=1}^k 2^{2D_j+1}\right)\cdot 1
= 1\,,
\end{multline*}
where in the last line we used the relation
$$
\sum_{j=1}^k (2D_j+1) = 6g-6+5k\,.
$$
\end{proof}


\section{Asymptotic volume contribution of square-tiled
surfaces with single critical layer}
\label{s:volume:contribution:from:graphs:with:single:vertex}

In Section~\ref{ss:Vol:Graph:k:g} we compute large genus
asymptotics for the
volume contribution $\Vol(\Petal_k(g))$ of the
stable graph $\Petal_k(g)\in\cG_g$, having single vertex
and $k$ loops, to the Masur--Veech volume $\Vol\cQ_g$. Our
computation is conditional to
Conjecture~\ref{conj:sum:of:correlators:fixed:k} on
asymptotics of the associated sums of $2k$-correlators.
Recall, that
Conjecture~\ref{conj:sum:of:correlators:fixed:k} is an
implication of the stronger
Conjecture~\ref{conj:sum:of:correlators}; the latter
conjecture is a simple reformulation of
Conjecture~\ref{conj:introduction:sum:of:correlators}.

In Section~\ref{ss:sum:over:single:vertex:graphs} we
compute the large genus asymptotics for the sum of volume
contributions $\sum_{k=1}^{[C\log(g)]} \Vol(\Petal_k(g))$
of such graphs, conditionally to
Conjecture~\ref{conj:introduction:sum:of:correlators}
(equivalently, to
Conjecture~\ref{conj:sum:of:correlators}), and to
Conjecture~\ref{conj:uniform:bound:for:error:term}. Here we
use any $C$ in the interval $]\tfrac{1}{2};2[$ for which
Conjectures~\ref{conj:introduction:sum:of:correlators}
and~\ref{conj:uniform:bound:for:error:term} are valid. We
will see that the resulting asymptotic value of such sum
coincides with the conjectural asymptotic value $\Vol\cQ_g$
given by Conjecture~\ref{conj:Vol:Qg}.

\subsection{Evaluation of $\Vol\Petal_k(g)$}
\label{ss:Vol:Graph:k:g}

Consider the stable graph $\Petal_k(g)$ having single
vertex, decorated with the label $g-k$, and having $k$
loops, see the left picture in
Figure~\ref{fig:two:non:separating}. This stable graph
represents those multicurves on a surface of genus $g$
without punctures, for which the underlying reduced
multicurve $\gamma_{\mathit{reduced}}$ is composed of $k$
pairwise non-isotopic non-separating simple closed curves.
The square-tiled surfaces associated to this stable graph
have single singular layer and $k$ maximal horizontal
cylinders.

\begin{figure}[htb]
\includegraphics{genus_two_graph_21.eps}
\begin{picture}(0,0)(135,10) 
\put(44,-14){$g-k$}
\put(0,0){$\overbrace{\rule{70pt}{0pt}}^{k\text{ loops}}$}
\put(110,7){$\overbrace{\rule{55pt}{0pt}}^k$}
\put(100,-26){$\underbrace{\rule{175pt}{0pt}}_g$}
\end{picture}
\includegraphics{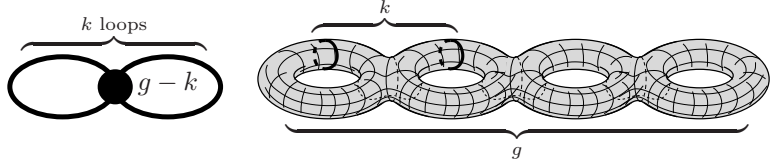}
\vspace*{40pt}
\caption{Graph $\Petal_k(g)$ (on the left)
representing $k$ pairwise non-homotopic
non-separating simple closed curves
on a surface of genus $g$ (on the right).}
\label{fig:two:non:separating}
\end{figure}

Notation $A(g,k) \preceq B(g,k)$ used in this section for
asymptotic inequalities between positive quantities means
that for any $\epsilon>0$ there exists $g_0\in\N$ such that
for all $g> g_0$ one has $A(g,k) \le (1+\epsilon)\cdot
B(g,k)$ for all integer $k$ in the interval $[1;C\log(g)]$,
where $C>0$ is specified separately. Notation $A(g,k)\sim
B(g,k)$ is used when
$\lim_{g\to+\infty}\frac{A(g,k)}{B(g,k)}=1$ uniformly for
all integer $k$ in the interval $[1;C\log(g)]$, where $C>0$
is specified separately.

\begin{Theorem}
\label{th:bounds:for:Vol:Gamma:k:g}
For any $C<2$ the following asymptotic
inequalities are valid (uniformly for
all integer $k$ in the interval $[1;C\log(g)]$)
\begin{multline}
\label{eq:bounds:in:terms:of:Z:k:gminus3}
c^{min}_{g-k,k}\cdot \sqrt{\frac{2}{\pi}}
\cdot\left(\frac{8}{3}\right)^{4g-4}
\cdot\frac{\sqrt{3g-3}}{2^{k-1}\cdot k!}
\cdot Z_k(3g-3)
\preceq \\ \preceq
\Vol(\Petal_k(g))
\preceq
c^{max}_{g-k,k}\cdot \sqrt{\frac{2}{\pi}}
\cdot\left(\frac{8}{3}\right)^{4g-4}
\cdot\frac{\sqrt{3g-3}}{2^{k-1}\cdot k!}
\cdot Z_k(3g-3)
\quad\text{as } g\to+\infty
\,,
\end{multline}
where $c^{max}_{g,k}, c^{min}_{g,k}$ are defined
in~\eqref{eq:c:max}, and~\eqref{eq:c:min} respectively and
$Z_k(3g-3)$ denotes the
sum~\eqref{eq:multiple:harmonic:zeta:sum:full:expansion}.
\end{Theorem}

We first state two Corollaries of this Theorem and
then prove the Theorem.

\begin{Corollary}
\label{cor:Vol:g:k}
Conjecture~\ref{conj:sum:of:correlators:fixed:k}
implies the following asymptotic expansion
for any fixed $k\in\N$
as $g\to+\infty$:
\begin{multline}
\label{eq:Vol:g:k}
\Vol(\Petal_k(g))
\sim
\sqrt{\frac{2}{\pi}}
\cdot\left(\frac{8}{3}\right)^{4g-4}
\cdot\frac{1}{\sqrt{3g-3}}
\,\cdot \\
\cdot\frac{1}{2^{k-1}}
\,\cdot
\left(\sum_{j=0}^{k-1}
\frac{B_j}{(k-1-j)!}
\cdot\big(\log(3g-3)\big)^{k-1-j}\right)
\sim\\ \sim
\sqrt{\frac{2}{\pi}}
\cdot\left(\frac{8}{3}\right)^{4g-4}
\cdot\frac{1}{\sqrt{3g-3}}
\cdot\frac{1}{(k-1)!}
\cdot
\left(\frac{\log(6g-6)+ \gamma}{2}\right)^{k-1}\,.
\end{multline}
\end{Corollary}
\begin{proof}
The first equivalence is a combination of the asymptotic
expansion~\eqref{eq:multiple:harmonic:zeta:sum:full:expansion}
from
Theorem~\ref{th:multiple:harmonic:zeta:sum:full:expansion}
for $Z_k(3g-3)$ with conjectural asymptotic
formula~\eqref{eq:sum:of:correlators:fixed:k} for
$c^{max}_{g,k}$ and $c^{min}_{g,k}$ from
Conjecture~\ref{conj:sum:of:correlators:fixed:k}. The
second equivalence corresponds to asymptotic
formula~\eqref{eq:multiple:harmonic:zeta:sum} from
Theorem~\ref{th:multiple:harmonic:sum}.
\end{proof}

Let $k\le C\log(g)$. Define
\begin{multline}
\label{eq:V:k:g:with:epsilon}
V_k(g)=
\sqrt{\frac{2}{\pi}}
\cdot\left(\frac{8}{3}\right)^{4g-4}
\cdot\frac{\sqrt{3g-3}}{2^{k-1}\cdot k!}
\cdot Z_k(3g-3)
\ =\
\sqrt{\frac{2}{\pi}}
\cdot\left(\frac{8}{3}\right)^{4g-4}
\cdot\\
\cdot\frac{1}{\sqrt{3g-3}}
\cdot\frac{1}{2^{k-1}}
\,\cdot
\left(\sum_{j=0}^{k-1}
\frac{B_j}{(k-1-j)!}
\cdot\big(\log(3g-3)\big)^{k-1-j}
+\epsilon_k^Z(3g-3)
\right)\,,
\end{multline}
where $\epsilon_k^Z(m)$ is defined
in~\eqref{eq:multiple:harmonic:zeta:sum:full:expansion:log}
from Conjecture~\ref{conj:uniform:bound:for:error:term}.

\begin{Corollary}
\label{cor:V:over:Gamma:uniform}
For any $C<2$ for which
Conjecture~\ref{conj:introduction:sum:of:correlators}
is valid we have:
\begin{equation}
\label{eq:Vol:g:k:uniform}
\lim_{g\to+\infty}
\max_{1\le k\le C\log(g)}
\frac{V_k(g)}{\Vol(\Petal_k(g))}
=
\lim_{g\to+\infty}
\min_{1\le k\le C\log(g)}
\frac{V_k(g)}{\Vol(\Petal_k(g))}
=1\,.
\end{equation}
\end{Corollary}
\begin{proof}
Consider expression~\eqref{eq:bounds:in:terms:of:Z:k:gminus3} from
Theorem~\ref{th:bounds:for:Vol:Gamma:k:g}.
Conjecture~\ref{conj:introduction:sum:of:correlators}
claims that $c^{min}_{g-k,k}$ and $c^{max}_{g-k,k}$ tend to
$1$ uniformly for all $k$ in the interval
$\big[1;[C\log(g)]\big]$ when $g\to+\infty$. Applying
expression~\eqref{eq:multiple:harmonic:zeta:sum:full:expansion:log}
for $Z_k(3g-3)$ we get the desired relation.
\end{proof}

\begin{NNRemark}
Note that at this stage we do
not yet make any assertions on asymptotic behavior of
$\epsilon_k^Z(3g-3)$, so at this stage
we do not rely on Conjecture~\ref{conj:uniform:bound:for:error:term}.
\end{NNRemark}

\begin{proof}[Proof of Theorem~\ref{th:bounds:for:Vol:Gamma:k:g}]
The automorphism group $\operatorname{Aut}(\Petal_k(g))$
consists of all possible permutations of loops composed
with all possible flips of individual loops, so
$$
|\operatorname{Aut}(\gamma)|=2^k\cdot k!\,.
$$
The graph $\Petal_k(g)$ has single vertex, so
$|V(\Graph)|=1$. The surface does not have any marked points,
so $n=0$.
Thus, applying~\eqref{eq:volume:contribution:of:stable:graph}
to $\Petal_k(g)$ we get
\begin{multline}
\label{eq:Vol:Gamma:k:g:init}
\Vol(\Petal_k(g))
=\frac{2^{6g-5} \cdot (4g-4)!}{(6g-7)!}
\cdot\\ \cdot
1\cdot\frac{1}{2^k\cdot k!}
\cdot\cZ\big(
b_1 b_2\dots b_k \cdot
N_{g-k,2k}(b_1,b_1,b_2,b_2,\dots,b_k,b_k)
\big)
=\\=
\frac{(4g-4)!}{(6g-7)!}
\cdot \frac{2^{6g-5}}{2^k\cdot k!}
\cdot\frac{1}{2^{5(g-k)-6+4k}}
\cdot\\
\sum_{\boldsymbol{d}\in\Pi(3g-3-k,2k)}
\frac{\langle\psi^{\mathbf{d}}\rangle_{g-k,2k}}{\mathbf{d}!}\,\cdot\,
\prod_{i=1}^k
\Big((2d_{2i-1}+2d_{2i}+1)!\,\cdot\,\zeta(2d_{2i-1}+2d_{2i}+2)\Big)\,.
\end{multline}

Rewrite the latter sum using notations
$\textbf{D}=(D_1,\dots,D_k)\in\Pi(3g-3-k,k)$
and $c_{g-k,k}(\boldsymbol{D})$ defined
by~\eqref{eq:sum:of:normalized:correlators}.
Adjusting
expression~\eqref{eq:sum:of:normalized:correlators}
given for genus $g$ to genus $g-k$
we get
\begin{multline*}
\sum_{\boldsymbol{d}\in\Pi(3g-3-k,2k)}
\frac{\langle\psi^{\mathbf{d}}\rangle_{g-k,2k}}{\mathbf{d}!}\,\cdot\,
\prod_{i=1}^k
\Big((2d_{2i-1}+2d_{2i}+1)!\,\cdot\,\zeta(2d_{2i-1}+2d_{2i}+2)\Big)
=\\=
\sum_{\boldsymbol{D}\in\Pi(3g-3-k,2k)}
c_{g-k,k}(\boldsymbol{D})\cdot
\frac{(6(g-k)+4k-5)!}{(g-k)!\cdot(3(g-k)-3+2k)!}
\cdot\frac{2^{3(g-k)-6+5k}}{3^{(g-k)}}
\cdot\\
\cdot\prod_{j=1}^k \frac{\zeta(2D_j+2)}{2D_j+2}\,.
\end{multline*}
which allows to rewrite~\eqref{eq:Vol:Gamma:k:g:init} as
\begin{multline*}
\Vol(\Petal_k(g))
=\\=
\left(\frac{(4g-4)!}{(6g-7)!}
\cdot 2^{g+1}
\cdot\frac{1}{k!}\right)
\cdot
\left(\frac{(6g-2k-5)!}{(g-k)!\cdot(3g-3-k)!}
\cdot\frac{2^{3g-6+2k}}{3^{(g-k)}}\right)\cdot
\\
\sum_{\boldsymbol{D}\in\Pi(3g-3-k,2k)}
\frac{c_{g-k,k}(\boldsymbol{D})}{2^k}
\cdot\prod_{j=1}^k \frac{\zeta(2D_j+2)}{D_j+1}\,.
\end{multline*}

Rearranging factors with factorials and powers of $2$ and
$3$, using~\eqref{eq:c:max} and~\eqref{eq:c:min} for the
bounds of $c_{g-k,k}(\boldsymbol{D})$ and passing to
notation~\eqref{eq:multiple:harmonic:zeta:sum:full:expansion}
we get the following bounds:
\begin{multline}
\label{eq:Vol:Gamma:k:g:intermed}
c^{min}_{g-k,k}\le \Vol(\Petal_k(g)):
\\
\left(
\frac{(6g-2k-5)!}{(6g-7)!}
\cdot\frac{(4g-4)!}{(g-k)!\cdot(3g-3-k)!}
\cdot \frac{2^{4g-5+3k}}{3^{(g-k)}}
\cdot\frac{1}{k!}
\cdot Z_k(3g-3)
\right)
\\
\le c^{max}_{g-k,k}
\end{multline}

We have the following asymptotic equivalence
uniform for all $k\in[1;C\log(g)]$:
\begin{multline}
\label{eq:equivalence:1}
\frac{(4g-4)!}{(g-k)!\cdot(3g-3-k)!}
\sim
\frac{(4g-4)!}{(g-1)!\cdot(3g-3)!}
\cdot g^{k-1}\cdot (3g)^k
\sim\\ \sim
\sqrt{\frac{2}{\pi(3g-3)}}
\cdot\left(\frac{2^8}{3^3}\right)^{g-1}
\cdot 3^k
\cdot g^{2k+1}
\quad\text{as }g\to+\infty
\,,
\end{multline}
where in the second line we
applied asymptotic formula~\eqref{eq:binom:4g:g}
for the binomial coefficient $\binom{4(g-1)}{g-1}$.
We also have the following asymptotic equivalence
uniform for all $k\in[1;C\log(g)]$:
\begin{equation}
\label{eq:equivalence:2}
\frac{(6g-2k-5)!}{(6g-7)!}
\sim
\frac{1}{(6g)^{2k-2}}
\quad\text{as }g\to+\infty
\,.
\end{equation}
Using the two above equivalences and
collecting powers of $2$, of $3$ and of $g$, we can
rewrite~\eqref{eq:Vol:Gamma:k:g:intermed} in the following
way:
$$
c^{min}_{g-k,k}\preceq
\cfrac{\Vol(\Petal_k(g))}
{\sqrt{\frac{2}{\pi}}
\cdot\left(\frac{8}{3}\right)^{4g-4}
\cdot\frac{\sqrt{3g-3}}{2^{k-1}\cdot k!}
\cdot Z_k(3g-3)}
\preceq c^{max}_{g-k,k}\,.
$$
\end{proof}

\subsection{Volume contribution of stable graphs with single vertex}
\label{ss:sum:over:single:vertex:graphs}

\begin{CondTheorem}
\label{cond:th:sum:over:one:cylinder:graphs}
Suppose that
both Conjectures~\ref{conj:introduction:sum:of:correlators}
and~\ref{conj:uniform:bound:for:error:term}
are valid for some $C\in]\tfrac{1}{2};2[$. Then
\begin{equation}
\label{eq:sum:of:Gamma:gives:all}
\sum_{k=1}^{[C\log(g)]} \Vol(\Graph_k(g))
\sim \frac{4}{\pi}
\cdot\left(\frac{8}{3}\right)^{4g-4}
\quad\text{as }g\to+\infty
\,.
\end{equation}
\end{CondTheorem}
\begin{proof}
Applying~\eqref{eq:Vol:g:k:uniform} from
Corollary~\ref{cor:V:over:Gamma:uniform}, passing to
notations~\eqref{eq:S:m:harmonic} and
applying~\eqref{eq:volume:coefficient:Z} from Conditional
Theorem~\ref{cond:th:volume:coefficient:total:sum} we get
\begin{multline*}
\sum_{k=1}^{[C\log(g)]} \Vol(\Graph_k(g))
\sim
\sum_{k=1}^{[C\log(g)]} V_k(g)
\sim\\
\sim\sqrt{\frac{2}{\pi}}
\cdot\left(\frac{8}{3}\right)^{4g-4}
\cdot
\sqrt{3g-3}\cdot
\sum_{k=1}^{[C\log(g)]}
\frac{Z_k(3g-3)}{2^{k-1}\cdot k!}
\sim\\
\sim\sqrt{\frac{2}{\pi}}\cdot\left(\frac{8}{3}\right)^{4g-4}
\cdot\sqrt{3g-3}\cdot \VSumZ(C,3g-3)
\sim\\
\sim\sqrt{\frac{2}{\pi}}
\cdot\left(\frac{8}{3}\right)^{4g-4}
\cdot2\sqrt{\frac{2}{\pi}}
=\frac{4}{\pi}
\cdot\left(\frac{8}{3}\right)^{4g-4}
\,.
\end{multline*}
\end{proof}

\subsection{Approximation by Poisson distribution}
\label{ss:Approximation:by:Poisson}

The distribution, where
\begin{equation}
\label{eq:Poisson:def}
p(n; \lambda)=
\operatorname{Probability}(X=n)=e^{-\lambda}\cdot\frac{\lambda^n}{n!}\ \text{ for } n=0,1,2,\dots
\end{equation}
is the Poisson distribution. It is known that the expected
value and the variance of Poisson distribution are equal to
$\lambda$:
$$
\operatorname{E}(X)=\operatorname{Var}(X)=\lambda\,.
$$
and that the Poisson distribution tends to normal
distribution for large values of $\lambda$.

Recall that by $\Vol_{k\textit{-cyl}}(\cQ_g)$
we denote the contribution of square-tiled surfaced
with exactly $k$ maximal horizontal
cylinders to the Masur--Veech volume $\Vol\cQ_g$.
The relations
$$
\sum_{k=1}^{3g-3}\Vol_{k\textit{-cyl}}=\Vol\cQ_g\,,
\qquad
\sum_{k=1}^{\log(g)}\Vol(\Petal_k(g))\overset{?}{\sim}\Vol\cQ_g
$$
(where the first one is an identity and the second one is
conjectural) imply that the collections of positive numbers
$\frac{\Vol(\Petal_k(g))}{\Vol\cQ_g}$ and
$\frac{\Vol_{k\textit{-cyl}}}{\Vol\cQ_g}$ can be
interpreted as probability distributions, see
Section~\ref{ss:large:genera} for details.
It would be convenient to complete every such finite
sequence by zeroes to work only with probability distributions
indexed by indices $k\in\N$.

We study large genus asymptotics of these distributions.
Conditional
Theorem~\ref{cond:th:sum:over:one:cylinder:graphs} and
Conjecture~\ref{conj:Vol:Qg} imply that the two
distributions converge in total variation, so we can
restrict our considerations to the large genus asymptotics
of the distribution $\frac{\Vol(\Petal_k(g))}{\Vol\cQ_g}$.

We study it in two regimes. In the first regime we consider
the ``left tail'' of the distribution of a bounded length
$K$. In plain terms, we consider the large genus
asymptotics of the finite collection of numbers
$\frac{\Vol(\Petal_1(g))}{\Vol\cQ_g},\dots,
\frac{\Vol(\Petal_K(g))}{\Vol\cQ_g}$.

In the second regime we consider $k$ in the full range
and let $g$ grow. We show that conditionally to
our Conjectures both distributions are asymptotically
well-approximated by Poisson distribution with parameter
$\frac{\log(g)}{2}+const$. We evaluate the constant $const$
and show that it changes when passing from the first regime
to the second.
\smallskip

\noindent\textbf{Best approximation for the
left tail of any fixed size.~}
In the first regime we assume that $k$ belongs to the
range $k=1,\dots,K$ of any fixed size $K$. Having fixed $K$
we let $g$ grow. In this regime, we are interested only in
the leading term of asymptotics of $\Vol(\Petal_k(g))$ as
$g\to+\infty$.

From now on $\gamma= 0.5772...$ denotes the Euler--Mascheroni
constant.

\begin{CondTheorem}
Define
\begin{equation}
\label{eq:total:mass:for:tail:distribution}
V_{tail}(g)=
\frac{2\cdot e^{\frac{\gamma}{2}}}{\sqrt{\pi}}
\cdot\left(\frac{8}{3}\right)^{4g-4}\,.
\end{equation}
Conjecture~\ref{conj:sum:of:correlators:fixed:k} implies
that the tail distribution
$\frac{\Vol(\Petal_k(g))}{V_{\mathit{tail}}(g)}$ corresponding
to any fixed range $[1,K]$ of $k$ tends to the left tail of
the Poisson distribution with parameter
\begin{equation}
\label{eq:lambda:tail}
\lambda_{tail}(g)=\frac{\log(6g-6)+\gamma}{2}\,,
\end{equation}
when $g\to\infty$ in the following sense:
$$
\lim_{g\to+\infty}
\frac{\Vol(\Petal_k(g))}{V_{\mathit{tail}}(g)}
\cdot\frac{1}{p(k-1; \lambda_{\mathit{tail}}(g))}
=1
\quad\text{for }k=1,\dots,K\,.
$$
\end{CondTheorem}
Note that we have to shift the index of
$p(\ast; \lambda_{\mathit{tail}}(g))$
by $1$: our count starts from $k=1$, while the Poisson
distribution starts from
$p(0; \lambda_{\mathit{tail}}(g))$.
\begin{proof}
Conditional to
Conjecture~\ref{conj:sum:of:correlators:fixed:k},
the leading term of the volume contribution
$\Vol(\Petal_k(g))$
is given by
the last line in~\eqref{eq:Vol:g:k}
in Corollary~\ref{cor:Vol:g:k}. Thus,
$$
\frac{\Vol(\Petal_k(g))}{V_{\mathit{tail}}(g)}
\sim
\frac{1}{\sqrt{2}\cdot e^{\frac{\gamma}{2}}}
\cdot\frac{1}{\sqrt{3g-3}}
\cdot\frac{1}{(k-1)!}
\left(\frac{\log(6g-6)+ \gamma}{2}\right)^{k-1}
\!\!\!
=\frac{1}{e^{\lambda_{\mathit{tail}}}}
\cdot\frac{\lambda_{\mathit{tail}}^{k-1}}{(k-1)!}
\,.
$$
\end{proof}

\medskip

\noindent\textbf{Best approximation near maximal values.~}
Let us study now the large genus asymptotics of the
distribution $\frac{\Vol(\Petal_k(g))}{\Vol\cQ_g}$ in the
full range range of $k$. We transform finite collections of
numbers into sequences completing any finite collection by
infinite sequence of zeroes. Speaking of
\textit{convergence} of sequences we always mean the
$\ell^1$-convergence (equivalently, convergence \textit{in
total variation}).
Using technical Lemma~\ref{lm:v:p} we prove in
Lemma~\ref{lm:convergence:to:Poisson} that certain
convenient auxiliary sequence $\{v_k(g)\}_{k\in\N}$ defined
below converges to the Poisson distribution. Finally, we
prove in Conditional Theorem~\ref{th:Poisson} that the
sequence $\{v_k(g)\}$ converges to
$\{\frac{\Vol(\Petal_k(g))}{\Vol\cQ_g}\}$ and that the
latter sequence converges to the sequence
$\left\{\frac{\Vol_{k\textit{-cyl}}\cQ_g}{\Vol\cQ_g}\right\}$.

Define the following function $c(x)$ on the interval
$C\in[0;2]$:
\begin{equation}
\label{eq:c:of:C}
c(C)=
\begin{cases}
\frac{1}{C-\tfrac{1}{2}}\cdot
\left(
\frac{\log(\pi)-3\log(2)}{2}
+C\log(2)-\log\Gamma(1+C)
\right)\,,
&\text{when }C\neq\tfrac{1}{2}\,,
\\
3\log(2)+\gamma-2\,,
&\text{when }C=\tfrac{1}{2}\,.
\end{cases}
\end{equation}
Note that
\begin{multline*}
\lim_{C\to\tfrac{1}{2}} c(C)=
\left(
\frac{\log(\pi)-3\log(2)}{2}
+C\log(2)-\log\Gamma(1+C)
\right)'\Big|_{C=\tfrac{1}{2}}
=\\=
\log(2)-\psi\left(\tfrac{3}{2}\right)
=\log(2)-(-\gamma-2\log(2)+2)=3\log(2)+\gamma-2
\,,
\end{multline*}
where we used~\eqref{eq:psi:3:2} for the value
$\psi\left(\tfrac{3}{2}\right)$. Thus, the function $c$ is
continuous on the interval $[0;2]$. It is monotonously
decreasing on this interval from
$c(0)=3\log(2)-\log(\pi)\approx 0.934712$ to
$c(2)=\frac{1}{3}(\log(\pi)-\log(2))\approx 0.150528$.

\begin{Lemma}
\label{lm:v:p}
Suppose that for some $C\in]0;2[$
Conjectures~\ref{conj:introduction:sum:of:correlators}
and~\ref{conj:uniform:bound:for:error:term} are valid.
Let
\begin{equation}
\label{eq:v:k:g}
v_k(g)=
\sqrt{\frac{\pi}{8}}
\cdot\frac{1}{\sqrt{3g-3}}
\,
\cdot\frac{1}{2^{k-1}}
\,\cdot
\left(\sum_{j=0}^{k-1}
\frac{B_j}{(k-1-j)!}
\cdot\big(\log(3g-3)\big)^{k-1-j}\right)
\end{equation}
for $k=1,\dots,[C\log(g)]$ and let $v_k(g)=0$ for
$k>[C\log(g)]$.
Let
\begin{equation}
\label{eq:lambda:C}
\lambda(C)=\frac{\log(3g-3)+c(C)}{2}\,,
\end{equation}
and let $K(C,g)=[C\log(3g-3)]$.
Then
\begin{equation}
\label{eq:v:p}
\lim_{g\to+\infty}\frac{v_{K(C,g)}(g)}{p(K(C,g)-1;\lambda(C))}
=1\,,
\end{equation}
where $p(n,\lambda)$ is defined in~\eqref{eq:Poisson:def}.
\end{Lemma}
\begin{proof}
Plugging the definition~\eqref{eq:lambda:C}
of $\lambda(C)$
into the definition~\eqref{eq:Poisson:def}
of $p(K-1,\lambda)$ we get
$$
p(K-1;\lambda)
=\frac{e^{-\tfrac{c}{2}}}{\sqrt{3g-3}}
\cdot\frac{(\log(3g-3)+c)^{K-1}}{(K-1)!\cdot 2^{K-1}}\,,
$$
where $K=K(C,g)$ and $c=c(C)$.
Thus, proving~\eqref{eq:v:p} is equivalent
to proving the following relation:
$$
\lim_{g\to+\infty}
\frac{(K-1)!}{(\log(3g-3)+c)^{K-1}}
\cdot
\left(\sum_{j=0}^{K-1}
\frac{B_j}{(K-1-j)!}
\cdot\big(\log(3g-3)\big)^{K-1-j}\right)
\overset{?}{=}
\frac{e^{-\tfrac{c}{2}}}{\sqrt{\tfrac{\pi}{8}}}\,,
$$
where, for brevity, we denote $K=K(C,g)$ and $c=c(C)$.
We can rewrite the expression in the left-hand side
of the above relation as
\begin{multline*}
B_0\cdot \left(\frac{\log(3g-3)}{\log(3g-3)+c}\right)^{K-1}
+B_1\cdot\frac{(K-1)}{\big(\log(3g-3)+c\big)}
\cdot\left(\frac{\log(3g-3)}{\log(3g-3)+c}\right)^{K-2}
+\\+
B_2\frac{(K-1)(K-2)}{\big(\log(3g-3)+c\big)^2}
\left(\frac{\log(3g-3)}{\log(3g-3)+c}\right)^{K-3}
\!\!\!\!+\dots
+B_{K-1}\frac{(K-1)!}{\big(\log(3g-3)+c\big)^{K-1}}
\,.
\end{multline*}
We have seen that $c(C)\in]0;1[$ for any $C\in]0;2[$.
Thus, for any fixed $n$ we get
$$
\left(\frac{K-n}{\log(3g-3)+c}\right)\sim C
\quad\text{as }g\to+\infty\,,
$$
and
$$
\left(\frac{\log(3g-3)}{\log(3g-3)+c}\right)^{K-n}
\sim
\left(1-\frac{c}{\log(3g-3)}\right)^{[C\log(3g-3)]}
\sim
e^{-c\cdot C}
\quad\text{as }g\to+\infty
\,.
$$
Hence, for any fixed $n\in\N$,
we get the following asymptotic expression
for the first $n+1$ terms of the above sum:
\begin{multline*}
B_0\cdot \left(\frac{\log(3g-3)}{\log(3g-3)+c}\right)^{K-1}
+B_1\cdot\frac{(K-1)}{\big(\log(3g-3)+c\big)}
\cdot\left(\frac{\log(3g-3)}{\log(3g-3)+c}\right)^{K-2}
+\dots\\+\dots
+B_n\cdot\frac{(K-1)\dots(K-n)}{\big(\log(3g-3)+c\big)^n}
\cdot\left(\frac{\log(3g-3)}{\log(3g-3)+c}\right)^{K-n+1}
\sim\\\sim
e^{-c\cdot C}
\cdot\left(
B_0+B_1\cdot C+ B_2\cdot C^2+\dots+B_n\cdot C^n
\right)
\quad\text{as }g\to+\infty
\,.
\end{multline*}

The absolute value of the remaining part of the sum
is trivially bounded from above by
$$
\sum_{j=n+1}^{+\infty} |B_j|\cdot C^j\,.
$$
By~\eqref{eq:A:B:j:tend:fast:to:zero} the latter expression
rapidly converges to $0$ as $n$ grows.
Hence the desired relation~\eqref{eq:v:p}
is equivalent to the following relation
\begin{equation*}
\frac{e^{-\tfrac{c}{2}}}{\sqrt{\tfrac{\pi}{8}}}
\overset{?}{=}
e^{-c\cdot C}
\cdot
\sum_{j=0}^{+\infty} B_j\cdot C^j\,.
\end{equation*}

Applying~\eqref{eq:sum:B:j:x:power:j}
for the sum on the right-hand side we
can rewrite the latter equation as:
$$
e^{c\cdot(C-\tfrac{1}{2})}
\overset{?}{=}
\sqrt{\tfrac{\pi}{8}}
\cdot\frac{2^C}{\Gamma(1+C)}\,.
$$

Note that
$$
\Gamma\left(1+\tfrac{1}{2}\right)=\frac{\sqrt{\pi}}{2}
$$
so for $C=\tfrac{1}{2}$ the relation is satisfied. Suppose
now that $C\neq\tfrac{1}{2}$. The above equation implicitly
defines $c$ as a function of $C$. Taking logarithms of both
sides and dividing the resulting expression by
$(C-\frac{1}{2})$ we get exactly the expression for $c(C)$
from~\eqref{eq:c:of:C}. Lemma~\ref{lm:v:p} is proved.
\end{proof}

\begin{Lemma}
\label{lm:convergence:to:Poisson}
Suppose that both
Conjectures~\ref{conj:introduction:sum:of:correlators}
and~\ref{conj:uniform:bound:for:error:term} are valid for
some $C$ in the interval $]\frac{1}{2};2[$. Then, the
sequence $\{v_k(g)\}_{k\in\N}$
converges in total variation to the Poisson
distribution $\{p(k-1,\lambda)\}_{k\in\N}$ with parameter
\begin{equation*}
\lambda(g)=\frac{\log(6g-6)+\gamma}{2}+(\log2-1)\,,
\end{equation*}
when $g\to\infty$.
\end{Lemma}
\begin{proof}
Choose any $\epsilon$ satisfying $0<\epsilon\ll\frac{1}{2}$
and any $M\gg 1$. Consider the following subsets
of $\N$:
\begin{align*}
I_1&=\{1,2,\dots,[\epsilon\log(g)]\}
\,,\\
I_2&=\left\{[\epsilon\log(g)]+1,[\epsilon\log(g)]+2,
\dots,[\lambda(g)-M\sqrt{\log(g)}]\right\}
\,,\\
I_3&=\left\{[\lambda(g)-M\sqrt{\log(g)}],
\dots, [\lambda(g)+M\sqrt{\log(g)}]\right\}
\,,\\
I_4&=\left\{[\lambda(g)+M\sqrt{\log(g)}]+1,
[\lambda(g)+M\sqrt{\log(g)}]+2,\dots,[C\log(g)]\right\}
\,,\\
I_5&=\big\{[C\log(g)]+1,[C\log(g)]+2,\dots\big\}
\,.
\end{align*}
We assume that $g$ is sufficiently large so that
all these intervals are nonempty.
We bound the sum $\sum|v_k(g)-p(k-1,\lambda(g))|$
from above separately on each of these intervals.

The contributions of the
left tails of both sequences tend to zero
as $g\to+\infty$:
$$
\lim_{g\to+\infty} \sum_{k\in I_1} v_k
=
\lim_{g\to+\infty} \sum_{k\in I_1} p(k-1,\lambda(g))
=0\,,
$$
so
$$
\lim_{g\to+\infty} \sum_{k\in I_1} |v_k-p(k-1,\lambda(g))|
= 0\,.
$$
Choosing $M$ sufficiently large, we can make the
contribution
$$
\lim_{g\to+\infty} \sum_{k\in I_2\sqcup I_4} p(k-1,\lambda(g))
$$
arbitrary small. It follows from Lemma~\ref{lm:v:p}
that for all sufficiently large $g$ we have
$$
v_k(g) \le T\cdot p(k-1,\lambda(g))
\quad\text{for }k\in I_2\sqcup I_4\,,
$$
with universal constant $T$ depending only on $\epsilon$.
Thus,
choosing $M$ sufficiently large, we can make the
variation
$$
\lim_{g\to+\infty} \sum_{k\in I_2\sqcup I_4}
|v_k(g)-p(k-1,\lambda(g))|
$$
on these two intervals arbitrary small.

For any fixed $M$ the ratio
$\frac{v_k(g)}{p(k-1,\lambda(g))}$
tends to $1$ uniformly for all $k\in I_3$.
Thus, we get
$$
\lim_{g\to+\infty} \sum_{k\in I_3}
|v_k(g)-p(k-1,\lambda(g))| =0\,.
$$
Finally
$$
\lim_{g\to+\infty} \sum_{k\in I_5}
|v_k(g)-p(k-1,\lambda(g))|
=\lim_{g\to+\infty} \sum_{k\in I_5}
|p(k-1,\lambda(g))|
=0\,.
$$
which completes the proof of
Lemma~\ref{lm:convergence:to:Poisson}.
\end{proof}

\begin{CondTheorem}
\label{th:Poisson}
Conjectures~\ref{conj:Vol:Qg},
\ref{conj:introduction:sum:of:correlators},
\ref{conj:uniform:bound:for:error:term}
together imply
that the distribution
$$
\frac{\Vol_{1\textit{-cyl}}\cQ_g}{\Vol\cQ_g}, \dots,
\frac{\Vol_{(3g-3)\textit{-cyl}}\cQ_g}{\Vol\cQ_g}
$$
tends to Poisson distribution $p(k-1,\lambda(g))$
with parameter
\begin{equation}
\label{eq:lambda:bis}
\lambda(g)=\frac{\log(6g-6)+\gamma}{2}+(\log2-1)\,,
\end{equation}
when $g\to\infty$. Note that we have to shift indices by
$1$: our count starts from $k=1$, while the Poisson
distribution starts from $n=0$.
\end{CondTheorem}
\begin{proof}
Let
$$
V(g)=\frac{4}{\pi}
\cdot\left(\frac{8}{3}\right)^{4g-4}\,.
$$
By Conjecture~\ref{conj:Vol:Qg} and
by~\eqref{eq:sum:of:Gamma:gives:all} in Conditional
Theorem~\ref{cond:th:sum:over:one:cylinder:graphs} we have
$$
\Vol\cQ_g\sim V(g)
\sim\sum_{k=1}^{[C\log(g)]} \Vol(\Graph_k(g))
\quad \text{ as } g\to+\infty\,.
$$
Since $\Vol_{k\textit{-cyl}}\cQ_g\ge \Vol(\Petal_k(g))\ge
0$ for all $k$, this implies $\ell^1$-convergence of
the sequences
\begin{multline}
\label{eq:cyl:2:Gamma}
\left(
\frac{\Vol_{1\textit{-cyl}}}{\Vol\cQ_g},
\frac{\Vol_{2\textit{-cyl}}}{\Vol\cQ_g},
\dots,
\frac{\Vol_{(3g-3)\textit{-cyl}}}{\Vol\cQ_g},
0,0,\dots
\right)
\to \\ \to
\left(
\frac{\Vol(\Petal_1(g))}{V(g)},
\frac{\Vol(\Petal_2(g))}{V(g)},
\dots,
\frac{\Vol(\Petal_{[C\log(g)]}(g))}{V(g)},
0,0,\dots
\right)
\quad\text{as } g\to+\infty
\,.
\end{multline}

Equation~\eqref{eq:Vol:g:k:uniform} from
Corollary~\ref{cor:V:over:Gamma:uniform} implies uniform
convergence and, hence, $\ell^1$-convergence of the
following sequences:
\begin{multline}
\label{eq:Gamma:2:V}
\left(
\frac{\Vol(\Petal_1(g))}{V(g)},
\frac{\Vol(\Petal_2(g))}{V(g)},
\dots,
\frac{\Vol(\Petal_{[C\log(g)]}(g))}{V(g)},
0,0,\dots
\right)
\to \\ \to
\left(
\frac{V_1(g)}{V(g)},
\frac{V_2(g)}{V(g)},
\dots,
\frac{V_{[C\log(g)]}(g)}{V(g)},
0,0,\dots
\right)
\quad\text{as } g\to+\infty
\,.
\end{multline}
where $V_k(g)$ is defined in~\eqref{eq:V:k:g:with:epsilon}.

Equation~\eqref{eq:max:epsilon:Z} from
Conjecture~\ref{conj:uniform:bound:for:error:term}
claims $\ell^1$-convergence
\begin{equation}
\label{eq:Gamma:2:v}
\left\{\frac{\Vol(\Petal_k(g))}{V}\right\}_{k\in\N}
\to \{v_k(g)\}_{k\in\N}
\quad\text{ as }g\to+\infty\,.
\end{equation}
Convergence of $\{v_k(g)\}$ to the Poisson distribution
with parameter $\lambda(g)$ is proved in
Lemma~\ref{lm:convergence:to:Poisson}.
\end{proof}

\begin{Remark}
Note that the parameter $\lambda_{tail}$ which provides
the best approximation by the Poisson distribution
for the left tail of the the distribution of
any fixed finite size differs from
the parameter $\lambda$ which provides
the best approximation by the Poisson distribution
globally for the entire distribution
$\frac{\Vol_{k\textit{-cyl}}\cQ_g}{\Vol\cQ_g}$, where
$k=1,\dots,3g-3$. Comparing~\eqref{eq:lambda:tail}
and~\eqref{eq:lambda:bis} we get
$$
\lambda(g)=\lambda_{tail}(g)+(\log2-1)\approx\lambda_{tail}(g)-0.306853\,.
$$
\end{Remark}

\begin{figure}[htb]
\includegraphics{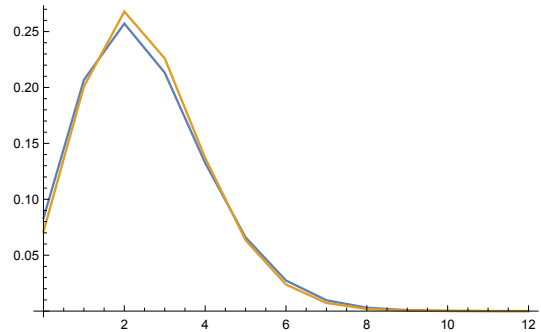}
\vspace*{120pt}
\caption{
\label{fig:poisson:versus:experimental}
Experimental distribution of contribution of $k$-cylinder
surfaces (in orange)
and the Poisson distribution
with $\lambda$ as in~\eqref{eq:lambda:bis} (in blue)}
\end{figure}

Figure~\ref{fig:poisson:versus:experimental} presents the
graph of the experimental statistics of frequencies of
$k$-cylinder surfaces in $g=26$
and the Poisson distribution with parameter $\lambda$
evaluated by~\eqref{eq:lambda:bis}. We joined the
neighboring points of the discrete graph by segments for
better visibility. The experimental statistics has slightly
higher maximum compared to Poisson distribution and its
left slope is shifted a little bit to the right compared to
the Poisson distribution. Already for $g=26$ we observe
that we would get slightly better approximation of the left
tail of the experimental statistics by Poisson distribution
choosing $\lambda(26)$ slighly larger. However, the optimal
approximation of the experimental data by all possible
Poisson distribution gives
$$
\lambda_{\textit{optimal}}(26)\approx 2.528\,,
$$
so the parameter $\lambda(26)\approx 2.487$
calculated using~\eqref{eq:lambda:bis}
and used for the graph in
Figure~\ref{fig:poisson:versus:experimental}
provides much better fit than
$\lambda_{tail}(26)\approx 2.794$
given by~\eqref{eq:lambda:tail}.

\subsection{Numerical evidence}

We have two independent sources of numerical evidence for
our Conjectures. The first source is exact computation of
all $\Vol\Petal_k(g)$ for $g\le 9$. Table~\ref{tab:Vol:k:g}
provides the rational coefficients $\frac{p}{q}$ in the
contribution $\frac{p}{q}\cdot\pi^{6g-6}$ of
$\Vol(\Graph_k(g))$ to $\Vol\cQ(1^{4g-4})$. We have all
exact values of analogous quantities for $g\le 9$.
\begin{table}[bht]
\small
$$
\hspace*{-20pt}
{
\begin{array}{|c|c|c|c|c|}
\hline &&&& \\ [-\halfbls]
g&k=1&k=2&k=3&k=4\\
&&&& \\ [-\halfbls] \hline &&&&\\ [-\halfbls]
2& \frac{16}{945}& \frac{8}{225}&&\\
&&&& \\ [-\halfbls] \hline &&&&\\ [-\halfbls]
3& \frac{204536}{273648375} & \frac{2206912}{1620840375} & \frac{2704}{3189375}&\\
&&&& \\ [-\halfbls] \hline &&&&\\ [-\halfbls]
4&
\frac{80320477}{2362381544250}&
\frac{16548755563}{251883751494375}&
\frac{18410248}{368225463375}&
\frac{16128416}{987779417625}\\
&&&& \\ [-\halfbls] \hline &&&&\\ [-\halfbls]
5&
\frac{10303583454872}{6451867979907013125}&
\frac{19854998108336976488}{6136611136849420367765625}&
\frac{2276745597432209408}{827792583677569525640625}&
\frac{1412757290717388688}{1158909617148597335896875}
\\
&&&& \\ \hline
\end{array}
}
$$
\caption{
\label{tab:Vol:k:g}
Exact values of $\cfrac{\Vol(\Petal_k(g))}{\pi^{6g-6}}$}
\end{table}

    %

We have also explicitly computed $\Vol\cQ_g$ for
$g=2,\dots,6$, see Table~\ref{tab:Vol:Q:g:n}.
Table~\ref{tab:Vol:from:one:cyl:and:Vol:Q:g} gives the
following approximate values of the
ratios
$\cfrac{\sum_{k=1}^g\Vol(\Petal_k(g))}{\Vol\cQ(1^{4g-4})}$
for these genera:
\begin{table}[hbt]
$$
\begin{array}{|c|c|c|c|c|}
\hline
&&&& \\ [-\halfbls]
g=2 & g=3 & g=4 & g=5 & g=6\\
&&&& \\ [-\halfbls] \hline &&&&\\ [-\halfbls]
0.787 & 0.855 & 0.909 & 0.940 & 0.959\\
[-\halfbls]
&&&& \\
\hline
\end{array}
$$
\caption{
\label{tab:Vol:from:one:cyl:and:Vol:Q:g}
Approximate values of $\cfrac{\sum_{k=1}^g\Vol(\Petal_k(g))}{\Vol\cQ(1^{4g-4})}$.}
\end{table}

One can argue that genus $6$ is still a bit far from
infinity. However, $\dim_{\C}\cQ_6=30$ and in the case of
Abelian differentials, the Masur--Veech volumes of all
strata are already quite close to their asymptotic values
for this range of dimensions.

Another source of numerical evidence is the following
experimental data. We can quite efficiently collect
experimental statistics
$\cfrac{\Vol_{k\textit{-cyl}}\cQ_g}{\Vol\cQ_g}$ of the
number $k=1,\dots,3g-3$ of maximal horizontal cylinders in
a ``random'' square-tiled surface in $\cQ_g$. We compare it
to statistics $\cfrac{v_k(g)}{V(g)}$. The resulting
normalized probability distributions for $g=26$ are
presented at
Figure~\ref{fig:frequencies:versus:single:vertex:formula}.

Recall that for $k=1$ and for sufficiently large genus,
$\Vol\Petal_1(g)$ gives practically all contribution of
$1$-cylinder square-tiled surfaces to $\Vol\cQ_g$, that is
$$
\frac{\Vol\Petal_1(g)}{\Vol_{1\textit{-cyl}}\cQ_{g}}
\sim 1-\sqrt{\frac{2}{3\pi g}}\cdot\frac{1}{4^g}
\quad\text{as }g\to+\infty\,.
$$
On the other hand, for $g=26$ the experimental value
of the first term of statistics is still reasonably large,
namely
$$
\frac{\Vol_{1\textit{-cyl}}\cQ_{26}}{\Vol\cQ_{26}}
\approx 0.071\,.
$$
This implies that the two distributions are still
reasonably close before normalizing the total mass to $1$.

\begin{figure}[htb]
\includegraphics{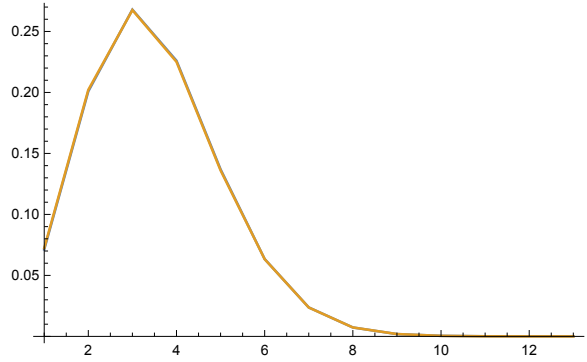}
\vspace*{130pt}
\caption{
\label{fig:frequencies:versus:single:vertex:formula}
Experimental distribution of frequencies of $k$-cylinder
square-tiled surfaces and the normalized distribution given
by expression~\eqref{eq:Vol:g:k} for $\Vol\Petal_k(g)$
in genus $g=26$}.
\end{figure}

The experimental numerical data for
frequencies of $k$-cylinder
square-tiled surfaces and the normalized distribution given
by expression~\eqref{eq:Vol:g:k} for $\Vol\Petal_k(g)$
given below for genus $g=26$,
\begin{align*}
0.0713,
0.2009,
0.2679,
0.2260,
0.1369,
0.0634,
0.0237,
0.0073,\dots
\\
0.0724,
0.2022,
0.2675,
0.2251,
0.1361,
0.0633,
0.0237,
0.0073,
\dots
\end{align*}
are already, basically, indistinguishable graphically,
see Figure~\ref{fig:frequencies:versus:single:vertex:formula}
(where we joined the neighboring points of the discrete
graph by segments for better visibility).


\end{document}